\documentclass[11pt]{article}
\usepackage[margin = 1in]{geometry}
\let\Horig\H
\usepackage[affil-it]{authblk}
\usepackage{lipsum}
\usepackage{amsfonts}
\usepackage{graphicx}
\graphicspath{ {images/} }
\usepackage{epstopdf}
\usepackage{dsfont}
\usepackage{mathtools}
\usepackage{empheq}
\usepackage{amsfonts}
\usepackage{amsgen,amsthm,amsmath,amstext,amsbsy,amsopn,amssymb,subfigure,stmaryrd,mathabx,mathrsfs}
\usepackage{comment}
\usepackage[ruled, vlined, lined, commentsnumbered]{algorithm2e}
\usepackage[font=footnotesize,labelfont=bf]{caption}

\usepackage{blkarray}

\usepackage{url}
\usepackage{longtable}
\usepackage{mathtools}
\usepackage{multirow}
\usepackage{array}

\usepackage{enumitem}
\usepackage{hyperref}
\usepackage{natbib}
\usepackage{booktabs}
\usepackage{algpseudocode}
\usepackage{xcolor}
\usepackage[utf8]{inputenc} 
\usepackage[T1]{fontenc}

\ifpdf
  \DeclareGraphicsExtensions{.eps,.pdf,.png,.jpg}
\else
  \DeclareGraphicsExtensions{.eps}
\fi

\newcommand{\cH}{\mathcal{H}}

\newcommand{\cE}{\mathcal{E}}

\newcommand{\cI}{\mathcal{I}}

\newcommand{\bbN}{\mathbb{N}}
\newcommand{\bbP}{\mathbb{P}}
\newcommand{\bbE}{\mathbb{E}}
\newcommand{\bbR}{\mathbb{R}}

\renewcommand{\exp}{{\rm{exp}}}
\newcommand{\TV}{{\sf TV}}

\renewcommand{\H}{{\rm{H}}}

\newcommand{\argmin}{\mathop{\rm arg\min}}

\newcommand{\indi}{{\mathds{1}}}

\newcommand{\CI}{{\textnormal{CI}}}

\newcommand{\wh}{\widehat}
\newcommand{\wt}{\widetilde}

\newcommand{\iprod}[2]{\left \langle #1, #2 \right\rangle}

\newtheorem{Theorem}{Theorem}

\newtheorem{Lemma}{Lemma}
\newtheorem{Remark}{Remark}

\newtheorem{Proposition}{Proposition}

\DeclareMathAlphabet\mathbfcal{OMS}{cmsy}{b}{n}

\hypersetup{
    colorlinks=true,
    linkcolor=blue,
    filecolor=blue,      
    urlcolor=cyan,
    citecolor=blue
}

\makeatletter
\newcommand*{\rom}[1]{\expandafter\@slowromancap\romannumeral #1@}
\makeatother

\allowdisplaybreaks

\begin{document}
	\title{Robust Confidence Intervals for a Binomial Proportion: Local Optimality and Adaptivity}
	
		\author[1]{Minjun Cho}
\author[2]{Yuetian Luo}
\author[1]{Chao Gao\thanks{The research of CG is supported in part by NSF Grants ECCS-2216912 and DMS-2310769, NSF Career Award DMS-1847590, and an Alfred Sloan fellowship.}}
\affil[1]{
University of Chicago
}
\affil[2]{
Rutgers University
}
	\date{}
	\maketitle

\begin{abstract}
This paper revisits the classical problem of interval estimation of a binomial proportion under Huber contamination. Our main result derives the rate of optimal interval length when the contamination proportion is unknown under a local minimax framework, where the performance of an interval is evaluated at each point in the parameter space. By comparing the rate with the optimal length of a confidence interval that is allowed to use the knowledge of contamination proportion, we characterize the exact adaptation cost due to the ignorance of data quality. Our construction of the confidence interval to achieve local length optimality builds on robust hypothesis testing with a new monotonization step, which guarantees valid coverage, boundary-respecting intervals, and an efficient algorithm for computing the endpoints. The general strategy of interval construction can be applied beyond the binomial setting, and leads to optimal interval estimation for Poisson data with contamination as well. We also investigate a closely related Erd\Horig{o}s--R\'{e}nyi model with node contamination. Though its optimal rate of parameter estimation agrees with that of the binomial setting, we show that adaptation to unknown contamination proportion is provably impossible for interval estimation in that setting.
\end{abstract}


\begin{sloppypar}
\section{Introduction} \label{sec: intro}

Interval estimation with i.i.d. samples from $\text{Binomial}(m,p)$ is arguably one of the most fundamental problems in statistics. It can be solved either using Gaussian approximation \citep{wilson1927probable,clopper1934use,vollset1993confidence,agresti1998approximate,brown2001interval,brown2002confidence} or by various concentration inequalities \citep{bernstein1924modification,hoeffding1963probability,arratia1989tutorial}. This paper revisits this classical problem with the presence of outliers. Given i.i.d. observations
\begin{equation}
X_1, \ldots, X_n \overset{i.i.d.}\sim P_{\epsilon,p,Q}=(1-\epsilon)\text{Binomial}(m,p)+\epsilon Q,\label{eq:Binomial}
\end{equation}
our goal is to construct a robust confidence interval that contains the model parameter $p$ with probability at least $1-\alpha$. The setting (\ref{eq:Binomial}) is known as Huber's contamination model \citep{huber1964robust}, where no assumptions are imposed on the contamination distribution $Q$. Roughly speaking, an $\epsilon$ fraction of the samples can take arbitrary values.

Unlike the classical setting without outliers, the difficulty of the problem under (\ref{eq:Binomial}) critically depends on whether the contamination proportion $\epsilon$ is known. With the knowledge of $\epsilon$, a valid confidence interval can be directly constructed from a high-probability estimation error bound. To be specific, one can first construct a robust estimator $\wh{p}$ that achieves the following local minimax error rate,\footnote{To the best of our knowledge, the minimax rate of estimating $p$ under (\ref{eq:Binomial}) was previously unknown. We characterize the rate in Theorem \ref{thm:est}.}
\begin{equation}
|\wh{p}-p|=O_{\mathbb{P}}\left(\sqrt{\frac{p(1-p)}{m}}\left(\frac{1}{\sqrt{n}}+\epsilon\right)+\frac{1}{m}\left(\frac{1}{n}+\epsilon\right)\right). \label{eq:CI<-er}
\end{equation}
Then, a confidence interval can be obtained by directly inverting the error bound (\ref{eq:CI<-er}).

On the other hand, when $\epsilon$ is unknown, there is no way to turn the estimation error bound (\ref{eq:CI<-er}) into a valid confidence interval. {What is even worse is that because we do not have any assumption on the contamination distribution $Q$, deriving an asymptotic distribution of $\wh{p}$ under \eqref{eq:Binomial} would be impossible. Therefore, subsampling based uncertainty quantification methods \citep{politis1994large,hall2013bootstrap}, such as bootstrap \citep{efron1994introduction}, would also fail in this setting. In fact, a simple example below illustrates that} confidence interval construction becomes fundamentally harder {when $\epsilon$ is unknown}. Consider the special case of $m=1$, and (\ref{eq:Binomial}) becomes
\begin{equation}
X_1, \ldots, X_n \overset{i.i.d.}\sim (1-\epsilon)\text{Bernoulli}(p)+\epsilon Q.\label{eq:bernoulli}
\end{equation}
{Suppose the true contamination proportion is $0$ but unknown, and $p = 0.1$.} It is interesting to note that
\begin{equation}
 \text{Bernoulli}(0.1)=0.9\text{Bernoulli}(0) + 0.1 \text{Bernoulli}(1). \label{eq:bern-eg}
\end{equation}
In other words, given a sequence of binary observations with roughly $90\%$ zeros and $10\%$ ones, there is no way to tell whether the sequence is generated by $\text{Bernoulli}(0.1)$ or by $\text{Bernoulli}(0)$ together with $10\%$ outliers. In this situation, if a statistician does not know in advance whether $\epsilon=0$ or $\epsilon=0.1$, the statistician will not be able to tell $p=0.1$ from $p=0$. Thus, any valid confidence interval needs to cover both $p=0.1$ and $p=0$, and its length is at least $0.1$ no matter how large the sample size $n$ is. 

This striking difference between known and unknown $\epsilon$ has been previously noted by \cite{luo2024adaptive} in the setting of the Gaussian location model
\begin{equation}
X_1, \ldots, X_n \overset{i.i.d.}\sim (1-\epsilon)N(\theta,1)+\epsilon Q.\label{eq:GLM}
\end{equation}
While the optimal length of a confidence interval for $\theta$ is of order $\frac{1}{\sqrt{n}}+\epsilon$ when $\epsilon$ is known, the best adaptive confidence interval without the knowledge of $\epsilon$ can only achieve the length
\begin{equation}
\frac{1}{\sqrt{\log n}} + \frac{1}{\sqrt{\log(1/\epsilon)}}. \label{eq:GLM-rate-un}
\end{equation}
The difference between the rates $\frac{1}{\sqrt{n}}+\epsilon$ and $\frac{1}{\sqrt{\log n}} + \frac{1}{\sqrt{\log(1/\epsilon)}}$ indicates a significant adaptation cost in the confidence interval construction.
The Bernoulli example (\ref{eq:bern-eg}) echoes this adaptation cost. In fact, the adaptation cost implied by the example (\ref{eq:bern-eg}) is even more severe, and the length lower bound $0.1$ holds for an arbitrary sample size $n$. In comparison, the Gaussian rate (\ref{eq:GLM-rate-un}) still tends to zero as $n\rightarrow \infty$ and $\epsilon\rightarrow 0$.

Despite the negative example (\ref{eq:bern-eg}) for the Bernoulli setting, there is still hope to solve the problem in a more interesting way for the binomial setting (\ref{eq:Binomial}). Indeed, the binomial distribution can be well approximated by a Gaussian as $m\rightarrow\infty$. We should, therefore, expect the possibility of constructing an adaptive confidence interval whose length scales as the rate (\ref{eq:GLM-rate-un}) multiplied by $\frac{1}{\sqrt{m}}$, {which shrinks polynomially with respect to $m$}. The main result of the paper confirms that this intuition is indeed true. We prove that the optimal length of a confidence interval with unknown $\epsilon$ under the setting (\ref{eq:Binomial}) scales as
\begin{equation}
\ell(n,\epsilon,m,p)=\left(\sqrt{\frac{p(1-p)}{m}}\left(\frac{1}{\sqrt{\log n}}+\frac{1}{\sqrt{\log(1/\epsilon)}}\right) + \frac{1}{m} \right) \wedge p\wedge (1-p) + \frac{1}{m}\left(\frac{1}{n}+\epsilon\right). \label{eq:main-rate-bmp}
\end{equation}
When $p$ is a constant bounded away from both $0$ and $1$, the rate (\ref{eq:main-rate-bmp}) can be simplified to $\frac{1}{\sqrt{m}}\left(\frac{1}{\sqrt{\log n}}+\frac{1}{\sqrt{\log(1/\epsilon)}}\right)+\frac{1}{m}$, which
 not only agrees with (\ref{eq:GLM-rate-un}) in the Gaussian setting for large $m$, but is also coherent with the impossibility example (\ref{eq:bern-eg}) in the Bernoulli setting when $m$ is small.
 {Moreover, the rate (\ref{eq:main-rate-bmp}) is derived under a {\it local optimality framework} for each individual $p\in[0,1]$. It extends the theory of \cite{luo2024adaptive} on robust confidence intervals with only worst-case optimality.} The local rate (\ref{eq:main-rate-bmp}) reflects the subtlety of the problem {--- the difficulty of interval estimation would vary across different $p$'s. This is particularly true} when the parameter $p$ is close to the boundary of $[0,1]$. In fact, the rate (\ref{eq:main-rate-bmp}) implies that adaptation to unknown $\epsilon$ is still possible even when $m=1$, i.e., the Bernoulli setting \eqref{eq:bernoulli}, as long as $p\wedge (1-p)$ tends to zero. This complements the important example (\ref{eq:bern-eg}).

According to \cite{luo2024adaptive}, an adaptive confidence interval can be constructed by inverting a family of robust testing functions. While this strategy works for the Gaussian location model (\ref{eq:GLM}), it does not guarantee to always output an interval in the binomial setting (\ref{eq:Binomial}). In this paper, {we propose a new principle, {\it monotonicity}, in designing robust tests, with which we are guaranteed to have intervals.} We will show that the monotone tests lead to an adaptive confidence interval under (\ref{eq:Binomial}) with its length shown to be of order (\ref{eq:main-rate-bmp}). This new construction also respects the boundary condition $0\leq p\leq 1$ that is not present in the previous location model (\ref{eq:GLM}), and can be discretized into an efficient algorithm that directly computes the two endpoints of the interval. To demonstrate the generality of the proposed method, we also consider Poisson data with contamination and construct an adaptive confidence interval that achieves both coverage and local length optimality under $X_1, \ldots, X_n \overset{i.i.d.}\sim (1-\epsilon)\text{Poisson}(\lambda)+\epsilon Q$.

In addition to the binomial model (\ref{eq:Binomial}), the paper also studies a closely related Erd\Horig{o}s--R\'{e}nyi model \citep{gilbert1959random,erdds1959random} with node contamination \citep{acharya2022robust}. This is an interesting random network model that allows outliers in the data. In such a setting, an $\epsilon$ fraction of network nodes are contaminated, and edges that are connected to these nodes have arbitrary connectivity. Equivalently, there is an $\epsilon$ fraction of rows and columns of the adjacency matrix that are contaminated, but the submatrix without contamination is still generated by $\text{Bernoulli}(p)$. In fact, the binomial model (\ref{eq:Binomial}) is equivalent to a setting where the adjacency matrix only has row contamination. This similarity between the binomial model and the Erd\Horig{o}s--R\'{e}nyi model with node contamination explains why the minimax rate of estimating $p$ derived by \cite{acharya2022robust} agrees with (\ref{eq:CI<-er}) when $m=n$. Somewhat surprisingly, in terms of constructing adaptive confidence intervals for $p$ when $\epsilon$ is unknown, the two models are drastically different. We show that adaptation to unknown $\epsilon$ is impossible for interval estimation under Erd\Horig{o}s--R\'{e}nyi model with node contamination; the optimal interval length cannot decrease as $\epsilon\rightarrow 0$. This is because an adaptive confidence interval there can be converted into a testing procedure that distinguishes a stochastic block model \citep{holland1983stochastic} from an Erd\Horig{o}s--R\'{e}nyi model, and will thus violate the lower bound of community detection in the literature.

\subsection{Paper Organization}

The mathematical formulation of adaptive confidence intervals under a local optimality framework is set up in Section \ref{sec:setting}. The solution to the binomial model will be given in Section \ref{sec:main}. Section \ref{sec:test} introduces a general framework of interval construction through inverting monotone tests. The same construction for the binomial model is also applied to Poisson data. The results for the Erd\Horig{o}s--R\'{e}nyi model will be given in Section \ref{sec:graph}. Finally, all technical proofs will be presented in Section \ref{sec:pfthm2} and the appendices.

\subsection{Notation}

Define $[n] = \{1,\ldots,n\}$ for any positive integer $n$. Given any two numbers $a, b \in \bbR$, let $a \wedge b := \min\{a,b\}$ and $a \vee b := \max\{a, b\}$. For any two sequences $\{a_n\}$ and $\{b_n\}$, we write $a_n \asymp b_n$ if there exist constants $c, C>0$ such that $ca_n \leq b_n\leq Ca_n$ for all $n$; $a_n \lesssim b_n$ means that $a_n \leq C b_n$ holds for some constant $C > 0$ independent of $n$. Given any real number $a \geq 0$, let $\lceil a \rceil = \min \{x \in \bbN_0: x \geq a \}$ and $\lfloor a \rfloor = \max \{x \in \bbN_0: x \leq a\}$ where $\bbN_0 = \bbN \cup \{0\}$ and $\bbN$ is the set of positive natural numbers. For an interval $B=[L,U]$, we write its length as $|B|=U - L$ when $U \geq L$ and implicitly assume it is empty when $U < L$. For any set $S$, we use $\#S$ to denote its cardinality. The notation $P^{\otimes n}$ means the product distribution of $P$ with $n$ i.i.d. copies. The total variation distance, Kullback–Leibler divergence and $\chi^2$-divergence between two distributions $P$ and $Q$ are defined by $\TV(P,Q)=\sup_B|P(B)-Q(B)|$, $D(P\|Q) = \int \log(dP/dQ) dP $ and $\chi^2(P\|Q) = \int \left( dP/dQ \right)^2 dQ - 1$, respectively. Given $n$ data points $X_1, \ldots, X_n$, we denote the empirical CDF as  
$
F_n(t) = \frac{1}{n} \sum_{i \in [n]} \indi\{X_i \leq t\}.
$ We use $\mathbb{E}$ and $\mathbb{P}$ for generic expectation and probability operators whenever the distribution is clear from the context.

\section{Problem Setting}\label{sec:setting}

\subsection{A Framework of Local Optimality}\label{sec:local}

Given i.i.d. observations
generated according to (\ref{eq:Binomial}),
our goal is to construct an adaptive robust confidence interval $\widehat{\CI}$ that contains the model parameter $p$ with probability at least $1-\alpha$. When the contamination proportion $\epsilon\in[0,\epsilon_{\max}]$ is unknown, we consider confidence intervals that satisfy the following coverage requirement,
$$\inf_{\epsilon\in[0,\epsilon_{\max}]}\inf_{p,Q}P_{\epsilon, p,Q}\left(p\in \widehat{\CI}\right)=\inf_{p,Q}P_{\epsilon_{\max}, p,Q}\left(p\in \widehat{\CI}\right)\geq 1-\alpha,$$ where the first equality is because $\{P_{\epsilon,p,Q}:Q\} \subseteq \{P_{\epsilon_{\max},p,Q}:Q\}$ for any $\epsilon \in [0, \epsilon_{\max}]$. Here $\epsilon_{\max}$ serves as a known conservative upper bound for $\epsilon$, which is assumed to be a constant throughout the paper unless otherwise stated. 
The set of all such intervals is defined by
$$\cI_{\alpha}(\epsilon_{\max}) = \left\{ \widehat{\CI} = [l(\{X_i \}_{i=1}^n), u(\{X_i \}_{i=1}^n) ]: \inf_{p,Q}P_{\epsilon_{\max}, p,Q}\left(p\in \widehat{\CI}\right) \geq 1- \alpha  \right\}.$$
We will find among $\cI_{\alpha}(\epsilon_{\max})$ a confidence interval with the smallest length. {In \cite{luo2024adaptive}, the optimal length is defined according to the worst-case performance over the class of all contamination distributions and the entire parameter space. This global notion of length optimality is suitable for location families {as the shape of the distribution there is location invariant, so is the difficulty of interval estimation. However, this attractive feature fails to hold beyond location families (an example will be provided shortly in Section \ref{sec:bernoulli}). Here, instead,} we will follow \cite{cai2013adaptive} and introduce a notion of local length optimality for the binomial model.} Given some $\epsilon\in[0,\epsilon_{\max}]$ and some $p\in[0,1]$, the locally optimal length of $\cI_{\alpha}(\epsilon_{\max})$ is defined by
\begin{equation}
r_{\alpha}(\epsilon,p,\epsilon_{\max})=\inf\left\{r\geq 0: \inf_{\widehat{\CI} \in  \cI_{\alpha}(\epsilon_{\max}) }\sup_{Q} P_{\epsilon, p, Q} \left( |\widehat{\CI}| \geq r \right) \leq \alpha\right\}. \label{eq:local-r-b}
\end{equation}
In other words, we allow the optimal length to depend on each specific parameter $p$.
Intuitively, when $p$ is close to $0$ or $1$, a shorter confidence interval is expected due to less variability of the data.

\subsection{Benchmark with Known $\epsilon$}

When $\epsilon_{\max}=\epsilon$, the quantity $r_{\alpha}(\epsilon,p,\epsilon)$ is reduced to the locally optimal length for confidence intervals that can depend on the knowledge of $\epsilon$. When $\epsilon_{\max}$ is a constant and $\epsilon=o(1)$, it is possible that $r_{\alpha}(\epsilon,p,\epsilon_{\max})$ is of greater order than $r_{\alpha}(\epsilon,p,\epsilon)$, in which case the unknown level of $\epsilon$ results in an \textit{adaptation cost}. In this paper, we will fully characterize the rates of both $r_{\alpha}(\epsilon,p,\epsilon_{\max})$ and $r_{\alpha}(\epsilon,p,\epsilon)$ and thus exactly quantify the cost of unknown $\epsilon$ in the construction of robust confidence intervals.

We first present a result on $r_{\alpha}(\epsilon,p,\epsilon)$ to benchmark the information-theoretic limit of the problem when $\epsilon$ is known.
\begin{Proposition}\label{prop:bench}
For any $\alpha \in (0,1/4)$, there exists some constant $c > 0 $ only depending on $\alpha$ such that \begin{equation*}\label{Ineq: Lower bound for known epsilon}
r_{\alpha}(\epsilon,p,\epsilon)\geq c\left[\sqrt{\frac{p(1-p)}{m}}\left(\frac{1}{\sqrt{n}}+\epsilon\right)+\frac{1}{m}\left(\frac{1}{n}+\epsilon\right)\right].
\end{equation*}
Moreover, for any $\alpha \in (0,1)$, if $\frac{\log(2/\alpha)}{n} + \epsilon$ is less than a sufficiently small constant, there is a robust confidence interval $\widehat{\CI}$ that satisfies
\begin{equation*}
		\begin{split}
			& \inf_{p, Q} P_{\epsilon, p, Q}\left( p \in\widehat{\CI} \right) \geq 1-\alpha,\\
			& \inf_{p, Q } P_{\epsilon, p, Q}\left( |\widehat{\CI}| \leq C \left[\sqrt{\frac{p(1-p)}{m}}\left(\frac{1}{\sqrt{n}}+\epsilon\right)+\frac{1}{m}\left(\frac{1}{n}+\epsilon\right)\right] \right) \geq 1-\alpha,
		\end{split}
	\end{equation*}
where $C > 0$ is some constant only depending on $\alpha$.
\end{Proposition}

The construction of $\widehat{\CI}$ that achieves the locally optimal length in Proposition \ref{prop:bench} is straightforward by considering a rate-optimal robust estimator $\wh{p}$. One can characterize the high-probability error bound of $\wh{p}$ as a function of $n,m,p$ and $\epsilon$. Since $\epsilon$ is known, the error bound can be regarded as a known function of $p$. This leads to a Wilson-type confidence interval that achieves the optimality. Details of the construction will be given in Section \ref{sec:est}.

On the other hand, the error bound of $\wh{p}$ will be unknown if one does not have the value of $\epsilon$. In this case, construction of a confidence interval using the error bound of $\wh{p}$ is infeasible. One can certainly still use the error bound with the unknown $\epsilon$ replaced by the known upper bound $\epsilon_{\max}$. However, this conservative strategy does not adapt to the level of $\epsilon$, and this paper will construct an optimal solution that is strictly better than the conservative one.

\subsection{Understanding Bernoulli}\label{sec:bernoulli}

To help readers build intuition for the general binomial setting (\ref{eq:Binomial}), we first discuss the special case of Bernoulli when $m=1$. This simple setting is already rich enough to understand the necessity of characterizing the local length optimality.

Starting with the example (\ref{eq:bern-eg}), we know that given {the true $p$ is 0.1 and the contamination proportion is unknown}, it is necessary for a confidence interval to have length at least $0.1$ with high probability, regardless of the sample size $n$. However, the example (\ref{eq:bern-eg}) is not the whole story. {Suppose the true contamination proportion is still $0$, but $p = 0$, it is not hard to convince ourselves that even if we have a contamination model such that}
\begin{equation}
\text{Bernoulli}(0)=(1-\epsilon)\text{Bernoulli}(p') + \epsilon Q, \label{eq:minjun-eg}
\end{equation}
we must have $p'=0$ and $Q=\text{Bernoulli}(0)$ in (\ref{eq:minjun-eg}) for any $\epsilon\in(0,1)$. In other words, given a sequence of all zeros, we should be very certain that the $p$ that generates the sequence has to be extremely close to $0$, even if we do not know the value of $\epsilon$. In fact, as long as we know $\epsilon<1/2$, the interval $\left[0,\frac{4\log(2/\alpha)}{n}\right]$ is guaranteed to cover any $p$ that could generate a sequence of all zeros, with probability at least $1-\alpha$.
The two examples (\ref{eq:bern-eg}) and (\ref{eq:minjun-eg}) suggest that for a constant order $\epsilon_{\max}$, we expect to have $r_{\alpha}(0,0.1,\epsilon_{\max})\asymp 1$, but $r_{\alpha}(0,0,\epsilon_{\max})\asymp \frac{1}{n}$. The local optimal length of an adaptive confidence interval critically depends on the magnitude of $p$.

With data generated according to (\ref{eq:bernoulli}), we illustrate how to construct an adaptive confidence interval with unknown $\epsilon\in[0,\epsilon_{\max}]$. Let us start with the estimator $\wh{p}=\frac{1}{n}\sum_{i=1}^n\indi\{X_i=1\}$. Its error bound under (\ref{eq:bernoulli}) is given by (\ref{eq:CI<-er}) with $m=1$. That is,
\begin{equation}
|\wh{p}-p|=O_{\mathbb{P}}\left(\sqrt{\frac{p(1-p)}{n}}+\frac{1}{n}+\epsilon\right). \label{eq:bern-phat-er}
\end{equation}
As we have already pointed out, one cannot invert this error bound into a confidence interval without the knowledge of $\epsilon$. However, it is actually possible to obtain a better one-sided error bound. Under (\ref{eq:bernoulli}), we have $\mathbb{E}\wh{p}\geq (1-\epsilon)p\geq (1-\epsilon_{\max})p$. Thus, the random variable $n\wh{p}$ is stochastically greater than $\text{Binomial}(n,(1-\epsilon_{\max})p)$, which implies
\begin{equation}
\wh{p} \geq (1-\epsilon_{\max})p - O_{\mathbb{P}}\left(\sqrt{\frac{p}{n}} + \frac{1}{n}\right).\label{eq:bern-1sd}
\end{equation}
The one-sided error bound (\ref{eq:bern-1sd}) holds for all $\epsilon\in[0,\epsilon_{\max}]$. In comparison, one can only obtain from (\ref{eq:bern-phat-er}) a worse lower bound $p-O_{\mathbb{P}}(1)$ that holds for all $\epsilon\in[0,\epsilon_{\max}]$. A similar argument also leads to
\begin{equation}
\wh{1-p} \geq (1-\epsilon_{\max})(1-p) - O_{\mathbb{P}}\left(\sqrt{\frac{1-p}{n}} +  \frac{1}{n}\right),\label{eq:bern-othsd}
\end{equation}
where $\wh{1-p}=\frac{1}{n}\sum_{i=1}^n\indi\{X_i=0\}$.
By inverting (\ref{eq:bern-1sd}) and (\ref{eq:bern-othsd}), we can obtain the following confidence interval
\begin{equation}
\left[1-C\left(\wh{1-p}+\frac{1}{n}\right),C\left(\wh{p}+\frac{1}{n}\right)\right] \bigcap [0,1] , \label{eq:sol-bern}
\end{equation}
for some constant $C>0$ that can be explicitly computed. The interval (\ref{eq:sol-bern}) covers $p$ with high probability because of (\ref{eq:bern-1sd}) and (\ref{eq:bern-othsd}). Its length is bounded above by $C\left(\wh{p}\wedge (\wh{1-p})+\frac{1}{n}\right)$, which is at most of order $p\wedge(1-p)+\frac{1}{n}+\epsilon$ using (\ref{eq:bern-phat-er}). Hence, the length of (\ref{eq:sol-bern}) matches the locally optimal rate (\ref{eq:main-rate-bmp}) with $m=1$, and decreases with $\epsilon$ even though the construction of (\ref{eq:sol-bern}) does not use the knowledge of $\epsilon$.

\section{Main Results}\label{sec:main}

In this section, we will characterize the order of $r_{\alpha}(\epsilon,p,\epsilon_{\max})$ for the binomial model and construct an adaptive confidence interval for $p$ when the contamination proportion $\epsilon$ is unknown.

\subsection{A Lower Bound for Locally Optimal Length} \label{sec:bino-test}

For any interval in $\cI_{\alpha}(\epsilon_{\max})$, we first present a lower bound for its local length at a given $p$.
\begin{Theorem}\label{thm:lower}
For any $\alpha \in (0,1/4)$, $\epsilon_{\max} \in [0, 1/2]$, and $n \geq 3$ satisfying $\epsilon_{\max} \geq \frac{2\alpha}{n}$, there exists some constant $c > 0$ only depending on $\alpha$ and $\epsilon_{\max}$, such that
\begin{equation}
r_{\alpha}(\epsilon,p,\epsilon_{\max})\geq c \ell(n,\epsilon,m,p). \label{eq:rate-main}
\end{equation}
The formula of $\ell(n,\epsilon,m,p)$ is given by (\ref{eq:main-rate-bmp}).
\end{Theorem}

Compared with the rate of $r_{\alpha}(\epsilon,p,\epsilon)$ in Proposition \ref{prop:bench}, the lower bound of $r_{\alpha}(\epsilon,p,\epsilon_{\max})$ implies an adaptation cost when $\epsilon$ is unknown. An interesting instance of (\ref{eq:rate-main}) is
$$r_{\alpha}(0,p,\epsilon_{\max}) \gtrsim \left(\sqrt{\frac{p(1-p)}{m\log n}} + \frac{1}{m}\right) \wedge p\wedge (1-p) + \frac{1}{mn}.$$
This means even when $\epsilon=0$ and there is literally no outlier in the entire data set, the adaptation cost of interval length is still necessary as long as a statistician is ignorant of the data quality.

The same problem has been considered with i.i.d. samples generated from a Gaussian location model with Huber contamination $(1-\epsilon)N(\theta,1)+\epsilon Q$. While the optimal length of a confidence interval with known $\epsilon$ is of order $\frac{1}{\sqrt{n}}+\epsilon$ \citep[Theorem 2.1]{chen2018robust}, it was shown in Theorem 1 of \cite{luo2024adaptive} that the rate deteriorates to $\frac{1}{\sqrt{\log n}}+\frac{1}{\sqrt{\log(1/\epsilon)}}$ when $\epsilon$ is unknown. Similarly, the term $\sqrt{\frac{p(1-p)}{m}}\left(\frac{1}{\sqrt{\log n}}+\frac{1}{\sqrt{\log(1/\epsilon)}}\right)$ in $\ell(n,\epsilon,m,p)$ resembles the Gaussian rate due to the fact that $\text{Binomial}(m,p)$ can be approximated by $N\left(mp, mp(1-p)\right)$ when $m$ is large. {But one nice feature in the contaminated binomial model is that the length of the interval still decreases polynomially in $m$.} For a general $m$, the function $\sqrt{\frac{p(1-p)}{m}}\left(\frac{1}{\sqrt{\log n}}+\frac{1}{\sqrt{\log(1/\epsilon)}}\right)+\frac{1}{m}$ interpolates between the Bernoulli rate and the Gaussian rate as $m$ ranges from $1$ to $\infty$. The additional minimum with $p\wedge (1-p)$ reflects the boundary effect of the problem as illustrated in Section \ref{sec:bernoulli}. Finally, the last term $\frac{1}{m}\left(\frac{1}{n}+\epsilon\right)$ is intrinsic to the confidence interval construction, and is needed even when $\epsilon$ is known (see Proposition~\ref{prop:bench}).

\subsection{A Locally Optimal Adaptive Confidence Interval}

We will introduce an algorithm to match the lower bound rate of Theorem \ref{thm:lower}. We first need a few quantities.
For each $p,\epsilon\in[0,1]$, define
\begin{subequations}\label{def:t-tau-r}
	\begin{align}
		\overline{t}(p,\epsilon) &= \left\{ \begin{array}{ll}
	 p - \min \left\{\frac{p (1 - p)}{2}, \frac{1}{8} \sqrt{\frac{p (1 - p) \log\left(\left(\epsilon + (\log(24/\alpha)/(2n))^{1/2}\right)^{-1}\right)}{m}} \right\} &  p \in [0, 1-1/m]\\
	 	 1-1/m & p \in (1-1/m, 1],
\end{array}  \right. \label{def:t}\\
\overline{r}(p,\epsilon) &= \left\{ \begin{array}{ll}
\frac{1}{2m} &  p = 0\\
	\frac{p(1-p)}{4m(p - \overline{t}(p,\epsilon) )}  &  p \in (0, 1-1/m]\\
	 (1- 1/(6e) )(1-p) &  p \in (1-1/m, 1],
\end{array}  \right. \label{def:r} \\
\overline{\tau}(p,\epsilon) &= \left\{ \begin{array}{ll}
	\frac{11}{10} P_{X\sim \text{Binomial}(m,p + \overline{r}(p, \epsilon))} \left( X \leq m \overline{t}(p,\epsilon) \right)  &   p \in [0, 1-1/m]\\
	  \frac{1}{2}(1 - p^m) - \frac{3 \log(24/\alpha) }{n}  &  p \in (1-1/m, 1].
\end{array}  \right. \label{def:tau}
	\end{align}
\end{subequations}
We also define $\underline{t}(p,\epsilon)=1-\overline{t}(1-p,\epsilon)$, $\underline{r}(p,\epsilon)=\overline{r}(1-p,\epsilon)$ and $\underline{\tau}(p,\epsilon)=\overline{\tau}(1-p,\epsilon)$.

With these quantities, the computation of the endpoints of our constructed interval is explicitly given by Algorithm \ref{alg:CI-endpoints}.
\begin{algorithm}[h]
\DontPrintSemicolon
\SetKwInOut{Input}{Input}\SetKwInOut{Output}{Output}
\Input{$\{X_i\}_{i=1}^n$}
\Output{$\wh{p}_{\rm{left}}$, $\wh{p}_{\rm{right}}$} 
\nl Set $p \leftarrow 1$, $S_m\leftarrow\left\{0,1/m, 2/m,\cdots,1\right\}$ and
\begin{equation*} \label{def:epsilon-set}
    \mathcal{E} \leftarrow \left\{ \frac{ 2^k \log (24 / \alpha)}{n}: k = 0,1, \ldots, \left\lfloor \log_2 \left( \frac{n \epsilon_{\max}}{\log(24/\alpha)} \right) \right\rfloor \right\} \cup \{\epsilon_{\max}\},
\end{equation*}

\nl Set
\begin{eqnarray}
\label{eq:end-start-left} \wh{p}_{\rm left} &\leftarrow& \left[1-\left(\frac{2}{n}\sum_{i=1}^n\indi\{X_i\leq m-1\}+\frac{6\log(24/\alpha)}{n}\right)\wedge 1\right]^{1/m}\vee\left(1-\frac{1}{m}\right),\\
\label{eq:end-start-right} \wh{p}_{\rm right} &\leftarrow& \left(1-\left[1-\left(\frac{2}{n}\sum_{i=1}^n\indi\{X_i\geq 1\}+\frac{6\log(24/\alpha)}{n}\right)\wedge 1\right]^{1/m}\right)\wedge\frac{1}{m}.
\end{eqnarray}

\nl For each $j\in[m]$, set $p \leftarrow p-1/m$,\;

\qquad For each $\epsilon\in\mathcal{E}$,\;

\qquad\qquad For each $q\in [0,p+1/m] \cap (S_m\setminus \{1\})$, compute 
\begin{equation}
\phi_{q,\epsilon}^+=\indi\left\{\frac{1}{n}\sum_{i=1}^n\indi\{X_i\leq m \overline{t}(q,\epsilon)\}< \overline{\tau}(q,\epsilon)\right\}.\label{eq:phi+}
\end{equation}
\qquad If $\max_{\epsilon\in\mathcal{E}}\min_{q\in [0,p+1/m] \cap (S_m\setminus \{1\})}\phi_{q,\epsilon}^+=0$, set $\wh{p}_{\rm{left}} \leftarrow p$.

\nl For each $j\in[m]$, set $p \leftarrow p+1/m$,\;

\qquad For each $\epsilon\in\mathcal{E}$,\;

\qquad\qquad For each $q\in [p-1/m,1] \cap (S_m \setminus \{0\} )$, compute 
\begin{equation}
\phi_{q,\epsilon}^-=\indi\left\{\frac{1}{n}\sum_{i=1}^n\indi\{X_i\geq m \underline{t}(q,\epsilon)\}< \underline{\tau}(q,\epsilon)\right\}.\label{eq:phi-}
\end{equation}
\qquad If $\max_{\epsilon\in\mathcal{E}}\min_{q\in [p-1/m,1] \cap (S_m \setminus \{0\} )}\phi_{q,\epsilon}^-=0$, set $\wh{p}_{\rm{right}} \leftarrow p$.

\caption{Computing Endpoints of Robust CI}
\label{alg:CI-endpoints}
\end{algorithm}
The output of Algorithm \ref{alg:CI-endpoints} can be concisely written as 
\begin{eqnarray}
\label{eq:left}\wh{p}_{\rm{left}} &=& \inf\left\{p\in S_m\cup\left[1-\frac{1}{m},1\right]: \max_{\epsilon\in\mathcal{E}}\left(\phi_{p,\epsilon}^+\wedge\min_{q\in [0,p+1/m] \cap (S_m \setminus \{1\} ) }\phi_{q,\epsilon}^+\right)=0\right\},\\
\label{eq:right}\wh{p}_{\rm{right}} &=& \sup\left\{p\in S_m\cup\left[0,\frac{1}{m}\right]: \max_{\epsilon\in\mathcal{E}}\left(\phi_{p,\epsilon}^-\wedge\min_{q\in [p-1/m,1] \cap (S_m \setminus \{0\} )}\phi_{q,\epsilon}^-\right)=0\right\},
\end{eqnarray}
where $S_m$ and $\mathcal{E}$ are discretizations of $[0,1]$ and $[0,\epsilon_{\max}]$ given in Algorithm \ref{alg:CI-endpoints},
 and the binary variables $\phi_{q,\epsilon}^+$ and $\phi_{q,\epsilon}^-$ are given by (\ref{eq:phi+}) and (\ref{eq:phi-}). At its core, Algorithm \ref{alg:CI-endpoints} computes $[\wh{p}_{\rm{left}},\wh{p}_{\rm{right}}]$ by inverting hypothesis tests. In particular, the formula (\ref{eq:left}) involves two testing functions $\min_{q\in [0,p+1/m] \cap (S_m \setminus \{1\} ) }\phi_{q,\epsilon}^+$ and $\phi_{p,\epsilon}^+$.
Intuitively, $\min_{q\in [0,p+1/m] \cap (S_m \setminus \{1\} ) }\phi_{q,\epsilon}^+$ can be viewed as a discretization of $\min_{q \in [0,p]} \phi_{q,\epsilon}^+$ when $p \in [0, 1-1/m]$. When it equals zero, it favors the null $H_0:p$ against an alternative $H_1:p+r$ for some $r \geq 0$ (more details will be given in Section \ref{sec:test}). Specifically, when there exists some $q\leq p$ such that the number of observations that $X_i\leq m\overline{t}(q,\epsilon)$ exceeds a certain threshold, this is evidence that the parameter generating the data is no greater than this $p$. For a larger $p> 1-\frac{1}{m}$, the {threshold values of $\phi_{p,\epsilon}^+$ are changed to accommodate the right boundary scenario and the test} can be directly inverted into the interval endpoint (\ref{eq:end-start-left}). Note that (\ref{eq:end-start-left}) and (\ref{eq:end-start-right}) can be regarded as extensions of the endpoints of the Bernoulli interval (\ref{eq:sol-bern}) for general $m\geq 1$. In the end, Algorithm~\ref{alg:CI-endpoints} computes the left endpoint $\wh{p}_{\rm{left}}$ (resp. right endpoint $\wh{p}_{\rm{right}}$) as the smallest (resp. largest) value such that the corresponding test is not rejected when tested against a larger (resp. smaller) alternative robustly over all levels of $\epsilon\in\mathcal{E}$.

The idea of constructing a confidence interval by inverting a testing procedure is very classical \citep{wilson1927probable}. Even when there is no outlier, inverting tests has advantages over other methods for general exponential families \citep{brown2001interval,brown2002confidence,brown2003interval}. For the purpose of constructing robust confidence intervals, the connection to robust hypothesis testing was established by \cite{luo2024adaptive} for general location families, though their results cannot be directly applied to the binomial setting (\ref{eq:Binomial}). 
Its reason and a formal connection to robust hypothesis testing in the setting of (\ref{eq:Binomial}) will be given in Section \ref{sec:test}. The guarantee for the output of Algorithm \ref{alg:CI-endpoints} is given as follows.

\begin{Theorem}\label{thm:upper}
Suppose $ \frac{\log(2/\alpha)}{n} + \epsilon_{\max}$ is less than a sufficiently small constant. Then, the interval $\widehat{\CI}$ with endpoints $\wh{p}_{\rm{left}}$ and $\wh{p}_{\rm{right}}$ computed by Algorithm \ref{alg:CI-endpoints} satisfies
\begin{equation*}
		\begin{split}
			&\inf_{ \epsilon \in [0, \epsilon_{\max}], p, Q} P_{\epsilon, p, Q}\left( p \in\widehat{\CI} \right) \geq 1-\alpha,\\
			&\inf_{\epsilon \in [0, \epsilon_{\max}], p, Q } P_{\epsilon, p, Q}\left( |\widehat{\CI}| \leq C \ell(n,\epsilon,m,p) \right) \geq 1-\alpha,
		\end{split}
	\end{equation*}
where $C>0$ is some constant only depending on $\alpha$. The formula of $\ell(n,\epsilon,m,p)$ is given by (\ref{eq:main-rate-bmp}).
\end{Theorem}

\subsection{Comparison with Estimation Error}\label{sec:est}
This section will compare confidence interval construction with estimation in the setting of (\ref{eq:Binomial}). Similar to the definition of $r_{\alpha}(\epsilon,p,\epsilon_{\max})$ and $r_{\alpha}(\epsilon,p,\epsilon)$, we first define the locally optimal estimation error. For any $\epsilon, p, q$, define
$$r_{\alpha}^{\rm est}(\epsilon, p,q)=\inf\left\{r\geq 0: \inf_{\wh{p}}\sup_{\theta\in\{p,q\},Q}P_{\epsilon,\theta,Q}\left(|\wh{p}-\theta|\geq r\right)\leq\alpha\right\}.$$
Then, the locally optimal estimation error at some given $\epsilon$ and $p$ is given by
$$r_{\alpha}^{\rm est}(\epsilon, p)=\sup_{q}r_{\alpha}^{\rm est}(\epsilon, p,q).$$
In words, $r_{\alpha}^{\rm est}(\epsilon, p)$ is the minimax estimation error at $p$ against its locally least-favorable alternative. A similar definition of local minimax risk was given by \cite{cai2015framework,duchi2016local}. The main difference in our definition is the uniformity over the contamination distribution $Q$, which corresponds to the lack of assumption on outliers. {Next, we provide the locally optimal estimation error in the setting of \eqref{eq:Binomial}.}

\begin{Theorem}\label{thm:est}
Suppose $\alpha < 1/3$. Then there exists some constant $c > 0$ only depending on $\alpha$ such that 
$$r_{\alpha}^{\rm est}(\epsilon, p)\geq c\left[\sqrt{\frac{p(1-p)}{m}}\left(\frac{1}{\sqrt{n}}+\epsilon\right)+\frac{1}{m} \left(\frac{1}{n} + \epsilon \right)\right].$$
 Moreover, for any $\alpha \in (0, 1)$, if $\frac{\log(2/\alpha)}{n} + \epsilon$ is less than a sufficiently small universal constant, there exists an adaptive estimator $\wh{p}$ that does not depend on $\epsilon$, such that
$$
\inf_{p,Q} P_{\epsilon, p, Q}\left(|\wh{p}-p|\leq C\left[\sqrt{\frac{p(1-p)}{m}}\left(\frac{1}{\sqrt{n}}+\epsilon\right)+\frac{1}{m}\left(\frac{1}{n} + \epsilon \right)\right]\right)\geq 1-\alpha,
$$
where $C > 0$ is some constant only depending on $\alpha$.
\end{Theorem}

Theorem \ref{thm:est} shows that the lower bound of $r_{\alpha}^{\rm est}(\epsilon, p)$ can be achieved by an estimator $\wh{p}$ that does not use the knowledge of $\epsilon$. This means that, contrary to the adaptation cost in confidence interval construction, rate-optimal adaptive estimation can be achieved without any cost. An optimal $\wh{p}$ can be constructed using the framework of total variation learning \citep{gao2018robust}. Its details will be given in Appendix \ref{app:proof-est}.

Given a rate-optimal estimator $\wh{p}$ and its error characterization in Theorem \ref{thm:est}, one can immediately obtain a Wilson-type confidence interval by inverting the high-probability error bound.
That is, the set
\begin{equation}
\left\{p\in[0,1]: |\wh{p}-p|\leq C\left[\sqrt{\frac{p(1-p)}{m}}\left(\frac{1}{\sqrt{n}}+\epsilon\right)+\frac{1}{m}\left(\frac{1}{n} + \epsilon \right)\right] \right\}\label{eq:bin-k-e}
\end{equation}
is an interval satisfying the coverage property.
When $\epsilon$ is known, this construction can be used in Proposition \ref{prop:bench} to achieve the lower bound of $r_{\alpha}(\epsilon,p,\epsilon)$.

When $\epsilon$ is unknown, the comparison between Theorem \ref{thm:est} and Theorem \ref{thm:lower} reveals a drastic difference between $r_{\alpha}^{\rm est}(\epsilon, p)$ and $r_{\alpha}(\epsilon,p,\epsilon_{\max})$. It is thus no longer possible to use the length of any confidence interval to accurately reflect the statistical error of an optimal point estimator.

\section{Locally Optimal Robust Test} \label{sec:test}
For the Gaussian location model with Huber contamination, it was established by \cite{luo2024adaptive} that the construction of an adaptive robust confidence interval is equivalent to solving robust hypothesis testing, and a length-optimal interval can be constructed by inverting a family of rate-optimal testing procedures. Following their strategy, we consider similar robust testing problems for $P_{\epsilon,p,Q}=(1-\epsilon)\text{Binomial}(m,p)+\epsilon Q$ in this section. With i.i.d. observations $X_1,\cdots,X_n$ drawn from some distribution $P$, define the following two pairs of robust hypothesis testing:
\begin{equation}\label{eq:test-def}
	\begin{split} &\cH(p, p+ r, \epsilon): \begin{array}{l}
		H_{0}: P \in \left\{ P_{\epsilon_{\max}, p, Q}: Q \right\} \quad  \textnormal { v.s. } \quad 
		H_{1}: P \in \left\{ P_{\epsilon, p + r, Q}: Q \right\},
	\end{array}\\
	&\cH(p, p - r, \epsilon): \begin{array}{l}
		H_{0}: P \in \left\{ P_{\epsilon_{\max}, p, Q}: Q \right\}\quad  \textnormal { v.s. } \quad
		H_{1}: P \in \left\{ P_{\epsilon, p - r, Q}: Q\right\},
	\end{array}
	\end{split}
\end{equation}
where $r$ is nonnegative and $\epsilon\in[0,\epsilon_{\max}]$. Suppose $\phi_{p,\epsilon}^+$ and $\phi_{p,\epsilon}^-$ are optimal testing functions for $\cH(p, p+ r, \epsilon)$ and $\cH(p, p - r, \epsilon)$, respectively. The confidence set constructed by \cite{luo2024adaptive} is given by the formula
\begin{equation}
\left\{p\in[0,1]: \phi_{p,\epsilon}^+=\phi_{p,\epsilon}^-=0\text{ for all }\epsilon\in[0,\epsilon_{\max}]\right\}.\label{eq:LGset}
\end{equation} When the set \eqref{eq:LGset}
 is an interval, which is the case for the Gaussian location model and other location families, \cite{luo2024adaptive} showed the coverage and optimal length guarantee of the confidence interval in \eqref{eq:LGset}. Unfortunately, for general statistical models, including the binomial model, the set (\ref{eq:LGset}) may not be an interval and therefore the result of \cite{luo2024adaptive} does not apply. A more delicate inversion of $\{\phi_{p,\epsilon}^+\}$ and $\{\phi_{p,\epsilon}^-\}$ is needed.

This section will first systematically study the robust testing (\ref{eq:test-def}). Then, we will discuss how to modify (\ref{eq:LGset}) into the formulas (\ref{eq:left}) and (\ref{eq:right}) that are applicable for more general parametric families.

\subsection{Locally Optimal Separation Rate}

In this section, we provide the lower and the upper bounds for solving testing problems in \eqref{eq:test-def}. For the clarity of presentation and technical convenience, we will present the lower bound result for $\cH(p\pm r, p, \epsilon)$ and the upper bound result for $\cH(p, p\pm r, \epsilon)$. 

Since the parameter $p$ is always bounded between $0$ and $1$, there is additional subtlety in characterizing the local optimal separation of the testing problem. In particular, the local separation rates of $\cH(p- r, p, \epsilon)$ and $\cH(p+ r, p, \epsilon)$ can be different when $p$ is close to $0$ or $1$. Fortunately for our purpose, it is easy to get around this problem when our goal is to lower bound the locally optimal confidence interval length by the testing rate. This is because the proof of Theorem \ref{thm:lower} only requires the lower bound of one of the two testing problems in $\cH(p\pm r, p, \epsilon)$, and for each $p\in[0,1]$ we can always choose the harder one. 

\begin{Theorem}\label{thm:test-low}
For any $\alpha \in (0,1)$, $\epsilon_{\max} \in [0, 1/2]$, and $n \geq 3$ satisfying $\epsilon_{\max} \geq \frac{2\alpha}{n}$, there exists some constant $c > 0$ only depending on $\alpha$ and $\epsilon_{\max}$, such that for any $\epsilon \in [0, \epsilon_{\max}]$ and $p \in [0,1]$, as long as $r\leq c\ell(n,\epsilon,m,p)$, we have
\begin{equation}\label{eq:l-s-p}
\begin{split}
	&\textnormal{either }\inf_{Q_0,Q_1} \TV\left(P^{\otimes n}_{\epsilon_{\max},p - r,Q_0}, P^{\otimes n}_{\epsilon, p, Q_1}\right) \leq \alpha \quad \textnormal{or}\quad  \inf_{Q_0,Q_1} \TV\left(P^{\otimes n}_{\epsilon_{\max},p + r,Q_0}, P^{\otimes n}_{\epsilon, p, Q_1}\right) \leq \alpha.
\end{split}
\end{equation}
\end{Theorem} {For any confidence interval $\widehat{\CI}\in\mathcal{I}_{\alpha}(\epsilon_{\max})$, it can be shown by that the test $\indi\{p-r\notin \widehat{\CI}\}$ (resp. $\indi\{p+r\notin \widehat{\CI}\}$) achieves small testing errors for $\cH(p-r, p, \epsilon)$ (resp. $\cH(p+r, p, \epsilon)$) as long as $r$ is greater than a high probability length bound of $\widehat{\CI}$.} Therefore, Theorem \ref{thm:lower} is immediately implied by Theorem \ref{thm:test-low} (see the proof of Theorem \ref{thm:lower} for details).

Next, we will characterize the local separation rates of the testing functions used in Algorithm~\ref{alg:CI-endpoints}. Recall the definitions of $\overline{r}(p, \epsilon)$ and $\underline{r}(p, \epsilon)$ in \eqref{def:t-tau-r}. The orders of the two quantities and their relations to $\ell(n,\epsilon,m,p)$ in (\ref{eq:main-rate-bmp}) are given by the following lemma.
\begin{Lemma} \label{lm:test-rate-CI-length-connection}
	    Suppose $\epsilon \in [0, 1/2]$ and $n \geq 2$. For $\overline{r}(p,\epsilon)$, $\underline{r}(p, \epsilon)$ and $\ell(n,\epsilon,m,p)$, we have
\begin{eqnarray*}
\overline{r}(p, \epsilon) &\asymp& \left(\sqrt{\frac{p(1-p)}{m}}\left(\frac{1}{\sqrt{\log n}}+\frac{1}{\sqrt{\log(1/\epsilon)}}\right) + \frac{1}{m} \right) \wedge (1-p), \\
\underline{r}(p, \epsilon) &\asymp& \left(\sqrt{\frac{p(1-p)}{m}}\left(\frac{1}{\sqrt{\log n}}+\frac{1}{\sqrt{\log(1/\epsilon)}}\right) + \frac{1}{m} \right) \wedge p, \\
\ell(n,\epsilon,m,p) &\asymp&  \overline{r}(p, \epsilon)\wedge \underline{r}(p, \epsilon) + \frac{1}{m}\left(\frac{1}{n}+\epsilon\right),\end{eqnarray*}
     where $\asymp$ suppresses dependence on $\alpha$. 
\end{Lemma}


The following result shows that $\overline{r}(p, \epsilon)$ and $\underline{r}(p, \epsilon)$ upper bound the locally optimal separation rates of $\cH(p, p+ r, \epsilon)$ and $\cH(p, p- r, \epsilon)$ respectively when $p$ is away from the boundary of $[0,1]$.

\begin{Theorem}\label{thm:test-up}
Suppose $\frac{\log(2/\alpha)}{n} + \epsilon_{\max}$ is less than a sufficiently small constant. The testing functions $\phi_{p,\epsilon}^+$ and $\phi_{p,\epsilon}^-$ defined by (\ref{eq:phi+}) and (\ref{eq:phi-}) satisfy the following simultaneous Type-1 error bounds,
\begin{eqnarray}
\label{ineq:type-1-phi+} \underset{Q}{\sup}P_{\epsilon_{\max},p,Q}\left(\underset{\epsilon\in[0,\epsilon_{\max}]}{\sup}\phi_{p,\epsilon}^+ = 1 \right) &\leq& \alpha/12, \\
\label{ineq:type-1-phi-} \underset{Q}{\sup}P_{\epsilon_{\max},p,Q}\left(\underset{\epsilon\in[0,\epsilon_{\max}]}{\sup}\phi_{p,\epsilon}^- = 1 \right) &\leq& \alpha/12,
\end{eqnarray}
for all $p\in \left[0,1\right]$. In addition, the testing function $\phi_{p,\epsilon}^+$ satisfies the following Type-2 error bound,
\begin{equation*}
    \underset{Q}{\sup}P_{\epsilon,p+r,Q}(\phi_{p,\epsilon}^+ = 0)\leq \alpha/12,   
\end{equation*}
for all $\epsilon \in [0, \epsilon_{\max}]$, all $p \in \left[0, 1 - \frac{4}{m}\left( \frac{10 \log(24/\alpha)}{n} + 3 \epsilon \right) \right ]$, and all
$r\in[\overline{r}(p, \epsilon),1-p]$.
Similarly, the testing function $\phi_{p,\epsilon}^-$ satisfies the following Type-2 error bound,
\begin{equation*}
    \underset{Q}{\sup}P_{\epsilon,p-r,Q}(\phi_{p,\epsilon}^- = 0)\leq \alpha/12,
\end{equation*}
for all $\epsilon \in [0, \epsilon_{\max}]$, all $p \in \left[ \frac{4}{m}\left( \frac{10 \log(24/\alpha)}{n} + 3 \epsilon \right),1 \right ]$, and all
$r\in[\underline{r}(p, \epsilon),p]$.
\end{Theorem}

Theorem \ref{thm:test-up} shows that the tests $\phi_{p,\epsilon}^+$ and $\phi_{p,\epsilon}^-$ achieve the local separation rates $\overline{r}(p, \epsilon)$ and $\underline{r}(p, \epsilon)$ for all $p$ that is bounded away from $0$ and $1$ by the order of $\frac{1}{m}\left(\frac{1}{n}+\epsilon\right)$. {This suggests that they can be inverted into a confidence interval with length of order $\overline{r}(p, \epsilon)+\frac{1}{m}\left(\frac{1}{n}+\epsilon\right)$ or $\underline{r}(p, \epsilon)+\frac{1}{m}\left(\frac{1}{n}+\epsilon\right)$ depending on whether $p$ is closer to $1$ or $0$. In either case, it will match the order of the locally optimal length $\ell(n,\epsilon,m,p)$ in view of Lemma \ref{lm:test-rate-CI-length-connection}.}

\subsection{From Monotone Tests to Confidence Interval} \label{sec:monotone-test-to-CI}

Since (\ref{eq:LGset}) may not be an interval with the testing functions (\ref{eq:phi+}) and (\ref{eq:phi-}), the length guarantee proved by \cite{luo2024adaptive} does not apply to (\ref{eq:LGset}).

To tackle this challenge, we need to find testing functions $\{\psi_{p,\epsilon}^{\pm}\}$ that solve $\cH(p, p\pm r, \epsilon)$ such that (\ref{eq:LGset}) is an interval. Note that (\ref{eq:LGset}), with $\phi_{p,\epsilon}^{\pm}$ replaced by $\psi_{p,\epsilon}^{\pm}$, can also be written as
$$\left\{p\in[0,1]: \psi_{p,\epsilon}^+=0\text{ for all }\epsilon\in[0,\epsilon_{\max}]\right\}\cap\left\{p\in[0,1]: \psi_{p,\epsilon}^-=0\text{ for all }\epsilon\in[0,\epsilon_{\max}]\right\}.$$
It is thus sufficient to require that $\left\{p\in[0,1]: \psi_{p,\epsilon}^+=0\text{ for all }\epsilon\in[0,\epsilon_{\max}]\right\}=[\wh{p}_{\rm left},1]$ for some $\wh{p}_{\rm left}\in[0,1]$ and $\left\{p\in[0,1]: \psi_{p,\epsilon}^-=0\text{ for all }\epsilon\in[0,\epsilon_{\max}]\right\}=[0,\wh{p}_{\rm right}]$ for some $\wh{p}_{\rm right}\in[0,1]$. This is clearly satisfied as long as $\psi_{p,\epsilon}^+$ is non-increasing in $p$ and $\psi_{p,\epsilon}^-$ is non-decreasing in $p$ for any $\epsilon \in [0, \epsilon_{\max}]$.

The monotonicity of the test is a very natural requirement. Intuitively, $\psi_{p,\epsilon}^+=1$ means that the data suggests that $p$ should be rejected in favor of some larger alternative. In other words, $p$ is too small to fit the observations. Therefore, for some even smaller $\wt{p}<p$, one should certainly reject $\wt{p}$ as well, which means $\psi_{\wt{p},\epsilon}^+=1$, or equivalently $\psi_{\wt{p},\epsilon}^+\geq \psi_{p,\epsilon}^+$. 

{To construct such a test, }we propose the following monotone variation of $\phi_{p,\epsilon}^{+}$, defined by
\begin{equation}
\psi_{p,\epsilon}^+ = \min_{q\in[0,p]}\phi_{q,\epsilon}^+. \label{eq:psi+}
\end{equation}
Thus, when $p\in[0,1]$, $\psi_{p,\epsilon}^+=1$ or $p$ is rejected in favor of a larger alternative if and only if the data suggests that every $q\leq p$ is too small in the sense that
$$\frac{1}{n}\sum_{i=1}^n\indi\{X_i\leq m \overline{t}(q,\epsilon)\}< \overline{\tau}(q,\epsilon).$$
Similarly, a monotone variation of $\phi_{p,\epsilon}^{-}$ is defined by
\begin{equation}
\psi_{p,\epsilon}^- = \min_{q\in[p,1]}\phi_{q,\epsilon}^-.\label{eq:psi-}
\end{equation}
This strategy also works for general parametric families beyond binomial. {In addition}, if the testing function $\phi_{p,\epsilon}^+$ (resp. $\phi_{p,\epsilon}^-$) is already non-increasing (resp. non-decreasing), we would simply get $\phi_{p,\epsilon}^+=\min_{q\in[0,p]}\phi_{q,\epsilon}^+$ (resp. $\phi_{p,\epsilon}^-=\min_{q\in[p,1]}\phi_{q,\epsilon}^-$).

Since $\psi_{p,\epsilon}^+\leq \phi_{p,\epsilon}^+$ and $\psi_{p,\epsilon}^-\leq \phi_{p,\epsilon}^-$, it is straightforward to see that the simultaneous Type-1 error guarantees in Theorem \ref{thm:test-up} still hold for $\{\psi_{p,\epsilon}^+\}$ and $\{\psi_{p,\epsilon}^-\}$. The following result shows that Type-2 error guarantees continue to hold as well.

\begin{Theorem}\label{thm:type2}
Suppose $ \frac{\log(2/\alpha)}{n} + \epsilon_{\max}$ is less than a sufficiently small constant. The testing function $\psi_{p,\epsilon}^+$ defined by (\ref{eq:psi+}) satisfies the Type-2 error bound,
\begin{equation*}
    \underset{Q}{\sup}P_{\epsilon,p+r,Q}(\psi_{p,\epsilon}^+ = 0)\leq \alpha/6,   
\end{equation*}
for all $\epsilon \in [0, \epsilon_{\max}]$, all $p \in \left[0, 1 - \frac{4}{m}\left( \frac{10 \log(24/\alpha)}{n} + 3 \epsilon \right) \right ]$, and all
$r\in[\overline{r}(p, \epsilon),1-p]$.
Similarly, the testing function $\psi_{p,\epsilon}^-$ defined by (\ref{eq:psi-}) satisfies the Type-2 error bound,
\begin{equation*}
    \underset{Q}{\sup}P_{\epsilon,p-r,Q}(\psi_{p,\epsilon}^- = 0)\leq \alpha/6,
\end{equation*}
for all $\epsilon \in [0, \epsilon_{\max}]$, all $p \in \left[ \frac{4}{m}\left( \frac{10 \log(24/\alpha)}{n} + 3 \epsilon \right),1 \right ]$, and all
$r\in[\underline{r}(p, \epsilon),p]$.
\end{Theorem}

With Theorem \ref{thm:type2}, the testing functions $\psi_{p,\epsilon}^+$ and $\psi_{p,\epsilon}^-$ are not only monotone in $p$, but they also achieve the same local separation rates as $\phi_{p,\epsilon}^+$ and $\phi_{p,\epsilon}^-$. With this modification, the set
\begin{equation}
\widetilde{\CI}=\left\{p\in[0,1]: \psi_{p,\epsilon}^+=\psi_{p,\epsilon}^-=0\text{ for all }\epsilon\in[0,\epsilon_{\max}]\right\} \label{eq:ci-con}
\end{equation}
is a well-defined interval. Its coverage and local optimality of length can be established from Theorem~\ref{thm:test-up} and Theorem \ref{thm:type2}.

A practical issue of (\ref{eq:ci-con}) is the difficulty of computing the endpoints of the interval. This motivates an additional discretization step. With $S_m=\left\{0,1/m, 2/m,\cdots,1\right\}$, we can replace (\ref{eq:psi+}) and (\ref{eq:psi-}) with
\begin{eqnarray} \label{eq:disc-test+}
\wh{\psi}_{p,\epsilon}^+ &=& \begin{cases}
\min_{q\in[0, \lceil mp \rceil/m ]\cap S_m}\phi_{q,\epsilon}^+ & p\in [0,1-\frac{1}{m}] \\
\phi_{p,\epsilon}^+\wedge \min_{q\in S_m\backslash\{1\}}\phi_{q,\epsilon}^+ & p \in (1-\frac{1}{m}, 1],
\end{cases} \\
\wh{\psi}_{p,\epsilon}^- &=& \begin{cases}
\phi_{p,\epsilon}^-\wedge \min_{q\in S_m\backslash\{0\}}\phi_{q,\epsilon}^- & p \in [0,\frac{1}{m}) \\
\min_{q\in[ \lfloor mp \rfloor/m  ,1]\cap S_m}\phi_{q,\epsilon}^- & p \in [\frac{1}{m}, 1]. \\
\end{cases}\label{eq:disc-test-}
\end{eqnarray}
Together with the grid $\mathcal{E}\subset[0,\epsilon_{\max}]$ used in Algorithm \ref{alg:CI-endpoints}, we define
\begin{equation}
\widehat{\CI}=\left\{p\in[0,1]: \wh{\psi}_{p,\epsilon}^+=\wh{\psi}_{p,\epsilon}^-=0\text{ for all }\epsilon\in\mathcal{E}\right\}. \label{eq:ci-dis}
\end{equation}
By the monotonicity of $\wh{\psi}_{p,\epsilon}^+$ and $\wh{\psi}_{p,\epsilon}^-$ , the formula (\ref{eq:ci-dis}) is still an interval to apply the theory of \cite{luo2024adaptive}, which leads to Theorem \ref{thm:upper}. Moreover, the endpoints of (\ref{eq:ci-dis}) are given by \eqref{eq:left} and \eqref{eq:right}, which can be explicitly computed by Algorithm \ref{alg:CI-endpoints}.

\begin{Proposition} \label{prop:binom-end-pionts}
The set $\widehat{\CI}$ defined by (\ref{eq:ci-dis}) is an interval whose endpoints are given by (\ref{eq:left}) and (\ref{eq:right}) and can be computed by Algorithm \ref{alg:CI-endpoints}.
\end{Proposition}

We remark that in \eqref{eq:disc-test+} and \eqref{eq:disc-test-}, the discretization over the grid $S_m$ only applies to $\psi_{p,\epsilon}^+$ when $0\leq p\leq 1-\frac{1}{m}$ and to $\psi_{p,\epsilon}^-$ when $\frac{1}{m}\leq p\leq 1$. This is because when $p\in[\frac{1}{m},1-\frac{1}{m}]$, the local optimal rate $\ell(n,\epsilon,m,p)$ is at least of order $\frac{1}{m}$, and therefore the discretization does not change the confidence interval length for more than the optimal rate. When $p<\frac{1}{m}$ or $p>1-\frac{1}{m}$, discretization is actually not needed, since $\phi_{p,\epsilon}^+$ (resp. $\phi_{p,\epsilon}^-$) is already monotone when $p>1-\frac{1}{m}$ (resp. $p<\frac{1}{m}$) and the tests can be explicity inverted into the closed form formulas (\ref{eq:end-start-left}) and (\ref{eq:end-start-right}).

\subsection{Application to Poisson Data}

The general strategy of interval construction by inverting monotone tests is not limited to the binomial model. In this section, we will demonstrate an application to Poisson data with contamination. Consider
\begin{equation}
X_1, \ldots, X_n \overset{i.i.d.}\sim (1-\epsilon)\text{Poisson}(\lambda)+\epsilon Q.\label{eq:Poisson}
\end{equation}
{Poisson distribution is not symmetric}, but similar to the binomial distribution, the variability of Poisson data also depends on the parameter $\lambda$. {We will use the framework of local optimality, and a special treatment will be given when $\lambda$ is close to $0$}. With slight abuse of notation, we define $r_{\alpha}(\epsilon,\lambda,\epsilon_{\max})$ in the same way as (\ref{eq:local-r-b}) except that the binomial probability is replaced by $P_{\epsilon,\lambda,Q}=(1-\epsilon)\text{Poisson}(\lambda)+\epsilon Q$. We first present a lower bound on the locally optimal length.
\begin{Theorem}\label{thm:lower-pois}
For any $\alpha \in (0,1/4)$, $\epsilon_{\max} \in [0, 1/2]$, and $n \geq 3$ satisfying $\epsilon_{\max} \geq \frac{2\alpha}{n}$, there exists some constant $c > 0$ only depending on $\alpha$ and $\epsilon_{\max}$, such that
\begin{equation}
r_{\alpha}(\epsilon,\lambda,\epsilon_{\max})\geq c\left(\left( \sqrt{\lambda}\left(\frac{1}{\sqrt{\log n}}+\frac{1}{\sqrt{\log(1/\epsilon)}}\right) + 1 \right)\wedge\lambda +\frac{1}{n}+\epsilon\right). \label{eq:rate-main-pois}
\end{equation}
\end{Theorem}
The rate (\ref{eq:rate-main-pois}) agrees with the binomial rate $m\ell(n,\epsilon,m,p)$ in (\ref{eq:main-rate-bmp}) with $p=\frac{\lambda}{m}$. This is quite natural given that $\text{Binomial}(m,p)$ can be approximated by $\text{Poisson}(\lambda)$ when $mp$ is close to $\lambda$.

To construct an optimal confidence interval, we need to solve the hypothesis testing problems (\ref{eq:test-def}) with a Poisson distribution. For each $\lambda \geq 0$, define
\begin{subequations}\label{def:t-tau-r-Poisson-upper}
	\begin{align}
		\overline{t}(\lambda, \epsilon) &= 
	 \lambda - \min \left\{\frac{\lambda}{2}, \frac{1}{8} \sqrt{ \lambda \log\left(\left(\epsilon + (\log(24/\alpha)/(2n))^{1/2}\right)^{-1}\right)} \right\},   \label{def:t-Poisson-upper} \\
\overline{r}(\lambda, \epsilon) &= \left\{ \begin{array}{ll}
\frac{1}{2} &  \lambda = 0\\
	\frac{\lambda}{4(\lambda - \overline{t}(\lambda,\epsilon) )}  &  \lambda \in (0, \infty),
\end{array}  \right.\label{def:r-Poisson-upper}\\
\overline{\tau}(\lambda, \epsilon) &= 
	\frac{11}{10} P_{X \sim \textnormal{Poisson}(\lambda + \overline{r}(\lambda, \epsilon))} \left( X \leq \overline{t}(\lambda, \epsilon) \right),
\label{def:tau-Poisson-upper}
	\end{align}
\end{subequations}
and
\begin{subequations}\label{def:t-tau-r-Poisson-lower}
\begin{align}
		\underline{t}(\lambda,\epsilon) &= \left\{ \begin{array}{ll}
        1 & \lambda \in [0,1) \\
	 \lambda +  \min \left\{\frac{\lambda}{2}, \frac{1}{8} \sqrt{ \lambda \log\left(\left(\epsilon + (\log(24/\alpha)/(2n))^{1/2}\right)^{-1}\right)} \right\} &  \lambda \in [1, \infty),
\end{array}  \right. \label{def:t-Poisson-under} \\
\underline{r}(\lambda,\epsilon) &= \left\{ \begin{array}{ll}
 (1- 1/(6e) )\lambda & \lambda \in [0,1)\\
	\frac{\lambda}{4(\underline{t}(\lambda,\epsilon) - \lambda )}  &  \lambda \in [1, \infty),
\end{array}  \right. \label{def:r-Poisson-under} \\
\underline{\tau}(\lambda, \epsilon) &= \left\{ \begin{array}{ll}
	  \frac{1}{2}(1 - e^{-\lambda}) - \frac{3 \log(24/\alpha) }{n}  & \lambda \in [0,1) \\
      \frac{11}{10} P_{X \sim \textnormal{Poisson}(\lambda - \underline{r}(\lambda, \epsilon))} \left( X \geq \underline{t}(\lambda, \epsilon)) \right)  &  \lambda \in [1, \infty).
\end{array}  \right. \label{def:tau-Poisson-under}
	\end{align}
\end{subequations}
Notice that we adopt the same notation for analogous quantities in both the binomial and Poisson settings. The intended meaning should be clear from the context, with $p$ referring to the binomial case and $\lambda$ to the Poisson case. Similar to (\ref{eq:phi+}) and (\ref{eq:phi-}), we define the testing functions
\begin{eqnarray}
\label{eq:pos-test+} \phi_{\lambda,\epsilon}^+ &=& \indi\left\{\frac{1}{n}\sum_{i=1}^n\indi\{X_i\leq  \overline{t}(\lambda,\epsilon)\}< \overline{\tau}(\lambda,\epsilon)\right\}, \\
\label{eq:pos-test-} \phi_{\lambda,\epsilon}^- &=& \indi\left\{\frac{1}{n}\sum_{i=1}^n\indi\{X_i\geq  \underline{t}(\lambda,\epsilon)\}< \underline{\tau}(\lambda,\epsilon)\right\}.
\end{eqnarray}
To obtain a confidence interval, we apply the following monotonization and discretization,
\begin{eqnarray}
\wh{\psi}_{\lambda,\epsilon}^+ &=& \begin{cases} \min_{\mu\in[0,\lceil\lambda\rceil]\cap \mathbb{N}_0}\phi_{\mu,\epsilon}^+ & \lambda  \in [0, \wh{\lambda}_{\max})\\
0 & \lambda = \wh{\lambda}_{\max},
\end{cases} \label{eq:psi+Poisson} \\
\wh{\psi}_{\lambda,\epsilon}^- &=& \begin{cases}
\phi_{\lambda,\epsilon}^-\wedge \min_{\mu\in [1,\wh{\lambda}_{\max}]\cap\mathbb{N}}\phi_{\mu,\epsilon}^- & \lambda \in [0,  1) \\
\min_{\mu\in[ \lfloor \lambda \rfloor,\wh{\lambda}_{\max}]\cap \mathbb{N}}\phi_{\mu,\epsilon}^- &  \lambda \in [1, \wh{\lambda}_{\max}], \\
\end{cases} \label{eq:psi-Poisson}
\end{eqnarray}
where
\begin{equation}
\wh{\lambda}_{\max}=X_{(\lceil 3n/4 \rceil)} + 1\label{eq:lambda-max}
\end{equation}
is a conservative upper bound for $\lambda$ to limit the search space to a bounded interval.
Together with the grid $\mathcal{E} = \left\{ \frac{ 2^k \log (24/\alpha)}{n}: k = 0,1, \ldots, \left\lfloor \log_2 \left( \frac{n \epsilon_{\max}}{\log(24/\alpha)} \right) \right\rfloor \right\} \cup \{\epsilon_{\max}\}$, we define
\begin{equation}
\widehat{\CI}=\left\{\lambda\in[0,\wh{\lambda}_{\max}]: \wh{\psi}_{\lambda,\epsilon}^+=\wh{\psi}_{\lambda,\epsilon}^-=0\text{ for all }\epsilon\in\mathcal{E}\right\}. \label{eq:ci-dis-Poisson}
\end{equation}
The endpoints of this interval can be computed in a similar way to Algorithm \ref{alg:CI-endpoints}. {The pseudocode will be presented as Algorithm \ref{alg:CI-endpoints-pois} in Appendix \ref{app:posi-alg}.} The guarantee for the interval \eqref{eq:ci-dis-Poisson} is given as follows.

\begin{Theorem}\label{thm:upper-pois}
Suppose $ \frac{\log(2/\alpha)}{n} + \epsilon_{\max}$ is less than a sufficiently small constant. Then, the interval $\widehat{\CI}$ defined by (\ref{eq:ci-dis-Poisson}) satisfies
\begin{equation*}
		\begin{split}
			&\inf_{ \epsilon \in [0, \epsilon_{\max}], \lambda, Q} P_{\epsilon, \lambda, Q}\left( p \in\widehat{\CI} \right) \geq 1-\alpha,\\
			&\inf_{\epsilon \in [0, \epsilon_{\max}], \lambda, Q } P_{\epsilon, \lambda, Q}\left( |\widehat{\CI}| \leq C \left( \left(\sqrt{\lambda}\left(\frac{1}{\sqrt{\log n}}+\frac{1}{\sqrt{\log(1/\epsilon)}}\right)  +1 \right) \wedge\lambda +\frac{1}{n}+\epsilon\right) \right) \geq 1-\alpha,
		\end{split}
	\end{equation*}
where $C>0$ is some constant only depending on $\alpha$.
\end{Theorem}

\section{Erd\Horig{o}s--R\'{e}nyi Model with Node Contamination} \label{sec:graph}

In this section, we consider statistical inference in the setting of Erd\Horig{o}s--R\'{e}nyi model \citep{gilbert1959random,erdds1959random} with contamination. In a standard Erd\Horig{o}s--R\'{e}nyi model with $n$ nodes, one observes a random graph encoded by an adjacency matrix $A\in\{0,1\}^{n\times n}$ such that $A_{ij}=A_{ji}\overset{i.i.d.}\sim \text{Bernoulli}(p)$ for all $1\leq i<j\leq n$ and $A_{ii}=0$ for all $1 \leq i \leq n$. We will consider a version of the problem with node contamination.

\subsection{Node Contamination}

There are two natural ways to model contamination on a random graph. One is the edge contamination, which allows an $\epsilon$ fraction of $A_{ij}$'s to be drawn from different distributions. The other one is node contamination. In this setting, an $\epsilon$ fraction of nodes are contaminated, and edges that are connected to the contaminated nodes are drawn from different distributions. Though both settings are relevant for different purposes of applications, the edge contamination setting is mathematically trivial given its resemblance to (\ref{eq:bernoulli}). We will thus focus on the setting of node contamination.

The Erd\Horig{o}s--R\'{e}nyi model with node contamination was proposed by \cite{acharya2022robust}. It is defined as a distribution set $\mathcal{G}(n,p,\epsilon)$, such that for any $P\in\mathcal{G}(n,p,\epsilon)$ the sampling process of $\{A_{ij}\}_{1\leq i<j\leq n}\sim P$ is described as follows.
\begin{enumerate}
\item Sample $z_1,\cdots,z_n \overset{i.i.d.}\sim \text{Bernoulli}(\epsilon)$.
\item Sample $\{A_{ij}\}$ independently given $\{z_i\}$. To be specific, for any $1\leq i<j\leq n$, sample $A_{ij}\sim \text{Bernoulli}(p)$ if $z_i=z_j=0$ and otherwise sample $A_{ij}$ from an arbitrary Bernoulli.
\end{enumerate}
In other words, the adjacency matrix $A$ has an $\epsilon$ fraction of rows and columns that are contaminated. We note that there is an even stronger contamination setting considered by \cite{acharya2022robust}, but we will present our results for $\mathcal{G}(n,p,\epsilon)$, even though the same conclusion continues to hold for the larger family.

Statistical inference under $\mathcal{G}(n,p,\epsilon)$ is a highly nontrivial task, even just for the estimation of $p$. The locally optimal estimation rate under the loss $|\wh{p}-p|$ is given by
\begin{equation}
\sqrt{\frac{p(1-p)}{n}}\left(\frac{1}{\sqrt{n}}+\epsilon\right)+\frac{1}{n}\left(\frac{1}{n}+\epsilon\right).\label{eq:ach}
\end{equation}
Remarkably, a delicate two-step spectral algorithm was proposed and analyzed by \cite{acharya2022robust} to achieve the above rate.\footnote{To be more precise, Theorem 21 of \cite{acharya2022robust} proved the lower bound of order (\ref{eq:ach}), together with an upper bound (Theorem 3 there) that matches (\ref{eq:ach}) up to some logarithmic factor.} It is interesting to compare (\ref{eq:ach}) with the locally optimal estimation error of the binomial model given by Theorem \ref{thm:est}. In particular, (\ref{eq:ach}) can be regarded as the binomial rate with $m=n$, though the two models are not equivalent. Further discussion on the difference and similarity of the two models will be given in Section \ref{eq:row-con}.

\subsection{A Conservative Confidence Interval}

Similar to (\ref{eq:local-r-b}), we define the locally optimal confidence interval length for the Erd\Horig{o}s--R\'{e}nyi model by
$$
r_{\alpha}^{\rm ER}(\epsilon,p,\epsilon_{\max})=\inf\left\{r\geq 0: \inf_{\widehat{\CI} \in  \cI_{\alpha}^{\rm ER}(\epsilon_{\max}) }\sup_{P\in\mathcal{G}(n,p,\epsilon)} P\left( |\widehat{\CI}| \geq r \right) \leq \alpha\right\},
$$
where
$$\cI_{\alpha}^{\rm ER}(\epsilon_{\max}) = \left\{ \widehat{\CI} = [l(A), u(A) ]: \inf_p\inf_{P\in\mathcal{G}(n,p,\epsilon_{\max})}P\left(p\in \widehat{\CI}\right) \geq 1- \alpha  \right\}.$$
Again, when $\epsilon_{\max}=\epsilon$ and $\epsilon$ is known, the locally optimal length $r_{\alpha}^{\rm ER}(\epsilon,p,\epsilon)$ can be achieved by (\ref{eq:bin-k-e}) with $m=n$, since the rate-optimal estimator $\wh{p}$ in \cite{acharya2022robust} achieves the high probability error bound (\ref{eq:ach}).

When $\epsilon_{\max}$ is a constant and $\epsilon$ is unknown, a conservative strategy is to use the error bound (\ref{eq:ach}) with $\epsilon=\epsilon_{\max}$. This leads to the following confidence interval
\begin{equation}
\widehat{\CI} = \left\{p\in[0,1]: |\wh{p}-p|\leq C\left(\sqrt{\frac{p(1-p)}{n}}+\frac{1}{n}\right)\right\}. \label{eq:er-conser}
\end{equation}
Somewhat surprisingly, this naive construction cannot be improved in terms of its local length when $\epsilon$ is unknown. In other words, for Erd\Horig{o}s--R\'{e}nyi model with node contamination, we have $r_{\alpha}^{\rm ER}(\epsilon,p,\epsilon_{\max})\asymp r_{\alpha}^{\rm ER}(\epsilon_{\max},p,\epsilon_{\max})$, and a shorter confidence interval when $\epsilon \ll\epsilon_{\max}$ is impossible. This is in stark contrast to the binomial model with Huber contamination given the results of Theorem \ref{thm:lower} and Theorem \ref{thm:upper}.

We will proceed to show the coverage and local length guarantees of (\ref{eq:er-conser}), while delaying the lower bound for $r_{\alpha}^{\rm ER}(\epsilon,p,\epsilon_{\max})$ to Section \ref{sec:sbm}. Since the spectral estimator of \cite{acharya2022robust} achieves (\ref{eq:ach}) up to some logarithmic factor, we still need to improve the estimator of \cite{acharya2022robust} in order that (\ref{eq:er-conser}) achieves the exact local length optimality.

To this end, we define a matrix norm
\begin{equation}\label{def: norm}
\|B\|_{\mathcal{U}}=
\sup_{U\in\mathcal{U}} |\iprod{B}{U}|,
\end{equation}
where
$$\mathcal{U}=\left\{U = (J-I)_{S\times S}: S\subset[n]\right\}.$$
Here $J\in\mathbb{R}^{n\times n}$ is the matrix with all ones, $I\in\mathbb{R}^{n\times n}$ is the identity matrix, and for any $B\in\mathbb{R}^{n\times n}$, the matrix $B_{S\times S}\in\mathbb{R}^{n\times n}$ is defined by zeroing out all entries of $B$ in $(S\times S)^c$. We consider the following estimator
\begin{equation}
\wh{p} = \frac{1}{\#\wh{S}(\#\wh{S} - 1)}\sum_{i\in\wh{S}}\sum_{j\in\wh{S}}A_{ij}, \label{eq:p-improved}
\end{equation}
where the subset $\wh{S}$ is computed according to
\begin{equation}
\wh{S}=\argmin_{S\subset[n]: \#{S} \geq \frac{3n}{4}}\left\|\left(A-\left(\frac{1}{\#{S}(\#{S} - 1)}\sum_{i\in{S}}\sum_{j\in{S}}A_{ij}\right)J\right)_{S\times S}\right\|_{\mathcal{U}}. \label{eq:find-whS}
\end{equation}

The first step of the spectral estimator of \cite{acharya2022robust} is defined in the same way, but they use the matrix operator norm instead of the norm $\|\cdot\|_{\mathcal{U}}$ to find $\wh{S}$. Their estimator achieves the rate $\sqrt{\frac{p(1-p)}{n}}+\frac{\sqrt{\log n}}{n}$. {By inspecting the proof of Theorem 12 in \cite{acharya2022robust}, we can see that the reason an extra $\sqrt{\log n}$ factor appears there is that the operator norm of a centered adjacency matrix of a standard  Erd\Horig{o}s--R\'{e}nyi graph with connection probability $p$ is bounded by $\sqrt{np(1-p)} + \sqrt{\log n}$. This result is in general sharp \citep{bandeira2016sharp}, thus the additive $\sqrt{\log n}$ term is not removable in the concentration of the operator norm.} We show that with our new matrix norm definition, the $\sqrt{\log n}$ factor can be removed, which immediately implies the desired coverage and local length guarantees for (\ref{eq:er-conser}).

\begin{Theorem}\label{thm:er-up}
Suppose $ \frac{\log(2/\alpha)}{n} + \epsilon_{\max}$ is less than a sufficiently small constant. Then, the estimator (\ref{eq:p-improved}) satisfies
\begin{equation}\label{Ineq: er-est}
\inf_p\inf_{P\in\mathcal{G}(n,p,\epsilon_{\max})} P\left(|\wh{p}-p|\leq C\left(\sqrt{\frac{p(1-p)}{n}}+\frac{1}{n}\right)\right)\geq 1- \alpha, 
\end{equation} where $C>0$ is some constant only depending on $\alpha$.
Consequently, the interval (\ref{eq:er-conser}) satisfies
\begin{equation*}
		\begin{split}
			&\inf_p\inf_{P\in\mathcal{G}(n,p,\epsilon_{\max})} P\left( p \in\widehat{\CI} \right) \geq 1-\alpha,\\
			&\inf_{\epsilon\in[0,\epsilon_{\max}],p}\inf_{P\in\mathcal{G}(n,p,\epsilon)} P\left( |\widehat{\CI}| \leq C' \left(\sqrt{\frac{p(1-p)}{n}}+\frac{1}{n}\right) \right) \geq 1-\alpha,
		\end{split}
	\end{equation*} where $C'>0$ is some constant only depending on $\alpha$.
\end{Theorem}

\subsection{Optimality via Community Detection}\label{sec:sbm}

The lower bound for $r_{\alpha}^{\rm ER}(\epsilon, p, \epsilon_{\max})$ is given by $r_{\alpha}^{\rm ER}(0, p, \epsilon_{\max})$ and the lower bound of $r_{\alpha}^{\rm ER}(0,p,\epsilon_{\max})$ can be derived from testing $H_0:P\in\mathcal{G}(n,p,0)$ against $H_1:P\in\mathcal{G}(n,q,\epsilon_{\max})$, where the notation $\mathcal{G}(n,p,0)$ is slightly abused to denote the standard Erd\Horig{o}s--R\'{e}nyi model without contamination. While $\mathcal{G}(n,p,0)$ is just the standard Erd\Horig{o}s--R\'{e}nyi model, the class $\mathcal{G}(n,q,\epsilon_{\max})$ contains a very important example called stochastic block model \citep{holland1983stochastic}. For $\eta,p_1,p_2,q\in[0,1]$, the sampling process of $\{A_{ij}\}_{1\leq i<j\leq n}\sim\text{SBM}(n,p_1,p_2,q,\eta)$ is given below.
\begin{enumerate}
\item Sample $z_1,\cdots,z_n \overset{i.i.d.}\sim \text{Bernoulli}(\eta)$.
\item Sample $\{A_{ij}\}$ independently given $\{z_i\}$ according to
$$A_{ij}\sim\begin{cases}
\text{Bernoulli}(p_1) & z_i=z_j = 0\\
\text{Bernoulli}(p_2) & z_i=z_j = 1\\
\text{Bernoulli}(q) & z_i\neq z_j.
\end{cases}$$
\end{enumerate}
Given the definition of $\mathcal{G}(n,p,\epsilon)$, we have
\begin{equation}
\text{SBM}(n,p_1,p_2,q, \eta)\in \mathcal{G}(n,p_1,\eta). \label{eq:sbm-er}
\end{equation}
Therefore, an even simpler testing problem is
\begin{equation}
H_0:A\sim\mathcal{G}(n,p,0)\,\,\,\, \textnormal{v.s.}\,\,\,\, H_1:A\sim \text{SBM}\left(n,p+r, p + \left(\frac{1 - \epsilon_{\max}}{\epsilon_{\max}}\right)^2 r,p-\frac{1-\epsilon_{\max}}{\epsilon_{\max}}r,\epsilon_{\max}\right) \label{eq:cd}
\end{equation} for some $r \geq 0$.
This is known as community detection, which has been thoroughly studied in the literature by \cite{decelle2011asymptotic,massoulie2014community,mossel2015reconstruction,mossel2014consistency,mossel2018proof,abbe2018community,gao2021minimax} and references therein. In particular, testing (\ref{eq:cd}) is possible if and only if the separation parameter $r$ is at least of order $\sqrt{\frac{p(1-p)}{n}}+\frac{1}{n}$. The following result is an adaptation of a lower bound in Theorem 3.3 of \cite{jin2021optimal}.
\begin{Proposition}\label{prop:cd-lower}
For any $\alpha \in (0, 1)$, $\epsilon_{\max} \in [0, 1/2]$, and $n \geq 2$, there exists some constant $c > 0$ only depending on $\alpha$ and $\epsilon_{\max}$, such that as long as
$$0\leq r\leq c\left(\sqrt{\frac{p(1-p)}{n}}+\frac{1}{n}\right),$$
we have
$$\TV\left(\mathcal{G}(n,p,0),\textnormal{SBM}\left(n,p+r, p + \left(\frac{1 - \epsilon_{\max}}{\epsilon_{\max}}\right)^2 r,p-\frac{1-\epsilon_{\max}}{\epsilon_{\max}}r,\epsilon_{\max}\right)\right)\leq \alpha,$$
for all $p\in[\frac{1}{n},\frac{1}{2}]$, and
$$\TV\left(\mathcal{G}(n,p,0),\textnormal{SBM}\left(n,p-r, p - \left(\frac{1 - \epsilon_{\max}}{\epsilon_{\max}}\right)^2 r,p+\frac{1-\epsilon_{\max}}{\epsilon_{\max}}r,\epsilon_{\max}\right)\right)\leq \alpha,$$
for all $p\in[\frac{1}{2},1 - \frac{1}{n}]$.
\end{Proposition}

On the other hand, suppose we have a confidence interval $\widehat{\CI}\in\mathcal{I}_{\alpha}^{\rm ER}(\epsilon_{\max})$, the relation (\ref{eq:sbm-er}) implies that the community detection problem (\ref{eq:cd}) can be solved by the test $\indi\{p + r\in \widehat{\CI}\}$ as long as $r$ exceeds a high probability length bound of $\widehat{\CI}$. The lower bound result of Proposition~\ref{prop:cd-lower} immediately implies the lower bound for $r_{\alpha}^{\rm ER}(\epsilon,p,\epsilon_{\max})$, which then leads to optimality of the conservative interval (\ref{eq:er-conser}).

\begin{Theorem}\label{thm:er-low}
For any $\alpha \in (0, 1/4)$, $\epsilon_{\max} \in [0, 1/2]$, and $n \geq 2$, there exists some constant $c > 0$ only depending on $\alpha$ and $\epsilon_{\max}$, such that
$$
r_{\alpha}^{\rm ER}(\epsilon,p,\epsilon_{\max})\geq c \left(\sqrt{\frac{p(1-p)}{n}}+\frac{1}{n}\right),
$$
for any $\epsilon\in[0,\epsilon_{\max}]$ and $p\in[\frac{1}{n},1-\frac{1}{n}]$.
\end{Theorem}

\begin{Remark}
Since the testing lower bound of Proposition \ref{prop:cd-lower} requires the condition $p\in[\frac{1}{n},1-\frac{1}{n}]$, the optimality of the conservative interval (\ref{eq:er-conser}) may not hold when $p$ is extremely close to $0$ or $1$. In fact, consider testing $H_0:P\in\mathcal{G}(n,0,0)$ against $H_1:P\in\mathcal{G}(n,r,\epsilon_{\max})$. It is straightforward to check that the two hypotheses can be distinguished whenever $r\geq \frac{C}{n^2}$, which suggests $r_{\alpha}^{\rm ER}(0,0,\epsilon_{\max})\asymp \frac{1}{n^2}$, a faster rate than $\sqrt{\frac{p(1-p)}{n}}+\frac{1}{n}$ {with $p = 0$}. However, fully understanding all the pathological cases in $[0,\frac{1}{n})\cup(1-\frac{1}{n},1]$ is quite technically involved. Theorem \ref{thm:er-up} and Theorem \ref{thm:er-low} show that at least for $p$'s that are mostly practically relevant, one cannot do better than the conservative strategy.
\end{Remark}

\subsection{Comparison with the Binomial Model}\label{eq:row-con}

To understand the relation between $\mathcal{G}(n,p,\epsilon)$ and the binomial model (\ref{eq:bernoulli}), it is helpful to introduce another set of distributions $\overline{\mathcal{G}}(n,p,\epsilon)$ on binary random matrices. The sampling process of $A\sim P$ for some $P\in \overline{\mathcal{G}}(n,p,\epsilon)$ is given as follow.
\begin{enumerate}
\item Sample $z_1,\cdots,z_n \overset{i.i.d.}\sim \text{Bernoulli}(\epsilon)$.
\item Sample $\{A_{ij}\}$ independently given $\{z_i\}$. To be specific, for any $i,j\in[n]$, sample $A_{ij}\sim \text{Bernoulli}(p)$ if $z_i=0$ and otherwise sample $A_{ij}$ from an arbitrary Bernoulli.
\end{enumerate}

In other words, a distribution in $\overline{\mathcal{G}}(n,p,\epsilon)$ only allows rows to be contaminated, while $\mathcal{G}(n,p,\epsilon)$ involves simultaneous row and column contamination. In fact, given some $P\in \overline{\mathcal{G}}(n,p,\epsilon)$ and sample $A\sim P$ according to the above process, we have $A_{i1},\cdots,A_{in}\overset{i.i.d.}\sim \text{Bernoulli}(p)$ for a row such that $z_i=0$. A sufficient statistic for this row is $\sum_{j=1}^nA_{ij}\sim\text{Binomial}(n,p)$. In other words, $\sum_{j=1}^nA_{1j},\cdots,\sum_{j=1}^nA_{nj}$ can be regarded as data generated by the binomial model (\ref{eq:bernoulli}) with $m=n$. In fact, it is not hard to check that Theorems \ref{thm:lower}, \ref{thm:upper} and \ref{thm:est} all hold for $\overline{\mathcal{G}}(n,p,\epsilon)$.

In comparison, the additional column contamination in the model $\mathcal{G}(n,p,\epsilon)$ makes statistical inference of $p$ significantly harder than that of $\overline{\mathcal{G}}(n,p,\epsilon)$ or (\ref{eq:bernoulli}). Though the locally optimal estimation errors are the same,\footnote{Though \cite{acharya2022robust} only proves the rate (\ref{eq:ach}) up to a logarithmic factor, we conjecture that (\ref{eq:ach}) is the exact locally optimal rate for the estimation problem.} construction of adaptive confidence intervals with unknown $\epsilon$ is completely different under the two settings. Theorem \ref{thm:er-up} and Theorem \ref{thm:er-low} imply that adaptation to unknown $\epsilon$ is impossible, and the optimal local length under $\mathcal{G}(n,p,\epsilon)$ is of the same order as that under $\overline{\mathcal{G}}(n,p,\epsilon_{\max})$.

\section{Proofs}\label{sec:pfthm2}

Due to page limits, we will prove Theorem \ref{thm:upper} in this section. The proofs of other results are given in the appendices. Throughout the section, we write $P_p=\text{Binomial}(m,p)$ so that $P_{\epsilon,p,Q}=(1-\epsilon)P_p+\epsilon Q$.

\begin{proof}[Proof of Theorem \ref{thm:upper}]
To show the guarantees of $\wh{\CI}$, we first analyze the properties of two related confidence intervals: the $\wt{\CI}$ in \eqref{eq:ci-con} and 
\begin{equation}
\widebar{\CI}=\left\{p\in[0,1]: {\psi}_{p,\epsilon}^+= {\psi}_{p,\epsilon}^-=0\text{ for all }\epsilon\in \cE\right\}, \label{eq:ci-semidis}
\end{equation} where ${\psi}_{p,\epsilon}^+$ and ${\psi}_{p,\epsilon}^-$ are defined in \eqref{eq:psi+} and \eqref{eq:psi-}. Intuitively, $\widebar{\CI}$ approximates $\wt{\CI}$ by discretizing $\epsilon$ and $\wh{\CI}$ further approximates $\widebar{\CI}$ by discretizing $p$. We divide the rest of the proof into three parts, each analyzing $\wt{\CI}$, $\widebar{\CI}$ and $\wh{\CI}$, respectively. 

\vskip.2cm
{\noindent \bf (Part I -- Guarantees for $\wt{\CI}$)} We are going to show the following result. 

\begin{Theorem} \label{thm:bino-upper-no-dis}
	Suppose $ \frac{\log(2/\alpha)}{n} + \epsilon_{\max}$ is less than a sufficiently small universal constant. The confidence interval (\ref{eq:ci-con}) satisfies
	\begin{equation*}
		\begin{split}
			&\inf_{ \epsilon \in [0, \epsilon_{\max}], p, Q} P_{\epsilon, p, Q}\left( p \in\widetilde{\CI} \right) \geq 1-\alpha/6\quad \textnormal{and}\quad\inf_{\epsilon \in [0, \epsilon_{\max}], p, Q } P_{\epsilon, p, Q}\left( |\widetilde{\CI}| \leq C' \ell(n,\epsilon,m,p) \right) \geq 1-\alpha/2,
		\end{split}
	\end{equation*}
where $C' > 0$ is some constant depending on $\alpha$ only. The formula of $\ell(n,\epsilon,m,p)$ is given by (\ref{eq:main-rate-bmp}).
\end{Theorem}
\begin{proof}[Proof of Theorem \ref{thm:bino-upper-no-dis}]
We begin with the coverage guarantee. For any $\epsilon \in [0, \epsilon_{\max}]$ and $p\in [0,1]$, we have
\begin{equation*}
    \begin{split}
        \sup_QP_{\epsilon,p,Q}\left(p \notin \widetilde{\mathrm{CI}}\right)&\leq \sup_QP_{\epsilon,p,Q}\left(\sup_{\epsilon'\in [0, \epsilon_{\max}]} {\psi}^{+}_{p, \epsilon'}=1 \right)+\sup_QP_{\epsilon,p,Q}\left(\sup_{\epsilon'\in [0, \epsilon_{\max}]}{\psi}^{-}_{p, \epsilon'}=1\right)\\
        &\overset{(a)}{\leq} \sup_QP_{\epsilon_{\max},p,Q}\left(\sup_{\epsilon'\in [0,\epsilon_{\max}]} {\psi}^{+}_{p, \epsilon'}=1 \right)+\sup_QP_{\epsilon_{\max},p,Q}\left(\sup_{\epsilon'\in [0,\epsilon_{\max}]} {\psi}^{-}_{p, \epsilon'}=1\right)\\
        & \overset{(b)}\leq \sup_QP_{\epsilon_{\max},p,Q}\left(\sup_{\epsilon'\in [0,\epsilon_{\max}]} {\phi}^{+}_{p, \epsilon'}=1 \right)+\sup_QP_{\epsilon_{\max},p,Q}\left(\sup_{\epsilon'\in [0,\epsilon_{\max}]} {\phi}^{-}_{p, \epsilon'}=1\right)\\
        & \overset{(c)}\leq \alpha/6,
    \end{split}
\end{equation*}
where (a) follows from the fact that
\begin{equation} \label{eq:model-inclusion}
	\{P_{\epsilon,p,Q}:Q\} \subseteq \{P_{\epsilon_0,p,Q}:Q\} \quad \textnormal{ for all } \, p \in [0,1], 0 \leq \epsilon \leq \epsilon_0 \leq 1,
\end{equation}
 (b) is because $ {\psi}^{+}_{p, \epsilon}$ (resp. ${\psi}^{-}_{p, \epsilon'}$) is no bigger than ${\phi}^{+}_{p, \epsilon}$ (resp. $ {\phi}^{-}_{p, \epsilon'}$) and (c) is by the simultaneous Type-1 error control of $\phi^+_{p, \epsilon}$ and $\phi^-_{p, \epsilon}$ given in Theorem \ref{thm:test-up}.

Now we consider the length guarantee. The following lemma will be useful, and its proof is given in Appendix \ref{sec:thm-upper-supporting-lemma}.
\begin{Lemma}\label{lm:r-property-length-bino}
	Suppose $ \frac{\log(2/\alpha)}{n} + \epsilon_{\max}$ is less than a sufficiently small universal constant. Given any $c \in (0,1)$, there exists a large constant $C_0 > 0$ only depending on $c$ and $\alpha$ such that for any $C \geq C_0$ and $\epsilon \in [0, \epsilon_{\max}]$,
    	\begin{equation}\label{ineq:r-length-prop2}
		\begin{split}
			\bar{r}( p - C\bar{r}(p, \epsilon), \epsilon ) \leq C\bar{r}(p, \epsilon), \quad \forall p \in [c/m, 1-c/m] \cap \{p: p - C\bar{r}(p, \epsilon) \geq 0 \}, \\
			\underline{r}( p + C\underline{r}(p, \epsilon), \epsilon ) \leq C\underline{r}(p, \epsilon), \quad \forall p \in [c/m, 1-c/m] \cap \{p: p + C\underline{r}(p, \epsilon) \leq 1 \}.
		\end{split}
	\end{equation}
\end{Lemma}
Next, we observe that $\widetilde{\CI}$ is an interval by the monotonicity of ${\psi}^{+}_{p, \epsilon}$ and ${\psi}^{-}_{p, \epsilon}$. For any $p\in [0,1]$, $\epsilon \in [0, \epsilon_{\max}]$, we have
\begin{equation}\label{Ineq: Length guarantee}
    \begin{split}
    \sup_QP_{\epsilon,p,Q}\left(|\widetilde{\mathrm{CI}}|\geq C' \ell(n, \epsilon, m, p)  \right)\leq& \sup_QP_{\epsilon,p,Q}\left(|\wt{\mathrm{CI}}|\geq C' \ell(n, \epsilon, m, p)  ,p\in\wt{\mathrm{CI}}\right)+\sup_QP_{\epsilon,p,Q}\left(p\notin \wt{\mathrm{CI}} \right) \\
    \leq & \sup_Q P_{\epsilon,p,Q}\left(|\wt{\mathrm{CI}}|\geq C' \ell(n, \epsilon, m, p)  ,p\in\wt{\mathrm{CI}}\right) + \alpha/6,
    \end{split}
\end{equation} where the last inequality is by the coverage guarantee we have just shown. We will bound
$$\sup_QP_{\epsilon,p,Q}\left(|\wt{\mathrm{CI}}|\geq C' \ell(n, \epsilon, m, p)  ,p\in\wt{\mathrm{CI}}\right),$$
according to three cases.
\begin{itemize}[leftmargin=*]
	\item (Case 1: $p \in [0, \frac{1}{4m})$) In this case, $p \leq 1/2$ and $\ell(n, \epsilon, m, p)  = p + \frac{1}{m}\left( \frac{1}{n} + \epsilon \right)$. In addition, without loss of generality, we can assume that $C' \ell(n,\epsilon,m,p) \leq 1/2$, as otherwise the statement is trivial since $ \sup_Q P_{\epsilon,p,Q}\left(|\wt{\mathrm{CI}}|\geq 2C' \ell(n, \epsilon, m, p)\right) = 0$. Thus, for any $C' > 1$,
	\begin{equation} \label{ineq:length-regime1}
		\begin{split}
			\sup_Q P_{\epsilon,p,Q}\left(|\wt{\mathrm{CI}}|\geq C' \ell(n, \epsilon, m, p)  ,p\in\wt{\mathrm{CI}}\right)   &\overset{(a)} \leq \sup_Q P_{\epsilon,p,Q}\left(p + (C'-1)\left(p + \frac{1}{m}\left( \frac{1}{n} + \epsilon \right)\right) \in \wt{\CI}\right) \\
			& \overset{(b)}\leq \sup_Q P_{\epsilon,p,Q}\left(\psi^-_{p + (C'-1)\left(p + \frac{1}{m}\left( \frac{1}{n} + \epsilon \right)\right), \epsilon  } = 0\right),
		\end{split}
	\end{equation} where (a) is because $|\wt{\mathrm{CI}}|\geq C' \ell(n, \epsilon, m, p) $ and $p\in\wt{\mathrm{CI}}$ together imply that $p + (C'-1)\left(p + \frac{1}{m}\left( \frac{1}{n} + \epsilon \right)\right) \in \wt{\CI}$ as $\wt{\CI}$ is an interval (note that $p + (C'-1)\left(p + \frac{1}{m}\left( \frac{1}{n} + \epsilon \right)\right) \leq 1/2$ as $C' \ell(n, \epsilon, m, p)  = C'\left(p + \frac{1}{m}\left( \frac{1}{n} + \epsilon \right)\right) \leq 1/2$); (b) is by the definition of $\wt{\CI}$.
	
	By Theorem \ref{thm:type2}, to control the Type-2 error of $\psi^-_{p + (C'-1)\left(p + \frac{1}{m}\left( \frac{1}{n} + \epsilon \right)\right), \epsilon  }$, we need 
	\begin{equation*}
		\begin{split}
			\left\{ \begin{array}{l}
				(i):p + (C'-1)\left(p + \frac{1}{m}\left( \frac{1}{n} + \epsilon \right)\right) \geq \frac{4}{m}\left( \frac{10 \log(24/\alpha)}{n} + 3 \epsilon \right), \\
				(ii):(C'-1)\left(p + \frac{1}{m}\left( \frac{1}{n} + \epsilon \right)\right) \geq \underline{r}(p + (C'-1)\left(p + \frac{1}{m}\left( \frac{1}{n} + \epsilon \right)\right), \epsilon).
			\end{array} \right.
		\end{split}
	\end{equation*} Notice that the first condition above holds as long as $C'$ is larger than a universal constant depending on $\alpha$ only. In addition, by Lemma \ref{lm:additional-r-property} (i) and the fact that 
\(\bar{r}(q, \epsilon) = \underline{r}(1 - q, \epsilon)\), 
it is easy to check that $\underline{r}(q, \epsilon) \leq \left(1 - \frac{1}{6e}\right)q$ for all $q \leq 1/2$. By the definition of $\underline{r}(p,\epsilon)$, and using the fact that $C'\left(p+\frac{1}{m}\left(\frac{1}{n}+\epsilon\right)\right) \leq 1/2$, a sufficient condition for the second requirement is
\begin{equation*}  (C'-1)\left(p + \frac{1}{m}\left( \frac{1}{n} + \epsilon \right)\right) \geq (1 - 1/(6e) ) \left( p + (C'-1)\left(p + \frac{1}{m}\left( \frac{1}{n} + \epsilon \right)\right) \right),  \end{equation*} which can be satisfied as long as $C' \geq 6e$. In summary, when $C'$ is sufficiently large, we have $\sup_Q P_{\epsilon,p,Q}\left(\psi^-_{p + (C'-1)\left(p + \frac{1}{m}\left( \frac{1}{n} + \epsilon \right)\right), \epsilon  } = 0\right) \leq \alpha/6$ by Theorem \ref{thm:type2}. Plugging it back into \eqref{ineq:length-regime1}, we have $\sup_QP_{\epsilon,p,Q}\left(|\wt{\mathrm{CI}}|\geq C' \ell(n, \epsilon, m, p)  ,p\in\wt{\mathrm{CI}}\right) \leq \alpha/6$.

	\item (Case 2: $p \in \left[\frac{1}{4m}, 1-\frac{1}{4m}\right]$) By Lemma \ref{lm:test-rate-CI-length-connection}, it is easy to check $\ell(n, \epsilon, m, p) \asymp\bar{r}(p, \epsilon) \wedge \underline{r}(p, \epsilon) \asymp \bar{r}(p, \epsilon) \vee \underline{r}(p, \epsilon)$ in this regime. Thus,
	\begin{equation} \label{ineq:length-regime2}
		\begin{split}
			&\sup_Q P_{\epsilon,p,Q}\left(|\wt{\mathrm{CI}}|\geq C' \ell(n, \epsilon, m, p)  ,p\in\wt{\mathrm{CI}}\right) \overset{(a)}\leq \sup_Q P_{\epsilon,p,Q}\Bigg(p - \frac{C''}{2} \bar{r}(p, \epsilon) \in \wt{\CI} \text{ or } p + \frac{C''}{2}  \underline{r}(p, \epsilon)   \in \wt{\CI}\Bigg) \\
			& \leq \sup_Q P_{\epsilon,p,Q}\Bigg(p - \frac{C''}{2} \bar{r}(p, \epsilon) \in \wt{\CI} \Bigg)  + \sup_Q P_{\epsilon,p,Q}\Bigg(p + \frac{C''}{2}  \underline{r}(p, \epsilon)   \in \wt{\CI}\Bigg),
		\end{split}
	\end{equation} where in (a), we use the fact $\wt{\CI}$ is an interval. 
	
	Next, we bound $\sup_Q P_{\epsilon,p,Q}\Bigg(p - \frac{C''}{2} \bar{r}(p, \epsilon) \in \wt{\CI} \Bigg) $. The analysis for bounding $ \sup_Q P_{\epsilon,p,Q}\Bigg(p + \frac{C''}{2}  \underline{r}(p, \epsilon)   \in \wt{\CI}\Bigg)$ is similar and we omit it here for simplicity. Notice that when $p - \frac{C''}{2} \bar{r}(p, \epsilon) < 0$, the corresponding probability is zero. So without loss of generality, we assume $p - \frac{C''}{2} \bar{r}(p, \epsilon) \geq 0$. By definition of $\wt{\CI}$, we have\begin{equation*}
	\begin{split}
		& \sup_Q P_{\epsilon,p,Q}\Bigg(p - \frac{C''}{2} \bar{r}(p, \epsilon) \in \wt{\CI} \Bigg) \leq \sup_Q P_{\epsilon,p,Q}\Bigg( \psi^+_{ p - \frac{C''}{2} \bar{r}(p, \epsilon), \epsilon }  = 0\Bigg).
	\end{split}
\end{equation*}		By Theorem \ref{thm:type2}, to control the Type-2 error of $\psi^+_{ p - \frac{C''}{2} \bar{r}(p, \epsilon), \epsilon }$ above, we need 
	\begin{equation*}
		\begin{split}
				(i):p - \frac{C''}{2} \bar{r}(p, \epsilon) \leq 1- \frac{4}{m}\left( \frac{10 \log(24/\alpha)}{n} + 3 \epsilon \right) \quad \textnormal{and}\quad (ii):
				\frac{C''}{2} \bar{r}(p, \epsilon) \geq \bar{r}\left(p - \frac{C''}{2} \bar{r}(p, \epsilon), \epsilon\right).
		\end{split}
	\end{equation*} The first condition above holds because $p \leq 1 - \frac{1}{4m}$ and $\frac{\log(2 / \alpha)}{n} + \epsilon_{\max}$ is sufficiently small; the second condition above holds as long as $C''$ is large by Lemma \ref{lm:r-property-length-bino}. In summary, as long as $C'$ is large enough to allow $C''$ to be taken sufficiently large, 
by Theorem \ref{thm:type2} and \eqref{ineq:length-regime2}, we have
$\sup_Q P_{\epsilon,p,Q}\left(|\wt{\mathrm{CI}}|\geq C' \ell(n, \epsilon, m, p)  ,p\in\wt{\mathrm{CI}}\right) \leq   \alpha/3.
$

	\item (Case 3: $p \in (1-\frac{1}{4m} , 1]$) The analysis is similar to the Case 1, and we omit the details for simplicity. 
\end{itemize}
 
 By plugging the upper bound of $\sup_Q P_{\epsilon,p,Q}\left(|\wt{\mathrm{CI}}|\geq C' \ell(n, \epsilon, m, p)  ,p\in\wt{\mathrm{CI}}\right) $ into \eqref{Ineq: Length guarantee}, we have
 \begin{equation*}
		 \sup_QP_{\epsilon,p,Q}\left(|\widetilde{\mathrm{CI}}|\geq C' \ell(n, \epsilon, m, p)  \right)\leq  \sup_QP_{\epsilon,p,Q}\left(|\wt{\mathrm{CI}}|\geq C' \ell(n, \epsilon, m, p)  ,p\in\wt{\mathrm{CI}}\right) + \alpha/6 \leq \alpha/2.
	\end{equation*}
 This finishes the length guarantee and also finishes the proof of Theorem \ref{thm:bino-upper-no-dis}.
 \end{proof}
\vskip.2cm
{\noindent \bf (Part II -- Guarantees for $\widebar{\CI}$)} 
By definition $\wt{\CI} \subseteq \widebar{\CI}$, so the coverage guarantee of $\widebar{\CI}$ directly follows from the coverage guarantee of $\wt{\CI}$. Now, we show the length guarantee of $\widebar{\CI}$. First, we observe that following the same analysis as the length analysis in Theorem \ref{thm:bino-upper-no-dis}, i.e., the analysis of \eqref{Ineq: Length guarantee}, we have
\begin{equation} \label{ineq:length-guarantee-grid}
	\begin{split}
		 \sup_{p, Q} P_{\epsilon,p,Q}\left(|\widebar{\mathrm{CI}}|\geq C' \ell(n, \epsilon, m, p) \right) \leq \alpha/2, \quad \forall \epsilon \in \cE.
	\end{split}
\end{equation} Now let us consider $\epsilon \in [0, \epsilon_{\max}] \setminus \cE$. For any $\epsilon \in [0, \epsilon_{\max}] \setminus \cE$, by construction of $\cE$, we can find $\epsilon_0 \in \cE$ such that $\epsilon_0 \geq \epsilon$ and $\ell(n, \epsilon, m, p) \asymp \ell(n, \epsilon_0, m, p)$. Then there exists a large $C > 0$ such that 
\begin{equation*}
		\begin{split}
			 \sup_{p, Q} P_{\epsilon,p,Q}\left(|\widebar{\mathrm{CI}}|\geq C \ell(n, \epsilon, m, p) \right) & \overset{ \eqref{eq:model-inclusion} }\leq   \sup_{p, Q} P_{\epsilon_0,p,Q}\left(|\widebar{\mathrm{CI}}|\geq C \ell(n, \epsilon, m, p) \right) \\
			 & \overset{(a)} \leq \sup_{p, Q} P_{\epsilon_0,p,Q}\left(|\widebar{\mathrm{CI}}|\geq C' \ell(n, \epsilon_0, m, p) \right) \overset{\eqref{ineq:length-guarantee-grid}} \leq  \alpha/2,
		\end{split}
	\end{equation*} where (a) is because $\ell(n, \epsilon, m, p) \asymp \ell(n, \epsilon_0, m, p)$.

\vskip.2cm
{\noindent \bf (Part III-1 -- Coverage Guarantee of $\widehat{\CI}$)} Recall the definition of $\widehat{\CI}$ in \eqref{eq:ci-dis}. For any $\epsilon \in [0, \epsilon_{\max}]$ and $p\in [0,1]$, we have
\begin{equation} \label{ineq:hatCI-coverage}
    \begin{split}
        \sup_QP_{\epsilon,p,Q}\left(p \notin \wh{\mathrm{CI}}\right)\leq \sup_QP_{\epsilon,p,Q}\left(\sup_{\epsilon'\in \cE} \wh{\psi}^{+}_{p, \epsilon'}=1 \right)+\sup_QP_{\epsilon,p,Q}\left(\sup_{\epsilon'\in \cE} \wh{\psi}^{-}_{p, \epsilon'}=1\right).
    \end{split}
\end{equation} Next, we will bound the two terms at the end of the above equation. Due to symmetry, we will just show the bound for $\sup_QP_{\epsilon,p,Q}\left(\sup_{\epsilon'\in \cE} \wh{\psi}^{+}_{p, \epsilon'}=1 \right)$. We divide the proof based on the range of $p$. 
\begin{itemize}[leftmargin=*]
	\item (Case 1: $p \in [0,1-1/m]$) In this case, given any $\epsilon \in [0, \epsilon_{\max}]$,
	\begin{equation*}\label{ineq:CI hat coverage}
		\begin{split}
			&\sup_{Q}P_{\epsilon,p,Q}\left(\sup_{\epsilon'\in \cE} \wh{\psi}^{+}_{p, \epsilon'}=1 \right) = \sup_{Q}P_{\epsilon,p,Q}\left(\sup_{\epsilon'\in \cE} \min_{q\in[0, \lceil mp \rceil/m ]\cap S_m}\phi_{q,\epsilon'}^+=1 \right) \quad (\text{by definition of }  \wh{\psi}^{+}_{p, \epsilon'}) \\
			& \leq \sup_{Q} P_{\epsilon,p,Q}\left(\sup_{\epsilon'\in \cE} \phi_{\lceil mp \rceil/m,\epsilon'}^+=1 \right)\overset{(a)}\leq \sup_{Q}P_{\epsilon,\lceil mp \rceil/m,Q}\left(\sup_{\epsilon'\in \cE} \phi_{\lceil mp \rceil/m,\epsilon'}^+=1 \right) \\
			& \overset{ \eqref{eq:model-inclusion} }\leq \sup_{Q} P_{\epsilon_{\max},\lceil mp \rceil/m,Q}\left(\sup_{\epsilon'\in [0, \epsilon_{\max}]} \phi_{\lceil mp \rceil/m,\epsilon'}^+=1 \right) \overset{(b)}\leq \alpha/12,
		\end{split}
	\end{equation*} where in (a) we use the fact that Binomial($m,\lceil mp \rceil/m$) stochastically dominates Binomial($m,p$) as $\lceil mp \rceil/m \geq p$, i.e., we may construct $X \sim \text{Binomial}(m,\lceil mp \rceil/m)$, $Y \sim \text{Binomial}(m,p)$, and some nonnegative random variable $Z$ such that $X = Y + Z$ almost surely; applying it to the clean data part in $ P_{\epsilon,p,Q}$ and $P_{\epsilon,\lceil mp \rceil/m,Q}$, we have 
	\begin{equation*}
		\begin{split}
			& P_{\frac{\lceil mp \rceil}{m}}\left(\sup_{\epsilon'\in \cE} \phi_{\lceil mp \rceil/m,\epsilon'}^+=1 \right) = P_{\frac{\lceil mp \rceil}{m}}\left(\sup_{\epsilon'\in \cE} \indi\left\{\frac{1}{n}\sum_{i=1}^n\indi\{X_i\leq m \overline{t}( \lceil mp \rceil/m ,\epsilon')\}< \overline{\tau}(\lceil mp \rceil/m,\epsilon')\right\} = 1 \right) \\
			 & \overset{(i)}= P_{ \{X_i, Z_i\}_{i=1}^n }\left(\sup_{\epsilon'\in \cE} \indi\left\{\frac{1}{n}\sum_{i=1}^n\indi\{X_i + Z_i\leq m \overline{t}( \lceil mp \rceil/m ,\epsilon')\}< \overline{\tau}(\lceil mp \rceil/m,\epsilon')\right\} = 1 \right) \\
			 & \geq  P_{ p}\left(\sup_{\epsilon'\in \cE} \indi\left\{\frac{1}{n}\sum_{i=1}^n\indi\{X_i \leq m \overline{t}( \lceil mp \rceil/m ,\epsilon')\}< \overline{\tau}(\lceil mp \rceil/m,\epsilon')\right\} = 1 \right) =P_{p}\left(\sup_{\epsilon'\in \cE} \phi_{\lceil mp \rceil/m,\epsilon'}^+=1 \right)
		\end{split}
	\end{equation*} where in (i), $Z_i \geq 0$ and $X_i \sim \textnormal{Binomial}(m,p)$; (b) is by the simultaneous Type-1 error control of $\phi_{q,\epsilon'}^+$ shown in Theorem \ref{thm:test-up}.
	
	\item (Case 2: $p \in (1-1/m,1]$) In this case, given any $\epsilon \in [0, \epsilon_{\max}]$,
	\begin{equation*}
		\begin{split}
			&\sup_{Q}P_{\epsilon,p,Q}\left(\sup_{\epsilon'\in \cE} \wh{\psi}^{+}_{p, \epsilon'}=1 \right) = \sup_{Q}P_{\epsilon,p,Q}\left(\sup_{\epsilon'\in \cE} \left( \phi_{p,\epsilon'}^+\wedge \min_{q\in S_m\backslash\{1\}}\phi_{q,\epsilon'}^+ \right)=1 \right) \quad (\text{by definition of }  \wh{\psi}^{+}_{p, \epsilon'}) \\
			& \leq \sup_{Q}P_{\epsilon,p,Q}\left(\sup_{\epsilon'\in \cE} \phi_{p,\epsilon'}^+=1 \right) \overset{ \eqref{eq:model-inclusion} }\leq \sup_{Q}P_{\epsilon_{\max},p,Q}\left(\sup_{\epsilon'\in [0, \epsilon_{\max}]} \phi_{p,\epsilon'}^+=1 \right)
			 \overset{(a)}\leq \alpha/12,
		\end{split}
	\end{equation*}
	where (a) is by the simultaneous Type-1 error control of $\phi^+_{p, \epsilon}$ given in Theorem \ref{thm:test-up}.
\end{itemize}
In summary, we have shown $\sup_QP_{\epsilon,p,Q}\left(\sup_{\epsilon'\in \cE} \wh{\psi}^{+}_{p, \epsilon'}=1 \right) \leq \alpha/12$. Following the same proof, we can also show $\sup_QP_{\epsilon,p,Q}\left(\sup_{\epsilon'\in \cE} \wh{\psi}^{-}_{p, \epsilon'}=1\right) \leq \alpha/12$. Plugging them back into \eqref{ineq:hatCI-coverage}, we have shown for any $\epsilon \in [0, \epsilon_{\max}]$ and $p\in [0,1]$, $ \sup_Q P_{\epsilon,p,Q}\left(p \notin \wh{\mathrm{CI}}\right) \leq \alpha/6$. This finishes the proof for the coverage guarantee.

\vskip.2cm
{\noindent \bf (Part III-2 -- Length Guarantee of $\widehat{\CI}$)} In this part, we will show 
\begin{equation} \label{ineq:hatCI-length}
	\sup_Q P_{\epsilon,p,Q}\left(|\widehat{\CI}|\geq C^* \ell(n, \epsilon, m, p) \right) \leq \alpha
\end{equation} for some $C^* >0$. Without loss of generality, we can assume $\wh{p}_{\rm left} \leq \wh{p}_{\rm right}$; the case $\wh{p}_{\rm left} > \wh{p}_{\rm right}$ is handled by our convention $|\widehat{\CI}|=0$ and need not be considered.

We will leverage the coverage and length guarantees of $\widebar{\CI}$ shown in Part II. Let us first denote the endpoints of $\widebar{\CI}$ as $\widebar{p}_{\rm{left}}$ and $\widebar{p}_{\rm{right}}$, i.e.,    
\begin{equation*} \label{eq:barCI-end-points}
		\begin{split}
			\widebar{p}_{\rm{left}} = \inf\{p \in [0,1]:  {\psi}_{p,\epsilon}^+ = 0 \text{ for all }\epsilon\in\mathcal{E} \} \quad \text{and} \quad \widebar{p}_{\rm{right}} = \sup\{p \in [0,1]:  {\psi}_{p,\epsilon}^- = 0 \text{ for all }\epsilon\in\mathcal{E} \}.
		\end{split}
	\end{equation*}  The following lemma shows a few connections of $\wh{p}_{\rm{left}}$/$\wh{p}_{\rm{right}}$ with $\widebar{p}_{\rm{left}}$/$\widebar{p}_{\rm{right}}$ and its proof is given in Appendix \ref{sec:thm-upper-supporting-lemma}. 

\begin{Lemma}\label{lm:connection-hatp-barp} (i) If $\widebar{p}_{\rm{left}} > 1-1/m$, then $\wh{p}_{\rm{left}} = \widebar{p}_{\rm{left}}$. Similarly, if $\widebar{p}_{\rm{right}} < 1/m$, then $\wh{p}_{\rm{right}} = \widebar{p}_{\rm{right}}$. (ii) We always have $[\wh{p}_{\rm{left}}, \wh{p}_{\rm{right}}] \subseteq [\widebar{p}_{\rm{left}}- 1/m, \widebar{p}_{\rm{right}} + 1/m]$. 
\end{Lemma}

 By the coverage and length guarantees of $\widebar{\CI}$ shown in Part II, we know for any $\epsilon \in [0, \epsilon_{\max}], p\in [0,1]$, and $Q$, with probability at least $1 - \alpha$ under $P_{\epsilon, p, Q}$, the following event $(G)$ happens:
\begin{equation*}
	\begin{split}
		(G) := \{p \in \widebar{\CI}\quad \text{ and }\quad |\widebar{\CI}| \leq C' \ell(n, \epsilon, m, p) \}.
	\end{split}
\end{equation*} Next, we will show that given $(G)$ happens, we have $|\widehat{\CI}| \leq C^*\ell(n, \epsilon, m, p)$ for some $C^* > 0$ depending on $C'$ only. We divide the proof into three cases based on the range of $p$. Let $c \in (0,1)$ to be a small constant we will specify later.
\begin{itemize}[leftmargin=*]
	\item (Case 1: $p \in [0,c/m)$) In this regime, $\ell(n, \epsilon, m, p)  = p + \frac{1}{m}\left( \frac{1}{n} + \epsilon \right) $. Then 
	\begin{equation*}
		\begin{split}
			\widebar{p}_{\rm{right}} \overset{(G)}\leq p + |\widebar{\CI}|  \overset{(G)}\leq (C' + 1) \left(p + \frac{1}{m}\left( \frac{1}{n} + \epsilon \right) \right) < \frac{(C' + 1)c}{m} + \frac{(C'+1)(1/n + \epsilon)}{m} \leq 1/m,
		\end{split}
	\end{equation*} where the last inequality holds as $ \frac{1}{n} + \epsilon_{\max}$ is small and we take $c$ to be a sufficiently small constant depending on $C'$ only. Then by Lemma \ref{lm:connection-hatp-barp} (i), we have $\wh{p}_{\rm{right}} = \widebar{p}_{\rm{right}}$. As a result,
	\begin{equation*}
		\begin{split}
			|\wh{\CI}| = \wh{p}_{\rm{right}} - \wh{p}_{\rm{left}} \leq \wh{p}_{\rm{right}} = \widebar{p}_{\rm{right}} \leq (C' + 1) \left(p + \frac{1}{m}\left( \frac{1}{n} + \epsilon \right) \right) \lesssim \ell(n, \epsilon, m, p).
		\end{split}
	\end{equation*}
	
\item (Case 2: $p \in [c/m,1-c/m]$) Note that in this regime,
 \begin{equation*}
 	\ell(n,\epsilon,m,p) \asymp \left(\sqrt{\frac{p(1-p)}{m}}\left(\frac{1}{\sqrt{\log n}}+\frac{1}{\sqrt{\log(1/\epsilon)}}\right) + \frac{1}{m} \right).
 \end{equation*}

Then by Lemma \ref{lm:connection-hatp-barp} (ii), we have
 \begin{equation*}
 	\begin{split}
 		|\wh{\CI}| \leq |\widebar{\CI}| + \frac{2}{m} \overset{(G)}\leq C' \ell(n,\epsilon,m,p) + \frac{2}{m} \lesssim \ell(n,\epsilon,m,p).
 	\end{split}
 \end{equation*}

\item (Case 3: $p \in (1-c/m,1]$) The proof is the same as Case 1 and we have $|\widehat{\CI}| \lesssim \ell(n, \epsilon, m, p)$. 
\end{itemize}
In summary, we have shown that given $(G)$ happens, there exists a $C^* > 0$ depending on $C'$ only such that $|\widehat{\CI}| \leq C^*\ell(n, \epsilon, m, p)$. This shows \eqref{ineq:hatCI-length} and finishes the proof for the length guarantee.
\end{proof}

\vskip 0.2in
\bibliographystyle{apalike}
\bibliography{reference}

\end{sloppypar}

		\newpage
		\appendix
		
			\begin{center}
			{\LARGE Supplement to "Robust Confidence Intervals for a Binomial Proportion: Local Optimality and Adaptivity"	
				
			}		
			\medskip
			
			\medskip
		\end{center}
		
		In this supplement, we provide the rest of the technical proofs. 

\appendix

\section{Proofs for the Binomial Model with Unknown $\epsilon$} \label{sec:proof-main}

This section collects proofs of Theorem \ref{thm:lower}, Theorem \ref{thm:test-low}, Theorem \ref{thm:test-up}, Theorem \ref{thm:type2}, Proposition \ref{prop:binom-end-pionts}, Lemma \ref{lm:p-hat-charac}, and the two supporting lemmas in the proof of Theorem \ref{thm:upper} (Lemma \ref{lm:r-property-length-bino} and Lemma \ref{lm:connection-hatp-barp}).

\subsection{Proof of Theorem \ref{thm:lower}}
The proof of this theorem relies on the following lemma.
\begin{Lemma}\label{Lem: TV to lower bound}
For any $\alpha \in (0,1/4)$, $\epsilon, \epsilon_0$ with $\epsilon \leq \epsilon_0$, and $p, p_0 \in [0,1]$ satisfying
$$
\inf_{Q_0, Q_1} \TV(P^{\otimes n}_{\epsilon_0, p_0, Q_0}, P^{\otimes n}_{\epsilon, p, Q_1}) \leq \alpha,
$$
we have
$
r_{\alpha}(\epsilon,p, \epsilon_0) > |p - p_0|.
$ Similarly, under the contaminated Poisson model (\ref{eq:Poisson}), for any $\lambda, \lambda_0 \geq 0$ satisfying
$
\inf_{Q_0, Q_1} \TV(P^{\otimes n}_{\epsilon_0, \lambda_0, Q_0}, P^{\otimes n}_{\epsilon, \lambda, Q_1}) \leq \alpha,
$ we have $
r_{\alpha}(\epsilon,\lambda, \epsilon_0) > |\lambda - \lambda_0|.
$
\end{Lemma}
\begin{proof}[Proof of Lemma \ref{Lem: TV to lower bound}]
	The proof of Lemma \ref{Lem: TV to lower bound} is similar to the proof of a combination of Proposition 2 and Theorem 1 in \cite{luo2024adaptive}. Basically, it is followed by a contradiction argument. If a confidence interval $\wh{\mathrm{CI}}$ has small length, then the test function defined by $\phi_{p_0} = \indi \left\{p_0 \notin \wh{\mathrm{CI}}\right\}$ can solve the hypothesis testing problem $$H_0: X_1, \ldots, X_n \overset{i.i.d.}\sim P_{\epsilon_0, p_0, Q} \quad \textnormal{v.s.}\quad H_1: X_1, \ldots, X_n \overset{i.i.d.}\sim P_{\epsilon, p, Q}, $$ with small Type-1 and Type-2 errors which contradicts the assumption that $\inf_{Q_0, Q_1} \TV(P^{\otimes n}_{\epsilon_0, p_0, Q_0}, P^{\otimes n}_{\epsilon, p, Q_1}) \leq \alpha$. Here, for simplicity, we would omit the details. 
\end{proof}

The result of Theorem \ref{thm:lower} is followed directly by a combination of Lemma \ref{Lem: TV to lower bound} and Theorem \ref{thm:test-low}.


\subsection{Proofs for Supporting Lemmas of Theorem \ref{thm:upper}} \label{sec:thm-upper-supporting-lemma}

\subsubsection{Proof of Lemma \ref{lm:r-property-length-bino}}
First, we extend the definitions of $\bar{r}(p, \epsilon)$ and $\underline{r}(p, \epsilon)$ by defining $\bar{r}(p, \epsilon)= \underline{r}(p, \epsilon) = 0$ when $ p < 0 $ or $p > 1$. With this extended definition, the relation $\underline{r}(p,\epsilon)=\overline{r}(1-p,\epsilon)$ holds for all $p \in \bbR$ and $\epsilon \in [0,1]$ and in order to show \eqref{ineq:r-length-prop2}, it is enough to show 
\begin{equation*}\label{ineq:r-length-prop2-simp}
	\begin{split}
		\bar{r}( p - C\bar{r}(p, \epsilon), \epsilon ) \leq C\bar{r}(p, \epsilon), \quad \forall p \in [c/m, 1-c/m], \\
		\underline{r}( p + C\underline{r}(p, \epsilon), \epsilon ) \leq C\underline{r}(p, \epsilon), \quad \forall p \in [c/m, 1-c/m].
	\end{split}
\end{equation*} We claim that in the above two equations, it is enough to show the first one, as the second equation is equivalent to the first one as 
\begin{equation*}
	\begin{split}
		& \underline{r}( p + C\underline{r}(p, \epsilon), \epsilon ) \leq C\underline{r}(p, \epsilon), \quad \forall p \in [c/m, 1-c/m] \\
	\Longleftrightarrow	& \bar{r}(1 - p - C\bar{r}(1-p, \epsilon), \epsilon ) \leq C \bar{r}(1-p, \epsilon), \quad \forall p \in [c/m, 1-c/m] \quad (\text{as }\underline{r}(p,\epsilon)=\overline{r}(1-p,\epsilon)) \\
	\overset{(a)}\Longleftrightarrow	& \bar{r}( p - C\bar{r}(p, \epsilon), \epsilon ) \leq C\bar{r}(p, \epsilon), \quad \forall p \in [c/m, 1-c/m], 
	\end{split}
\end{equation*} where (a) is because if $p \in [c/m, 1-c/m]$, we also have $1-p \in [c/m, 1-c/m]$.  

Fix any $\epsilon \in [0, \epsilon_{\max}]$, let us denote $A = \epsilon + \sqrt{\frac{\log(24/\alpha)}{2n}}$. By assumption, $A$ is less than a sufficiently small constant. We divide the rest of the proof into two cases.
\begin{itemize}[leftmargin=*]
	\item (Case 1: $\log(1/A) \geq 4m$) In this case, by Lemma \ref{lm:test-rate-CI-length-connection}, we have $\bar{r}(p, \epsilon) \gtrsim 1/m$ for all $p \in [c/m, 1-c/m]$. In addition, by Lemma \ref{lm:additional-r-property} (ii) and $\bar{r}( p - C'\bar{r}(p, \epsilon), \epsilon ) \leq \frac{1}{m}$ for any $C' > 0$. So there exists $C_0$ such that $\bar{r}(p - C_0\bar{r}(p, \epsilon), \epsilon) \leq C_0 \bar{r}(p, \epsilon)$. Note that the same conclusion holds if we replace $C_0$ by any $C \geq C_0$.

\item (Case 2: $\log(1/A) < 4m$)  By Lemma \ref{lm:test-rate-CI-length-connection}, it is easy to check that for any $p \in [c/m, 1-c/m]$, 
\begin{equation*}
\begin{split}
      \overline{r}(p, \epsilon) &\asymp \left(\sqrt{\frac{p(1-p)}{m}}\left(\frac{1}{\sqrt{\log n}}+\frac{1}{\sqrt{\log(1/\epsilon)}}\right) + \frac{1}{m} \right) \wedge (1-p)\\
      & \asymp \sqrt{\frac{p(1-p)}{m}}\left(\frac{1}{\sqrt{\log n}}+\frac{1}{\sqrt{\log(1/\epsilon)}}\right) + \frac{1}{m}.
\end{split}
\end{equation*} Thus, for any $C_0 > 1$ and $p \in [c/m,1 - c/m]$ such that $p - C_0\bar{r}(p, \epsilon) \geq 0$,
\begin{equation*} \label{ineq:r-bar-shift-upper}
\begin{split}
		\bar{r}(p - C_0\bar{r}(p, \epsilon), \epsilon) & \lesssim \sqrt{ \frac{(p - C_0\bar{r}(p, \epsilon))(1-p + C_0\bar{r}(p, \epsilon))}{m} } \left(\frac{1}{\sqrt{\log n}}+\frac{1}{\sqrt{\log(1/\epsilon)}}\right) +  \frac{1}{m} \\
		& \lesssim \sqrt{ (C_0+1)\frac{p(1-p)}{m} }\left(\frac{1}{\sqrt{\log n}}+\frac{1}{\sqrt{\log(1/\epsilon)}}\right) +  \frac{1}{m} \quad(\text{as } \bar{r}(p, \epsilon) \lesssim 1-p) \\
        & \lesssim \sqrt{C_0} \bar{r}(p, \epsilon).
\end{split}
\end{equation*} 
Therefore, as long as $C_0$ is large enough, we have $\bar{r}(p - C_0\bar{r}(p, \epsilon), \epsilon) \leq C_0  \bar{r}(p, \epsilon)$ and the same conclusion holds if we replace $C_0$ by any $C \geq C_0$.

\end{itemize}
In summary, there exists $C_0 > 0$ such that  
\begin{equation*}
	\bar{r}( p - C_0\bar{r}(p, \epsilon), \epsilon ) \leq  C_0\bar{r}(p, \epsilon), \quad \forall p \in [c/m, 1-c/m],
\end{equation*} and the same conclusion holds if we replace $C_0$ by any $C \geq C_0$. Finally, we note that the above analysis holds for all $\epsilon \in [0, \epsilon_{\max}]$. This finishes the proof of this lemma.

\subsubsection{Proof of Lemma \ref{lm:connection-hatp-barp}}
{\noindent \bf (Proof of the First Statement)} We will prove the statement regarding $\widebar{p}_{\rm{left}}$ and $\wh{p}_{\rm{left}}$, while a similar proof works for the other part. Let us give a similar characterization of $\widebar{p}_{\rm{left}}$ as $\wh{p}_{\rm{left}}$ in \eqref{eq:pleft-third}. Specifically, we will show the following claim: 
\begin{equation} \label{eq:barp-characterization}
	\begin{split}
		\widebar{p}_{\rm{left}} = \left\{ \begin{array}{ll}
			 \inf\{p \in [0,1-1/m]:  \max_{\epsilon \in \cE} \min_{q \in [0,p] } \phi_{q, \epsilon}^+ = 0  \} & \text{ if }  \max_{\epsilon \in \cE} \min_{q \in [0,1-1/m]} \phi_{q, \epsilon}^+ = 0\\
			 \inf \{ p \in (1- \frac{1}{m} ,1]: \phi_{p, \epsilon}^+ = 0  \} & \text{ if }  \max_{\epsilon \in \cE} \min_{q \in [0,1-1/m]} \phi_{q, \epsilon}^+ = 1.
		\end{array}  \right.
	\end{split}
\end{equation} The proof of this claim is very similar to the proof of Lemma \ref{lm:p-hat-charac}, so we omit the proof here for simplicity. 

As a result, if $\widebar{p}_{\rm{left}} > 1-1/m$, \eqref{eq:barp-characterization} implies that $\max_{\epsilon \in \cE} \min_{q \in [0,1-1/m]} \phi_{q, \epsilon}^+ = 1$, which further implies that $\max_{\epsilon \in \cE} \min_{q \in S_m \setminus \{1\}} \phi_{q, \epsilon}^+ = 1$. Then
\begin{equation*}
	\wh{p}_{\rm{left}} \overset{ \textnormal{Lemma }\ref{lm:p-hat-charac} }= \inf \{ p \in (1- 1/m ,1]: \phi_{p, \epsilon}^+ = 0  \} \overset{ \eqref{eq:barp-characterization} }= \widebar{p}_{\rm{left}}.
\end{equation*}
\vskip.2cm
{\noindent \bf (Proof of the Second Statement)} If we can show that $\wh{p}_{\rm{left}} \geq \widebar{p}_{\rm{left}} -1/m$ and $\wh{p}_{\rm{right}} \leq \widebar{p}_{\rm{right}} +1/m$, then the claim follows. In the following, we will show $ \wh{p}_{\rm{left}} \geq \widebar{p}_{\rm{left}} -1/m$, while the proof for $ \wh{p}_{\rm{right}} \leq \widebar{p}_{\rm{right}} +1/m$ is similar. If $\widebar{p}_{\rm{left}} > 1-1/m$, the conclusion follows directly from the first statement. Now we consider $\widebar{p}_{\rm{left}} \leq 1-1/m$. In addition, we note that if $\wh{p}_{\rm{left}} \geq 1-1/m$, we clearly have $\wh{p}_{\rm{left}} \geq \widebar{p}_{\rm{left}} -1/m$. So without loss of generality, we just need to focus on $\wh{p}_{\rm{left}} < 1-1/m$ as well. Then by Lemma \ref{lm:p-hat-charac}, we know when $\wh{p}_{\rm{left}} < 1-1/m$, then it is determined by the following formula: 
\begin{equation*}
	\begin{split}
		\wh{p}_{\rm{left}} =  \inf\{p \in S_m\setminus \{1\}:  \max_{\epsilon \in \cE} \min_{q \in [0,p+1/m] \cap (S_m \setminus \{1\}) } \phi_{q, \epsilon}^+ = 0  \}.
	\end{split}
\end{equation*}  In another way of speaking, $\max_{\epsilon \in \cE} \min_{q \in [0,\wh{p}_{\rm{left}}+1/m] \cap (S_m \setminus \{1\}) } \phi_{q, \epsilon}^+ = 0.$
This implies that $$\max_{\epsilon \in \cE} \min_{q \in [0,\wh{p}_{\rm{left}}+1/m] }\phi_{q, \epsilon}^+ = 0.$$ Recall that $\widebar{p}_{\rm{left}}=  \inf \{p \in [0,1]:  {\psi}_{p,\epsilon}^+ = 0 \text{ for all }\epsilon\in\mathcal{E} \} = \inf \{p \in [0,1]: \max_{\epsilon \in \cE} \min_{q \in [0,p]}{\phi}_{q,\epsilon}^+ = 0 \}$. Thus, $\widebar{p}_{\rm{left}} \leq \wh{p}_{\rm{left}}+1/m$ and this finishes the proof.

\subsection{Proof of Theorem \ref{thm:test-low}}
For simplicity, we will show \eqref{eq:l-s-p} holds for all $p \in [0, 1/2]$. By symmetry, a similar argument can also be applied to show \eqref{eq:l-s-p} also holds for $p \in [1/2, 1]$, since $X \sim P_{\epsilon,p, Q}$ implies $m - X \sim P_{\epsilon, 1-p, Q'}$ for some $Q'$. We will divide the rest of the proof into two parts based on different ranges of $p$.
\vskip.2cm
{\noindent \bf (Part I: $p \in [0, \frac{1}{4m} \left( \frac{\alpha}{n} + \epsilon  \right) ] $)} We will divide the proof into two cases based on the magnitude of $n$ and $\epsilon$.
\begin{itemize}[leftmargin=*]
\item (Case 1: $\alpha/n \geq \epsilon$) In this case, $p \in [0, \frac{\alpha}{2mn}] $ and $\ell(n,\epsilon,m,p) = p + \frac{1}{m} \left( \frac{1}{n} + \epsilon \right) \leq \frac{\alpha}{2mn} + \frac{2}{mn}$. We will show that when
	$r\leq c\ell(n,\epsilon,m,p) \leq c (2 + \alpha/2) \frac{1}{mn}$ for some small enough $c > 0$, then $\inf_{Q_0,Q_1} \TV\left(P^{\otimes n}_{\epsilon_{\max},p + r,Q_0}, P^{\otimes n}_{\epsilon, p, Q_1}\right) \leq \alpha$. First, it is easy to check that when $c$ is small, we have $p + r \leq 1$. Now, we take $Q_0 = \mathrm{Binomial}(m,p+r)$ and $Q_1 = \mathrm{Binomial}(m,p)$. Then 
	\begin{equation} \label{ineq:lower-part1-case1}
\TV\left(P_{\epsilon_{\max},p + r,Q_0}, P_{\epsilon, p, Q_1}\right) = \TV(P_{p+r},P_{p}) \overset{\textnormal{Lemma }\ref{lm:TV monotonicity}}\leq  \TV(P_{p+r},P_{0}) = 1 - (1-p-r)^m \leq m (p + r) \leq \frac{\alpha}{n},
	\end{equation} where the last inequality holds as long as we take $c \leq \frac{\alpha}{4 + \alpha}$.
	Then by the property of TV distance on the product measure, we get 
\begin{equation*}
  \inf_{Q_0,Q_1} \TV\left(P^{\otimes n}_{\epsilon_{\max},p + r,Q_0}, P^{\otimes n}_{\epsilon, p, Q_1}\right)  \leq n\TV(P_{p+r},P_{p}) \overset{\eqref{ineq:lower-part1-case1}} \leq \alpha.
\end{equation*}

	\item (Case 2: $\alpha /n \leq \epsilon$) In this case, $p \in [0, \frac{1}{2m} \epsilon] $ and $\ell(n,\epsilon,m,p) = p + \frac{1}{m} \left( \frac{1}{n} + \epsilon \right) \leq \frac{1}{2m} \epsilon + \frac{(1+ 1/\alpha) \epsilon}{m} $. We will show that when
	$r\leq c\ell(n,\epsilon,m,p) \leq c ((1 + 1/\alpha) + 1/2) \frac{\epsilon}{m}$ for some small enough $c > 0$, then $\inf_{Q_0,Q_1} \TV\left(P^{\otimes n}_{\epsilon_{\max},p + r,Q_0}, P^{\otimes n}_{\epsilon, p, Q_1}\right) =0$. First, it is easy to check that when $c$ is small, we have $p + r \leq 1$. Now, we take $Q_0 = \mathrm{Binomial}(m,p+r)$, and $Q_1$ to have the following probability mass function:
\begin{equation}\label{q_1 for large p}
    q_1(k)= \frac{1}{\epsilon}\binom{m}{k}(p+r)^k(1 - p-r)^{m-k} - \frac{1 - \epsilon}{\epsilon} \binom{m}{k}p^k(1 - p)^{m-k}, \quad \forall k \in [m] \cup \{0\}.
\end{equation} It is easy to see that $P_{ \epsilon_{\max}, p+r, Q_0}$ and $P_{\epsilon,p,Q_1}$ exactly match as long as $Q_1$ is a valid distribution. This will directly imply the result. So, we just need to verify $Q_1$ is a valid distribution. It is easy to check that $\sum_{k=0}^{m}q_1(k)=1$. To show that the formula \eqref{q_1 for large p} is a valid probability mass function, we only need to verify that $q_1(k)\geq 0$ for all $k \in [m] \cup \{0\}$. This is true because for any $p > 0$ and $k \in [m] \cup \{0\}$, we have
\begin{equation*}
    \begin{split}
        &\frac{\binom{m}{k}(p+r)^k(1 - r-p)^{m-k}}{\binom{m}{k}p^k(1 - p)^{m-k}} = \left(\frac{p+r}{p}\right)^k \left(\frac{1 - r-p}{1 - p}\right)^{m - k} \\
       & \geq \left(\frac{1 - r-p}{1 - p}\right)^m \geq \left(1 - r -p\right)^m \geq 1 - m(p+r) \geq 1 - \epsilon \left( 1/2 +  c (1.5 + 1/\alpha) \right)\geq 1- \epsilon
    \end{split}
\end{equation*}
where the last inequality holds as long as $c\leq \frac{1}{2(1.5 + 1/\alpha)}$. When $p = 0$, we also have $q_1(k) \geq 0$ for all $k \in [m] \cup \{0\}$ since $$(1 - r)^m \geq 1 - mr \geq 1 -  c\epsilon(1.5 + 1 / \alpha) \geq 1 - \epsilon,$$
where the last inequality holds as long as $c\leq \frac{1}{1.5 + 1/\alpha}$.
\end{itemize}

\vskip.2cm
{\noindent \bf (Part II: $p \in (\frac{1}{4m} \left( \frac{\alpha}{n} + \epsilon  \right),1/2 ] $)} We first divide the proof into two cases based on the value of $m$. 

\begin{itemize}[leftmargin=*]
	\item (Case 1: $ \frac{1}{\sqrt{m}} \geq \frac{1}{2} ( \frac{1}{\sqrt{\log n}}+\frac{1}{\sqrt{\log(1/\epsilon)}} ) $) In this case, $ \frac{1}{m} \geq \sqrt{\frac{p(1-p)}{m}}\left(\frac{1}{\sqrt{\log n}}+\frac{1}{\sqrt{\log(1/\epsilon)}}\right) $ for all $p \in [0,1/2]$. Thus, 
	\begin{equation*}
		\begin{split}
			\ell(n,\epsilon,m,p) \leq \frac{2}{m} \wedge p + \frac{1}{m} \left( \frac{1}{n} + \epsilon \right) \leq \frac{2}{m} \wedge p + \frac{1}{m} \wedge \frac{4}{\alpha} p \leq \frac{3}{m} \wedge \frac{5}{\alpha}p,
		\end{split}
	\end{equation*} where the second inequality is because of the regime of $p$ in this case. Next, we will show that when
	$r\leq c\ell(n,\epsilon,m,p) \leq c \left( \frac{3}{m} \wedge \frac{5}{\alpha} p \right)$ for some small enough $c > 0$, 
	$
		\inf_{Q_0,Q_1} \TV\left(P^{\otimes n}_{\epsilon_{\max},p - r,Q_0}, P^{\otimes n}_{\epsilon, p, Q_1}\right) \leq \alpha.
	$ Note that as long as $c \leq \alpha/5$, we have $p-r \geq 0$.
	Now, we construct $Q_1 = \mathrm{Binomial}(m, p)$. To construct $Q_0$, we define its probability mass function as
  \begin{equation}\label{q_0 for small p}
    q_0(k)=\frac{1}{\epsilon_{\max}}\binom{m}{k}p^{k}(1-p)^{m-k}-\frac{1-\epsilon_{\max}}{\epsilon_{\max}}\binom{m}{k}(p - r)^k (1 - (p - r))^{m-k}, \quad \forall k \in [m] \cup \{0\}.
\end{equation} As long as the formula \eqref{q_0 for small p} is valid probability mass function, the distributions $P_{\epsilon_{\max}, p - r, Q_0}$ and $P_{\epsilon, p, Q_1}$ exactly match, i.e., $\TV(P^{\otimes n}_{\epsilon_{\max},p - r,Q_0}, P^{\otimes n}_{\epsilon, p, Q_1}) = 0$ and it implies our result. Next, we verify $Q_0$ is a valid distribution. It is easy to check $\sum_{k=0}^{m} q_0(k) = 1$, so we just need to show $q_0(k) \geq 0$ for all $k \in [m] \cup \{0\}$ to confirm that $Q_0$ is a valid distribution.
\begin{equation}\label{Cond: lower bound first case}
    \begin{split}
        & q_0(k) \geq 0, \quad \forall k \in [m] \cup \{0\} 
    \Longleftrightarrow    \left(\frac{p - r}{p}\right)^k \left(\frac{1 - (p - r)}{1 - p}\right)^{m - k} \leq \frac{1}{1-\epsilon_{\max}}, \quad \forall k\in[m]\cup\{0\}\\
        \overset{(a)}{\Longleftrightarrow} & \left(\frac{1 - p + r}{1 - p}\right)^{m}\leq \frac{1}{1-\epsilon_{\max}} 
        \Longleftrightarrow  m\log\left( \frac{1 - p + r}{1 - p} \right) \leq \log \left(\frac{1}{1 - \epsilon_{\max}}\right)\\
        \overset{(b)}{\Longleftarrow} & \frac{mr}{1 - p} \leq \log\left(\frac{1}{1-\epsilon_{\max}}\right)
        \overset{(c)}\Longleftarrow  r \leq \frac{\log\left(\frac{1}{1-\epsilon_{\max}}\right)}{2m},
    \end{split}
\end{equation}
where (a) is because $\left(\frac{p - r}{p}\right)^k \left(\frac{1 - (p - r)}{1 - p}\right)^{m - k}$ is decreasing in $k$; (b) is because $\log (1+x) \leq x$ for all $x > - 1$; (c) is because $p \leq 1/2$. Notice that the last condition in \eqref{Cond: lower bound first case} is satisfied as long as we choose $c \leq \log\left(\frac{1}{1-\epsilon_{\max}}\right)/6$.

	\item (Case 2: $ \frac{1}{\sqrt{m}} < \frac{1}{2} ( \frac{1}{\sqrt{\log n}}+\frac{1}{\sqrt{\log(1/\epsilon)}} ) $) In this case, there exists $p^* \in [0,1/2]$ such that $ \frac{1}{m} = \sqrt{\frac{p^*(1-p^*)}{m}}\left(\frac{1}{\sqrt{\log n}}+\frac{1}{\sqrt{\log(1/\epsilon)}}\right)$. It is easy to verify $p^* \asymp \frac{\log n \wedge \log(1/\epsilon) }{m} \geq \frac{1}{4m} \left( \frac{\alpha}{n} + \epsilon  \right)$. 

    When $p \in (\frac{1}{4m} \left( \frac{\alpha}{n} + \epsilon  \right), C^* p^*]$, where $C^* > 0$ is some large constant which could depend on $\alpha$ and will be specified later, we have 
$
			\ell(n,\epsilon,m,p)\lesssim \frac{1}{m} \wedge p,
$ where the notation $\lesssim$ hide the dependence on $\alpha$ as well in this proof. Then following the analysis as in the Part II Case 1, we have as long as $c$ is sufficiently small, we can construct $Q_0$ and $Q_1$ such that $\TV\left(P^{\otimes n}_{\epsilon_{\max},p - r,Q_0}, P^{\otimes n}_{\epsilon, p, Q_1}\right)  = 0$.
	
Now, let us move to the case $p \in (C^* p^*, 1/2]$ (notice that if $C^* p^* > 1/2$, then this regime does not exist and we are done). Next, we consider two scenarios based on the magnitude of $n$ and $1/\epsilon$.
\vskip.2cm
{\noindent \bf (Scenario 1: $n \leq 1/\epsilon$)} In this case $\log n \leq \log(1/\epsilon)$. Thus $p > C^* p^* \geq C_1 \frac{\log (n / \alpha)}{m}$ for some $C_1$ depending on $\alpha$ and $C^*$. In addition, 
\begin{equation*}
		\begin{split}
			\ell(n,\epsilon,m,p)\lesssim \sqrt{\frac{p(1-p)}{m}}\left(\frac{1}{\sqrt{\log n}}+\frac{1}{\sqrt{\log(1/\epsilon)}}\right) \lesssim \sqrt{\frac{p(1-p)}{m}}\frac{1}{\sqrt{\log (n/\alpha)}}.
		\end{split}
	\end{equation*} Next, we are going to show that when $r \leq  c \sqrt{\frac{p (1 - p)}{m \log (n / \alpha)}}$ for some small enough $c>0$, then we can construct $Q_0$ and $Q_1$ such that $\TV\left(P^{\otimes n}_{\epsilon_{\max},p - r,Q_0}, P^{\otimes n}_{\epsilon, p, Q_1}\right)  \leq \alpha$.

In particular, we construct $Q_1 = \mathrm{Binomial}(m, p)$ and define the probability mass function of $Q_0$ by
\begin{equation*}
    q_0(k)=\frac{1}{\epsilon_{\max}}\binom{m}{k}p^k(1 - p)^{m - k} - \frac{1-\epsilon_{\max}}{\epsilon_{\max}}\frac{\indi \{k\geq mt_n\}}{\bbP(\mathrm{Binomial}(m, p - r)\geq mt_n)} \binom{m}{k}(p - r)^k(1 - p + r)^{m - k},
\end{equation*}
for all $ k\in[m]\cup\{0\}$, where 
\begin{equation*}
    t_n = p - 8\sqrt{\frac{p (1 - p) \log (n / \alpha)}{m}}.
\end{equation*}
Then, we have
\begin{equation} \label{ineq:t-positive-check}
    t_n \geq p - 8\sqrt{\frac{p \log (n / \alpha)}{m}} \geq \sqrt{\frac{C_1p\log (n / \alpha)}{m}} - 8 \sqrt{\frac{p\log (n / \alpha)}{m}} \geq 
    \Big(\sqrt{C_1} - 8\Big)\sqrt{\frac{p\log(n / \alpha)}{m}},
\end{equation}
where in the first inequality we use the fact $p \geq C_1\frac{\log (n / \alpha)}{m}$. Therefore, we have $t_n \geq 0$ as long as $C_1 \geq 64$, which also adds a lower bound on $C^*$. Also, when $n \geq 3$, we have
\begin{equation*}
    p - t_n = 8\sqrt{\frac{p (1 - p) \log(n / \alpha)}{m}} \geq 8\sqrt{\frac{p (1 - p)}{m \log(n / \alpha)}} \geq \frac{8r}{c}
\end{equation*}
where the first inequality is because $\log(n / \alpha) \geq \log(3) > 1$ and in the second inequality we use the fact that $r \leq c \sqrt{\frac{p (1 - p)}{m \log (n / \alpha)}}$. Thus, we have $(p - t_n)/2 > r$ when $c < 4$. In particular, this implies that $p > 2r \geq r$ when $c$ is sufficiently small.

Next, we derive the conditions for $Q_0$ to be a valid distribution. It is easy to check that $\sum_{k=0}^{m}q_0(k)=1$. Therefore, to ensure that $Q_0$ is a valid distribution,  we need to show that $q_0(k) \geq 0$ for all $k \in [m] \cup \{0\}$, which is guaranteed by
\begin{equation}\label{Cond: Lower bound with n}
    \begin{split}
        & p^k (1 - p)^{m - k} \geq  \frac{(1-\epsilon_{\max})(p - r)^k (1 - p + r)^{m - k}}{\bbP(\mathrm{Binomial}(m,p - r)\geq mt_n)}, \quad \forall k\geq mt_n\\
        \Longleftrightarrow & \left(\frac{p}{p - r}\right)^k\left(\frac{1 - p}{1-p + r}\right)^{m-k} \geq \frac{1-\epsilon_{\max}}{\bbP(\mathrm{Binomial}(m, p - r)\geq mt_n)}, \quad \forall k \geq mt_n \\
        \overset{(a)}\Longleftarrow & \left(\frac{p}{p - r}\right)^{mt_n}\left(\frac{1 - p}{1 - p + r}\right)^{m(1-t_n)} \geq \frac{1-\epsilon_{\max}}{\bbP(\mathrm{Binomial}(m, p - r)\geq mt_n)} \\
        \Longleftrightarrow & t_n\log\left(\frac{p}{p - r}\right) + (1 - t_n) \log \left(\frac{1 - p}{1- p + r}\right) \geq \frac{\log((1-\epsilon_{\max})/\bbP(\mathrm{Binomial}(m,p - r)\geq mt_n))}{m} \\
        \overset{(b)}{\Longleftarrow} & \frac{(t_n - p) r}{p(1 - p)} \geq \frac{\log((1-\epsilon_{\max})/\bbP(\mathrm{Binomial}(m,p - r)\geq mt_n))}{m} \\
        \Longleftrightarrow & r \leq \frac{p (1 - p)\log(\bbP(\mathrm{Binomial}(m, p - r)\geq mt_n)/(1-\epsilon_{\max}))}{m(p - t_n)},
    \end{split}
\end{equation}
where (a) is because $\left(\frac{p}{p - r}\right)^k\left(\frac{1 - p}{1-p + r}\right)^{m-k}$ is increasing in $k$; (b) is because $\log (1+x)\geq x/(1+x)$ for all $x>-1$. 
In addition, when $Q_0$ is a valid distribution, we have
\begin{equation} \label{Ineq: Bound for TV with n}
    \begin{split}
        &\TV(P_{\epsilon_{\max},p - r, Q_0},P_{\epsilon, p, Q_1}) \\
        &= \frac{1}{2}\sum_{k=0}^{m}(1-\epsilon_{\max})\left|1-\frac{\indi \{k\geq mt_n\}}{\bbP(\mathrm{Binomial}(m,p - r)\geq mt_n)}\right| \binom{m}{k}(p - r)^k(1 - p + r)^{m - k}\\
        &= (1-\epsilon_{\max})\bbP(\mathrm{Binomial}(m,p - r)< m t_n) \leq \bbP(\mathrm{Binomial}(m,p - r)< m t_n)\\
        &\overset{\textnormal{Lemma } \ref{Lem: Chernoff bound}}{\leq} \exp(-mD(\mathrm{Bernoulli}(t_n)\parallel \mathrm{Bernoulli}(p - r)))\\
        &\overset{(a)}{\leq}  \exp\left(-\frac{m(p - t_n - r)^2}{2}\left(\frac{1}{(\sqrt{t_n}+\sqrt{p - r})^2}+\frac{1}{(\sqrt{1 - t_n}+\sqrt{1 - p + r})^2}\right)\right)\\
        & \overset{(b)}\leq \exp\left(-\frac{m(p - t_n - (p - t_n)/2)^2}{2}\left(\frac{1}{4(p - r)}+\frac{1}{4(1-t_n)}\right)\right)\\
        & \overset{(c)}{\leq} \exp\left(-\frac{m(p - t_n)^2}{32}\left(\frac{1}{2p} + \frac{1}{2(1 - p)}\right)\right)  = \exp\left(-\frac{m(p-t_n)^2}{64p(1-p)}\right) \overset{(d)}{=} \frac{\alpha}{n},
    \end{split}
\end{equation}
where in (a) we use the property that squared Hellinger distance between two distributions is less than or equal to the Kullback-Leibler divergence; in (b) we use the fact that $(p - t_n)/2 > r$ when $c$ is sufficiently small, which implies both $t_n \leq p - r$ and $1 - t_n \geq 1 - p + r$; in (c) we use $p -r \leq 2p$ and $1-t_n \leq 1 \leq 2(1-p)$; (d) is by the choice of $t_n$. Then, the following holds by the property of TV distance on the product measure:
\begin{equation*}
    \TV\left(P^{\otimes n}_{\epsilon_{\max},p - r,Q_0},P^{\otimes n}_{\epsilon, p, Q_1}\right) \leq  n\TV(P_{\epsilon_{\max},p - r, Q_0},P_{\epsilon,p,Q_1}) \leq \alpha.
\end{equation*}
On the other hand, from the derivation of \eqref{Ineq: Bound for TV with n}, we also obtain the inequality $\mathbb{P}(\mathrm{Binomial}(m, p-r) \geq m t_n) \geq 1 - \frac{\alpha}{n}$, so the last condition in \eqref{Cond: Lower bound with n} for ensuring $Q_0$ is a valid distribution is implied by
\begin{equation*}
    \begin{split}
       & r \leq \frac{\log((1 - \frac{\alpha}{n})/(1-\epsilon_{\max}))p(1-p)}{m(p-t_n)} 
         \overset{(a)}{\Longleftarrow}  r \leq \frac{\log((1 -\epsilon_{\max} / 2)/(1-\epsilon_{\max}))p(1-p)}{m(p-t_n)} \\
        \overset{(b)} \Longleftrightarrow & r \leq  \frac{\log((1 - \epsilon_{\max} / 2)/(1-\epsilon_{\max}))}{8}\sqrt{\frac{p(1-p)}{m\log(n/\alpha)}},
    \end{split}
\end{equation*}
where (a) holds when $\epsilon_{\max}\geq \frac{2\alpha}{n}$ and (b) is by the setting of $t_n$. The above conditions are satisfied whenever $r \leq c \sqrt{\frac{p (1 - p)}{m \log (n / \alpha)}}$ for some sufficiently small $c > 0$ only depending on $\epsilon_{\max}$.

\vskip.2cm
{\noindent \bf (Scenario 2: $n \geq 1/\epsilon$)} In this case $\log n \geq \log(1/\epsilon)$. Thus $p > C^* p^* \geq C_2 \frac{\log (1/\epsilon)}{m}$ for some $C_2$ depending on $\alpha$ and $C^*$. In addition, 
\begin{equation*}
		\begin{split}
			\ell(n,\epsilon,m,p)\lesssim \sqrt{\frac{p(1-p)}{m}}\left(\frac{1}{\sqrt{\log n}}+\frac{1}{\sqrt{\log(1/\epsilon)}}\right) \lesssim \sqrt{\frac{p(1-p)}{m}}\frac{1}{\sqrt{\log (1/\epsilon)}}.
		\end{split}
	\end{equation*} Next, we are going to show that when $p\geq C_2\frac{\log(1 / \epsilon)}{m}$ and $r \leq  c \sqrt{\frac{p (1 - p)}{m \log (1/ \epsilon)}}$ for large $C_2>0$ and some small enough $c>0$, then $\inf_{Q_0, Q_1}\TV\left(P^{\otimes n}_{\epsilon_{\max},p - r,Q_0}, P^{\otimes n}_{\epsilon, p, Q_1}\right)  \leq \alpha$.

 In this case, let us define
 \begin{equation*}
     r_{\epsilon} = \frac{c(\epsilon_{\max})}{16}\sqrt{\frac{p (1 - p)}{m \log(1 / \epsilon)}},
 \end{equation*}
 where $c(\epsilon_{\max}) = \log((1 - \epsilon_{\max}/2)/ (1 - \epsilon_{\max}))$. It is sufficient to verify that when $\epsilon \in [0, \epsilon_{\max}/2]$, if $r \leq r_{\epsilon}$, then $\inf_{Q_0,Q_1} \TV\left(P^{\otimes n}_{\epsilon_{\max},p - r,Q_0},P^{\otimes n}_{\epsilon, p, Q_1}\right)\leq \alpha$. This is because when $\epsilon \in [\epsilon_{\max}/2, \epsilon_{\max}]$, there exist constants $C_2'$ and $c'$ such that when $p \geq C_2'\frac{\log(1 / \epsilon)}{m} \geq C_2\frac{\log(2 / \epsilon_{\max})}{m}$ and $r \leq c'\sqrt{\frac{p (1 - p)}{m \log(1 / \epsilon )}} \leq \frac{c(\epsilon_{\max})}{16}\sqrt{\frac{p (1 - p)}{m \log(2 / \epsilon_{\max})}} = r_{\epsilon_{\max}/2}$, then
\begin{equation*}
    \inf_{Q_0,Q_1} \TV\left(P^{\otimes n}_{\epsilon_{\max}, p - r, Q_0},P^{\otimes n}_{\epsilon, p, Q_1}\right) \leq \inf_{Q_0,Q_1} \TV\left(P^{\otimes n}_{\epsilon_{\max}, p - r,Q_0},P^{\otimes n}_{\epsilon_{\max}/2, p, Q_1}\right)\leq \alpha,
\end{equation*}
where the first inequality is because $\{P_{\epsilon_{\max}/2, p, Q_1} : Q_1\} \subseteq \{P_{\epsilon, p, Q_1}: Q_1\}$ when $\epsilon \geq \epsilon_{\max}/2$. From now on, we assume $\epsilon \in [0, \epsilon_{\max}/2]$. Then, we have
\begin{equation*}
    r_{\epsilon} =\frac{c(\epsilon_{\max})}{16}\sqrt{\frac{p (1 - p)}{m \log(1 / \epsilon)}} \leq \frac{c(\epsilon_{\max}) p}{16\sqrt{C_2}\log(1 / \epsilon)} \leq \frac{c(\epsilon_{\max}) p}{16 \sqrt{C_2}\log(2 / \epsilon_{\max})},
\end{equation*}
where the first inequality is because $p \geq C_2 \frac{\log(1 / \epsilon)}{m}$ and in the second inequality we use the fact that $\epsilon \leq \epsilon_{\max}/2$. Therefore, we have $r_\epsilon \leq p/2$ when $C_2 \geq \left(\frac{c(\epsilon_{\max})}{8\log(2 / \epsilon_{\max})}\right)^2$, which also adds a lower bound on $C^*$. Accordingly, we choose $C_2$ sufficiently large, depending only on $\epsilon_{\max}$, and consider the case where $r \leq r_{\epsilon} \leq p/2$ is satisfied. In this case, we define the probability mass function of $Q_1$ by
\begin{equation}\label{Def: q_1 with epsilon}
q_1(k)=
\begin{cases}
a_k & (1-\epsilon)p^k (1 - p)^{m - k} > (1-\epsilon_{\max})(p - r)^k (1 - p + r)^{m - k} \\
\frac{(1-\epsilon_{\max})\binom{m}{k}(p - r)^k (1 - p + r)^{m - k}}{\epsilon} & (1-\epsilon)p^k (1 - p)^{m - k} \leq (1-\epsilon_{\max})(p - r)^k (1 - p + r)^{m - k},
\end{cases}
\end{equation}
for all $k \in [m] \cup \{0\}$, where $a_k \geq 0$ are arbitrary nonnegative values chosen so that $\sum_{k=0}^m q_1(k) = 1$ if such a choice is possible. Suppose that the formula \eqref{Def: q_1 with epsilon} is a valid probability mass function. Then, we can define the valid probability mass function of $Q_0$ by
\begin{equation*}
    q_0(k) = \frac{(1-\epsilon)\binom{m}{k}p^k ( 1 - p)^{m - k} + \epsilon q_1(k) - (1-\epsilon_{\max})\binom{m}{k}(p - r)^k (1 - p + r)^{m - k}}{\epsilon_{\max}},
\end{equation*}
for all $k \in [m] \cup \{0 \}$. Then, the distributions $P_{\epsilon_{\max}, p-r ,Q_0}$ and $P_{\epsilon, p, Q_1}$ exactly match, which implies our result. Note that we can define $q_1(k)$ as in \eqref{Def: q_1 with epsilon} only if there exist nonnegative values $a_k \geq 0$ such that $\sum_{k=0}^m q_1(k) = 1$. This is guaranteed if 
\begin{equation}\label{Cond: Valid q_1 with epsilon}
    \sum_{k \in \mathcal{S}(p, r, \epsilon)} \frac{(1-\epsilon_{\max}) \binom{m}{k}(p - r)^k (1 - p + r)^{m - k}}{\epsilon} \leq 1,
\end{equation}
where
\begin{equation*}
\mathcal{S}(p, r, \epsilon) 
= \left\{k \in [m] \cup \{0\} : (1-\epsilon)p^k (1 - p)^{m - k} \leq (1-\epsilon_{\max})(p - r)^k (1 - p + r)^{m - k} \right\}.    
\end{equation*}
It is easy to check that
\begin{equation*}
    \mathcal{S}(p, r, \epsilon) 
= \left\{k \in [m] \cup \{0\} : 
\frac{k}{m} \leq 
\frac{
\log \left(\frac{1 - p + r}{1 - p}\right) - \frac{1}{m} \log \left( \frac{1 - \epsilon}{1 - \epsilon_{\max}} \right)
}{
\log \left(\frac{p}{p - r}\right) + \log \left(\frac{1 - p + r}{1 - p}\right) 
}
\right\}.
\end{equation*}
For $\epsilon \in [0, \epsilon_{\max}/2]$, we upper bound the threshold value appearing in the definition of $\mathcal{S}(p, r, \epsilon)$ as follows:
\begin{equation}\label{Ineq: Upper bound for threshold}
     \begin{split}
         &\frac{
\log \left(\frac{1 - p + r}{1 - p}\right) - \frac{1}{m} \log \left( \frac{1 - \epsilon}{1 - \epsilon_{\max}} \right)
}{
\log \left(\frac{p}{p - r}\right) + \log \left(\frac{1 - p + r}{1 - p}\right)} \overset{(a)}{\leq}  \frac{1}{\frac{\log \left(\frac{p}{p - r}\right)}{\log\left(\frac{1 - p + r}{1 - p}\right)} + 1} - \frac{\frac{c(\epsilon_{\max})}{m}}{\log \left(\frac{p}{p - r}\right) + \log \left(\frac{1 - p + r}{1 - p}\right) } \\
         \overset{(b)}{\leq} & p -\frac{c(\epsilon_{\max})(p - r) (1 - p)}{m r(1 - r)}
         \overset{(c)}{\leq}  p -\frac{c(\epsilon_{\max})p(1-p)}{2 m r}
         \leq   p -\frac{c(\epsilon_{\max})p(1-p)}{2 m r_{\epsilon}} \\
         = & p - 8\sqrt{\frac{p (1 - p) \log(1 / \epsilon)}{m}},
     \end{split}
\end{equation}
where (a) is because $\epsilon \leq \epsilon_{\max}/2$; in (b) we use the fact that $x/(1+x)\leq \log (1+x) \leq x$ when $x>-1$; (c) is because $r \leq p/2$. Now, define
\begin{equation*}
    t_\epsilon =  p - 8\sqrt{\frac{p (1 - p) \log(1 / \epsilon)}{m}}.
\end{equation*}
Since $p \geq C_2\frac{\log(1 / \epsilon)}{m}$, choosing $C_2 \geq 64$ ensures that $t_\epsilon \geq 0$ following the same analysis as in \eqref{ineq:t-positive-check}. Then for any $\epsilon \in [0, \epsilon_{\max}/2]$, we have
\begin{equation}\label{Ineq: r and p - t}
    \begin{split}
        2r_{\epsilon} & = \frac{\log((1 - \epsilon_{\max}/2)/(1 - \epsilon_{\max}))}{8}\sqrt{\frac{p(1 - p)}{m\log(1 / \epsilon)}} \\
        & = \frac{\log \Big( 1 + \frac{\epsilon_{\max}}{2(1 - \epsilon_{\max})}\Big)}{8}\sqrt{\frac{p(1 - p)}{m\log(1 / \epsilon)}} \overset{(a)}\leq \frac{\log (1 + \epsilon_{\max})}{8}\sqrt{\frac{p(1 - p)}{m\log(1 / \epsilon)}} \\
        & \overset{(b)}\leq \frac{\log (2/\epsilon_{\max})}{8}\sqrt{\frac{p(1 - p)}{m\log(1 / \epsilon)}} \overset{(c)}\leq \frac{\log (1/\epsilon)}{8}\sqrt{\frac{p(1 - p)}{m\log(1 / \epsilon)}} \leq 8\sqrt{\frac{p(1 - p)\log(1 / \epsilon)}{m}} = p - t_{\epsilon},
    \end{split}
\end{equation}
where (a) and (b) hold for all $\epsilon_{\max} \leq 1/2$, and in (c) we use the fact that $\epsilon \leq \epsilon_{\max}/2$. Then the condition \eqref{Cond: Valid q_1 with epsilon} is implied by
\begin{equation*} \label{Sufficient condition for valid q_1 Case 2}
    \begin{split}
        \eqref{Cond: Valid q_1 with epsilon} \Longleftrightarrow & \bbP\left(\mathrm{Binomial}(m, p - r) \leq m \frac{
\log \left(\frac{1 - p + r}{1 - p}\right) - \frac{1}{m} \log \left( \frac{1 - \epsilon}{1 - \epsilon_{\max}} \right)
}{
\log \left(\frac{p}{p - r}\right) + \log \left(\frac{1 - p + r}{1 - p}\right)}  \right)\leq \frac{\epsilon}{1-\epsilon_{\max}}\\
    \overset{\textnormal{Lemma }\ref{lm:Binomial-prop}\,(iii), \eqref{Ineq: Upper bound for threshold} }{\Longleftarrow} & \bbP \left(\mathrm{Binomial}(m, p - r_{\epsilon}) \leq m t_{\epsilon}\right)\leq \frac{\epsilon}{1-\epsilon_{\max}} \\
    \overset{(a)}{\Longleftarrow} & D\left(\mathrm{Bernoulli}(t_{\epsilon}) \| \mathrm{Bernoulli}(p - r_{\epsilon})\right)\geq \frac{\log\left((1-\epsilon_{\max})/\epsilon\right)}{m}\\
    \overset{(b)}{\Longleftarrow} & \frac{(p - r_{\epsilon} - t_{\epsilon})^2}{2}\left(\frac{1}{(\sqrt{t_{\epsilon}}+\sqrt{p -r_{\epsilon}})^2}+\frac{1}{(\sqrt{1-t_{\epsilon}}+\sqrt{1 - p + r_{\epsilon}})^2}\right) \geq \frac{\log\left( 1/\epsilon\right)}{m}\\
    \Longleftarrow & \frac{(p - r_{\epsilon} - t_{\epsilon})^2}{8}\left(\frac{1}{p - r_{\epsilon}}+\frac{1}{1-t_{\epsilon}}\right) \geq \frac{\log\left(1/\epsilon\right)}{m}
    \overset{(c)} \Longleftarrow  \frac{(p - r_{\epsilon} - t_\epsilon)^2}{16 p (1 - p)} \geq \frac{\log\left(1/\epsilon\right)}{m} \\
    \overset{\eqref{Ineq: r and p - t}} \Longleftarrow &  \frac{(p - t_\epsilon)^2}{64 p (1 - p)} \geq \frac{\log\left(1/\epsilon\right)}{m},
    \end{split}  
\end{equation*}
where in (a) we use the Chernoff bound for the binomial distribution (see Lemma~\ref{Lem: Chernoff bound}); in (b) we use the property that squared Hellinger distance between two distributions is less than or equal to the Kullback-Leibler divergence; in (c) we use the fact that $p - r_{\epsilon}\leq 2p$ and $1 - t_{\epsilon} \leq 1 \leq 2(1 -p)$. The above conditions are satisfied by the setting of $t_{\epsilon}$. This finishes the proof for Case 2 of Part~II and also finishes the proof of this theorem.
\end{itemize}


\subsection{Proof of Theorem \ref{thm:test-up}}

By symmetry, we only show the guarantee for $\phi_{p,\epsilon}^+$ as the guarantee for $\phi_{p,\epsilon}^-$ follows directly by the choice of $\underline{t}(p,\epsilon)$, $\underline{r}(p,\epsilon)$ and $\underline{\tau}(p,\epsilon)$. 
We divide the rest of the proof into two parts: simultaneous Type-1 error control and Type-2 error control.
\vskip.2cm

{\noindent \bf (Part 1: Simultaneous Type-1 error control)} Let us first focus on $p \in [0, 1-1/m]$. A key lemma we are going to use is the following and its proof is given in the subsequent subsections.

\begin{Lemma} \label{lm:r-t-property}
Suppose $ \frac{\log(2/\alpha)}{n} + \epsilon_{\max}$ is less than a sufficiently small universal constant. Then for the choices of $\overline{t}(p, \epsilon)$ and $\overline{r}(p, \epsilon)$ defined in (\ref{def:t}) and (\ref{def:r}), we have 
	\begin{itemize}
		\item (i) $\overline{t}(p, \epsilon) \in [p/2,p]$ and $1-\overline{t}(p, \epsilon) \in [1-p, \frac{3}{2} (1-p) ]$ for all $p \in [0,1-1/m]$ and all $ \epsilon \in [0, \epsilon_{\max}]$;
		\item (ii) $p + \bar{r}(p,\epsilon) \in [0,1]$ for all $p \in \left[0, 1\right]$ and all $\epsilon \in [0, \epsilon_{\max}]$, moreover, the function $p \mapsto p + \bar{r}(p,\epsilon)$ is strictly increasing in $p$ for $p \in [0,1]$ given any fixed $\epsilon \in [0, \epsilon_{\max}]$;
		\item (iii) For all $\epsilon \in [0, \epsilon_{\max}]$ and all $p\in \left[0,1-1/m\right]$, \begin{equation*} \label{ineq:r-t-property-1}
		P_{p + \overline{r}(p, \epsilon)}\left( X \leq m  \overline{t}(p, \epsilon) \right) > 10 \left( \epsilon + \sqrt{\frac{\log(24/\alpha)}{2n}} \right).
	\end{equation*} 
	In addition, for all $\epsilon \in [0, \epsilon_{\max}]$ and all $p\in \left(1- \frac{1}{m}, 1- \frac{4}{m} \left( \frac{10 \log(24/\alpha) }{n} + 3 \epsilon  \right) \right]$, 
	\begin{equation*} \label{ineq:r-t-property-2}
		P_{p}\left( X \leq m  \overline{t}(p, \epsilon) \right) \geq 6 \epsilon + \frac{20\log(24/\alpha)}{n}.
	\end{equation*}
	\end{itemize}
\end{Lemma}  
To show the simultaneous Type-1 error control in this regime, we show the equivalent statement: for all $p \in [0,1-1/m]$,
\begin{equation*}
    \underset{Q}{\inf} P_{\epsilon_{\max},p,Q}(\phi_{p, \epsilon}^+=0 \text{ for all }\epsilon\in[0,\epsilon_{\max}])\geq 1-\alpha/12.
\end{equation*} First, by the DKW inequality (see Lemma \ref{lm:DKW}), we have with probability at least $1 - \alpha/12$, the following event holds:
	\begin{equation*}
		(E) = \{ \textnormal{for all } x \in \bbR, |F_n(x) - P_{\epsilon_{\max}, p, Q}(X \leq x)| \leq \sqrt{ \log (24/\alpha)/(2n) } \}.
	\end{equation*} Then for any $Q$ and $p \in [0,1-1/m]$,
\begin{equation} \label{ineq:type-1-regime1}
	\begin{split}
		&\phi_{p, \epsilon}^+=0 \text{ for all }\epsilon\in[0,\epsilon_{\max}]
		\Longleftrightarrow  F_n(m \bar{t}(p, \epsilon ) ) \geq \bar{\tau}(p, \epsilon), \forall \epsilon\in[0,\epsilon_{\max}] \\
		\overset{(E)}\Longleftarrow &P_{\epsilon_{\max}, p, Q }(X \leq m \bar{t}(p, \epsilon )) \geq \bar{\tau}(p, \epsilon) + \sqrt{ \frac{\log(24/\alpha)}{2n} }, \forall \epsilon\in[0,\epsilon_{\max}]\\
		\Longleftarrow & (1 - \epsilon_{\max} ) P_{p}(X \leq m\bar{t}(p, \epsilon )) \geq \bar{\tau}(p, \epsilon) + \sqrt{ \frac{\log(24/\alpha)}{2n} }, \forall \epsilon\in[0,\epsilon_{\max}] \\
		\overset{(a)}\Longleftrightarrow &  (1 - \epsilon_{\max} ) P_{p}(X \leq m\bar{t}(p, \epsilon )) \geq \frac{11}{10} P_{p + \overline{r}(p, \epsilon)} \left( X \leq m \overline{t}(p,\epsilon) \right) + \sqrt{ \frac{\log(24/\alpha)}{2n} }, \forall \epsilon\in[0,\epsilon_{\max}] \\
		\overset{\text{Lemma } \ref{lm:r-t-property}\,(iii) }\Longleftarrow &  (1 - \epsilon_{\max} ) P_{p}(X \leq m\bar{t}(p, \epsilon )) > \frac{12}{10} P_{p + \overline{r}(p, \epsilon)} \left( X \leq m \overline{t}(p,\epsilon) \right), \forall \epsilon\in[0,\epsilon_{\max}]\\
		\Longleftrightarrow &  P_{p}(X \leq m\bar{t}(p, \epsilon )) > \frac{12}{10(1 - \epsilon_{\max} )} P_{p + \overline{r}(p, \epsilon)} \left( X \leq m \overline{t}(p,\epsilon) \right), \forall \epsilon\in[0,\epsilon_{\max}],
	\end{split}
\end{equation} where (a) is by the definition of $\bar{\tau}(p, \epsilon)$. When $p = 0$, the sufficient condition at the end of \eqref{ineq:type-1-regime1} is implied by the following condition:
\begin{equation*}
	\begin{split}
		P_0(X = 0) > \frac{12}{10(1 - \epsilon_{\max} )} P_{1/(2m)}(X = 0) & \Longleftrightarrow 1 > \frac{12}{10(1 - \epsilon_{\max} )} \left( 1- \frac{1}{2m} \right)^m 
		\Longleftarrow 1 > \frac{12e^{-1/2}}{10(1 - \epsilon_{\max} )},
	\end{split}
\end{equation*} which is satisfied as long as $\epsilon_{\max}$ is small. For $p \in (0, 1-1/m]$, the sufficient condition at the end of \eqref{ineq:type-1-regime1} is implied by
\begin{equation} \label{ineq:type-1-regime2}
    \begin{split}
    &P_{p}(X \leq m\bar{t}(p, \epsilon )) > \frac{12}{10(1 - \epsilon_{\max} )} P_{p + \overline{r}(p, \epsilon)} \left( X \leq m \overline{t}(p,\epsilon) \right)\\
       \Longleftrightarrow &\sum_{k\leq m\bar{t}(p,\epsilon)}\binom{m}{k}p^k(1-p)^{m-k}>  \frac{12}{10(1 - \epsilon_{\max} )} \sum_{k\leq m\bar{t}(p,\epsilon)}\binom{m}{k}(p+\bar{r}(p,\epsilon))^k(1-p-\bar{r}(p,\epsilon))^{m-k}\\
        \Longleftarrow& p^k(1-p)^{m-k} > \frac{12}{10(1 - \epsilon_{\max} )} (p+\bar{r}(p,\epsilon))^k(1-p-\bar{r}(p,\epsilon))^{m-k}, \quad \forall k\leq m\bar{t}(p,\epsilon)\\
        \Longleftrightarrow& \min_{k\leq m\bar{t}(p,\epsilon)} \left(\frac{p}{p+\bar{r}(p,\epsilon)}\right)^k\left(\frac{1-p}{1-p-\bar{r}(p,\epsilon)}\right)^{m-k}> \frac{12}{10(1 - \epsilon_{\max} )} \\
        \Longleftarrow&\left(\frac{p}{p+\bar{r}(p,\epsilon)}\right)^{m\bar{t}(p,\epsilon)}\left(\frac{1-p}{1-p-\bar{r}(p,\epsilon)}\right)^{m(1-\bar{t}(p,\epsilon))}>  \frac{12}{10(1 - \epsilon_{\max} )} \\
        \Longleftrightarrow & \bar{t}(p,\epsilon)\log\left(1-\frac{\bar{r}(p,\epsilon)}{p+\bar{r}(p,\epsilon)}\right)+(1-\bar{t}(p,\epsilon))\log\left(1+\frac{\bar{r}(p,\epsilon)}{1-p-\bar{r}(p,\epsilon)}\right)> \frac{\log\left(  \frac{12}{10(1 - \epsilon_{\max} )} \right)}{m} \\
        \overset{(a)}{\Longleftarrow}&-\bar{t}(p,\epsilon)\frac{\bar{r}(p,\epsilon)}{p}+(1-\bar{t}(p,\epsilon))\frac{\bar{r}(p,\epsilon)}{1-p} > \frac{\log\left(  \frac{12}{10(1 - \epsilon_{\max} )} \right)}{m} 
        \overset{(b)}\Longleftarrow  \bar{r}(p,\epsilon) \geq \frac{p(1-p)}{4m(p-\bar{t}(p,\epsilon))},
    \end{split}
\end{equation} where (a) is because $\log(1+x)\geq x/(1+x)$ for all $x>-1$ and (b) is because $\log\left(  \frac{12}{10(1 - \epsilon_{\max} )} \right) < 1/4$ as long as $\epsilon_{\max} \leq 0.05$. Notice that the last condition in \eqref{ineq:type-1-regime2} is satisfied by the choice of $\bar{r}(p,\epsilon)$ for all $\epsilon\in[0,\epsilon_{\max}]$ and all $p \in (0,1-1/m]$. Thus, we have shown the simultaneous Type-1 error control for all $p \in [1-1/m]$.

Now, we move onto the regime $p \in (1-1/m,1]$. A key lemma we are going to use is the following one regarding the estimation under the corrupted Bernoulli model and its proof is given in the subsequent subsections.
\begin{Lemma} \label{lm:bernoulli-estimation} Suppose $X_1, \ldots, X_n \overset{i.i.d.}\sim (1 - \epsilon )\textnormal{Bernoulli}(p) + \epsilon Q $. Then for any $\epsilon \in [0,1/4]$,  we have
\begin{equation*}\label{ineq:ber-est-pro2}
	\bbP\left(\frac{\sum_{i=1}^n \indi\{X_i = 1\} }{n} \leq \frac{1}{2} p - \frac{3\log(2/\alpha)}{n} \right) \leq \alpha.
\end{equation*}
\end{Lemma}

Now we prove the simultaneous Type-1 error control for $p \in (1-1/m,1]$. Note that in this case
\begin{equation*}
\begin{split}
	\phi_{p,\epsilon}^+ &= \indi\left\{ \frac{1}{n} \sum_{i=1}^n \indi \{X_i \leq m-1\} < \frac{1}{2}(1-p^m) - \frac{3\log(24/\alpha)}{n} \right\},
\end{split}
\end{equation*}
and it is independent of $\epsilon$. In addition, given $X_1, \ldots, X_n \overset{i.i.d.}\sim P_{\epsilon_{\max}, p, Q}$, $\indi\{X_1 \leq m - 1 \}, \ldots, \indi\{X_n \leq m - 1 \} \overset{i.i.d.}\sim (1- \epsilon_{\max})\textnormal{Bernoulli}(1 - p^m ) + \epsilon_{\max} \textnormal{Bernoulli}( Q(X \leq m - 1))$. Thus for all $p \in (1-1/m,1]$,
\begin{equation*}
	\begin{split}
		\underset{Q}{\sup}P_{\epsilon_{\max},p,Q}\left(\underset{\epsilon\in[0,\epsilon_{\max}]}{\sup}\phi_{p,\epsilon}^+ = 1 \right) 
		& = \underset{Q}{\sup}P_{\epsilon_{\max},p,Q}\left(\frac{1}{n} \sum_{i=1}^n \indi \{X_i \neq m\} < \frac{1}{2}(1-p^m) - \frac{3\log(24/\alpha)}{n} \right)\\
		& \overset{\text{Lemma } \ref{lm:bernoulli-estimation} }\leq \alpha/12.
	\end{split}
\end{equation*} This finishes the proof for the simultaneous Type-1 error control for all $p \in [0,1]$. 

\vskip.2cm

{\noindent \bf (Part 2: Type-2 error control)} First, we note that as we have proved in Lemma \ref{lm:r-t-property} (ii), $p + \bar{r}(p, \epsilon ) \in [0,1]$, so the quantity $P_{\epsilon,p + r,Q}(\phi_{p, \epsilon}^+ = 0)$ is well defined. Also note that given any $\epsilon \in [0, \epsilon_{\max}]$, $p \in \left[0, 1 - \frac{4}{m}\left( \frac{10 \log(24/\alpha)}{n} + 3 \epsilon \right) \right ]$ and $r \in [\bar{r} (p, \epsilon), 1 - p]$,
\begin{equation} \label{ineq:type2 with rbar}
    \begin{split}
        & P_{\epsilon,p + r,Q}(\phi_{p, \epsilon}^+ = 0)= P_{\epsilon,p + r,Q}\left(\frac{1}{n}\sum_{i=1}^n \indi \{X_i\leq m\bar{t}(p,\epsilon)\} \geq \bar{\tau}(p,\epsilon)\right)\\
        & = (1 - \epsilon ) P_{p+r} \left(\frac{1}{n}\sum_{i=1}^n \indi \{X_i\leq m\bar{t}(p,\epsilon)\} \geq \bar{\tau}(p,\epsilon)\right) + \epsilon Q\left(\frac{1}{n}\sum_{i=1}^n \indi \{X_i\leq m\bar{t}(p,\epsilon)\} \geq \bar{\tau}(p,\epsilon)\right) \\
    & \overset{(a)}\leq  (1 - \epsilon ) P_{p+\bar{r}(p, \epsilon)} \left(\frac{1}{n}\sum_{i=1}^n \indi \{X_i\leq m\bar{t}(p,\epsilon)\} \geq \bar{\tau}(p,\epsilon)\right) + \epsilon Q\left(\frac{1}{n}\sum_{i=1}^n \indi \{X_i\leq m\bar{t}(p,\epsilon)\} \geq \bar{\tau}(p,\epsilon)\right) \\
    &= P_{\epsilon,p + \bar{r}(p, \epsilon),Q}(\phi_{p, \epsilon}^+=0),
    \end{split}
\end{equation} where in (a) we use the fact that Binomial($m,p+r$) stochastically dominates Binomial($m,p+\bar{r}(p, \epsilon)$), i.e., there exists a coupling of $X$ and $Y$ such that $X \sim \textnormal{Binomial}(m, p+r)$, $Y \sim \textnormal{Binomial}(m, p+\bar{r}(p, \epsilon))$, and $X \geq Y$ almost surely. Therefore, to demonstrate Type-2 error control, it suffices to show that $ P_{\epsilon,p + \bar{r}(p, \epsilon),Q}(\phi_{p, \epsilon}^+=0)\leq \alpha/12$ for all $\epsilon \in [0, \epsilon_{\max}]$, all $p \in \left[0, 1 - \frac{4}{m}\left( \frac{10 \log(24/\alpha)}{n} + 3 \epsilon \right) \right ]$.

 Again, we first focus on $p \in [0, 1-1/m]$. In this case, for any $\epsilon \in [0, \epsilon_{\max}]$ and $Q$,
\begin{equation} \label{ineq:type-2-regime1}
	\begin{split}
		&P_{\epsilon,p + \bar{r}(p, \epsilon),Q}(\phi_{p, \epsilon}^+ = 0) = P_{\epsilon,p + \bar{r}(p, \epsilon),Q}\left(\frac{1}{n}\sum_{i=1}^n \indi \{X_i\leq m\bar{t}(p,\epsilon)\} \geq \bar{\tau}(p,\epsilon)\right) \\
		& = P_{\epsilon,p + \bar{r}(p, \epsilon),Q}\Bigg(\frac{1}{n}\sum_{i=1}^n \indi \{X_i\leq m\bar{t}(p,\epsilon)\} - P_{\epsilon,p + \bar{r}(p, \epsilon),Q}(X \leq m\bar{t}(p,\epsilon)) \\
		& \quad \quad \quad \quad \quad \quad \geq \bar{\tau}(p,\epsilon) - P_{\epsilon,p + \bar{r}(p, \epsilon),Q}(X \leq m\bar{t}(p,\epsilon))\Bigg).
	\end{split}
\end{equation} By Hoeffding's inequality (see Lemma \ref{lm:Binomial-prop} (i)), we know that the probability in \eqref{ineq:type-2-regime1} is bounded by $\alpha/12$ if 
\begin{equation} \label{ineq:type-2-regime1-suff-cond}
	\begin{split}
		& \bar{\tau}(p,\epsilon) - P_{\epsilon,p + \bar{r}(p, \epsilon),Q}(X \leq m\bar{t}(p,\epsilon)) \geq \sqrt{ \frac{\log(24/\alpha)}{2n} } \\
		\overset{(a)}\Longleftarrow & \bar{\tau}(p,\epsilon) \geq \epsilon +  \sqrt{ \frac{\log(24/\alpha)} {2n} } + P_{p + \bar{r}(p, \epsilon)}(X \leq m\bar{t}(p,\epsilon)) \\
		\overset{(b)}\Longleftrightarrow & \frac{1}{10} P_{p + \bar{r}(p, \epsilon)}(X \leq m\bar{t}(p,\epsilon)) \geq \epsilon +  \sqrt{ \frac{\log(24/\alpha)} {2n} }
	\end{split}
\end{equation} where (a) is because $P_{\epsilon,p + \bar{r}(p, \epsilon),Q}(X \leq m\bar{t}(p,\epsilon))  \leq \epsilon + P_{p + \bar{r}(p, \epsilon)}(X \leq m\bar{t}(p,\epsilon))  $ and in (b) we plug in the definition of $\bar{\tau}(p, \epsilon)$ when $p \in [0,1-1/m]$. Notice that the last condition at the end of \eqref{ineq:type-2-regime1-suff-cond} is satisfied for all $\epsilon \in [0, \epsilon_{\max}]$ and all $p \in [0, 1-1/m]$ by Lemma \ref{lm:r-t-property} (iii). Thus, we have shown that $ P_{\epsilon,p + \bar{r}(p, \epsilon),Q}(\phi_{p, \epsilon}^+=0)\leq \alpha$ for all $\epsilon \in [0, \epsilon_{\max}]$ and all $p \in [0, 1-1/m]$.

Now, let us show the Type-2 error control when $p \in \left(1-1/m, 1 - \frac{4}{m}\left( \frac{10 \log(24/\alpha)}{n} + 3 \epsilon \right) \right ]$. For any $\epsilon \in [0, \epsilon_{\max}]$ and $Q$,
\begin{equation}\label{ineq:type-2-regime2}
	\begin{split}
&P_{\epsilon,p + \bar{r}(p, \epsilon),Q}(\phi_{p, \epsilon}^+ = 0) = P_{\epsilon,p + \bar{r}(p, \epsilon),Q}\left(\frac{1}{n}\sum_{i=1}^n \indi \{X_i\leq m\bar{t}(p,\epsilon)\} \geq \bar{\tau}(p,\epsilon)\right) \\
		& = P_{\epsilon,p + \bar{r}(p, \epsilon),Q}\left(\frac{1}{n}\sum_{i=1}^n \indi \{X_i \neq m\} \geq  \frac{1}{2}(1 - p^m) - \frac{3 \log(24/\alpha)}{n} \right)\\
		& = P_{\epsilon,p + \bar{r}(p, \epsilon),Q}\Bigg(\frac{1}{n}\sum_{i=1}^n \indi \{X_i \neq m\} - P_{\epsilon,p + \bar{r}(p, \epsilon),Q}(X \neq m) \\
		& \quad \quad \quad \geq  \frac{1}{2}(1 - p^m) - \frac{3 \log(24/\alpha)}{n} - P_{\epsilon,p + \bar{r}(p, \epsilon),Q}(X \neq m) \Bigg).
	\end{split}
\end{equation} By Bernstein's inequality (see Lemma \ref{lm:Binomial-prop} (ii) and take $C$ there to be $2$, $\alpha = \alpha/12$), we have the last probability in \eqref{ineq:type-2-regime2} is bounded by $\alpha/12$ if 
\begin{equation}\label{ineq:type-2-regime2-suff-cond}
	\begin{split}
		&\frac{1}{2}(1 - p^m) - \frac{3 \log(24/\alpha)}{n} - P_{\epsilon,p + \bar{r}(p, \epsilon),Q}(X \neq m) \geq \frac{1}{2} P_{\epsilon,p + \bar{r}(p, \epsilon),Q}(X \neq m) + \frac{2 \log(24/\alpha)}{n}\\
		&\Longleftarrow  \frac{1}{2}(1 - p^m)  \geq \frac{3}{2} P_{\epsilon,p + \bar{r}(p, \epsilon),Q}(X \neq m) + \frac{5 \log(24/\alpha)}{n}\\
		& \Longleftarrow 1 - p^m  \geq 3 P_{p + \bar{r}(p, \epsilon)}(X \neq m) + 3 \epsilon + \frac{10 \log(24/\alpha)}{n}  \Longleftarrow \left\{ \begin{array}{l}
			1 - p^m \geq 6 \epsilon + \frac{20\log(24/\alpha)}{n} \\
			1-p^m \geq 6 P_{p+\bar{r}(p, \epsilon)}(X \neq m)
		\end{array} \right. \\
		& \overset{(a)}\Longleftrightarrow \left\{ \begin{array}{l}
			P_p(X \leq m \bar{t}(p, \epsilon) ) \geq 6 \epsilon + \frac{20\log(24/\alpha)}{n} \\
			1-p^m \geq 6 (1 - (p +\bar{r}(p, \epsilon))^m ),
		\end{array} \right.
	\end{split}
\end{equation} where in (a) we plug in the definition of $\bar{t}(p, \epsilon)$ in this regime. Notice that the first condition at the end of \eqref{ineq:type-2-regime2-suff-cond} is satisfied for all $p \in \left(1-1/m, 1 - \frac{4}{m}\left( \frac{10 \log(24/\alpha)}{n} + 3 \epsilon \right) \right ]$ as we have shown in Lemma~\ref{lm:r-t-property}~(iii). Next, we are going to show that given any $\epsilon \in [0, \epsilon_{\max}]$, $$1-p^m \geq 6 (1 - (p +\bar{r}(p, \epsilon))^m ) \,\text{ holds for all }\,  p \in (1-1/m,1] \text{  with  } \bar{r}(p, \epsilon) = (1 - 1/(6e) )(1-p).$$

When $m = 1$,
\begin{equation*}
	\begin{split}
		1-p^m \geq 6 (1 - (p +\bar{r}(p, \epsilon))^m ) \Longleftrightarrow 1-p \geq 6 (1 - (p +\bar{r}(p, \epsilon)) ) \Longleftrightarrow 1-p \geq \frac{1}{e} (1-p)
	\end{split}
\end{equation*} and it clearly holds. Now we consider $m \geq 2$. Let $f(p) = 1-p^m - 6\left( 1- ( p +(1-1/(6e))(1-p) )^m \right)$. Then $f(1) = 0$. If we can show $f'(p) \leq 0$ for all $p \in (1-1/m,1]$, then it implies that $f(p) \geq 0$ for all $p \in (1-1/m,1]$. 
\begin{equation*}
	\begin{split}
		f'(p)& = mp^{m-1} \left( \frac{1}{e} \left( 1 + \frac{1 - 1/(6e)}{p} - (1-1/(6e)) \right)^{m-1} - 1  \right) \\
		& \leq  mp^{m-1} \left( \frac{1}{e} \left( 1 + \frac{1 - 1/(6e)}{1-1/m} - (1-1/(6e)) \right)^{m-1} - 1  \right) \\
		& =  mp^{m-1} \left( \frac{1}{e} \left( 1 + \frac{1 - 1/(6e)}{m-1}  \right)^{m-1} - 1  \right) \leq mp^{m-1} \left( \frac{1}{e} e^{1-1/(6e)} - 1  \right) = mp^{m-1} \left( e^{-1/(6e)} - 1 \right) \\
		& < 0.
	\end{split}
\end{equation*} This shows $f(p) \geq 0$ for all $p \in (1-1/m,1]$ and finishes the proof for the Type-2 error control. This also finishes the proof of this theorem.

\subsubsection{Proof of Lemma \ref{lm:r-t-property}}

For convenience, let us denote $A = \epsilon + \sqrt{\frac{\log(24/\alpha)}{2n}}$. By assumption, $A$ is less than a sufficiently small constant. We will also simply write $\overline{t}(p, \epsilon)$ and $\overline{r}(p, \epsilon)$ as $t$ and $r$ in the proof, while we should keep in mind that they depend on $(p, \epsilon)$. Recall that when $p \in [0, 1-1/m]$,
\begin{equation*}
	t =  p - \min \left\{\frac{p (1 - p)}{2}, \frac{1}{8} \sqrt{\frac{p (1 - p) \log\left(1/A\right)}{m}} \right\}
\end{equation*} and $r = \frac{1}{2m}$ when $p = 0$ and $r = \frac{p(1-p)}{4m(p - t )}$ when $p \in (0, 1-1/m]$. Now let us also define two quantities $0\leq p_1 \leq 1/2 \leq p_2 \leq 1$ as follows: if $\log(1/A) \geq 4m$, set $p_1 = p_2 = 1/2$ and if $\log(1/A) < 4m$, set $0\leq p_1 < 1/2 < p_2 \leq 1$ be the solution of the equation
\begin{equation*}
	\log(1/A) = 16 m p(1-p).
\end{equation*} Notice that when $p = 1/m$ or $1-1/m$, $16 m p(1-p) \leq 16 \leq \log(1/A)$ as $A$ is less than a sufficiently small constant. As a result, we have $1/m \leq p_1 \leq 1/2 \leq p_2 \leq 1-1/m$ when $m \geq 2$. By the same argument, we also have $p_2 \leq 1-2/m$ when $m \geq 4$.
\vskip.2cm
{\noindent \bf (Part I: Proof of Claim (i))} When $p \in [0,1-1/m]$, $t \leq p$ as $\log(1/A) \geq 0$. At the same time, 
\begin{equation*}
	\begin{split}
		&t \geq p - p(1-p)/2 \geq p - \frac{\min \{p, 1-p\}}{2} = \max \{ p/2, (3p-1)/2 \},
	\end{split}
\end{equation*} and it implies that $1 - p \leq 1 - t \leq \frac{3}{2}(1 - p)$. 

\vskip.2cm
{\noindent \bf (Part II: Proof of Claim (ii))} We first show that the function $p \mapsto p + \bar{r}(p,\epsilon)$ is strictly increasing in $p$ for $p \in [0,1]$ given any fixed $\epsilon \in [0, \epsilon_{\max}]$. 
\begin{itemize}[leftmargin=*]
	\item (Case 1: $\log(1/A) \geq 4m$) By Lemma \ref{lm:additional-r-property} (ii),
	\begin{equation*}
		\begin{split}
			p + r = \left\{ \begin{array}{ll}
				p+\frac{1}{2m}, & \textnormal{ if } p \in [0,1-1/m],\\
				1 + \frac{1}{6e} p - \frac{1}{6e}, & \textnormal{ if } p \in (1-1/m, 1].
			\end{array}  \right.
		\end{split}
	\end{equation*} It is easy to check that the function $p \mapsto p + r$ is strictly increasing in $p$ for $p \in [0,1]$ given any fixed $\epsilon \in [0, \epsilon_{\max}]$. Notice that $p + r \in [0,1]$ for any $p, \epsilon \in [0,1]$.
	\item (Case 2: $\log(1/A) < 4m$) In this case, by Lemma \ref{lm:additional-r-property} (iii), 
\begin{equation*}
	\begin{split}
		p + r = \left\{ \begin{array}{ll}
				 \frac{1}{2m} + p , & \textnormal{ if } p \in [0,p_1],\\
				2 \sqrt{ \frac{p(1-p)}{m \log(1/A)} } + p, & \textnormal{ if } p \in (p_1, p_2), \\
				\frac{1}{2m} + p, & \textnormal{ if } p \in [p_2,1-1/m], \\
				1 + \frac{1}{6e} p - \frac{1}{6e}, & \textnormal{ if } p \in (1-1/m, 1].
			\end{array}  \right.
	\end{split}
\end{equation*} It is easy to check $p+r$ is increasing with respect to $p$ when $p \in [p_2,1]$. In addition, by design, $p+r$ is a continuous function with respect to $p$ when $p \in [0,1-1/m]$. So to prove the claim, we just need to show $p+r$ is increasing with respect to $p$ when $p \in [0,1-1/m]$. In view of the expression of $p+r$, it is sufficient to show $p+r$ is increasing with respect to $p$ when $p \in (p_1, p_2)$. Notice that when $p \in (p_1, p_2)$, by construction $\log(1/A) < 16mp(1-p)$. Then
\begin{equation*}
	\begin{split}
		\frac{\partial (p+r)}{\partial p} &= 1 + \left(\frac{p(1-p)}{m \log(1/A)} \right)^{-1/2} \frac{1-2p}{m \log(1/A)} = 1 + \frac{1-2p}{\sqrt{mp(1-p) \log(1/A)}} \\
		& \left\{ \begin{array}{ll}
			\geq 0, & \text{ if } p\in (p_1, \frac{1}{2}],\\
			  > 1 - \frac{4(2p-1)}{\log(1/A)} \geq 0, & \text{ if } p \in (1/2, p_2),
		\end{array} \right.
	\end{split}
\end{equation*} where the inequality in $p \in (1/2, p_2)$ is because $A$ is less than a sufficiently small constant.
 
\end{itemize}

\vskip.2cm
{\noindent \bf (Part III: Proof of Claim (iii))}
First, notice that when $m = 1$, $[0,1-1/m] = \{0\}$. In addition, when $p = 0$, we have $t =0$ and $r = 1/2$. So
\begin{equation*}
	P_{p+r}(X \leq m t) = P_{1/2}( X = 0) = 1/2 > 10 A
\end{equation*} since $A$ is less than a sufficiently small constant. For the rest of the proof, we focus on the setting $m$ is a positive integer greater than or equal to $2$ and we divide the rest of the proof into four cases. 
\vskip.2cm
\begin{itemize}[leftmargin=*]
	\item (Case 1: $p \in [0,p_1]$) In this case, $t = p - \frac{p (1 - p)}{2}$, $r = 1/(2m)$ and $\log(1/A) \geq 16mp_1 (1-p_1) \geq 8mp_1$ where the second inequality is because $p_1 \leq 1/2$. Then
\begin{equation*}
	\begin{split}
		P_{p+r}(X \leq m t) &= P_{p+ \frac{1}{2m} } \left( X \leq m t \right) \geq P_{p+ \frac{1}{2m} }\left( X = 0 \right) = \left( 1-p - \frac{1}{2m} \right)^m  \geq \left( 1-p_1 - \frac{1}{2m} \right)^m.
	\end{split}
\end{equation*} A sufficient condition for $\left( 1-p_1 - \frac{1}{2m} \right)^m > 10 A$ is derived as follows:
\begin{equation*}
	\begin{split}
		 & \left( 1-p_1 - \frac{1}{2m} \right)^m > 10 A
		\Longleftrightarrow  m \log \left( 1-p_1 - \frac{1}{2m} \right) > \log(A) + \log(10) \\
		\Longleftarrow & m \frac{-p_1 - \frac{1}{2m} }{1 - p_1 - \frac{1}{2m}} > \log(A) + \log(10) \quad (\textnormal{as }\log(1+x) \geq \frac{x}{1+x}, \forall x > -1 )\\
		\Longleftrightarrow & \log(1/A) - \frac{mp_1 + \frac{1}{2} }{1 - p_1 - \frac{1}{2m}} > \log (10) \\ 
		\Longleftarrow & \log(1/A) - \frac{mp_1 + \frac{1}{2} }{1 - 1/2 - 1/4} > \log (10) \quad (\textnormal{as } m \geq 2, p_1 \leq 1/2)\\
		 \Longleftarrow & \log(1/A) - \frac{1}{2} \log(1/A) > \log(10) + 2 \quad (\textnormal{as } \log(1/A) \geq 8mp_1) 
		 \Longleftrightarrow  \log(1/A) > 4 + 2 \log(10).
	\end{split}
\end{equation*} Notice that the last condition above is satisfied since $A$ is less than a sufficiently small constant. So we have shown $P_{p+r}(X \leq m t) > 10A$ when $p \in [0, p_1]$. 
\item (Case 2: $p \in (p_1,p_2)$) Note that when $\log(1/A) \geq 4m$, this case does not exist. But when this regime is nonempty, we have 
\begin{equation} \label{ineq:t-r-range-case2}
	\begin{split}
		&t = p - \frac{1}{8} \sqrt{\frac{p (1 - p) \log\left(1/A\right)}{m}},\,\log(1/A) < 16mp(1-p)\, \textnormal{ and } \,  r = 2 \sqrt{ \frac{p(1-p)}{m \log(1/A)} }.
	\end{split}
\end{equation} Moreover, since $A$ is less than a sufficiently small constant, we have 
\begin{equation}\label{ineq:t-r-range-prop3-case2}
	\begin{split}
		p +r -t \leq 2 (p - t) \quad \textnormal{ and } \quad 1 - p -r \geq \frac{1}{2}(1-p),
	\end{split}
\end{equation} where the first inequality is easy to check and the second inequality is because 
\begin{equation*}
	\begin{split}
		1 - p -r \geq \frac{1}{2}(1-p) \Longleftrightarrow 1 -p \geq 2 r \Longleftrightarrow 1- p \geq 4 \sqrt{ \frac{p(1-p)}{m \log(1/A)} } \Longleftarrow \sqrt{m p(1-p)} \geq 4/\sqrt{\log(1/A)}, 
	\end{split}
\end{equation*} and the last condition is clearly true as $\log(1/A) < 16mp(1-p)$ and $A$ is small.

Next, we aim to provide a lower bound for $P_{p+r}(X\leq mt)$. 
\begin{equation}\label{Lower bound for CDF of Binomial distribution}
    \begin{split}
        &P_{p+r}(X\leq mt) = \sum_{k\leq mt}\binom{m}{k}(p+r)^k(1-p-r)^{m-k}\\
        &\overset{(a)}\geq \sum_{mt-2\sqrt{mt(1-t)}< k\leq mt}\binom{m}{k}t^k(1-t)^{m-k}\left(\frac{p+r}{t}\right)^k\left(\frac{1-p-r}{1-t}\right)^{m-k}\\
        &\geq \sum_{mt-2\sqrt{mt(1-t)}< k\leq mt}\binom{m}{k}t^{k}(1-t)^{m-k}\min_{mt-2\sqrt{mt(1-t)}<k'\leq mt}\left(\frac{p+r}{t}\right)^{k'}\left(\frac{1-p-r}{1-t}\right)^{m-k'}\\
         &\overset{(b)}{\geq} \sum_{mt-2\sqrt{mt(1-t)}< k\leq mt}\binom{m}{k}t^{k}(1-t)^{m-k}\left(\frac{p+r}{t}\right)^{mt-2\sqrt{mt(1-t)}}\left(\frac{1-p-r}{1-t}\right)^{m(1-t)+2\sqrt{mt(1-t)}}\\
        &=P_t\left(mt-2\sqrt{mt(1-t)}<X\leq mt\right)\\
        & \quad \times \exp\left(-mD(\mathrm{Bernoulli}(t) \parallel \mathrm{Bernoulli}(p+r))-2\sqrt{mt(1-t)}\left(\log\left(\frac{p+r}{t}\right)+\log\left(\frac{1-t}{1-p-r}\right)\right)\right)\\
        &\overset{(c)}{\geq}P_t\left(mt-2\sqrt{mt(1-t)} < X\leq mt\right)\\
        & \quad \times  \exp\left(-\frac{m(p+r-t)^2}{(p+r)(1-p-r)}-2\sqrt{mt(1-t)}\left(\frac{p+r-t}{t}+\frac{p+r-t}{1-p-r}\right)\right),
    \end{split}
\end{equation} where in (a), we use the fact 
\begin{equation*}
	\begin{split}
		mt - 2 \sqrt{mt(1-t)} &\geq \sqrt{mt} \left( \sqrt{mt} - 2 \right) \overset{ \textnormal{Claim (i)} }\geq   \sqrt{mt} \left( \sqrt{mp/2} - 2 \right) \geq \sqrt{mt} \left( \sqrt{mp(1-p)/2} - 2 \right) \\
		& \overset{\eqref{ineq:t-r-range-case2}}\geq \sqrt{mt} \left( \frac{1}{4} \sqrt{\log (1/A)/2} - 2 \right) \geq 0;
	\end{split}
\end{equation*} (b) is because $\left(\frac{p+r}{t}\right)^{k}\left(\frac{1-p-r}{1-t}\right)^{m-k}$ is increasing in $k$ when $k \geq 0$ as $p \geq t$; in (c), we use the inequality $\log(1+x) \leq x$ for all $x > -1$ and the fact that $\chi^2(P\|Q) \geq D(P\|Q)$ for any two distributions $P, Q$.

Next, we bound the two terms $P_t\left(mt-2\sqrt{mt(1-t)} < X\leq mt\right)$ and $$\exp\left(-\frac{m(p+r-t)^2}{(p+r)(1-p-r)}-2\sqrt{mt(1-t)}\left(\frac{p+r-t}{t}+\frac{p+r-t}{1-p-r}\right)\right)$$ at the end of \eqref{Lower bound for CDF of Binomial distribution} separately. Let $\Phi(\cdot)$ denote the CDF of standard Gaussian. Then, by Berry-Esseen theorem (see Lemma \ref{Lem: Berry esseen}), we have
\begin{equation}\label{Application of Berry Esseen Theorem}
    \begin{split}
         &P_{t} \Big (mt-2\sqrt{m t (1 - t)} < X \leq m t \Big)
         \geq (\Phi(0)-\Phi(-2)) - \frac{7(1 - 2 t (1 - t))}{10\sqrt{m t (1 - t)}} -\frac{1}{3\sqrt{m}} \\
        & \overset{\textnormal{Claim (i)} }\geq (\Phi(0)-\Phi(-2)) - \frac{7 \sqrt{2} }{10\sqrt{m p (1 - p)}} - \frac{1}{3\sqrt{m}}\overset{\eqref{ineq:t-r-range-case2}}{\geq}  (\Phi(0)-\Phi(-2)) - \frac{7\sqrt{2} \times 4 }{10 \sqrt{\log(1/A)}} - \frac{1}{3\sqrt{m}} \\ 
        & \geq 0.1,
    \end{split}
\end{equation}
where the last inequality holds as long as $A$ is less than a sufficiently small constant. At the same time,
\begin{equation} \label{ineq:case2-exponent-bound}
	\begin{split}
		&-\frac{m(p+r-t)^2}{(p+r)(1-p-r)}-2\sqrt{mt(1-t)}\left(\frac{p+r-t}{t}+\frac{p+r-t}{1-p-r}\right) \\
		& \overset{\textnormal{Claim (i)},\eqref{ineq:t-r-range-prop3-case2}}\geq - \frac{ 4 m (p-t)^2 }{ \frac{1}{2} p(1-p) } - 2 \sqrt{ \frac{3}{2} mp (1-p) } \left( \frac{4(p-t)}{p} + \frac{4(p-t)}{1-p} \right) \\
		& =  - \frac{ 8 m (p-t)^2 }{ p(1-p) } - 8 \sqrt{ \frac{3}{2} \frac{m(p-t)^2}{p(1-p)}}  \geq - \frac{ 16 m (p-t)^2 }{ p(1-p) }, 
	\end{split}
\end{equation} where in the last inequality, we use the fact that $ \frac{m(p-t)^2}{p(1-p)} = \frac{\log(1/A)}{64}$, which is greater than a sufficiently large constant as long as $A$ is less than a sufficiently small constant. 

By plugging \eqref{Application of Berry Esseen Theorem} and \eqref{ineq:case2-exponent-bound} into \eqref{Lower bound for CDF of Binomial distribution}, we have 
\begin{equation*}
	P_{p+r}(X\leq mt) \geq \frac{1}{10} \exp \left( - \frac{ 16 m (p-t)^2 }{ p(1-p) } \right),
\end{equation*} and a sufficient condition to guarantee $P_{p+r}(X\leq mt) > 10A$ is given as follows,
\begin{equation*}
	\begin{split}
		P_{p+r}(X\leq mt) > 10A &\Longleftarrow \exp \left( - \frac{ 16 m (p-t)^2 }{ p(1-p) } \right) > 100 A \\
		& \Longleftrightarrow \frac{ 16 m (p-t)^2 }{ p(1-p) } < \log(1/A) - \log(100)  \Longleftarrow \frac{ 16 m (p-t)^2 }{ p(1-p) } \leq  \frac{1}{4} \log(1/A) \\
		& \Longleftrightarrow |p - t| \leq \frac{1}{8} \sqrt{\frac{p (1 - p) \log\left(1/A\right)}{m}}.
	\end{split}
\end{equation*} Notice that the last condition is satisfied by the choice of $t$. Thus, we have shown that in this regime, we also have $P_{p+r}(X\leq mt) > 10A$.

\item (Case 3: $p \in [p_2,1-1/m]$) In this case,
\begin{equation} \label{ineq:t-r-property-case3}
	t = p(1+p)/2,\quad  \log(1/A) \geq 16mp(1-p)\quad  \textnormal{ and }\quad r = \frac{1}{2m}. 
\end{equation}

 Let us first consider the case $2\leq m \leq 40$. In this case,
\begin{equation*}
	\begin{split}
		P_{p+r}(X\leq mt)   \geq P_{p + \frac{1}{2m} } (X = 0) = (1-p - \frac{1}{2m} )^m \geq (\frac{1}{m} -  \frac{1}{2m} )^m > 10 A,
	\end{split}
\end{equation*} where the last inequality holds as $m \leq 40$ and $A$ is sufficiently small. 

Next, we consider $m \geq 40$. We further divide the proof into two scenarios: $p \in [p_2, 1 - \frac{2}{m}  ]$ and $p \in [1 - \frac{2}{m}, 1 - \frac{1}{m}]$ since $p_2 \leq 1- \frac{2}{m}$. 
\vskip.2cm
{\bf Scenario 1: $p \in [p_2, 1 - \frac{2}{m} ]$}. In this regime, we have
\begin{equation}\label{ineq:t-r-range-prop3-case3}
	\begin{split}
		p +r -t= \frac{p(1-p)}{2} + \frac{1}{2m} \leq p(1-p) \quad \textnormal{ and } \quad 1 - p -r = 1 - p - \frac{1}{2m} \geq \frac{1}{2}(1-p),
	\end{split}
\end{equation} where the first inequality is because $p(1-p) \geq 1/2 \cdot \frac{2}{m} = \frac{1}{m} $ for all $p \in [p_2, 1 - \frac{2}{m}]$ and the second inequality is also straightforward to check. Next, we observe that the lower bound of $P_{p+r}(X\leq mt)$ derived in \eqref{Lower bound for CDF of Binomial distribution} still holds in this regime, as conditions required there still hold:
\begin{equation*}
	mt - 2 \sqrt{mt(1-t)} \geq \sqrt{mt} \left( \sqrt{mt} - 2 \right) \overset{ \textnormal{Claim (i)} }\geq   \sqrt{mt} \left( \sqrt{mp/2} - 2 \right) \geq \sqrt{mt} \left( \sqrt{m/4} - 2 \right) > 0 
\end{equation*} since $p \geq p_2 \geq 1/2$ and $m \geq 40$. Next, we bound the two terms $P_t\left(mt-2\sqrt{mt(1-t)} < X\leq mt\right)$ and $\exp\left(-\frac{m(p+r-t)^2}{(p+r)(1-p-r)}-2\sqrt{mt(1-t)}\left(\frac{p+r-t}{t}+\frac{p+r-t}{1-p-r}\right)\right)$ at the end of \eqref{Lower bound for CDF of Binomial distribution} separately. First, by Berry-Esseen theorem (see Lemma \ref{Lem: Berry esseen}), we have
\begin{equation}\label{ineq:case3-berry-ess-const}
	\begin{split}
		 & P_{t} \Big (mt-2\sqrt{m t (1 - t)} < X \leq m t \Big)  \geq  (\Phi(0)-\Phi(-2)) - \frac{7}{10\sqrt{m t (1 - t)}} -\frac{1}{3\sqrt{m}} \\
        & =  (\Phi(0)-\Phi(-2)) - \frac{7}{10\sqrt{m \frac{p(1+p)}{2} (1 - p(1+p)/2)}} -\frac{1}{3\sqrt{m}} \\
        & \overset{(a)}\geq (\Phi(0)-\Phi(-2)) - \frac{7}{10\sqrt{m \frac{(1 - \frac{2}{m} )(2 - \frac{2}{m} )}{2} (1 - (1 - \frac{2}{m} )(2 - \frac{2}{m} )/2)}} -\frac{1}{3\sqrt{m}}\\
        & = (\Phi(0)-\Phi(-2)) - \frac{7}{10\sqrt{(1 - \frac{2}{m} )(1 - \frac{1}{m} ) (3 - \frac{2}{m} )}} -\frac{1}{3\sqrt{m}} \\
        & \geq (\Phi(0)-\Phi(-2)) - \frac{7}{10\sqrt{(1 - \frac{2}{40} )(1 - \frac{1}{40} ) (3 - \frac{2}{40} )}} -\frac{1}{3\sqrt{40}} \quad (\text{as } m \geq 40) \\
        & \geq 1/1000
	\end{split}
\end{equation} where (a) is because when $p \in [p_2, 1 - \frac{2}{m} ]$, since $p_2 \geq 1/2$ and $m \geq 40$, $\frac{p(1+p)}{2} (1 - p(1+p)/2)$ achieves its minimum when $p = 1 - \frac{2}{m} $. At the same time,
\begin{equation} \label{ineq:case3-exponent-bound}
	\begin{split}
		&-\frac{m(p+r-t)^2}{(p+r)(1-p-r)}-2\sqrt{mt(1-t)}\left(\frac{p+r-t}{t}+\frac{p+r-t}{1-p-r}\right) \\
		& \overset{ \text{Claim (i), } \eqref{ineq:t-r-range-prop3-case3} }\geq \frac{-m p^2 (1-p)^2 }{ \frac{1}{2} p(1-p)} - 2 \sqrt{ \frac{3}{2} mp(1-p)} \left(  \frac{p(1-p)}{p/2} + \frac{p(1-p)}{ (1-p)/2  } \right) \\
		& = -2mp(1-p) - 2\sqrt{6mp(1-p)}. 
	\end{split}
\end{equation}
By plugging \eqref{ineq:case3-berry-ess-const} and \eqref{ineq:case3-exponent-bound} into \eqref{Lower bound for CDF of Binomial distribution}, we have 
\begin{equation*}
	P_{p+r}(X\leq mt) \geq \frac{1}{1000} \exp \left( -2mp(1-p) - 2\sqrt{6mp(1-p)}\right),
\end{equation*} and a sufficient condition to guarantee $P_{p+r}(X\leq mt) > 10A$ is given as follows,
\begin{equation*}
	\begin{split}
		P_{p+r}(X\leq mt) > 10A & \Longleftrightarrow \frac{1}{1000} \exp \left( -2mp(1-p) - 2\sqrt{6mp(1-p)}\right) > 10A \\
		& \Longleftrightarrow -2mp(1-p) - 2\sqrt{6mp(1-p)} > 4\log(10) + \log(A)\\
		& \Longleftrightarrow \log(1/A) -2mp(1-p) - 2\sqrt{6mp(1-p)} > 4\log(10)\\
		& \overset{\eqref{ineq:t-r-property-case3} }\Longleftarrow \log(1/A) - \frac{1}{8} \log(1/A) - 2 \sqrt{ \frac{6}{16} \log(1/A) } > 4 \log(10), 
	\end{split}
\end{equation*} and the last inequality holds as $A$ is sufficiently small.

\vskip.2cm
{\bf Scenario 2: $p \in [1 - \frac{2}{m}, 1 - \frac{1}{m}]$}.
In this regime,
\begin{equation*}
	\begin{split}
		&P_{p+r}(X\leq mt)= P_{p + \frac{1}{2m} } (X \leq mp(1+p)/2) \overset{\textnormal{Lemma }\ref{lm:Binomial-prop}\,(iii) } \geq P_{1 - \frac{1}{m} + \frac{1}{2m}  }(X \leq mp(1+p)/2) \\
		& \geq P_{1 - \frac{1}{2m} }( X \leq m(1- 2/m )( 2 - 2/m )/2) \quad (\textnormal{as } p \geq 1 -2/m ) \\
		& = P_{1 - \frac{1}{2m} }( X \leq (1- 2/m )( m - 1 )) \geq P_{1 - \frac{1}{2m} }(X \leq m - 3) \\
& = 1 - ( 1 - \frac{1}{2m} )^m - m \cdot \frac{1}{2m}\cdot ( 1 - \frac{1}{2m} )^{m-1} - \frac{m(m-1)}{2} \cdot \frac{1}{(2m)^2} \cdot ( 1 - \frac{1}{2m} )^{m-2}\\
		& \geq 1 - \exp(-1/2) - \frac{1}{2} \frac{1}{1 - \frac{1}{2m} } \exp(-1/2) - \frac{1}{8 ( 1- \frac{1}{2m})^2} \exp(-1/2) \quad (\text{as } (1+x/n)^n \leq e^x) \\
		& \geq 1 - \exp(-1/2) \left( 1 + \frac{1}{2 - \frac{1}{40} } + \frac{1}{2( 2 - \frac{1}{40} )^2} \right)  \quad (\textnormal{as } m \geq 40 )\\
		& \geq 1/200 > 10 A,
	\end{split}
\end{equation*} where the last inequality holds as $A$ is sufficiently small.

\item (Case 4: $p \in \Big(1-1/m,1- \frac{4}{m} \left( \frac{10 \log(24/\alpha) }{n} + 3 \epsilon  \right) \Big]$ )
Notice that in this regime $t = 1 - 1/m$. Thus, $P_{p}(X\leq mt) = P_p(X \neq m) = 1 - p^m$. Given any positive integer $m \geq 1$, let $f(x) = (1-x)^{1/m} - (1 - \frac{2x}{m} ) $. So $f(0) = 0 $ and $ f'(x) = \frac{2 - (1 -x)^{1/m - 1}}{m} \geq \frac{2 - (1 - x)^{-1}}{m} \geq 0$ if $x \in [0, 1/2]$. Thus $f(x) \geq 0$ when $x \in [0,1/2]$. Then 
\begin{equation*}
	\begin{split}
		&P_{p}(X\leq mt) \geq 6 \epsilon + \frac{20 \log(24/\alpha) }{n} \\
		&\Longleftrightarrow \left(1 - \left(6 \epsilon + \frac{20 \log(24/\alpha) }{n}\right) \right)^{1/m} \geq p \\
		& \Longleftarrow 1 - \frac{4}{m} \left( 3 \epsilon + \frac{10 \log(24/\alpha) }{n}\right) \geq p \quad (\text{as } f(x) \geq 0, \forall x \in [0, 1/2] \text{ and } 6\epsilon_{\max} + \frac{20\log(24/\alpha)}{n} \leq \frac{1}{2}).
	\end{split}
\end{equation*} The last condition is satisfied for $p$ in this regime.

\end{itemize}
This finishes the proof of this lemma.

\subsubsection{Proof of Lemma \ref{lm:bernoulli-estimation}} 
Notice that $\indi\{X_i = 1\}$ stochastically dominates $(1-\epsilon)\textnormal{Bernoulli}(p)$ for any $i$. Since $X_i$s are independent, $\sum_{i=1}^n \indi\{X_i = 1\}$ stochastically dominates $(1-\epsilon)$Binomial$(n,p)$. By property of stochastic dominance and Bernstein's inequality (see Lemma \ref{lm:Binomial-prop} (ii)), we have for any $C \geq 4/3$,
\begin{equation*}
\begin{split}
	& \bbP \left(  \frac{\sum_{i=1}^n \indi\{X_i = 1\}}{n} \leq (1-\epsilon) \left(p - \frac{p(1-p)}{C} - \frac{C \log(2/\alpha)}{n} \right) \right) \\
	\leq & \bbP\left( (1-\epsilon)\textnormal{Binomial}(n,p) \leq (1-\epsilon)n \left(p - \frac{p(1-p)}{C} - \frac{C \log(2/\alpha)}{n} \right) \right) \leq \alpha.
\end{split}
\end{equation*} By taking $C= 3$, the result follows as long as $2(1-\epsilon)/3 \geq 1/2$, i.e., $\epsilon \leq 1/4$. This finishes the proof of this lemma.

\subsection{Proof of Theorem \ref{thm:type2}}

The proof for the Type-2 error control of $\psi_{p,\epsilon}^-$ is similar to the one of $\psi_{p,\epsilon}^+$ by the choice of $\underline{t}(p,\epsilon)$, $\underline{r}(p,\epsilon)$ and $\underline{\tau}(p,\epsilon)$. We focus on the proof for $\psi_{p,\epsilon}^+$. We divide the proof into two cases.

\vskip.2cm
{\noindent \bf (Case 1: $p \in [0,1-1/m]$ for $\psi_{p,\epsilon}^+$ and $p \in [1/m,1]$ for $\psi_{p,\epsilon}^-$)}  Given any $Q$, $p \in [0,1-1/m]$, $\epsilon \in [0, \epsilon_{\max}]$ and $r\in[\overline{r}(p, \epsilon),1-p]$, by the DKW inequality (see Lemma \ref{lm:DKW}), we have with probability at least $1 - \alpha/12$, the following event holds:
	\begin{equation*}
		(E) = \{ \textnormal{for all } x \in \bbR, |F_n(x) - P_{\epsilon, p+r, Q}(X \leq x)| \leq \sqrt{ \log (24/\alpha)/(2n) } \}.
	\end{equation*} In this case, we will show
	\begin{equation*}
		\begin{split}
			\underset{Q}{\sup}P_{\epsilon,p+r,Q}\left(\psi_{q,\epsilon}^+ = 0\right)= \underset{Q}{\sup}P_{\epsilon,p+r,Q}\left( \min_{q \in [0,p]} \phi_{q,\epsilon}^+ = 0\right)\leq \alpha/12,
		\end{split}
	\end{equation*} and it is equivalent to show 
	\begin{equation} \label{ineq:monotone-type2-regime1+-goal}
		\inf_Q P_{\epsilon,p+r,Q}\left( \phi_{q,\epsilon}^+ = 1, \forall q \in [0,p]\right)\geq 1-\alpha/12.
	\end{equation} Given $(E)$ happens, a sufficient condition for $\phi_{q,\epsilon}^+ = 1, \forall q \in [0,p]$ is given as follows.
\begin{equation} \label{ineq:monotone-type2-regime1+}
	\begin{split}
			& \forall q \in [0,p], \phi_{q,\epsilon} ^ + = 1 
		\Longleftrightarrow  \forall q \in [0,p], \frac{1}{n} \sum_{i=1}^n \indi \left\{ X_i \leq m \bar{t} (q, \epsilon) \right\} < \overline{\tau} (q, \epsilon) \\
		\overset{(E)}\Longleftarrow & \forall q \in [0,p], P_{\epsilon, p+r, Q}(X \leq m \bar{t} (q, \epsilon)) + \sqrt{ \frac{\log(24/\alpha)}{2n} } < \overline{\tau} (q, \epsilon)\\
		\overset{(a)}\Longleftrightarrow & \forall q \in [0,p], P_{\epsilon, p+r, Q}(X \leq m \bar{t} (q, \epsilon)) + \sqrt{ \frac{\log(24/\alpha)}{2n} } < \frac{11}{10} P_{q+\bar{r}(q, \epsilon)}(X \leq m \bar{t} (q, \epsilon)) \\
		\overset{(b)}\Longleftarrow & \forall q \in [0,p], P_{p+r}(X \leq m \bar{t} (q, \epsilon)) + \epsilon + \sqrt{ \frac{\log(24/\alpha)}{2n} } < \frac{11}{10} P_{q+\bar{r}(q, \epsilon)}(X \leq m \bar{t} (q, \epsilon)) \\
		\Longleftarrow & \left\{ \begin{array}{l}
			\forall q \in [0,p],  \epsilon + \sqrt{ \frac{\log(24/\alpha)}{2n} } < \frac{1}{10} P_{q+\bar{r}(q, \epsilon)}(X \leq m \bar{t} (q, \epsilon)) \\
			\forall q \in [0,p], P_{p+r}(X \leq m \bar{t} (q, \epsilon)) \leq P_{q+\bar{r}(q, \epsilon)}(X \leq m \bar{t} (q, \epsilon)) 
		\end{array} \right. \\
		\Longleftarrow & \left\{ \begin{array}{l}
			\forall q \in [0,p],  \epsilon + \sqrt{ \frac{\log(24/\alpha)}{2n} } < \frac{1}{10} P_{q+\bar{r}(q, \epsilon)}(X \leq m \bar{t} (q, \epsilon))\\
			\forall q \in [0,p], p +r \geq  q+\bar{r}(q, \epsilon) \,\, (\text{as }\text{Lemma } \ref{lm:Binomial-prop}\,(iii))
		\end{array} \right.\\
		\Longleftarrow & \left\{ \begin{array}{l}
			\forall q \in [0,p],  \epsilon + \sqrt{ \frac{\log(24/\alpha)}{2n} } < \frac{1}{10} P_{q+\bar{r}(q, \epsilon)}(X \leq m \bar{t} (q, \epsilon)) \\
			 p +r \geq  p+\bar{r}(p, \epsilon) \,\, (\text{as } \text{Lemma } \ref{lm:r-t-property}\,(ii)).
		\end{array} \right.
	\end{split}
\end{equation} Here (a) is by the definition of $\bar{\tau}(q, \epsilon)$ and (b) is because $P_{\epsilon, p+r, Q}(X \leq m \bar{t} (q, \epsilon)) \leq \epsilon + P_{p+r}(X \leq m \bar{t} (q, \epsilon))$. Note that the last two conditions in \eqref{ineq:monotone-type2-regime1+} are satisfied since $r \geq \bar{r}(p, \epsilon)$ and by Lemma~\ref{lm:r-t-property}~(iii) given $p\in [0,1-1/m]$. Thus, we have shown \eqref{ineq:monotone-type2-regime1+-goal} for all $p \in [0,1-1/m]$, all $\epsilon \in [0, \epsilon_{\max}]$ and all $r\in[\overline{r}(p, \epsilon),1-p]$. 

\vskip.2cm
{\noindent \bf (Case 2: $p \in \left(1-1/m, 1 - \frac{4}{m}\left( \frac{10 \log(24/\alpha)}{n} + 3 \epsilon \right) \right ]$ for $\psi_{p,\epsilon}^+$ and $p \in \left[\frac{4}{m}\left( \frac{10 \log(24/\alpha)}{n} + 3 \epsilon \right) ,1/m\right)$ for $\psi_{p,\epsilon}^-$)} In this case, we will leverage the following observation that while in general $\phi_{p, \epsilon}^{+}$ and $\phi_{p, \epsilon}^{-}$ are not monotone with respect to $p$ when $p \in [0,1]$, they are monotone with respect to $p$ when $p$ is at the boundary. Its proof is also straightforward by directly checking and we omit it here.
\begin{Lemma}\label{lm:monotone-test-boundary}
	When $p \in [0,1/m)$, $\phi_{p, \epsilon}^{-}$ is non-decreasing as $p$ increases; when $p \in (1-1/m,1]$, $\phi_{p, \epsilon}^{+}$ is non-increasing as $p$ increases.
\end{Lemma}

Now, let us consider the Type-2 error control of $\psi_{p,\epsilon}^+$. Given any $Q$, $\epsilon \in [0, \epsilon_{\max}]$, $r\in[\overline{r}(p, \epsilon),1-p]$ and $p \in \left(1-1/m, 1 - \frac{4}{m}\left( \frac{10 \log(24/\alpha)}{n} + 3 \epsilon \right) \right ]$,
\begin{equation}\label{ineq:monotone-type2-regime2+}
	\begin{split}
		\underset{Q}{\sup}P_{\epsilon,p+r,Q}\left(\psi_{p,\epsilon}^+ = 0\right)&= \underset{Q}{\sup}P_{\epsilon,p+r,Q}\left( \min_{q \in [0,p]} \phi_{q,\epsilon}^+ = 0\right) \\
		& = \sup_Q P_{\epsilon,p+r,Q}\left( \left(\min_{q \in [0,1-1/m]} \phi_{q,\epsilon}^+\right) \wedge \left(\min_{q \in (1-1/m,p]}\phi_{q,\epsilon}^+ \right)  = 0\right) \\
		& \overset{\text{Lemma } \ref{lm:monotone-test-boundary} }= \sup_Q P_{\epsilon,p+r,Q}\left( \left(\min_{q \in [0,1-1/m]} \phi_{q,\epsilon}^+\right) \wedge \phi_{p,\epsilon}^+  = 0\right) \\
		& \leq \sup_Q P_{\epsilon,p+r,Q}\left( \min_{q \in [0,1-1/m]} \phi_{q,\epsilon}^+  = 0\right) + \sup_Q P_{\epsilon,p+r,Q}\left(  \phi_{p,\epsilon}^+  = 0\right) \\
		& \overset{\text{Theorem } \ref{thm:test-up} }\leq \sup_Q P_{\epsilon,p+r,Q}\left( \min_{q \in [0,1-1/m]} \phi_{q,\epsilon}^+  = 0\right)  + \alpha/12.
	\end{split}
\end{equation} Next we bound $\sup_Q P_{\epsilon,p+r,Q}\left( \min_{q \in [0,1-1/m]} \phi_{q,\epsilon}^+  = 0\right)$ and its analysis is almost the same as the one in \eqref{ineq:monotone-type2-regime1+}, so we only sketch for simplicity. By the DKW inequality, we have with probability at least $1 - \alpha/12$, the following event holds:
	\begin{equation*}
		(E) = \{ \textnormal{for all } x \in \bbR, |F_n(x) - P_{\epsilon, p+r, Q}(X \leq x)| \leq \sqrt{ \log (24/\alpha)/(2n) } \}.
	\end{equation*}
Given $(E)$ happens, a sufficient condition for $\phi_{q,\epsilon}^+ = 1, \forall q \in [0,1-1/m]$ is given as follows.
\begin{equation*}
	\begin{split}
			& \forall q \in [0,1-1/m], \phi_{q,\epsilon} ^ + = 1 
		\Longleftrightarrow  \forall q \in [0,1-1/m], \frac{1}{n} \sum_{i=1}^n \indi \left\{ X_i \leq m \bar{t} (q, \epsilon) \right\} < \overline{\tau} (q, \epsilon) \\
		\overset{(E)}\Longleftarrow & \forall q \in [0,1-1/m], P_{\epsilon, p+r, Q}(X \leq m \bar{t} (q, \epsilon)) + \sqrt{ \frac{\log(24/\alpha)}{2n} } < \overline{\tau} (q, \epsilon)\\
		\Longleftrightarrow & \forall q \in [0,1-1/m], P_{\epsilon, p+r, Q}(X \leq m \bar{t} (q, \epsilon)) + \sqrt{ \frac{\log(24/\alpha)}{2n} } < \frac{11}{10} P_{q+\bar{r}(q, \epsilon)}(X \leq m \bar{t} (q, \epsilon)) \\
		\Longleftarrow & \forall q \in [0,1-1/m], P_{p+r}(X \leq m \bar{t} (q, \epsilon)) + \epsilon + \sqrt{ \frac{\log(24/\alpha)}{2n} } < \frac{11}{10} P_{q+\bar{r}(q, \epsilon)}(X \leq m \bar{t} (q, \epsilon)) \\
		\Longleftarrow & \left\{ \begin{array}{l}
			\forall q \in [0,1-1/m],  \epsilon + \sqrt{ \frac{\log(24/\alpha)}{2n} } < \frac{1}{10} P_{q+\bar{r}(q, \epsilon)}(X \leq m \bar{t} (q, \epsilon)) \\
			\forall q \in [0,1-1/m], P_{p+r}(X \leq m \bar{t} (q, \epsilon)) \leq P_{q+\bar{r}(q, \epsilon)}(X \leq m \bar{t} (q, \epsilon)) 
		\end{array} \right. \\
		\Longleftarrow & \left\{ \begin{array}{l}
			\forall q \in [0,1-1/m],  \epsilon + \sqrt{ \frac{\log(24/\alpha)}{2n} } < \frac{1}{10} P_{q+\bar{r}(q, \epsilon)}(X \leq m \bar{t} (q, \epsilon))\\
			\forall q \in [0,1-1/m], p +r \geq  q+\bar{r}(q, \epsilon) \,\, (\text{as }\text{Lemma } \ref{lm:Binomial-prop}\,(iii))
		\end{array} \right.\\
		\Longleftarrow & \left\{ \begin{array}{l}
			\forall q \in [0,1-1/m],  \epsilon + \sqrt{ \frac{\log(24/\alpha)}{2n} } < \frac{1}{10} P_{q+\bar{r}(q, \epsilon)}(X \leq m \bar{t} (q, \epsilon)) \\
			 p +r \geq  1-1/m+\bar{r}(1-1/m, \epsilon) \,\, (\text{as } \text{Lemma } \ref{lm:r-t-property}\,(ii))
		\end{array} \right.\\
		\Longleftarrow & \left\{ \begin{array}{l}
			\forall q \in [0,1-1/m],  \epsilon + \sqrt{ \frac{\log(24/\alpha)}{2n} } < \frac{1}{10} P_{q+\bar{r}(q, \epsilon)}(X \leq m \bar{t} (q, \epsilon)) \\
			 p +\bar{r}(p, \epsilon) \geq  1-1/m+\bar{r}(1-1/m, \epsilon) \,\, (\text{as } r \geq \bar{r}(p, \epsilon)).
		\end{array} \right.
	\end{split}
\end{equation*}
 Note that the last two conditions above are satisfied due to Lemma \ref{lm:r-t-property} (iii) given $q\in [0,1-1/m]$ and Lemma \ref{lm:r-t-property} (ii). This shows that $\sup_Q P_{\epsilon,p+r,Q}\left( \min_{q \in [0,1-1/m]} \phi_{q,\epsilon}^+  = 0\right) \leq \alpha/12$. In view of \eqref{ineq:monotone-type2-regime2+}, we have shown $\sup_QP_{\epsilon,p+r,Q}\left(\psi_{p,\epsilon}^+ = 0\right) \leq \alpha/6$.

\subsection{Proof of Proposition \ref{prop:binom-end-pionts}}

We first present a convenient lemma, which provides another characterization for $\wh{p}_{\rm{left}} $ and $\wh{p}_{\rm{right}}$.
\begin{Lemma}\label{lm:p-hat-charac}
	The set $\widehat{\CI}$ defined by (\ref{eq:ci-dis}) is an interval whose endpoints are given by
	\begin{equation} \label{eq:pleft-third}
		\wh{p}_{\rm{left}} = \left\{ \begin{array}{ll}
			 \inf\{p \in S_m\setminus \{1\}:  \max_{\epsilon \in \cE} \min_{q \in [0,p+1/m] \cap (S_m \setminus \{1\}) } \phi_{q, \epsilon}^+ = 0  \} & \text{ if }  \max_{\epsilon \in \cE} \min_{q \in S_m \setminus \{1\}} \phi_{q, \epsilon}^+ = 0\\
			 \inf \{ p \in (1- \frac{1}{m} ,1]: \phi_{p, \epsilon'}^+ = 0  \} & \text{ if }  \max_{\epsilon \in \cE} \min_{q \in S_m \setminus \{1\}} \phi_{q, \epsilon}^+ = 1,
		\end{array}  \right.
	\end{equation} and 
	\begin{equation}\label{eq:pright-third}
		\begin{split}
					\wh{p}_{\rm{right}} = \left\{ \begin{array}{ll}
			 \sup\{p \in S_m\setminus \{0\}:  \max_{\epsilon \in \cE} \min_{q \in [0,p-1/m] \cap (S_m \setminus \{0\}) } \phi_{q, \epsilon}^- = 0  \} & \text{ if }  \max_{\epsilon \in \cE} \min_{q \in S_m \setminus \{0\}} \phi_{q, \epsilon}^- = 0\\
			 \sup \{ p \in [0, \frac{1}{m} ): \phi_{p, \epsilon'}^- = 0  \} & \text{ if }  \max_{\epsilon \in \cE} \min_{q \in S_m \setminus \{0\}} \phi_{q, \epsilon}^- = 1,
		\end{array}  \right.
		\end{split}
	\end{equation}
    where $\epsilon' \in [0,\epsilon_{\max}]$ can be chosen arbitrarily.
\end{Lemma} The proof of Lemma \ref{lm:p-hat-charac} is provided in the subsequent subsections. Next, we first show that $\wh{p}_{\rm{left}} $ and $\wh{p}_{\rm{right}}$ defined in \eqref{eq:left} and \eqref{eq:right} are the unification of the ones in \eqref{eq:pleft-third} and \eqref{eq:pright-third}. So they are also the endpoints of $\wh{\CI}$ by Lemma \ref{lm:p-hat-charac}. Then we show the output of the Algorithm \ref{alg:CI-endpoints} are the $\wh{p}_{\rm{left}}$ and $\wh{p}_{\rm{right}}$ defined in \eqref{eq:pleft-third} and \eqref{eq:pright-third}.
\vskip.2cm
{\noindent \bf (Part I: Equivalence of \eqref{eq:left}/\eqref{eq:right} and \eqref{eq:pleft-third}/\eqref{eq:pright-third})} We will present the proof for the equivalence between \eqref{eq:left} and \eqref{eq:pleft-third}, while the proof for the equivalence between \eqref{eq:right} and \eqref{eq:pright-third} is similar and we omit it here. We will show separately that when $\max_{\epsilon \in \cE} \min_{q \in S_m \setminus \{1\}} \phi_{q, \epsilon}^+ = 0$ and $\max_{\epsilon \in \cE} \min_{q \in S_m \setminus \{1\}} \phi_{q, \epsilon}^+ = 1$, \eqref{eq:left} can be written as the cases in \eqref{eq:pleft-third}. Let us denote 
	\begin{equation*}
		\begin{split}
			S_{\rm{left}} = \left\{p\in S_m\cup\left[1-\frac{1}{m},1\right]: \max_{\epsilon\in\mathcal{E}}\left(\phi_{p,\epsilon}^+\wedge\min_{q\in [0,p+1/m] \cap (S_m \setminus \{1\} )}\phi_{q,\epsilon}^+\right)=0\right\}.
		\end{split}
	\end{equation*} Note that $\max_{\epsilon\in\mathcal{E}}\left(\phi_{p,\epsilon}^+\wedge\min_{q\in [0,p+1/m] \cap (S_m \setminus \{1\} ) }\phi_{q,\epsilon}^+\right)$ is non-increasing in $p$ when $p \in S_m\cup\left[1-\frac{1}{m},1\right]$. Thus, for any $p_1,p_2 \in S_m \cup [1-\tfrac{1}{m},1]$ with $p_1 \leq p_2$, if $p_1 \in S_{\rm left}$ then $p_2 \in S_{\rm left}$.
\begin{itemize}[leftmargin=*]
	\item (Case 1: $\max_{\epsilon \in \cE} \min_{q \in S_m \setminus \{1\}} \phi_{q, \epsilon}^+ = 0$) In this case, it is easy to check $1-1/m \in S_{\rm{left}}$. Thus $\wh{p}_{\rm{left}}$ in \eqref{eq:left} is less than or equal to $1-1/m$. Then 
		\begin{equation*}
		\begin{split}
			\wh{p}_{\rm{left}} \text{ in } \eqref{eq:left} &= \inf\left\{p\in S_m\cup\left[1-\frac{1}{m},1\right]:  \max_{\epsilon\in\mathcal{E}}\left(\phi_{p,\epsilon}^+\wedge\min_{q\in [0,p+1/m] \cap (S_m \setminus \{1\} ) }\phi_{q,\epsilon}^+\right)=0\right\} \\
			& = \inf\left\{p\in S_m\setminus \{1\} :  \max_{\epsilon\in\mathcal{E}}\left(\phi_{p,\epsilon}^+\wedge\min_{q\in [0,p+1/m] \cap (S_m \setminus \{1\} ) }\phi_{q,\epsilon}^+\right)=0\right\} \\
			& \overset{(a)}= \inf\left\{p\in S_m\setminus \{1\} :  \max_{\epsilon\in\mathcal{E}}\left(\min_{q\in [0,p+1/m] \cap (S_m \setminus \{1\} ) }\phi_{q,\epsilon}^+\right)=0\right\} =\wh{p}_{\rm{left}} \text{ in } \eqref{eq:pleft-third},
		\end{split}
	\end{equation*} where (a) is because for any $p\in S_m\setminus \{1\} $, $\min_{q\in [0,p+1/m] \cap (S_m \setminus \{1\} ) }\phi_{q,\epsilon}^+ \leq \phi_{p,\epsilon}^+$.
	\item (Case 2: $\max_{\epsilon \in \cE} \min_{q \in S_m \setminus \{1\}} \phi_{q, \epsilon}^+ = 1$) In this case, it is easy to check $1-1/m \notin S_{\rm{left}}$. Thus $\wh{p}_{\rm{left}}$ in \eqref{eq:left} is greater than or equal to $1-1/m$. As a result,
	\begin{equation} \label{eq:pleft-equi-case2}
		\begin{split}
			&\wh{p}_{\rm{left}} \text{ in } \eqref{eq:left} = \inf\left\{p\in S_m\cup\left[1-\frac{1}{m},1\right]:  \max_{\epsilon\in\mathcal{E}}\left(\phi_{p,\epsilon}^+\wedge\min_{q\in [0,p+1/m] \cap (S_m \setminus \{1\} ) }\phi_{q,\epsilon}^+\right)=0\right\} \\
			& \overset{(a)}= \inf\left\{p\in \left(1-\frac{1}{m},1\right]:  \max_{\epsilon\in\mathcal{E}}\left(\phi_{p,\epsilon}^+\wedge\min_{q\in [0,p+1/m] \cap (S_m \setminus \{1\} ) }\phi_{q,\epsilon}^+\right)=0\right\}\\
			& = \inf\left\{p\in \left(1-\frac{1}{m},1\right]:  \max_{\epsilon\in\mathcal{E}}\left(\phi_{p,\epsilon}^+\wedge\min_{q\in S_m \setminus \{1\}  }\phi_{q,\epsilon}^+\right)=0\right\} \\
			& \overset{(b)}= \inf\left\{p\in \left(1-\frac{1}{m},1\right]:  \max_{\epsilon\in\mathcal{E}^*}\left(\phi_{p,\epsilon}^+\wedge\min_{q\in S_m \setminus \{1\}  }\phi_{q,\epsilon}^+\right)=0 \text{ and } \max_{\epsilon\in\cE \setminus \mathcal{E}^*}\left(\phi_{p,\epsilon}^+\wedge\min_{q\in S_m \setminus \{1\}  }\phi_{q,\epsilon}^+\right)=0\right\} \\
			& = \inf\left\{p\in \left(1-\frac{1}{m},1\right]:  \max_{\epsilon\in\mathcal{E}^*}\phi_{p,\epsilon}^+=0 \right\} \overset{(c)}= \inf\left\{p\in \left(1-\frac{1}{m},1\right]:  \phi_{p,\epsilon'}^+=0 \right\}  = \wh{p}_{\rm{left}} \text{ in } \eqref{eq:pleft-third},
		\end{split}
	\end{equation} where (a) is because $ 1-1/m \notin S_{\rm{left}}$, which implies that any $p \in S_m \setminus \{1\}$ does not belong to $S_{\rm left}$; in (b), $\cE^*:= \{ \epsilon \in \cE: \min_{q \in S_m \setminus \{1\}} \phi_{q, \epsilon}^+ = 1 \}$ and (c) holds for any $\epsilon' \in [0, \epsilon_{\max}]$ because on $p \in \left(1-\frac{1}{m},1\right]$, $\phi_{p,\epsilon'}^+$ is independent of $\epsilon'$. 
\end{itemize}

\vskip.2cm
{\noindent \bf (Part II: Algorithm \ref{alg:CI-endpoints} computes \eqref{eq:pleft-third} and \eqref{eq:pright-third})} We show the output of $\wh{p}_{\rm{left}}$ of Algorithm \ref{alg:CI-endpoints} is the same as the one defined in \eqref{eq:pleft-third}, while the proof for $\wh{p}_{\rm{right}}$ is similar. First, it is easy to check that 
\begin{equation} \label{eq:pleft-boundary}
	\begin{split}
		 &\{ p \in (1- 1/m ,1]: \phi_{p, \epsilon}^+ = 0  \} \\
		 & = \left[  \left[1-\left(\frac{2}{n}\sum_{i=1}^n\indi\{X_i\leq m-1\}+\frac{6\log(24/\alpha)}{n}\right)\wedge 1\right]^{1/m} \vee \left( 1- \frac{1}{m} \right),1 \right] \, \setminus\, \left\{1 - \frac{1}{m}\right\}.
	\end{split}
\end{equation} Note that the left boundary at the right-hand side of \eqref{eq:pleft-boundary} is exactly $\wh{p}_{\rm{left}}$ in \eqref{eq:end-start-left}. 

Now we walk to the first if statement in Step 3 of Algorithm \ref{alg:CI-endpoints}. The for loop tries to check whether the condition $\max_{\epsilon \in \cE} \min_{q \in [0,p+1/m] \cap (S_m \setminus \{1\}) } \phi_{q, \epsilon}^+ = 0$ holds for some $p \in S_m\setminus \{1\}$. If there is a $p \in S_m\setminus \{1\}$ such that $\max_{\epsilon \in \cE} \min_{q \in [0,p+1/m] \cap (S_m \setminus \{1\}) } \phi_{q, \epsilon}^+ = 0$ holds, then the for loop computes
\begin{equation*}
	\begin{split}
		   \inf\{p \in S_m\setminus \{1\}:  \max_{\epsilon \in \cE} \min_{q \in [0,p+1/m] \cap (S_m \setminus \{1\}) } \phi_{q, \epsilon}^+ = 0  \}
	\end{split}
\end{equation*} and by \eqref{eq:pleft-third}, we know this quantity is $\wh{p}_{\rm{left}}$.

If there is no $p \in S_m \setminus \{1\}$ such that $\max_{\epsilon \in \cE} \min_{q \in [0,p+1/m] \cap (S_m \setminus \{1\}) } \phi_{q, \epsilon}^+ = 0$ holds, then it implies $\max_{\epsilon\in\mathcal{E}}\min_{q\in S_m \setminus \{1\} }\phi_{q,\epsilon}^+=1$, then by \eqref{eq:pleft-third}, we know 
\begin{equation*}
	\begin{split}
		\wh{p}_{\rm{left}} \text{ in } \eqref{eq:pleft-third}& = \inf \{ p \in (1- 1/m ,1]: \phi_{p, \epsilon}^+ = 0  \} \overset{\eqref{eq:pleft-boundary}}= \wh{p}_{\rm{left}} \text{ in } \eqref{eq:end-start-left},
	\end{split}
\end{equation*} which is also the output of the algorithm. This finishes the proof of this proposition.

\subsubsection{Proof of Lemma \ref{lm:p-hat-charac}}
We will present the proof for the equivalence of the left endpoint of $\wh{\CI}$ and \eqref{eq:pleft-third}, while the proof for the equivalence of the right endpoint of $\wh{\CI}$ and \eqref{eq:pright-third} is similar and we omit it here. We will show separately that when $\max_{\epsilon \in \cE} \min_{q \in S_m \setminus \{1\}} \phi_{q, \epsilon}^+ = 0$ and $\max_{\epsilon \in \cE} \min_{q \in S_m \setminus \{1\}} \phi_{q, \epsilon}^+ = 1$, the left endpoint of $\wh{\CI}$ can be written as the cases in \eqref{eq:pleft-third}. Let us denote 
	\begin{equation*}
		\begin{split}
			S_{\rm{left}} = \left\{p \in [0,1]:  \wh{\psi}_{p,\epsilon}^+ = 0 \text{ for all }\epsilon\in\mathcal{E}\right\}.
		\end{split}
	\end{equation*} Note that $\wh{\psi}_{p,\epsilon}^+$ is non-increasing in $p$ when $p \in [0,1]$, thus $S_{\rm{left}}$ is an interval. Thus $\wh{\CI}$ is also an interval.
\begin{itemize}[leftmargin=*]
	\item (Case 1: $\max_{\epsilon \in \cE} \min_{q \in S_m \setminus \{1\}} \phi_{q, \epsilon}^+ = 0$) In this case, it is easy to check $1-1/m \in S_{\rm{left}}$. Thus the left endpoint of $\wh{\CI}$ is less or equal to $1-1/m$. When $m = 1$, it is easy to check that $\textnormal{left endpoint of }\wh{\CI} =0 = \wh{p}_{\rm{left}} \text{ in } \eqref{eq:pleft-third}$. When $m \geq 2$,
	\begin{equation*}
		\begin{split}
			&\textnormal{left endpoint of }\wh{\CI}  = \inf \left\{p \in [0,1]:  \wh{\psi}_{p,\epsilon}^+ = 0 \text{ for all }\epsilon\in\mathcal{E}\right\} \\
			& = \inf \left\{p \in [0,1-1/m]:  \min_{q\in[0, \lceil mp \rceil/m ]\cap S_m}\phi_{q,\epsilon}^+ = 0 \text{ for all }\epsilon\in\mathcal{E}\right\}\\
			& = \inf \left\{p \in (0,1-1/m]:  \min_{q\in[0, \lceil mp \rceil/m ]\cap S_m}\phi_{q,\epsilon}^+ = 0 \text{ for all }\epsilon\in\mathcal{E}\right\}\\
			& = \inf\left\{p\in S_m\setminus \{1\} :  \max_{\epsilon\in\mathcal{E}}\left(\min_{q\in [0,p+1/m] \cap (S_m \setminus \{1\} ) }\phi_{q,\epsilon}^+\right)=0\right\} =\wh{p}_{\rm{left}} \text{ in } \eqref{eq:pleft-third}.
		\end{split}
	\end{equation*}

		\item (Case 2: $\max_{\epsilon \in \cE} \min_{q \in S_m \setminus \{1\}} \phi_{q, \epsilon}^+ = 1$) In this case, it is easy to check $1-1/m \notin S_{\rm{left}}$. Thus the left endpoint of $\wh{\CI}$ is greater or equal to $1-1/m$. Then 
		\begin{equation*}
			\begin{split}
				&\textnormal{left endpoint of }\wh{\CI} = \inf \left\{p \in [0,1]:  \wh{\psi}_{p,\epsilon}^+ = 0 \text{ for all }\epsilon\in\mathcal{E}\right\} \\
			& \overset{(a)}= \inf \left\{p \in  \left(1-\frac{1}{m},1\right]:  \wh{\psi}_{p,\epsilon}^+ = 0 \text{ for all }\epsilon\in\mathcal{E}\right\}\\
			& = \inf\left\{p\in \left(1-\frac{1}{m},1\right]: \max_{\epsilon \in \cE} \left(\phi_{p,\epsilon}^+\wedge \min_{q\in S_m\backslash\{1\}}\phi_{q,\epsilon}^+\right)=0  \right\}  \overset{(b)}= \wh{p}_{\rm{left}} \text{ in } \eqref{eq:pleft-third},
			\end{split}
		\end{equation*} where (a) $ 1-1/m \notin S_{\rm{left}}$ and $S_{\rm{left}}$ is an interval; (b) follows the same analysis as in the second half of the proof in \eqref{eq:pleft-equi-case2}. 
\end{itemize}

\subsection{Proof of Lemma \ref{lm:test-rate-CI-length-connection}}
Note that the order of $\ell(n,\epsilon,m,p)$ can be directly obtained by combining the order of $\overline{r}(p, \epsilon)$ and $\underline{r}(p, \epsilon)$. Next, we will show the order of $\overline{r}(p, \epsilon)$, while the proof for $\underline{r}(p, \epsilon)$ is similar.

When $p \in [0,1-1/m]$, we have $1-p \geq 1/m$ and $1-p \gtrsim \sqrt{\frac{p(1-p)}{m}}\left(\frac{1}{\sqrt{\log n}}+\frac{1}{\sqrt{\log(1/\epsilon)}}\right)$, thus 
\begin{equation} \label{eq:order-1}
    \begin{split}
        \left(\sqrt{\frac{p(1-p)}{m}}\left(\frac{1}{\sqrt{\log n}}+\frac{1}{\sqrt{\log(1/\epsilon)}}\right) + \frac{1}{m} \right) \wedge (1-p) \asymp \sqrt{\frac{p(1-p)}{m}}\left(\frac{1}{\sqrt{\log n}}+\frac{1}{\sqrt{\log(1/\epsilon)}}\right) + \frac{1}{m}.
    \end{split}
\end{equation}
As a result,
\begin{equation*}
    \begin{split}
        \overline{r}(p, \epsilon) \overset{ \textnormal{Lemma }\ref{lm:additional-r-property} \,(i) }=  \frac{1}{2m} \vee 2 \sqrt{ \frac{p(1-p)}{m \log(1/A)} } \overset{  \eqref{eq:order-1}}\asymp   \left(\sqrt{\frac{p(1-p)}{m}}\left(\frac{1}{\sqrt{\log n}}+\frac{1}{\sqrt{\log(1/\epsilon)}}\right) + \frac{1}{m} \right) \wedge (1-p).
    \end{split}
\end{equation*}

When $p \in (1-1/m,1]$, we have $1 -p \leq \sqrt{\frac{p(1-p)}{m}}\left(\frac{1}{\sqrt{\log n}}+\frac{1}{\sqrt{\log(1/\epsilon)}}\right) + \frac{1}{m}$, thus
\begin{equation*}
    \begin{split}
         \overline{r}(p, \epsilon) \overset{ \eqref{def:r} }=  (1- 1/(6e) )(1-p) \asymp   \left(\sqrt{\frac{p(1-p)}{m}}\left(\frac{1}{\sqrt{\log n}}+\frac{1}{\sqrt{\log(1/\epsilon)}}\right) + \frac{1}{m} \right) \wedge (1-p).
    \end{split}
\end{equation*} This finishes the proof of this lemma.

\section{Proofs for the Binomial Model with Known $\epsilon$}

This section collects the proofs of Proposition \ref{prop:bench} and Theorem \ref{thm:est}.

\subsection{Proof of Proposition \ref{prop:bench}}

We begin by deriving a lower bound for $r_{\alpha}(\epsilon, p, \epsilon)$, followed by establishing an upper bound by constructing a robust confidence interval $\widehat{\CI}$ that achieves the desired inequality.

\vskip.2cm
{\noindent \bf (Part I: Lower bound)} In this part, we begin by presenting a lemma, and its proof is provided in the subsequent subsections.

\begin{Lemma}\label{Lem: TV for known epsilon}
    For any $\alpha \in (0,1)$, $\epsilon \in [0, 1]$, and $p \in [0, 1]$, there exists some constant $c > 0$ only depending on $\alpha$, such that as long as
$$r\leq c\left[\sqrt{\frac{p(1-p)}{m}}\left(\frac{1}{\sqrt{n}}+\epsilon\right)+\frac{1}{m}\left(\frac{1}{n}+\epsilon\right)\right],$$
we have
\begin{equation*}
   \textnormal{either }  \inf_{Q_0,Q_1} \TV\left(P^{\otimes n}_{\epsilon,p,Q_0}, P^{\otimes n}_{\epsilon, p + r, Q_1}\right) \leq \alpha \quad \textnormal{or}\quad    \inf_{Q_0,Q_1} \TV\left(P^{\otimes n}_{\epsilon,p,Q_0}, P^{\otimes n}_{\epsilon, p - r, Q_1}\right) \leq \alpha.
\end{equation*}
\end{Lemma}

A combination of Lemma \ref{Lem: TV to lower bound} and Lemma~\ref{Lem: TV for known epsilon} shows that $$r_{\alpha}(\epsilon, p, \epsilon) \geq c\left[\sqrt{\frac{p(1-p)}{m}}\left(\frac{1}{\sqrt{n}}+\epsilon\right) + \frac{1}{m}\left(\frac{1}{n}+\epsilon\right)\right],$$ for some constant $c>0$ only depending on $\alpha$. 

 \vskip.2cm
{\noindent \bf (Part II: Upper bound)} In this part, we show that there exists a robust confidence interval $\wh{\CI}$ that achieves the desired coverage and length guarantees. Fix $p$ and $Q$. By Theorem~\ref{thm:est}, there exist $\widehat{p}$ that does not depend on $\epsilon$, and some constant $C' > 0$ only depending on $\alpha$, such that
$$
 (E) = \left\{ |\widehat{p} - p| \leq C' \left[ \sqrt{\frac{p(1 - p)}{m}} \left( \frac{1}{\sqrt{n}} + \epsilon \right) + \frac{1}{m} \left( \frac{1}{n} + \epsilon \right) \right] \right\},
$$
holds with probability at least $1 - \alpha$, assuming that $\frac{\log(2/\alpha)}{n} + \epsilon$ is less than a sufficiently small universal constant. Let
\begin{equation}\label{Def: CI using estimator}
    \widehat\CI=\left\{ q \in [0,1] : |\widehat{p} - q| \leq C' \left( A\sqrt{q(1 - q)} + B \right) \right\},
\end{equation}
where $A = \frac{1}{\sqrt{m}}\left( \frac{1}{\sqrt{n}} + \epsilon \right)$ and $ B = \frac{1}{m} \left( \frac{1}{n} + \epsilon \right)$. Then, on the event $(E)$, it is easy to check that $p \in \wh{\CI}$, which implies that $p \in \wh{\CI}$ with probability at least $1 - \alpha$.

First, it is relatively easy to see that $\widehat{\mathrm{CI}}$ is an interval. This is because $f(q) = |\wh{p} - q|$ is a convex function in $q$, which is decreasing when $q \leq \wh{p}$ and increasing when $q > \wh{p}$. In addition, $h(q) = C' \left( A\sqrt{q(1 - q)} + B \right)$ is a concave function which is increasing when $q \leq 1/2$ and decreasing when $q > 1/2$. These two functions have at most two intersection points over $q \in [0,1]$. Moreover, by their shape, it is easy to see $\widehat{\mathrm{CI}}$ is an interval by plot. 

Next, we show the length guarantee of $\widehat{\mathrm{CI}}$. On the event $(E)$, for any $q \in \widehat{\mathrm{CI}}$, we have
\begin{equation*}
    \begin{split}
        q - p & \leq  |\wh{p} - q| + |\wh{p} - p|  \overset{(a)}\leq C'\left(A\sqrt{q} + A\sqrt{p} + 2B\right) \leq C'\left(\frac{q}{2C'} + \frac{C'A^2}{2} + \frac{p}{2C'} + \frac{C'A^2}{2} + 2B\right) \\
        & = \frac{q}{2}+\frac{p}{2} + C'^2A^2+ 2 C'B,
    \end{split}
\end{equation*}
where (a) follows from the definition of $\widehat{\mathrm{CI}}$ and the fact that both $p, q \in \widehat{\mathrm{CI}}$ given $(E)$ happens. It is easy to check that $A^2 \lesssim B$. Therefore, we have
\begin{equation} \label{Ineq: q < 3phat + term}
    q \leq 3p + C_1'B,
\end{equation}
for some constant $C_1'>0$ that depends on $C'$, and hence only on $\alpha$. Therefore, on the event $(E)$, for any $q \in \wh{\CI}$, we have
\begin{equation}\label{Ineq: |q-p| upper bound 1}
    \begin{split}
       |q - p| & \overset{(E)}\leq   C'(A\sqrt{q}+A\sqrt{p}+2B) \overset{\eqref{Ineq: q < 3phat + term}} \leq C' \left( A\sqrt{3p + C_1'B} + A\sqrt{p} + 2B \right) \leq   C_2' \left(A\sqrt{p} + B\right),
    \end{split}
\end{equation}
where the last inequality holds for some constant $C_2' > 0$ that depends only on $C'$ and $C_1'$, since $A \sqrt{B} \lesssim B$. By a similar argument, given $(E)$ happens, we also have
\begin{equation}\label{Ineq: |q-p| upper bound 2}
|q - p| \leq   C_2' \left(A\sqrt{1 - p} + B\right),
\end{equation}
for any $q \in \wh{\CI}$. Recall that we have shown $\widehat{\mathrm{CI}}$ is an interval, i.e., $ = [\widehat{p}_l, \widehat{p}_r]$ for some $\widehat{p}_l \in [0, \widehat{p}]$ and $\widehat{p}_r \in [\widehat{p}, 1]$. Notice that $\wh{p}_l, \wh{p}_r \in \wh{\CI}$ and $\wh{p}_l \leq p \leq \wh{p}_r$ on the event $(E)$. Therefore, we have
\begin{equation*}
|\widehat{\mathrm{CI}}| = |p - \widehat{p}_l| + |p - \widehat{p}_r| \overset{\eqref{Ineq: |q-p| upper bound 1}, \eqref{Ineq: |q-p| upper bound 2}}\leq 2C_2' \left( A \sqrt{p \wedge (1 - p)} + B \right) \leq 2\sqrt{2}C_2'\left(A\sqrt{p(1-p)}+B\right),
\end{equation*}
with probability at least $1 - \alpha$. By setting $C = 2\sqrt{2} C_2'$, we have established the length guarantee of $\widehat{\mathrm{CI}}$, and this finishes the proof of this proposition.

\subsubsection{Proof of Lemma \ref{Lem: TV for known epsilon}}
It suffices to prove the claim for $\epsilon \in [0,1)$ and $p \in [0,1/2]$, since the case $p \in [1/2, 1]$ follows by the symmetry of the binomial distribution. To this end, it is enough to show that there exist $c_1, c_2 >0$ only depending on $\alpha$, such that as long as $r \leq c_1 \left( \sqrt{\frac{p(1 - p)}{mn}} + \frac{1}{mn} \right)$ or $r \leq c_2 \left( \sqrt{\frac{p(1 - p)}{m}}\epsilon + \frac{\epsilon}{m} \right)$, we have $ \inf_{Q_0, Q_1} \TV \Big(P^{\otimes n}_{\epsilon,p,Q_0}, P^{\otimes n}_{\epsilon,p+r, Q_1} \Big) \leq \alpha$ for all $p \in [0,1/2]$.

\vskip.2cm
{\noindent \bf (Part I: $r \leq c_1\left(\sqrt{\frac{p (1 - p)}{mn}} + \frac{1}{mn}\right) $)} In this part, we will divide the proof into two cases based on different ranges of $p$.
\begin{itemize}[leftmargin=*]
\item (Case 1: $p \leq \frac{\alpha}{2mn}$) In this case, it is easy to check that $r \leq \frac{\alpha}{2mn}$ as long as $c_1$ is sufficiently small. Then, by taking $Q_0 = \mathrm{Binomial}(m,p)$ and $Q_1 = \mathrm{Binomial}(m, p+r)$, we have 
	\begin{equation*}
		\begin{split}
	\TV\left(P_{\epsilon,p,Q_0}, P_{\epsilon, p+r, Q_1}\right) = \TV(P_{p},P_{p+r})  \overset{\textnormal{Lemma }\ref{lm:TV monotonicity}}\leq \TV(P_{p+r},P_{0})= 1 - (1-(p+r))^m  \leq \frac{\alpha}{n},
		\end{split}
	\end{equation*} 
    where the last inequality holds since $(1 - x)^m \geq 1 - m x$ for all $x \in [0,1]$ and $p + r \leq \alpha/(mn)$. Therefore, we have $\inf_{Q_0,Q_1} \TV\left(P^{\otimes n}_{\epsilon,p,Q_0}, P^{\otimes n}_{\epsilon, p, Q_1}\right) \leq n\inf_{Q_0,Q_1} \TV\left(P_{\epsilon,p,Q_0}, P_{\epsilon, p, Q_1}\right) \leq \alpha.$
\item (Case 2: $\frac{\alpha}{2mn} < p \leq 1/2$) In this case, it is easy to check that $r \leq \sqrt{2} \alpha \sqrt{\frac{p(1-p)}{mn}}$ as long as $c_1$ is sufficiently small. Then, by taking $Q_0 = \mathrm{Binomial}(m,p)$ and $Q_1 = \mathrm{Binomial}(m, p+r)$, we have 
\begin{equation*}\label{Ineq: TV known epsilon 1}
    \begin{split}
         &\TV \Big(P^{\otimes n}_{\epsilon,p,Q_0}, P^{\otimes n}_{\epsilon,p+r, Q_1} \Big)  = \TV \Big(P^{\otimes n}_{p}, P^{\otimes n}_{p+r} \Big)  \overset{(a)} \leq \sqrt{\frac{D(P^{\otimes n}_{p+r} \|P^{\otimes n}_{p})}{2}}   \\
         & = \sqrt{\frac{mn}{2}D(\mathrm{Bernoulli}(p+r)\| \mathrm{Bernoulli}(p))} \overset{(b)}\leq \sqrt{\frac{mn}{2}\chi^2(\mathrm{Bernoulli}(p+r)\| \mathrm{Bernoulli}(p))} \\
        & = \sqrt{\frac{mnr^2}{2 p (1 - p)}}  \leq \alpha,
        \end{split}
\end{equation*}
where (a) is by the Pinsker's inequality; (b) follows from the inequality $D(P\|Q) \leq \chi^2(P\|Q)$, which holds for any distribution $P$ and $Q$.

\end{itemize}

{\noindent \bf (Part II: $r \leq c_2\left(\sqrt{\frac{p (1 - p)}{m}}\epsilon + \frac{\epsilon}{m}\right) $)}  In this part, it suffices to show that $\TV(P_{p}, P_{p+r}) \leq \frac{\epsilon}{1 - \epsilon}$ and the conclusion follows from Lemma \ref{lem: TV}. We will divide the proof into two cases based on different ranges of $p$.

\begin{itemize}[leftmargin=*]
\item (Case 1: $p \leq \frac{1}{2m}$) In this case, we have $r \leq  \frac{\epsilon}{2m}$ as long as $c_2$ is sufficiently small. Then, it is easy to check that $P_p(X=i) \leq P_{p+r}(X = i)$ for all $i \in [m]$ whereas the inequality is reversed when $i = 0$. Thus $\TV(P_{p}, P_{p + r}) = P_p(X = 0) - P_{p + r}(X = 0)$. Therefore, we have
\begin{equation*}
\TV(P_{p}, P_{p + r}) =(1 - p)^m\left(1 - \left(1 - \frac{r}{1 - p}\right)^m\right) \overset{(a)}{\leq} \frac{mr}{1 - p} \overset{(b)}\leq  \frac{\epsilon}{1 - \epsilon},
\end{equation*}
where (a) holds because $(1 - x)^m \geq 1 - mx$ for all $x \in [0,1]$, and $(1 - p)^m \leq 1$; (b) is because $r \leq \frac{\epsilon}{2m}$ and $1 - p \geq 1/2$. 
\item (Case 2: $\frac{1}{2m} < p \leq 1/2$) In this case, we have $r \leq  \sqrt{\frac{2p (1 - p)}{m}}\epsilon$.
Then, it follows that
\begin{equation*}
    \begin{split}
        \TV(P_{p}, P_{p+r})  & \overset{(a)} \leq \sqrt{\frac{D(P_{p + r} \| P_{p})}{2}} = \sqrt{\frac{m}{2}D(\mathrm{Bernoulli}(p + r)\| \mathrm{Bernoulli}(p))} \\
        & \overset{(b)} \leq \sqrt{\frac{m}{2}\chi^2(\mathrm{Bernoulli}(p + r)\| \mathrm{Bernoulli}(p))}
        = \sqrt{\frac{mr^2}{2 p (1 - p)}} \leq \frac{\epsilon}{1 - \epsilon},
    \end{split}
\end{equation*}
where (a) is by the Pinsker's inequality; (b) follows from the inequality $D(P\|Q) \leq \chi^2(P\|Q)$, which holds for any distribution $P$ and $Q$. This finishes the proof of this lemma. 
 \end{itemize}

\subsection{Proof of Theorem \ref{thm:est}} \label{app:proof-est}
We begin by deriving the lower bound, followed by the upper bound.
\vskip.2cm
{\noindent \bf (Part I: Lower bound)} 
In this part, we prove the lower bound for $r_{\alpha}^{\rm est}(\epsilon, p)$ for all $p \in [0,1/2]$ with a contradiction argument; the case $p \in [1/2, 1]$ follows by symmetry. Take $c$ to be the small constant in Lemma \ref{Lem: TV for known epsilon} and let $r = \frac{c}{2}\left[\sqrt{\frac{p(1-p)}{m}}\left(\frac{1}{\sqrt{n}}+\epsilon\right)+\frac{1}{m} \left(\frac{1}{n} + \epsilon \right)\right]$, so that $p+2r \leq 1$. Suppose that $r_{\alpha}^{\rm est}(\epsilon, p) > r$ does not hold. By the definition of $r_{\alpha}^{\rm est}(\epsilon, p)$, we have $r_{\alpha}^{\rm est}(\epsilon, p, p + 2r) \leq r_{\alpha}^{\rm est}(\epsilon, p) \leq r$. Then, there exists an estimator $\wh{p}$ such that the following conditions hold:
\begin{equation}\label{Ineq: estimation error}
    \sup_Q P_{\epsilon, p, Q}(|\wh{p} - p| \geq r) \leq \alpha \quad \textnormal{and} \quad \sup_Q P_{\epsilon, p + 2r, Q}(|\wh{p} - (p + 2r)| \geq r) \leq \alpha.
\end{equation}
Since $2r = c\left[\sqrt{\frac{p(1-p)}{m}}\left(\frac{1}{\sqrt{n}}+\epsilon\right)+\frac{1}{m} \left(\frac{1}{n} + \epsilon \right)\right]$, we have
\begin{equation*}
    \begin{split}
        \alpha & \overset{\textnormal{Lemma }\ref{Lem: TV for known epsilon}}\geq \inf_{Q_0,Q_1} \TV\left(P^{\otimes n}_{\epsilon,p,Q_0}, P^{\otimes n}_{\epsilon, p + 2r, Q_1}\right)  \geq \inf_{Q} P_{\epsilon, p, Q}(\wh{p}  < p +  r) - \sup_{Q} P_{\epsilon, p + 2r, Q}(\wh{p} < p + r) \\
        & \geq \inf_{Q} P_{\epsilon, p, Q}(|\wh{p} - p| <  r) - \sup_{Q} P_{\epsilon, p + 2r, Q}(|\wh{p} - (p + 2r)| \geq  r) \overset{\eqref{Ineq: estimation error}} \geq ( 1 - \alpha) - \alpha = 1 - 2\alpha,
    \end{split}
\end{equation*}
where the second inequality is by the definition of TV distance. Since $\alpha < 1/3$, this leads to a contradiction. Therefore, the assumption does not hold, and we conclude that $r_{\alpha}^{\rm est}(\epsilon, p) > r$.   
\vskip.2cm
{\noindent \bf (Part II: Upper bound)}
In this part, we assume that $\frac{\log(2/ \alpha)}{n} + \epsilon$ is smaller than a sufficiently small constant. This further implies that both $\frac{\log(2/ \alpha)}{n}$ and $\epsilon$ are also sufficiently small. Given $X_1, \ldots, X_n$, let us denote $\cI \subseteq [n]$ as the index set of inliers. Also, let $
\mathcal{I}^c = [n] \setminus \mathcal{I}$ denote the complement set, corresponding to the observations generated from the contamination distribution. Then, by Bernstein's inequality (see Lemma \ref{lm:Binomial-prop} (ii)), we have
\begin{equation}\label{event: A}
 \# \mathcal{I}^c  < \frac{3}{2}n\epsilon + 2 \log(4 / \alpha) ,
\end{equation}
with probability at least $1 - \alpha/2$.
We note that given \eqref{event: A} holds, $\#\mathcal{I} \geq\frac{n}{2}$ as long as $\frac{\log(2/\alpha)}{n} + \epsilon$ is sufficiently small. By applying Bernstein's inequality again, we have
\begin{equation}\label{event: B1}
     \left| \frac{1}{m \, \#\mathcal{I} } \sum_{i \in \mathcal{I}} (X_i - mp) \right| < \frac{4 \log (4 / \alpha)}{3 m\, \# \mathcal{I} }  + 2\sqrt{\frac{p (1 - p) \log(4 / \alpha)}{m\, \# \mathcal{I} }},
\end{equation}
with probability at least $1-\alpha/2$. Let 
$$\wh{p} = \frac{1}{mn}\sum_{i = 1}^n X_i \indi\{0\leq X_i \leq m\}.$$ Given that \eqref{event: A} and \eqref{event: B1} hold simultaneously, we have
\begin{equation}\label{eq:estimation-error-bound}
    \begin{split}
        &|\wh{p} - p| = \left|\frac{ \# \mathcal{I} }{n}\frac{1}{m \,  \# \mathcal{I} }\sum_{i \in \mathcal{I}} (X_i \indi\{0\leq X_i \leq m\} - mp) + \frac{1}{mn}\sum_{i \in \mathcal{I}^c}(X_i \indi\{0\leq X_i \leq m\} - mp)  \right| \\
        & \leq \left| \frac{1}{m\,  \# \mathcal{I} }\sum_{i \in \mathcal{I}} (X_i - mp)\right| + \frac{ \# \mathcal{I}^c }{n}\overset{\eqref{event: A}, \eqref{event: B1}}  \leq \frac{8 \log (4 / \alpha)}{3 m n}  + 2\sqrt{\frac{ 2 p (1 - p) \log(4 / \alpha)}{mn}} + \frac{3}{2}\epsilon + \frac{2 \log (4 / \alpha) }{n}.
    \end{split}
\end{equation}
 By the union bound, \eqref{eq:estimation-error-bound} holds with probability at least $1 - \alpha$. Since \eqref{eq:estimation-error-bound} implies the conclusion we want to prove when $m$ is smaller than any constant $C$, we may assume $m$ is larger than a sufficiently large universal constant throughout the rest of the proof.

Let us define $\wh{p}_s$, $\wh{p}_l$, and $\wh{p}_g$ by 
\begin{equation*}
\begin{split}
&\wh{p}_s = 1 - \left(\frac{1}{n} \sum_{i=1}^{n} \indi \{X_i = 0\}\right)^{1/m}, \qquad \wh{p}_l = \left(\frac{1}{n} \sum_{i=1}^{n} \indi \{X_i = m\}\right)^{1/m},
\end{split}
\end{equation*}
and
\begin{equation*}
    \wh{p}_g = \argmin _{p \in [0,1]} \max_{t \in \bbR} \left|  F_n(t) - P_p(X \leq t) \right|.
\end{equation*}
 Intuitively, $\wh{p}_s$ and $\wh{p}_l$ are good estimators of $p$ when $p$ is small and large, respectively, and $\wh{p}_g$ is an estimator for the middle range of $p$ based on total variation learning \citep{gao2018robust}. To take advantage of these properties, define $\wh{p}$ as follows 
 \begin{equation*}
     \wh{p} = \begin{cases}
         \wh{p}_s & \text{if }\frac{1}{n} \sum_{i=1}^{n} \indi \{X_i = 0\} \geq \exp(-\frac{3}{2}C) \textnormal{ and } \frac{1}{n} \sum_{i=1}^{n} \indi \{X_i = m\} < \exp(-\frac{3}{2}C)\\
         \wh{p}_l & \text{if } \frac{1}{n} \sum_{i=1}^{n} \indi \{X_i = 0\} < \exp(-\frac{3}{2}C) \textnormal{ and } \frac{1}{n} \sum_{i=1}^{n} \indi \{X_i = m\} \geq \exp(-\frac{3}{2}C) \\
         \wh{p}_{g} & \text{otherwise.}
     \end{cases}
 \end{equation*}
 where $C>0$ is some sufficiently large universal constant. Suppose $m$ is sufficiently large, ensuring that $m \geq 4C$. When $p \leq C/m$, we have for any distribution $Q$,
 \begin{equation*}
     P_{\epsilon, p, Q}(X = 0) \geq (1 - \epsilon)(1 - p)^{m} \geq (1 - \epsilon)(1 - C/m)^{m} \geq \frac{1}{2}\exp(-C),
 \end{equation*}
as long as $\epsilon$ is sufficiently small and $m$ is sufficiently large. Therefore, by Bernstein's inequality, there exists an event that happens with probability at least $1 - \alpha/6$ such that
$$
     \frac{1}{n} \sum_{i=1}^{n} \indi \{X_i = 0\} \geq \exp \left(-\frac{3C}{2}C \right),
$$
as long as $\frac{\log(2 / \alpha)}{n}$ is sufficiently small and $C$ is large enough. By a similar argument, when $p \leq 1/2$, we have
$$
\frac{1}{n} \sum_{i=1}^{n}\indi\{ X_i = m \} < \exp\left( -\frac{3}{2} C \right),
$$
with probability at least $1 - \alpha/6$. In addition, when $p > 2C / m$, it holds that
$$
\frac{1}{n} \sum_{i=1}^{n} \indi \{ X_i = 0 \} < \exp\left( -\frac{3}{2} C \right),
$$
with probability at least $1 - \alpha/6$. Therefore, with probability at least $1 - \alpha/3$, $\widehat{p} = \widehat{p}_s$ when $p \leq C/m$; $\widehat{p} \in \{\widehat{p}_s, \widehat{p}_g\}$ when $C/m < p \leq 2C/m$; and $\widehat{p} = \widehat{p}_g$ when $2C/m< p \leq 1/2$. Due to the symmetry of the estimator $\widehat{p}$ and the binomial distribution, it suffices to consider the case $p \leq 1/2$. It is therefore enough to show that $\widehat{p}_s$ achieves the desired performance with probability at least $1 - 2\alpha/3$ when $p \leq 2C/m$, and $\widehat{p}_g$ does so when $C/m < p \leq 1/2$.  

\begin{itemize}[leftmargin=*]
\item (Case 1: $\widehat{p}_s$ when $p \leq 2C/m$)
 By Bernstein's inequality, we have with probability at least $1 - 2\alpha/3$, the following event happens:
\begin{equation*}
   (A) = \left\{ \left|\frac{1}{n} \sum_{i=1}^{n} \indi \{X_i = 0\} - P_{\epsilon, p, Q}(X = 0)\right| \leq \frac{4 \log(3 / \alpha)}{3n} + 2\sqrt{\frac{(1 - P_{\epsilon, p, Q}(X = 0))\log(3 / \alpha)}{n}} \right\}.
\end{equation*}
Given $(A)$ happens, we have
\begin{equation}\label{Ineq:diff bound}
    \begin{split}
        &\left| (1 - \wh{p}_s)^m - (1 - p)^m \right| 
        \leq \left| \frac{1}{n} \sum_{i=1}^{n} \indi \{X_i = 0\} - P_{\epsilon, p, Q}(X = 0) \right| + \left|P_{\epsilon, p, Q}(X = 0) - P_p(X = 0)\right| \\
        & \overset{(A)}\leq \frac{4 \log(3 / \alpha)}{3n} + 2\sqrt{\frac{(1 - (1 - \epsilon)(1 - p)^m - \epsilon Q(X= 0))\log(3 / \alpha)}{n}} + \epsilon \\
        & \leq \frac{4 \log(3 / \alpha)}{3n} + 2\sqrt{\frac{(1 - (1 - p)^m )\log(3 / \alpha)}{n}} + 2\sqrt{\frac{\epsilon\log(3 / \alpha)}{n}} + \epsilon \\
        & \leq \frac{7 \log(3 / \alpha)}{3n} + 2\sqrt{\frac{mp\log(3 / \alpha)}{n}} + 2\epsilon.
    \end{split}
\end{equation}
We now show, by contradiction, that $\wh{p}_s < \frac{3C}{m}$ under the event $(A)$, and then apply the mean value theorem to relate $(1 - \wh{p}_s)^m$ and $(1 - p)^m$. Suppose that $(A)$ occurs and $\wh{p}_s \geq 3C/m$. Then we have
\begin{equation*}\label{Ineq: ineq for contradiction}
\left| (1 - \wh{p}_s)^m - (1 - p)^m \right| \geq \left(1 - \frac{2C}{m} \right)^m - \left(1 - \frac{3C}{m} \right)^m> \frac{1}{2}\left(\exp(-2C) - \exp(-3C)\right),
\end{equation*}
where the last inequality holds when $m$ is large enough. Since $\frac{\log(2 / \alpha)}{n} + \epsilon$ is sufficiently small, this contradicts \eqref{Ineq:diff bound} and we conclude that $\wh{p}_s < 3C /m$ whenever $(A)$ happens. By the mean value theorem, there exists some $\widetilde{p}$ between $p$ and $\wh{p}_s$ such that
\begin{equation}\label{Ineq: MVT1}
(1 - \wh{p}_s)^m - (1 - p)^m = m(1 - \widetilde{p})^{m - 1} (\wh{p}_s - p).    
\end{equation}
When \eqref{Ineq:diff bound} holds, we have $\widetilde{p} \leq  \max\{\wh{p}_s, p \}< 3C/m$. Therefore, with probability at least $1 - 2\alpha/3$, we have
\begin{equation}\label{Ineq:ps-final-bound}
    \begin{split}
        m \left(1 - \frac{3C}{m}\right)^{m-1} \, |\wh{p}_s - p| 
        \leq m(1 - \widetilde{p})^{m-1} \, |\wh{p}_s - p|  \overset{\eqref{Ineq:diff bound}, \eqref{Ineq: MVT1}}\leq \frac{7 \log(3 / \alpha)}{3n} + 2\sqrt{\frac{mp\log(3 / \alpha)}{n}} + 2\epsilon.
    \end{split}
\end{equation}
We have $\left(1 - \frac{3C}{m} \right)^{m - 1} > \frac{\exp(-3C)}{2}$ when $m$ is sufficiently large. Therefore, \eqref{Ineq:ps-final-bound} implies that with probability at least $1 - 2\alpha/3$,
\begin{equation*}
        |\wh{p}_s - p|  \overset{(a)}\lesssim \sqrt{\frac{p}{mn}} + \frac{1}{m} \left( \frac{1}{n} + \epsilon \right)  \overset{(b)}{\asymp} \sqrt{\frac{p(1 - p)}{m}} \left( \frac{1}{\sqrt{n}} + \epsilon \right) + \frac{1}{m} \left( \frac{1}{n} + \epsilon \right),
\end{equation*}
where in (a), $\lesssim$ holds up to a constant depending on $\alpha$; (b) follows from the fact that $p(1-p) \asymp p$ since $p \leq 1/2$, and we are considering the case $p \lesssim 1/m$, under which $\epsilon \sqrt{\frac{p(1 - p)}{m}} \lesssim \frac{\epsilon}{m}$. This finishes the proof of Case 1.

\item (Case 2: $\widehat{p}_g$ when $C/m < p \leq 1/2$)
 By the DKW inequality (see Lemma \ref{lm:DKW}), we have with probability at least $1 - 2\alpha/3$, the following event happens:
\begin{equation*}
    (B) = \left\{\sup_{t \in \bbR} \left|F_n(t) - P_{\epsilon, p,Q}(X \leq t)\right| \leq \sqrt{\log(3 / \alpha) / (2n)}\right\}.
\end{equation*}
From now on, we work under the assumption that $(B)$ occurs. For any $Q$, we have
\begin{equation}\label{Ineq: hat p m}
    \begin{split}
        &\sup_{t \in \bbR} \left|P_{\wh{p}_g}(X \leq t) - P_p(X \leq t) \right| \leq \sup_{t \in \bbR} \left|P_{\wh{p}_g}(X \leq t) - F_n(t) \right| + \sup_{t \in \bbR} \left|F_n (t) - P_p(X \leq t) \right|   \\
        & \leq 2\sup_{t \in \bbR} \left|F_n(t) - P_p(X \leq t)\right| \leq 2\left(\sup_{t \in \bbR} \left|F_n(t) - P_{\epsilon, p, Q}(X \leq t)\right| + \sup_{t \in \bbR} \left|P_{\epsilon, p, Q}(X \leq t) - P_p(X \leq t)\right|\right) \\
        &  \leq 2\sup_{t \in \bbR} \left|F_n(t) - P_{\epsilon, p, Q}(X \leq t)\right| + 2\epsilon \overset{(B)} \leq 2\sqrt{\log(3 / \alpha) / (2n)} + 2\epsilon,
    \end{split}
\end{equation}
where the second inequality is by the definition of $\wh{p}_g$. Then, we get
\begin{equation}\label{Ineq: Berry Esseen for contradiction}
    \begin{split}
        |P_{\wh{p}_g}(X \leq mp) - 1/2| &  \leq \left|P_{\wh{p}_g}(X \leq mp) - P_p(X \leq mp) \right| + |P_{p}(X \leq mp) - \bbP(N(0,1) \leq 0)| \\
   & \overset{(a)} \leq  \sup_{t \in \bbR} \left|P_{\wh{p}_g}(X \leq t) - P_p(X \leq t) \right| + \frac{7}{20\sqrt{m p (1 - p)}} + \frac{1}{6 \sqrt{m}} \\
   & \overset{(b)}\leq 2\sqrt{\log(3 / \alpha) / (2n)} + 2\epsilon + \frac{7}{20\sqrt{C/2}} + \frac{1}{6 \sqrt{m}}  \overset{(c)} \leq \frac{1}{40},
    \end{split}
\end{equation}
where (a) is by Berry-Esseen theorem (see Lemma \ref{Lem: Berry esseen}); (b) follows from \eqref{Ineq: hat p m} and the fact that $C/m < p \leq 1/2$; (c) holds since $\frac{\log(2 / \alpha)}{n}$ and $\epsilon$ are sufficiently small, and $C$ and $m$ are large enough. 

We now show that $|\wh{p}_g - p | < \frac{7}{4}\sqrt{\frac{p(1-p)}{m}}$ by contradiction. Suppose $\wh{p}_g - p \geq \frac{7}{4}\sqrt{\frac{p(1-p)}{m}}$. Then,
\begin{equation*}\label{Ineq: Bernstein for contradiction}
    \begin{split}
        P_{\wh{p}_g}(X \leq mp) & \overset{\textnormal{Lemma }\ref{lm:Binomial-prop}\, (iii)}\leq P_{p + \frac{7}{4}\sqrt{\frac{p(1-p)}{m}}}(X \leq mp)  \leq \exp\left(- \frac{147\sqrt{mp}(1-p)}{96\sqrt{mp} + 224\sqrt{1 - p}}\right) <\frac{19}{40},
    \end{split}
\end{equation*}
where the second inequality is by Bernstein's inequality and the last inequality holds since $C/m <p$ and $C$ is sufficiently large. However, this contradicts \eqref{Ineq: Berry Esseen for contradiction}. Therefore, the assumption is invalid, and we have $\wh{p}_g - p < \frac{7}{4} \sqrt{\frac{p(1 - p)}{m}}$. By a similar argument, we also obtain $p - \wh{p}_g < \frac{7}{4} \sqrt{\frac{p(1 - p)}{m}}$, and thus $\left|\wh{p}_g - p\right| < \frac{7}{4} \sqrt{\frac{p(1 - p)}{m}}$. By the mean value theorem, for any $k \in [mp - 2\sqrt{mp(1 - p)}] \cup \{0\}$, there exists some $\widetilde{p}_k$ between $p$ and $\wh{p}_g$ such that
\begin{equation}\label{Eq: MVT pg}
\wh{p}_g^k(1 - \wh{p}_g)^{m - k} - p^k(1 - p)^{m - k}
= \widetilde{p}_k^k(1 - \widetilde{p}_k)^{m - k}\left(\frac{k}{\widetilde{p}_k} - \frac{m - k}{1 - \widetilde{p}_k}\right)(\wh{p}_g - p).
\end{equation}
Then, using $|\wh{p}_g - p| < \frac{7}{4}\sqrt{\frac{p(1-p)}{m}}$, it is easy to check that for any $k \in [mp - 2\sqrt{mp(1 - p)}] \cup \{0\}$,
\begin{equation}\label{Ineq: k-over-p-tilde}
\frac{k}{\widetilde{p}_k}
 \leq m - \frac{1}{4\sqrt{2}}\sqrt{\frac{m}{p}} \quad \text{and} \quad \frac{m - k}{1 - \widetilde{p}_k} \geq  m.
\end{equation}
Therefore, we have
\begin{equation} \label{Ineq: Ineq: hat p m 2}
    \begin{split}
        & \sup_{t \in \bbR} \left|P_{\wh{p}_g}(X \leq t) - P_p(X \leq t) \right| \geq \left|P_{\wh{p}_g}\left(X \leq mp - 2\sqrt{mp(1-p)}\right) - P_{p}\left(X \leq mp - 2\sqrt{mp(1-p)}\right) \right| \\
        &  \overset{\eqref{Eq: MVT pg}}= \left|\sum_{k \leq mp - 2\sqrt{mp(1-p)}} \binom{m}{k} \widetilde{p}_k^k(1 - \widetilde{p}_k)^{m - k}\left(\frac{k}{\widetilde{p}_k} - \frac{m - k}{1 - \widetilde{p}_k}\right) \right| \left| \wh{p}_g - p \right| \\
        & \overset{\eqref{Ineq: k-over-p-tilde}} \geq \frac{1}{4\sqrt{2}}\sqrt{\frac{m}{p}} \left| \wh{p}_g - p \right|P_{\widetilde{p}_k}\left(X \leq mp - 2\sqrt{m p (1 - p)}\right) \\
        & \overset{\textnormal{Lemma }\ref{lm:Binomial-prop}\,(iii)}\geq \frac{1}{4\sqrt{2}}\sqrt{\frac{m}{p}} \left| \wh{p}_g - p \right| P_{\max\{p, \wh{p}_g\}}\left(X \leq mp - 2\sqrt{m p (1 - p)}\right) \\
        & \overset{\eqref{Ineq: hat p m}}\geq \frac{1}{4\sqrt{2}}\sqrt{\frac{m}{p}} \left| \wh{p}_g - p \right| \left(P_{p}\left(X \leq mp - 2\sqrt{m p (1 - p)}\right) - \left(2 \sqrt{\log(3 / \alpha) / (2n)} + 2 \epsilon \right)\right) \\
        & \overset{(a)}\geq \frac{1}{4\sqrt{2}}\sqrt{\frac{m}{p}} \left| \wh{p}_g - p \right| \left(\Phi(-2) - \left(\frac{7}{20\sqrt{mp(1-p)}} + \frac{1}{6\sqrt{m}}\right) - \left(2 \sqrt{\log(3 / \alpha) / (2n)} + 2 \epsilon \right)\right) \\
        & \overset{(b)}\geq \frac{1}{4\sqrt{2}}\sqrt{\frac{m}{p}} \left| \wh{p}_g - p \right| \left(\Phi(-2) - \left(\frac{7}{20\sqrt{C/2}} + \frac{1}{6\sqrt{m}}\right) - \left(2 \sqrt{\log(3 / \alpha) / (2n)} + 2 \epsilon \right)\right) \\
        & \overset{(c)}\geq \frac{\Phi(-2)}{8\sqrt{2}}\sqrt{\frac{m}{p}} \left| \wh{p}_g - p \right|, 
    \end{split}
\end{equation} where (a) is by Berry-Esseen theorem (see Lemma \ref{Lem: Berry esseen}) and $\Phi(\cdot)$ denotes the CDF of standard Gaussian; (b) uses the facts that $p \geq C/m$ and $1 - p \geq 1/2$; (c) holds since $\epsilon$ and $\frac{\log(2 / \alpha)}{n}$ are sufficiently small, and $C$ and $m$ are large enough. Thus, by \eqref{Ineq: hat p m} and \eqref{Ineq: Ineq: hat p m 2}, we have with probability at least $1 - 2\alpha/3$,
\begin{equation*}
    |\wh{p}_g - p| \overset{(a)} \lesssim \sqrt{\frac{p}{m}}\left(\frac{1}{\sqrt{n}} + \epsilon \right) \overset{(b)} \asymp \sqrt{\frac{p(1 - p)}{m}} \left( \frac{1}{\sqrt{n}} + \epsilon \right) + \frac{1}{m} \left( \frac{1}{n} + \epsilon \right),
\end{equation*}
where in (a), $\lesssim$ holds up to a constant depending on $\alpha$; (b) follows from the fact that $p(1-p) \asymp p$ since $p \leq 1/2$, and we are considering the case $p \gtrsim 1/m$, under which $\sqrt{\frac{p(1 - p)}{m}}\Big(\frac{1}{\sqrt{n}} + \epsilon\Big) \gtrsim \frac{1}{m}\Big(\frac{1}{n} + \epsilon \Big)$. This finishes the proof of Case 2.

\end{itemize}

\section{Proofs for the Poisson Model} \label{app:proof-posi-add} 

This section presents an algorithm that computes the confidence interval (\ref{eq:ci-dis-Poisson}). Then we state the proofs of Theorem \ref{thm:lower-pois}, and Theorem \ref{thm:upper-pois}. Throughout the section, we write $P_{\lambda}=\text{Poisson}(\lambda)$ so that $P_{\epsilon,\lambda,Q}=(1-\epsilon)P_{\lambda}+\epsilon Q$.

\subsection{Pseudocode for Computing (\ref{eq:ci-dis-Poisson})} \label{app:posi-alg}

\begin{algorithm}[h]
\DontPrintSemicolon
\SetKwInOut{Input}{Input}\SetKwInOut{Output}{Output}
\Input{$\{X_i\}_{i=1}^n$}
\Output{$\wh{\lambda}_{\rm{left}}$, $\wh{\lambda}_{\rm{right}}$} 
\nl Set $\lambda \leftarrow \wh{\lambda}_{\max}$ as in (\ref{eq:lambda-max}) and $\mathcal{E}\leftarrow \left\{ \frac{ 2^k \log (24 / \alpha)}{n}: k = 0,1, \ldots, \left\lfloor \log_2 \left( \frac{n \epsilon_{\max}}{\log(24/\alpha)} \right) \right\rfloor \right\} \cup \{\epsilon_{\max}\}$.\;

\nl Set
\begin{eqnarray*}
\wh{\lambda}_{\rm left} &\leftarrow& \wh{\lambda}_{\max},\\
\wh{\lambda}_{\rm right} &\leftarrow& -\log\left(\left[1-\left(\frac{2}{n}\sum_{i=1}^n\indi\{X_i\geq 1\}+\frac{6\log(24/\alpha)}{n}\right)\wedge  (1- 1/e)\right]\right).
\end{eqnarray*}

\nl For each $j\in[\wh{\lambda}_{\max}]$, set $\lambda \leftarrow \lambda -1$,\;

\qquad For each $\epsilon\in\mathcal{E}$,\;

\qquad\qquad For each $\mu\in [0,\lambda + 1] \cap \mathbb{N}_0$, compute $\phi_{\mu,\epsilon}^+$ in (\ref{eq:pos-test+}).

\qquad If $\max_{\epsilon\in\mathcal{E}}\min_{\mu\in [0,\lambda+1] \cap \mathbb{N}_0}\phi_{\mu,\epsilon}^+=0$, set $\wh{\lambda}_{\rm{left}} \leftarrow \lambda$.

\nl For each $j\in[\wh{\lambda}_{\max}]$, set $\lambda \leftarrow \lambda+1$,\;

\qquad For each $\epsilon\in\mathcal{E}$,\;

\qquad\qquad For each $\mu \in [\lambda-1,\wh{\lambda}_{\max}] \cap \bbN$, compute $\phi_{\mu,\epsilon}^-$ in (\ref{eq:pos-test-}).

\qquad If $\max_{\epsilon\in\mathcal{E}}\min_{\mu\in [\lambda-1,\wh{\lambda}_{\max}] \cap \mathbb{N}}\phi_{\mu,\epsilon}^-=0$, set $\wh{\lambda}_{\rm{right}} \leftarrow \lambda$.
\caption{Computing Endpoints of Robust CI with Poisson Data}
\label{alg:CI-endpoints-pois}
\end{algorithm}
The following result shows that $\widehat{\CI}$ defined by (\ref{eq:ci-dis-Poisson}) is computed by Algorithm \ref{alg:CI-endpoints-pois}.
\begin{Proposition}\label{prop:pois-end-points}
    The set $\widehat{\CI}$ defined by (\ref{eq:ci-dis-Poisson}) is an interval whose endpoints are given by
    \begin{eqnarray*}
\label{eq:left-pois}\wh{\lambda}_{\rm{left}} &=& \inf\left\{\lambda\in [0, \wh{\lambda}_{\max} - 1] \cap \bbN_0: \max_{\epsilon\in\mathcal{E}}\left(\min_{\mu\in [0,\lambda+1] \cap \bbN_0  }\phi_{\mu,\epsilon}^+\right)=0\right\},\\
\label{eq:right-pois}\wh{\lambda}_{\rm{right}} &=& \sup\left\{\lambda \in [\wh{\lambda}_{\max}]\cup\left[0,1\right]: \max_{\epsilon\in\mathcal{E}}\left(\phi_{\lambda,\epsilon}^-\wedge\min_{\mu \in [\lambda-1, \wh{\lambda}_{\max}] \cap  \bbN}\phi_{\mu,\epsilon}^-\right)=0\right\},
\end{eqnarray*}
where $\mathcal{E}$ is the discretization of $[0,\epsilon_{\max}]$ given in Algorithm \ref{alg:CI-endpoints-pois},
 and the binary variables $\phi_{\mu,\epsilon}^+$ and $\phi_{\mu,\epsilon}^-$ are given by (\ref{eq:pos-test+}) and (\ref{eq:pos-test-}). In addition, they can be computed by Algorithm \ref{alg:CI-endpoints-pois}.
\end{Proposition}

\begin{proof}[Proof of Proposition \ref{prop:pois-end-points}]
    The proof of Proposition \ref{prop:pois-end-points} is similar to that of Proposition \ref{prop:binom-end-pionts}. We will omit most of the details, but present the following lemma that corresponds to Lemma~\ref{lm:p-hat-charac} in the proof of Proposition \ref{prop:binom-end-pionts}. Again, the proof of Lemma \ref{lm:lambda-hat-charac-poi} is similar to that of Lemma~\ref{lm:p-hat-charac}, and we omit it here.
\begin{Lemma}\label{lm:lambda-hat-charac-poi}
	The set $\widehat{\CI}$ defined by (\ref{eq:ci-dis-Poisson}) is an interval whose endpoints are given by
	\begin{equation*} \label{eq:lambda-left-third}
	\wh{\lambda}_{\rm{left}} = \inf\left\{\lambda\in [0, \wh{\lambda}_{\max} - 1] \cap \bbN_0: \max_{\epsilon\in\mathcal{E}}\left(\min_{\mu\in [0,\lambda+1] \cap \bbN_0  }\phi_{\mu,\epsilon}^+\right)=0\right\} ,
	\end{equation*} and 
	\begin{equation*}\label{eq:lambda-right-third}
		\begin{split}
					\wh{\lambda}_{\rm{right}} = \left\{ \begin{array}{ll}
			 \sup\{\lambda \in [\wh{\lambda}_{\max}]  :  \max_{\epsilon \in \cE} \min_{\mu \in [\lambda - 1,\wh{\lambda}_{\max}] \cap \bbN } \phi_{\mu, \epsilon}^- = 0  \} & \text{ if }  \max_{\epsilon \in \cE} \min_{\mu \in [\wh{\lambda}_{\max}]} \phi_{\mu, \epsilon}^- = 0\\
			 \sup \{ \lambda \in [0, 1 ): \phi_{\lambda, \epsilon'}^- = 0  \} & \text{ if }  \max_{\epsilon \in \cE} \min_{\mu \in [\wh{\lambda}_{\max}]} \phi_{\mu, \epsilon}^- = 1,
		\end{array}  \right.
		\end{split}
	\end{equation*}
    where $\epsilon' \in [0,\epsilon_{\max}]$ can be chosen arbitrarily.
\end{Lemma}

\end{proof}

\subsection{Proof of Theorem \ref{thm:lower-pois}}

A key theorem is given below and its proof will be given in the subsequent subsections.

\begin{Theorem}\label{thm:test-low-Poisson}
For any $\alpha \in (0,1)$, $\epsilon_{\max} \in [0, 1/2]$, and $n \geq 3$ satisfying $\epsilon_{\max} \geq \frac{2\alpha}{n}$, there exists some constant $c > 0$ only depending on $\alpha$ and $\epsilon_{\max}$, such that for any $\epsilon \in [0, \epsilon_{\max}]$ and $\lambda \geq 0$, as long as $$r\leq c\left(\left( \sqrt{\lambda}\left(\frac{1}{\sqrt{\log n}}+\frac{1}{\sqrt{\log(1/\epsilon)}}\right) + 1 \right)\wedge\lambda +\frac{1}{n}+\epsilon\right),$$ we have
\begin{equation*}\label{eq:l-s-p Poisson}
\begin{split}
	&\textnormal{either }\inf_{Q_0,Q_1} \TV\left(P^{\otimes n}_{\epsilon_{\max},\lambda - r,Q_0}, P^{\otimes n}_{\epsilon, \lambda, Q_1}\right) \leq \alpha \quad \textnormal{or}\quad  \inf_{Q_0,Q_1} \TV\left(P^{\otimes n}_{\epsilon_{\max},\lambda + r,Q_0}, P^{\otimes n}_{\epsilon, \lambda, Q_1}\right) \leq \alpha.
\end{split}
\end{equation*}
\end{Theorem}

The result of Theorem \ref{thm:lower-pois} is followed directly by a combination of Lemma \ref{Lem: TV to lower bound} and Theorem \ref{thm:test-low-Poisson}.

\subsubsection{Proof of Theorem \ref{thm:test-low-Poisson}}
For simplicity, let us define
\begin{equation} \label{eq:l-poi}
	\ell(n,\epsilon,\lambda)=\left( \sqrt{\lambda}\left(\frac{1}{\sqrt{\log n}}+\frac{1}{\sqrt{\log(1/\epsilon)}}\right) + 1 \right)\wedge\lambda +\frac{1}{n}+\epsilon.
\end{equation}
We will divide the rest of the proof into three parts based on different ranges of $\lambda$: $\lambda \in [0, \frac{1}{4} \left( \frac{\alpha}{n} + \epsilon  \right) ]$, $\lambda \in (\frac{1}{4} \left( \frac{\alpha}{n} + \epsilon  \right), C^*(\log n \wedge \log(1 / \epsilon))]$, and $\lambda \in (C^*(\log n \wedge \log(1 / \epsilon)), \infty)$, where $C^*$ is some large constant only depending on $\alpha$ and $\epsilon_{\max}$ and the conditions it needs to satisfy will be specified later.
\vskip.2cm
{\noindent \bf (Part I: $\lambda \in [0, \frac{1}{4} \left( \frac{\alpha}{n} + \epsilon  \right) ] $)} In this part, we will divide the proof into two cases based on the magnitude of $n$ and $\epsilon$.
\begin{itemize}[leftmargin=*]
\item (Case 1: $\alpha/n \geq \epsilon$) In this case, $\lambda \in [0, \frac{\alpha}{2n}] $ and $\ell(n,\epsilon,\lambda) = \lambda + \frac{1}{n} + \epsilon \leq \frac{\alpha}{2n} + \frac{2}{n}$. We will show that when
	$r\leq c\ell(n,\epsilon,\lambda) \leq c (2 + \alpha/2) \frac{1}{n}$ for some small enough $c > 0$, then $\inf_{Q_0,Q_1} \TV\left(P^{\otimes n}_{\epsilon_{\max},\lambda + r,Q_0}, P^{\otimes n}_{\epsilon, \lambda, Q_1}\right) \leq \alpha$. We take $Q_0 = \mathrm{Poisson}(\lambda + r)$ and $Q_1 = \mathrm{Poisson}(\lambda)$. Then 
	\begin{equation} \label{ineq:lower-part1-case1-Poisson}
\TV\left(P_{\epsilon_{\max},\lambda + r,Q_0}, P_{\epsilon, \lambda, Q_1}\right) = \TV(P_{\lambda + r},P_{\lambda})  \overset{\textnormal{Lemma }\ref{lm:TV monotonicity}}\leq \TV(P_{\lambda + r},P_{0})  = 1 - \exp(-\lambda-r)   \leq \lambda + r \leq \frac{\alpha}{n},
	\end{equation} where the last inequality holds as long as we take $c \leq \frac{\alpha}{4 + \alpha}$. Then by the property of TV distance on the product measure, we get 
\begin{equation*}
  \inf_{Q_0,Q_1} \TV\left(P^{\otimes n}_{\epsilon_{\max},\lambda + r,Q_0}, P^{\otimes n}_{\epsilon, \lambda, Q_1}\right)  \leq n\TV(P_{\lambda + r},P_{\lambda}) \overset{\eqref{ineq:lower-part1-case1-Poisson}} \leq \alpha.
\end{equation*}
	\item (Case 2: $\alpha /n \leq \epsilon$) In this case, $\lambda \in [0, \frac{1}{2}\epsilon] $ and $\ell(n,\epsilon,\lambda) = \lambda + \frac{1}{n} + \epsilon \leq \frac{1}{2} \epsilon + (1+ 1/\alpha) \epsilon$. We will show that when
	$r\leq c\ell(n,\epsilon,\lambda) \leq c ((1 + 1/\alpha) + 1/2) \epsilon$ for some small enough $c > 0$, then $\inf_{Q_0,Q_1} \TV\left(P^{\otimes n}_{\epsilon_{\max},\lambda + r,Q_0}, P^{\otimes n}_{\epsilon, \lambda, Q_1}\right) \leq \alpha$. Now, we take $Q_0 = \mathrm{Poisson}(\lambda + r)$, and $Q_1$ to have the following probability mass function:
\begin{equation}\label{q_1 for small lambda}
    q_1(k)= \frac{1 }{\epsilon} \frac{\exp(-\lambda - r)(\lambda + r)^k}{k!} - \frac{1 - \epsilon} {\epsilon}\frac{\exp(- \lambda) \lambda^k}{k!}, \quad \forall k \in \bbN_0.
\end{equation} It is easy to see that $P_{ \epsilon_{\max}, \lambda+r, Q_0}$ and $P_{\epsilon,\lambda,Q_1}$ exactly match as long as $Q_1$ is a valid distribution. This will directly imply the result. So, we just need to verify $Q_1$ is a valid distribution. It is easy to check that $\sum_{k=0}^{\infty}q_1(k)=1$. To show that the formula \eqref{q_1 for small lambda} is a valid probability mass function, we only need to verify that $q_1(k)\geq 0$ for all $k \in \bbN_0$. This is true because for any $\lambda > 0$ and $k \in \bbN_0$, we have
\begin{equation*}
    \begin{split}
    \frac{\frac{\exp(-\lambda - r)(\lambda + r)^k}{k!}}{\frac{\exp(-\lambda)\lambda^k}{k!}} & = \exp(-r )\left(\frac{\lambda + r}{\lambda}\right)^k 
   \geq \exp(-r) \geq 1 - r \geq  1 - \epsilon c(1.5 + 1/\alpha)    \geq 1 - \epsilon,
    \end{split}
\end{equation*}
where the last inequality holds as long as $c\leq \frac{1}{1.5 + 1/\alpha}$. When $\lambda = 0$, we also have $q_1(k) \geq 0$ for all $k \in \bbN_0$ since $\exp(-r) \geq 1 - \epsilon$ as long as $c\leq \frac{1}{1.5 + 1/\alpha}$.
\end{itemize}

\vskip.2cm
{\noindent \bf (Part II: $\lambda \in (\frac{1}{4} \left( \frac{\alpha}{n} + \epsilon  \right), C^*(\log n \wedge \log(1 / \epsilon)) ] $)} In this part, we have
\begin{equation*}
        \ell(n,\epsilon,\lambda) = \left( \sqrt{\frac{\lambda}{\log n}}+\sqrt{\frac{\lambda}{\log(1/\epsilon)}} + 1 \right)\wedge\lambda +\frac{1}{n}+\epsilon  \leq (2\sqrt{C^*} + 1) \wedge \lambda + 1 \wedge \frac{4}{\alpha}\lambda \leq (2\sqrt{C^*} +2) \wedge \frac{5}{\alpha}\lambda,
\end{equation*}
where the first inequality is because of the regime of $\lambda$. Next, we will show that when $r \leq c\ell(n,\epsilon,\lambda) \leq  c((2\sqrt{C^*} + 2) \wedge \frac{5}{\alpha}\lambda)$ for some small enough $c > 0$, $\inf_{Q_0,Q_1} \TV\left(P^{\otimes n}_{\epsilon_{\max},\lambda - r,Q_0}, P^{\otimes n}_{\epsilon, \lambda, Q_1}\right) \leq \alpha.$ Note that as long as $c \leq \alpha/5$, we have $ \lambda - r \geq 0$.
	Now, we construct $Q_1 = \mathrm{Poisson}(\lambda)$. To construct $Q_0$, we define its probability mass function as
 \begin{equation}\label{q_0 for small lambda Poisson}
    q_0(k)=\frac{1}{\epsilon_{\max}} \frac{\exp(-\lambda)\lambda^k}{k!} -\frac{1-\epsilon_{\max}}{\epsilon_{\max}}\frac{\exp(-\lambda + r)(\lambda - r)^k}{k!}, \quad \forall k \in \bbN_0.
\end{equation} As long as the formula \eqref{q_0 for small lambda Poisson} is valid probability mass function, the distributions $P_{\epsilon_{\max}, \lambda - r, Q_0}$ and $P_{\epsilon, \lambda, Q_1}$ exactly match, i.e., $\TV(P^{\otimes n}_{\epsilon_{\max},\lambda - r,Q_0}, P^{\otimes n}_{\epsilon, \lambda, Q_1}) = 0$ and it implies our result. Next, we verify $Q_0$ is a valid distribution. It is easy to check $\sum_{k=0}^{\infty} q_0(k) = 1$, so we just need to show $q_0(k) \geq 0$ for all $k \in \bbN_0$ to confirm that $Q_0$ is a valid distribution.
\begin{equation}\label{Cond: lower bound first case Poisson}
    \begin{split}
        & q_0(k) \geq 0, \quad \forall k \in \bbN_0 
        \Longleftrightarrow  \exp(-\lambda)\lambda^k - (1 - \epsilon_{\max})\exp(-\lambda + r)(\lambda - r)^k \geq 0, \quad \forall k \in \bbN_0 \\ \\
    \Longleftrightarrow &  \exp(r)\left(\frac{\lambda - r}{\lambda}\right)^k \leq \frac{1}{1 - \epsilon_{\max}}, \quad \forall k \in \bbN_0 
    \overset{(a)} \Longleftrightarrow   \exp(r) \leq \frac{1}{1 - \epsilon_{\max}}
    \Longleftrightarrow r \leq \log\left(\frac{1}{1 - \epsilon_{\max}}\right),
    \end{split}
\end{equation}
where (a) is because $\left(\frac{\lambda - r}{\lambda}\right)^k$ is decreasing in $k$. Notice that the last condition in \eqref{Cond: lower bound first case Poisson} is satisfied as long as we choose $c \leq \frac{\log\left(\frac{1}{1-\epsilon_{\max}}\right)}{2\sqrt{C^*} + 2}$.

\vskip.2cm
{\noindent \bf (Part III: $\lambda \in (C^*(\log n \wedge \log(1 / \epsilon)), \infty) $)} In this part, we have
\begin{equation}\label{Ineq: lambda is greater}
\begin{split}
 &\sqrt{\lambda} \left(\frac{1}{\sqrt{\log n }} + \frac{1}{ \sqrt{\log(1 / \epsilon)}}\right) + 1   \leq 2\sqrt{\frac{\lambda}{\log n \wedge \log(1 / \epsilon)}} +  1 \\
 & \overset{(a)} \leq 2\sqrt{3\lambda (\log n \wedge \log(1 / \epsilon))} + 3(\log n \wedge \log(1 / \epsilon)) \overset{(b)} < 2\sqrt{\frac{3}{C^*}}\lambda + \frac{3}{C^*}\lambda  \overset{(c)} < \lambda,
\end{split}
\end{equation}
where (a) holds since $\log n \wedge \log(1 / \epsilon) \geq \log (2)$, and both $1/x \leq 3x$ and $1 \leq 3x$ hold for all $x \geq \log (2)$; (b) is because of the regime of $\lambda$ in this part; (c) is satisfied as long as $C^{*} \geq 18$. Thus, we have 
\begin{equation*}
    \begin{split}
        \ell(n, \epsilon, \lambda) & \overset{\eqref{eq:l-poi} ,\eqref{Ineq: lambda is greater}}= \sqrt{\lambda} \left(\frac{1}{\sqrt{\log n }} + \frac{1}{ \sqrt{\log(1 / \epsilon)}}\right) + 1 +\frac{1}{n}+\epsilon  \leq 2\sqrt{\frac{\lambda}{\log n \wedge \log(1 / \epsilon)}} + 2\\
        & \leq 2\sqrt{\frac{\lambda}{\log n \wedge \log(1 / \epsilon)}} + \frac{2}{\sqrt{C^*}}\sqrt{\frac{\lambda}{\log n \wedge \log(1 / \epsilon)}} \leq 3\sqrt{\frac{\lambda}{\log n \wedge \log(1 / \epsilon)}},
    \end{split}
\end{equation*}
where the last inequality holds as long as we choose $C^* \geq 4$.
Hence, it suffices to show that when
	$r\leq c\ell(n,\epsilon,\lambda) \leq 3c\sqrt{\frac{\lambda}{\log n \wedge \log(1 / \epsilon)}}$ for some small enough $c > 0$, then $\inf_{Q_0,Q_1} \TV\left(P^{\otimes n}_{\epsilon_{\max},\lambda + r,Q_0}, P^{\otimes n}_{\epsilon, \lambda, Q_1}\right) \leq \alpha$. We will divide the proof into two cases based on the magnitude of $n$ and $\epsilon$.

\begin{itemize}[leftmargin=*]
\item (Case 1: $n \leq  1 / \epsilon$) In this case, $\log n \leq \log(1 / \epsilon)$. Let $C_1 = \frac{\log (3)}{\log (3 / \alpha)}C^*$. Then, we have
\begin{equation*}
    \lambda > C^* \log n \geq C^*  \frac{\log (3)}{\log (3/ \alpha)}\log (n / \alpha) =  C_1 \log(n / \alpha),
\end{equation*}
where the second inequality is because $ \log n = \frac{\log n}{\log n + \log(1 / \alpha)} \log(n/ \alpha) \geq \frac{\log (3)}{\log (3) + \log(1 / \alpha)} \log(n/ \alpha)$. In addition, 
\begin{equation*}
     \ell(n, \epsilon, \lambda)  \lesssim \sqrt{\lambda} \left(\frac{1}{\sqrt{\log n }} + \frac{1}{ \sqrt{\log(1 / \epsilon)}}\right) \lesssim \sqrt{\frac{\lambda}{\log (n / \alpha)}},
\end{equation*}
where in this proof, we use $a \lesssim b$ to mean there exists a constant $C$ depending on $\alpha$ only so that $a \leq C b$. Next, we are going to show that when $\lambda > C_1 \log(n / \alpha)$ and $r \leq  c \sqrt{\frac{\lambda}{\log (n / \alpha)}}$ for some large enough $C_1 > 0$ and small enough $c>0$, then we can construct $Q_0$ and $Q_1$ such that $\TV\left(P^{\otimes n}_{\epsilon_{\max},\lambda - r,Q_0}, P^{\otimes n}_{\epsilon, \lambda, Q_1}\right)  \leq \alpha$.

In particular, we construct $Q_1 = \mathrm{Poisson}(\lambda)$ and define the probability mass function of $Q_0$ by
\begin{equation*}
    q_0(k)=\frac{1}{\epsilon_{\max}}\exp(-\lambda)\lambda^{k} - \frac{1-\epsilon_{\max}}{\epsilon_{\max}}\frac{\indi \{k \geq t_n\}}{\bbP(\mathrm{Poisson}(\lambda - r)\geq t_n)} \exp(-\lambda+r)(\lambda - r)^{k},
\end{equation*}
for all $k \in \bbN_0$, where 
\begin{equation*}
    t_n = \lambda - 3\sqrt{\lambda \log (n / \alpha)}.
\end{equation*}
Then, we have
\begin{equation}\label{Ineq: tn is greater than 0 Poisson}
    t_n = \lambda - 3\sqrt{\lambda \log (n / \alpha)} >\sqrt{C_1 \lambda \log (n / \alpha)} - 3\sqrt{\lambda \log (n / \alpha)} = \Big(\sqrt{C_1} - 3\Big) \sqrt{\lambda \log (n / \alpha)},
\end{equation}
where in the first inequality we use the fact $\lambda > C_1 \log (n / \alpha)$. Therefore, we have $t_n \geq 0$ as long as $C_1 \geq 9$. Also, when $n \geq 3$, we have
\begin{equation*}
    \lambda - t_n = 3\sqrt{\lambda \log (n / \alpha)} \geq 3\sqrt{\frac{\lambda}{\log(n / \alpha)}} \geq \frac{3r}{c},
\end{equation*}
where the first inequality is because $\log(n / \alpha) \geq \log(3) > 1$ and in the second inequality we use the fact that $r \leq c \sqrt{\frac{\lambda}{\log (n / \alpha)}}$. Thus, when $c < 1$ we have $\sqrt{\lambda\log(n / \alpha)} > r$ and $\lambda - t_n > 2r$. In particular, this implies that $\lambda > 2r \geq r$ when $c$ is sufficiently small.

Next, we derive the conditions for $Q_0$ to be a valid distribution. It is easy to check that $\sum_{k=0}^{\infty}q_0(k)=1$. Therefore, to ensure that $Q_0$ is a valid distribution, we need to show that $q_0(k) \geq 0$ for all $k \in \bbN_0$, which is guaranteed by
\begin{equation}\label{Cond: Lower bound with n Poisson}
    \begin{split}
       & \exp(-\lambda)\lambda^{k} \geq  \frac{(1-\epsilon_{\max})\exp(-\lambda+r)(\lambda - r)^{k}}{\bbP(\mathrm{Poisson}(\lambda - r)\geq t_n)}, \quad \forall k \geq t_n\\
        \Longleftrightarrow & \exp(-r) \left(\frac{\lambda}{\lambda - r}\right)^k \geq \frac{1-\epsilon_{\max}}{\bbP(\mathrm{Poisson}(\lambda - r)\geq t_n)}, \quad \forall k \geq t_n \\
        \overset{(a)}\Longleftarrow & \exp(-r) \left(\frac{\lambda}{\lambda - r}\right)^{t_n} \geq \frac{1-\epsilon_{\max}}{\bbP(\mathrm{Poisson}(\lambda - r)\geq t_n)} \\
        \Longleftrightarrow & - r + t_n\log\left(\frac{\lambda}{\lambda - r}\right) \geq \log((1 - \epsilon_{\max})/\bbP(\mathrm{Poisson}(\lambda - r)\geq t_n)) \\
        \overset{(b)}{\Longleftarrow} & \frac{(t_n - \lambda) r}{\lambda} \geq \log((1 - \epsilon_{\max})/\bbP(\mathrm{Poisson}(\lambda - r)\geq t_n)) \\
        \Longleftrightarrow & r \leq \frac{\lambda\log(\bbP(\mathrm{Poisson}(\lambda - r)\geq t_n)/(1 - \epsilon_{\max}))}{\lambda - t_n},
    \end{split}
\end{equation}
where (a) is because $\left(\frac{\lambda}{\lambda - r}\right)^k$ is increasing in $k$; (b) is because $\log (1+x)\geq x/(1+x)$ for all $x>-1$. Also, when $C_1 \geq 49$, we have
\begin{equation} \label{Ineq: lambda, r, tn}
    \frac{\lambda - r - t_n}{t_n} \leq \frac{\lambda - t_n}{t_n} = \frac{3\sqrt{\lambda\log(n / \alpha)}}{\lambda - 3\sqrt{\lambda\log(n / \alpha)}} \overset{(a)}\leq \frac{3\sqrt{\lambda\log(n / \alpha)}}{\sqrt{C_1\lambda\log(n / \alpha)} - 3\sqrt{\lambda\log(n / \alpha)}} = \frac{3}{\sqrt{C_1} - 3} \leq \frac{3}{4},
\end{equation} 
where in (a) we use the fact that $\lambda \geq C_1\log(n / \alpha)$. 

In addition, when $Q_0$ is a valid distribution, we have
\begin{equation} \label{Ineq: Bound for TV with n Poisson}
    \begin{split}
        &\TV(P_{\epsilon_{\max},\lambda - r, Q_0},P_{\epsilon, \lambda, Q_1}) = \frac{1}{2}\sum_{k=0}^{\infty}(1-\epsilon_{\max})\left|1-\frac{\indi \{k\geq t_n\}}{\bbP(\mathrm{Poisson}(\lambda - r)\geq t_n)}\right| \frac{\exp(-\lambda + r)(\lambda - r)^k}{k!} \\
        &= (1-\epsilon_{\max})\bbP(\mathrm{Poisson}(\lambda - r) < t_n) \leq \bbP(\mathrm{Poisson}(\lambda - r) < t_n)\\
        & \overset{(a)}{\leq} \exp \left( - \lambda + r + t_n + t_n\log \Big( \frac{\lambda - r }{t_n} \Big) \right) \overset{(b)}\leq \exp \left( - \frac{(\lambda - r - t_n)^2}{2t_n} + \frac{(\lambda - r - t_n)^3}{3t_n^2} \right) \\
        &\overset{\eqref{Ineq: lambda, r, tn}}\leq \exp \left( - \frac{(\lambda - r - t_n)^2}{2t_n} + \frac{3}{4} \frac{(\lambda - r - t_n)^2}{3t_n}\right)  = \exp \left(-\frac{(\lambda - r - t_n)^2}{4t_n} \right) \\
        &\overset{(c)}\leq \exp \left(-\frac{\left(\lambda - t_n - \sqrt{\lambda\log(n / \alpha)}\right)^2}{4\lambda} \right) \overset{(d)}= \frac{\alpha}{n},
    \end{split}
\end{equation}
where in (a) we use the Chernoff bound for the Poisson distribution (see Lemma~\ref{Lem: Chernoff bound for Poisson}); in (b) we use the fact that $\log( 1 + x ) \leq x - \frac{x^2}{2} + \frac{x^3}{3}$ for all $x \geq 0$ and $\lambda - t_n  - r > 2r - r \geq 0$ when $c$ is sufficiently small; (c) is because $r < \sqrt{\lambda\log(n / \alpha)}$ and $t_n < \lambda$; (d) is by the definition of $t_n$. Then, the following holds by the property of TV distance on the product measure:
\begin{equation*}
    \TV\left(P^{\otimes n}_{\epsilon_{\max},\lambda - r,Q_0},P^{\otimes n}_{\epsilon, \lambda, Q_1}\right) \leq  n\TV(P_{\epsilon_{\max},\lambda - r, Q_0},P_{\epsilon,\lambda,Q_1}) \leq \alpha.
\end{equation*}
On the other hand, from the derivation of \eqref{Ineq: Bound for TV with n Poisson}, we also obtain the inequality $\mathbb{P}(\mathrm{Poisson}(\lambda-r) \geq t_n) \geq 1 - \frac{\alpha}{n}$, so the last condition in \eqref{Cond: Lower bound with n Poisson} is implied by
\begin{equation*}
    \begin{split}
       & r \leq \frac{\lambda\log((1 - \frac{\alpha}{n})/(1 - \epsilon_{\max}))}{\lambda - t_n}
         \overset{(a)}{\Longleftarrow}  r \leq \frac{\lambda\log((1 - \epsilon_{\max}/2)/(1 - \epsilon_{\max}))}{\lambda - t_n} \\
        \overset{(b)} \Longleftrightarrow & r \leq  \frac{\log((1 - \epsilon_{\max} / 2)/(1-\epsilon_{\max}))}{3}\sqrt{\frac{\lambda}{\log(n/\alpha)}},
    \end{split}
\end{equation*}
where (a) holds when $\epsilon_{\max}\geq \frac{2\alpha}{n}$ and (b) is by the setting of $t_n$. The above conditions are satisfied whenever $r \leq c \sqrt{\frac{\lambda}{\log (n / \alpha)}}$ for some sufficiently small $c > 0$ only depending on $\epsilon_{\max}$.

\item (Case 2: $n \geq  1 / \epsilon$) In this case, we have $\log n \geq \log(1 / \epsilon)$. To keep the notation consistent with the previous case, we let $C_2 = C^*$. Then, we have $\lambda > C_2\log(1 / \epsilon)$. In addition,
\begin{equation*}
     \ell(n, \epsilon, \lambda)  \lesssim \sqrt{\lambda} \left(\frac{1}{\sqrt{\log n }} + \frac{1}{ \sqrt{\log(1 / \epsilon)}}\right) \lesssim \sqrt{\frac{\lambda}{\log (1 / \epsilon)}}.
\end{equation*} Next, we are going to show that when $\lambda \geq C_2\log(1 / \epsilon)$ and $r \leq  c \sqrt{\frac{\lambda}{\log (1/ \epsilon)}}$ for large $C_2>0$ and some small enough $c>0$, then $\inf_{Q_0, Q_1}\TV\left(P^{\otimes n}_{\epsilon_{\max},\lambda - r,Q_0}, P^{\otimes n}_{\epsilon, \lambda, Q_1}\right)  \leq \alpha$. 

 In this case, let us define
\begin{equation*}
    r_{\epsilon} = \frac{c(\epsilon_{\max})}{4}\sqrt{\frac{\lambda}{\log(1 / \epsilon)}},
\end{equation*}
where $c(\epsilon_{\max}) = \log((1 - \epsilon_{\max}/2)/ (1 - \epsilon_{\max}))$. It is sufficient to verify that for all $\epsilon \in [0, \epsilon_{\max}/2]$, if $r \leq r_{\epsilon}$, then $\inf_{Q_0,Q_1} \TV\left(P^{\otimes n}_{\epsilon_{\max},\lambda - r_{\epsilon},Q_0},P^{\otimes n}_{\epsilon, \lambda, Q_1}\right)\leq \alpha$. This is because when $\epsilon \in [ \epsilon_{\max}/2, \epsilon_{\max}]$, there exist constants $C_2'$ and $c'$ such that when $\lambda \geq C_2' \log (1 / \epsilon) \geq  C_2 \log (2 / \epsilon_{\max})$ and $r \leq c'\sqrt{\frac{\lambda}{\log(1 / \epsilon )}} \leq \frac{c(\epsilon_{\max})}{4}\sqrt{\frac{\lambda}{\log(2 / \epsilon_{\max})}} = r_{\epsilon_{\max}/2}$, then
\begin{equation*}
    \inf_{Q_0,Q_1} \TV\left(P^{\otimes n}_{\epsilon_{\max},\lambda - r,Q_0},P^{\otimes n}_{\epsilon, \lambda, Q_1}\right) \leq \inf_{Q_0,Q_1} \TV\left(P^{\otimes n}_{\epsilon_{\max},\lambda - r,Q_0},P^{\otimes n}_{\epsilon_{\max}/2, \lambda, Q_1}\right)\leq \alpha,
\end{equation*}
where the first inequality is because $\{P_{\epsilon_{\max}/2,\lambda, Q_1} : Q_1\} \subseteq \{P_{\epsilon, \lambda, Q_1}: Q_1\}$ when $\epsilon \geq \epsilon_{\max}/2$. From now on, we assume $\epsilon \in [0, \epsilon_{\max}/2]$. Then, we have
\begin{equation*}
    r_{\epsilon} = \frac{c(\epsilon_{\max})}{4}\sqrt{\frac{\lambda}{\log(1 / \epsilon)}} \leq \frac{c(\epsilon_{\max})\lambda}{4\sqrt{C_2}\log(1 / \epsilon)} \leq \frac{c(\epsilon_{\max})\lambda}{4\sqrt{C_2}\log(2 / \epsilon_{\max})},
\end{equation*}
where the first inequality is because $\lambda \geq C_2 \log(1 / \epsilon)$ and in the second inequality we use the fact that $\epsilon \leq \epsilon_{\max}/2$. Therefore, whenever $C_2 > \left(\frac{c(\epsilon_{\max})}{4\log(2 / \epsilon_{\max})}\right)^2$, we have $r_{\epsilon} < \lambda$. Accordingly, we choose $C_2$ sufficiently large, depending only on $\epsilon_{\max}$, and consider the case where $r \leq r_{\epsilon} < \lambda$ is satisfied. In this case, we define the probability mass function of $Q_1$ by
\begin{equation}\label{Def: q_1 with epsilon Poisson}
q_1(k)=
\begin{cases}
a_k & (1-\epsilon) \exp(- \lambda) \lambda^k  > (1-\epsilon_{\max}) \exp(- \lambda + r) (\lambda - r)^k \\
\frac{1-\epsilon_{\max}}{\epsilon} \frac{\exp(- \lambda + r) (\lambda - r)^k}{k!} & (1-\epsilon) \exp(- \lambda) \lambda^k  \leq (1-\epsilon_{\max}) \exp(- \lambda + r) (\lambda - r)^k,
\end{cases}
\end{equation}
for all $k \in \bbN_0$, where $a_k \geq 0$ are arbitrary nonnegative values chosen so that $\sum_{k=0}^{\infty} q_1(k) = 1$ if such a choice is possible. Suppose that the formula \eqref{Def: q_1 with epsilon Poisson} is a valid probability mass function. Then, we can define the valid probability mass function of $Q_0$ by
\begin{equation*}
    q_0(k) = \frac{(1-\epsilon)\frac{\exp(- \lambda) \lambda^k}{k!} + \epsilon q_1(k) - (1-\epsilon_{\max})\frac{\exp(- \lambda + r) (\lambda - r)^k}{k!}}{\epsilon_{\max}},
\end{equation*}
for all $k \in \bbN_0$. Then, the distributions $P_{\epsilon_{\max}, \lambda - r ,Q_0}$ and $P_{\epsilon, \lambda, Q_1}$ exactly match, which implies our result. Note that we can define $q_1(k)$ as in \eqref{Def: q_1 with epsilon Poisson} only if there exist nonnegative values $a_k \geq 0$ such that $\sum_{k=0}^{\infty} q_1(k) = 1$. This is guaranteed if 
\begin{equation}\label{Cond: Valid q_1 with epsilon Poisson}
    \sum_{k \in \mathcal{S}(\lambda, r, \epsilon)} \frac{1 - \epsilon_{\max}}{\epsilon} \frac{\exp(- \lambda + r) (\lambda - r)^k}{k!} \leq 1,
\end{equation}
where
\begin{equation*}
\mathcal{S}(\lambda, r, \epsilon) 
= \left\{k \in \bbN_0 : (1-\epsilon) \exp(- \lambda) \lambda^k  \leq (1-\epsilon_{\max}) \exp(- \lambda + r) (\lambda - r)^k \right\}.    
\end{equation*}
It is easy to check that
\begin{equation*}
    \mathcal{S}(\lambda, r, \epsilon) 
= \left\{k \in \bbN_0 : 
k \leq \frac{r - \log\left(\frac{1 - \epsilon}{1 - \epsilon_{\max}}\right)}{\log\left(\frac{\lambda}{\lambda - r}\right)}
\right\}.
\end{equation*}
For $\epsilon \in [0, \epsilon_{\max}/2]$, the threshold value appearing in the definition of $\mathcal{S}(\lambda, r, \epsilon)$ can be upper bounded as follows:
\begin{equation}\label{Ineq: Upper bound for threshold Poisson}
         \frac{r - \log\left(\frac{1 - \epsilon}{1 - \epsilon_{\max}}\right)}{\log\left(\frac{\lambda}{\lambda - r}\right)}  \overset{(a)}\leq  \frac{r - c(\epsilon_{\max})}{\log\left(\frac{\lambda}{\lambda - r}\right)}  \overset{(b)}\leq \lambda - \frac{c(\epsilon_{\max})\lambda}{r} \leq \lambda - \frac{c(\epsilon_{\max})\lambda}{r_{\epsilon}}  \overset{(c)}= \lambda - 4\sqrt{\lambda \log (1 / \epsilon)},
\end{equation}
where (a) is because $\epsilon \leq \epsilon_{\max}/2$; in (b) we use the fact that $\log (1+x) \geq x/(1+x)$ when $x>-1$; (c) is by the definition of $r_{\epsilon}$. Now, define
\begin{equation*}
    t_\epsilon =   \lambda - 4\sqrt{\lambda \log (1 / \epsilon)}.
\end{equation*}
Since $\lambda \geq C_2 \log (1 / \epsilon)$, choosing $C_2 \geq 16$ ensures that $t_\epsilon \geq 0$ following the same analysis as in \eqref{Ineq: tn is greater than 0 Poisson}. Then for any $\epsilon \in [0, \epsilon_{\max}/2]$, we have
\begin{equation}\label{Ineq: r and lambda - t Poisson}
    \begin{split}
        2r_{\epsilon} & = \frac{\log((1 - \epsilon_{\max}/2)/(1 - \epsilon_{\max}))}{2}\sqrt{\frac{\lambda}{\log(1 / \epsilon)}}  \leq \frac{\log (2/\epsilon_{\max})}{2}\sqrt{\frac{\lambda}{\log(1 / \epsilon)}} \leq \frac{\log (1/\epsilon)}{2}\sqrt{\frac{\lambda}{\log(1 / \epsilon)}} \\
        & \leq 4 \sqrt{\lambda\log (1/\epsilon)} = \lambda - t_{\epsilon},
    \end{split}
\end{equation}
where the first inequality holds for all $\epsilon_{\max} \leq 1/2$ and in the second inequality we use the fact that $\epsilon \leq \epsilon_{\max}/2$. Also, we have
\begin{equation*}
    \frac{\lambda - t_{\epsilon} + r_{\epsilon}}{t_{\epsilon}} = \frac{4\log(1 / \epsilon) + c(\epsilon_{\max})/4}{\sqrt{\lambda \log (1 / \epsilon)} - 4 \log(1 / \epsilon)} \leq \frac{(4 + c(\epsilon_{\max})/4)\log(1 / \epsilon) }{\sqrt{\lambda \log (1 / \epsilon)} - 4 \log(1 / \epsilon)} \leq \frac{4 + c(\epsilon_{\max}) / 4}{\sqrt{C_2} - 4},
\end{equation*}
where the first inequality follows from the fact that $\log(1 / \epsilon) \geq \log(2 / \epsilon_{\max}) \geq \log(4) \geq 1$, and the second inequality uses the fact that $\lambda \geq C_2 \log(1 / \epsilon)$. Therefore, by choosing a sufficiently large $C_2$ that depends only on $\epsilon_{\max}$, we can ensure that $\frac{\lambda - t_{\epsilon} + r_{\epsilon}}{t_{\epsilon}} \leq \frac{3}{4}$ holds, and we restrict our attention to this setting. Then, the condition \eqref{Cond: Valid q_1 with epsilon Poisson} is implied by
\begin{equation*}
    \begin{split}
        \eqref{Cond: Valid q_1 with epsilon Poisson} \Longleftrightarrow & \bbP\left(\mathrm{Poisson}(\lambda - r) \leq  \frac{r - \log\left(\frac{1 - \epsilon}{1 - \epsilon_{\max}}\right)}{\log\left(\frac{\lambda}{\lambda - r}\right)} \right)\leq \frac{\epsilon}{1-\epsilon_{\max}}\\
    \overset{\text{Lemma }\ref{lm:Poisson-prop},\eqref{Ineq: Upper bound for threshold Poisson} }{\Longleftarrow} & \bbP\left(\mathrm{Poisson}(\lambda - r_{\epsilon}) \leq t_{\epsilon}\right)\leq \frac{\epsilon}{1-\epsilon_{\max}}\\
    \overset{(a)}{\Longleftarrow} & - \lambda + r_{\epsilon} + t_{\epsilon} + t_{\epsilon} \log \left(\frac{\lambda - r_{\epsilon}}{t_{\epsilon}}\right) \leq \log \left( \frac{\epsilon}{1 - \epsilon_{\max}} \right) \\
    \overset{(b)} \Longleftarrow &  - \frac{(\lambda - r_{\epsilon} - t_{\epsilon})^2}{2t_{\epsilon}} + \frac{(\lambda - r_{\epsilon} - t_{\epsilon})^3}{3t_{\epsilon}^2} \leq \log \left( \frac{\epsilon}{1 - \epsilon_{\max}} \right) \\
    \overset{(c)} \Longleftarrow &  - \frac{(\lambda - r_{\epsilon} - t_{\epsilon})^2}{2t_{\epsilon}} + \frac{3}{4}\frac{(\lambda - r_{\epsilon} - t_{\epsilon})^2}{3t_{\epsilon}} \leq \log \left(\epsilon \right) 
    \Longleftrightarrow   \frac{(\lambda - r_{\epsilon} - t_{\epsilon})^2}{4t_{\epsilon}} \geq \log \left( 1 / \epsilon \right) \\
        \Longleftarrow &  \frac{(\lambda - r_{\epsilon} - t_{\epsilon})^2}{4\lambda} \geq \log \left( 1 / \epsilon \right) 
    \overset{\eqref{Ineq: r and lambda - t Poisson}} \Longleftarrow   \frac{(\lambda - t_{\epsilon})^2}{16\lambda} \geq \log \left( 1 / \epsilon \right),
    \end{split}  
\end{equation*}
where in (a) we use the Chernoff bound for the Poisson distribution (see Lemma~\ref{Lem: Chernoff bound for Poisson}); in (b) we use the fact that $\log( 1 + x ) \leq x - \frac{x^2}{2} + \frac{x^3}{3}$ for all $x \geq 0$ and $\lambda - t_{\epsilon} + r_{\epsilon} =4\sqrt{\lambda \log(1 / \epsilon)} + r_{\epsilon} \geq 0$; in (c) we use the fact that $\frac{\lambda - t_{\epsilon} + r_{\epsilon}}{t_{\epsilon}} \leq \frac{3}{4}$. Notice that the above conditions are satisfied by the setting of $t_{\epsilon}$. This finishes the proof for Case 2 of Part III and also finishes the proof of this theorem.

\end{itemize}

\subsection{Proof of Theorem \ref{thm:upper-pois}}

The proof of Theorem \ref{thm:upper-pois} is similar to that of Theorem \ref{thm:upper}. We first state a theorem that establishes the simultaneous Type-1 error and Type-2 error guarantees for the testing functions $\phi_{\lambda,\epsilon}^+$ defined by \eqref{eq:pos-test+} and $\phi_{\lambda,\epsilon}^-$ defined by \eqref{eq:pos-test-}. Its proof is provided in the subsequent subsections.

 \begin{Theorem}\label{thm:test-up-Poisson}
Suppose $ \frac{\log(2/\alpha)}{n} + \epsilon_{\max}$ is less than a sufficiently small universal constant. The testing functions $\phi_{\lambda,\epsilon}^+$ and $\phi_{\lambda,\epsilon}^-$  defined by (\ref{eq:pos-test+}) and (\ref{eq:pos-test-}) satisfy the following simultaneous Type-1 error bounds
\begin{eqnarray}
\label{ineq:type-1-phi+ Poisson} \underset{Q}{\sup}P_{\epsilon_{\max},\lambda,Q}\left(\underset{\epsilon\in[0,\epsilon_{\max}]}{\sup}\phi_{\lambda,\epsilon}^+ = 1 \right)\leq \alpha/12, \\
\label{ineq:type-1-phi- Poisson} \underset{Q}{\sup}P_{\epsilon_{\max},\lambda,Q}\left(\underset{\epsilon\in[0,\epsilon_{\max}]}{\sup}\phi_{\lambda,\epsilon}^- = 1 \right)\leq \alpha/12,
\end{eqnarray}
for all $\lambda \in [0, \infty)$. In addition, the testing function $\phi_{\lambda,\epsilon}^+$ satisfies the following Type-2 error bound,
\begin{equation*}
    \underset{Q}{\sup}P_{\epsilon,\lambda+r,Q}(\phi_{\lambda,\epsilon}^+ = 0)\leq \alpha/12   ,
\end{equation*} for all $\epsilon \in [0, \epsilon_{\max}]$, all $\lambda \in [0, \infty)$ and all
$r \geq \overline{r}(\lambda, \epsilon)$, where $\overline{r}(\lambda, \epsilon)$ is defined by (\ref{def:r-Poisson-upper}). Similarly, the testing function $\phi_{\lambda,\epsilon}^-$ satisfies the following Type-2 error bound,
\begin{equation*}
\underset{Q}{\sup}P_{\epsilon,\lambda-r,Q}(\phi_{\lambda,\epsilon}^- = 0)\leq \alpha/12,  
\end{equation*} for all $\epsilon \in [0, \epsilon_{\max}]$, all $\lambda \in \left[4\left( \frac{10 \log(24/\alpha)}{n} + 3 \epsilon \right), \infty \right)$ and all
$r\in[\underline{r}(\lambda, \epsilon),\lambda]$, where $\underline{r}(\lambda, \epsilon)$ is defined by (\ref{def:r-Poisson-under}).
\end{Theorem}

To show the guarantees of $\wh{\CI}$, we first analyze the properties of two related confidence intervals. Let us define $\psi_{\lambda, \epsilon}^+$ and $\psi_{\lambda ,\epsilon}^-$ by
\begin{equation}\label{def: psi+ pois real}
    \psi_{\lambda ,\epsilon}^+ = \min_{\mu \in[0,\lambda]}\phi_{\mu,\epsilon}^+,
\end{equation}
and 
\begin{equation}\label{def: psi- pois real}
        \psi_{\lambda ,\epsilon}^- = \min_{\mu \in[\lambda, \infty)}\phi_{\mu,\epsilon}^-.
\end{equation}
With these functions, define
\begin{equation*}
\wt{\CI}=\left\{\lambda \in [0, \infty): {\psi}_{\lambda,\epsilon}^+= {\psi}_{\lambda,\epsilon}^-=0\text{ for all }\epsilon\in [0, \epsilon_{\max}]\right\}, \label{eq:ci-wt-pois}
\end{equation*}
and
\begin{equation*}
\widebar{\CI}=\left\{\lambda \in [0, \infty): {\psi}_{\lambda,\epsilon}^+= {\psi}_{\lambda,\epsilon}^-=0\text{ for all }\epsilon\in \cE\right\}. \label{eq:ci-semidis-pois}
\end{equation*}

The coverage and length guarantees for $\wt{\CI}$ and $\widebar{\CI}$ are given as follows:
\begin{equation} \label{eq:CI-tilde-poi-guarantee}
		\begin{split}
			&\inf_{ \epsilon \in [0, \epsilon_{\max}], \lambda, Q} P_{\epsilon, \lambda, Q}\left( \lambda \in\widetilde{\CI} \right) \geq 1-\alpha/6 \quad \text{and}\quad  \inf_{\epsilon \in [0, \epsilon_{\max}], \lambda, Q } P_{\epsilon, 
            \lambda, Q}\left( |\widetilde{\CI}| \leq C \ell(n,\epsilon,\lambda) \right) \geq 1-\alpha/2,
		\end{split}
	\end{equation}
	\begin{equation}\label{eq:CI-bar-poi-guarantee}
		\begin{split}
			&\inf_{ \epsilon \in [0, \epsilon_{\max}], \lambda, Q} P_{\epsilon, \lambda, Q}\left( \lambda \in\widebar{\CI} \right) \geq 1-\alpha/6 \quad \text{and}\quad \inf_{\epsilon \in [0, \epsilon_{\max}], \lambda, Q } P_{\epsilon, 
            \lambda, Q}\left( |\widebar{\CI}| \leq C \ell(n,\epsilon,\lambda) \right) \geq 1-\alpha/2.
		\end{split}
	\end{equation}
where $\ell(n,\epsilon,\lambda)$ is defined in \eqref{eq:l-poi} and $C > 0$ is some constant depending on $\alpha$ only. The proofs for \eqref{eq:CI-tilde-poi-guarantee} and \eqref{eq:CI-bar-poi-guarantee} are similar to Part I and Part II in the proof of Theorem \ref{thm:upper}, we defer them to Appendix \ref{app:poi-additional-proof}. For the rest of the proof, we show the coverage and length guarantees for $\wh{\CI}$. 

\vskip.2cm
{\noindent \bf (Coverage Guarantee of $\wh{\CI}$)} Before establishing the coverage guarantee of $\wh{\CI}$, we define a new confidence interval. Let
\begin{equation*}
\wh{\psi}_{\lambda,\epsilon, \lambda_{\max}}^- = 
\begin{cases}
\phi_{\lambda,\epsilon}^- \wedge \min_{\mu \in [1,\lambda_{\max}]\cap\mathbb{N}} \phi_{\mu,\epsilon}^- & \lambda \in [0,1)  \\
\min_{\mu \in [\lfloor \lambda \rfloor,\lambda_{\max}]\cap\mathbb{N}} \phi_{\mu,\epsilon}^- & \lambda \in [1, \lambda_{\max}], \\
\end{cases}
\end{equation*}
where $\lambda_{\max} \in \mathbb{N}$ is a constant. With this, we define
\begin{equation*}\label{eq:CI-lambda-max}
\widehat{\CI}_{\lambda_{\max}}
=\left\{\lambda \in [0,\lambda_{\max}]: 
\wh{\psi}_{\lambda,\epsilon}^+=\wh{\psi}_{\lambda,\epsilon,\lambda_{\max}}^-=0 
\text{ for all } \epsilon \in \mathcal{E}\right\},
\end{equation*}
where $\wh{\psi}_{\lambda,\epsilon}^+$ is defined by \eqref{eq:psi+Poisson}. Then, by the definitions of $\wh{\psi}_{\lambda,\epsilon}^-$ in \eqref{eq:psi-Poisson} and $\wh{\CI}$ in \eqref{eq:ci-dis-Poisson}, it is easy to check that $\wh{\CI}_{\lambda_{\max}} \subseteq \wh{\CI}$ as long as $\lambda_{\max} \leq \wh{\lambda}_{\max}$. Therefore, if we can show that $\wh{\CI}_{\lambda_{\max}}$ achieves the coverage guarantee for some $\lambda_{\max}$ satisfying $\lambda_{\max} \leq \wh{\lambda}_{\max}$ with high probability, then $\wh{\CI}$ also achieves the coverage guarantee. 

We now show that $\wh{\CI}_{\lambda_{\max}}$ achieves the coverage guarantee when $\lambda_{\max} = \lceil \lambda \rceil \vee 1$. For any $\epsilon \in [0, \epsilon_{\max}]$ and $\lambda \in [0, \infty)$, we have
\begin{equation} \label{ineq:hatCI-coverage-pois}
\sup_QP_{\epsilon,\lambda,Q}\left(\lambda \notin \wh{\mathrm{CI}}_{\lceil\lambda \rceil \vee 1} \right) \leq \sup_QP_{\epsilon,\lambda,Q}\left(\sup_{\epsilon'\in \cE} \wh{\psi}^{+}_{\lambda, \epsilon'}=1 \right)+\sup_QP_{\epsilon,\lambda,Q}\left(\sup_{\epsilon'\in \cE} \wh{\psi}^{-}_{\lambda, \epsilon', \lceil\lambda \rceil \vee 1}=1\right).
\end{equation}
Next, we will bound the two terms at the end of the above equation. For any $\epsilon \in [0, \epsilon_{\max}]$ and $\lambda \in [0, \infty)$, by the same analysis as that used for the coverage guarantee of $\wh{\CI}$ in the proof of Theorem \ref{thm:upper} Part III, we have 
\begin{equation*}
		\begin{split}
    \sup_{Q}P_{\epsilon,\lambda,Q}\left(\sup_{\epsilon'\in \cE} \wh{\psi}^{+}_{\lambda, \epsilon'}=1 \right) \leq \alpha/12\quad \textnormal{ and }\quad \sup_{Q}P_{\epsilon,\lambda,Q}\left(\sup_{\epsilon'\in \cE} \wh{\psi}^{-}_{\lambda, \epsilon', \lceil \lambda\rceil \vee 1}=1 \right) \leq \alpha/12,
		\end{split}
	\end{equation*} by leveraging the definition of $\wh{\psi}^{+}_{\lambda, \epsilon'}$ and $\wh{\psi}^{-}_{\lambda, \epsilon',\lceil\lambda \rceil \vee 1}$, the stochastic dominance property of the Poisson distribution and the simultaneous Type-1 error control of $\phi_{\mu,\epsilon'}^+$/$\phi^-_{\lambda, \epsilon}$ shown in Theorem \ref{thm:test-up-Poisson}. By plugging them into \eqref{ineq:hatCI-coverage-pois}, we have shown for any $\epsilon \in [0, \epsilon_{\max}]$ and $\lambda \in [0,\infty)$, 
    \begin{equation} \label{Ineq: coverage guarantee of CI lambda ceiling}
        \sup_Q P_{\epsilon,\lambda,Q}\left(\lambda \notin \wh{\mathrm{CI}}_{\lceil \lambda \rceil \vee 1}\right) \leq \alpha/6.
    \end{equation} 
We now present a lemma which states that $\lceil \lambda \rceil \vee 1 \leq \wh{\lambda}_{\max}$ with high probability. Its proof is provided in the subsequent subsections.

\begin{Lemma} \label{Lem: poisson median}
For any $\alpha \in (0,1)$, $\epsilon \in [0, 1]$, and $n \geq 1$ satisfying $\frac{\epsilon}{2} + \sqrt{\frac{\log(12 / \alpha)}{2n}} \leq \frac{1}{4}$, we have
\begin{equation*}
    \inf_{\lambda \in [0, \infty),Q}P_{\epsilon,\lambda,Q}(\lceil\lambda\rceil\vee 1 \leq \wh{\lambda}_{\max} ) \geq 1 - \alpha/6,
\end{equation*}
where $\wh \lambda_{\max}$ is defined by (\ref{eq:lambda-max}).
\end{Lemma}
We note that the condition in Lemma \ref{Lem: poisson median} is satisfied since $\frac{\log(2/\alpha)}{n} + \epsilon_{\max}$ is less than a sufficiently small constant. We now proceed to establish the coverage guarantee for $\wh{\CI}$. For any $\epsilon \in [0, \epsilon_{\max}]$ and $\lambda \in [0, \infty)$, we have
\begin{equation*}
    \begin{split}
       &\sup_QP_{\epsilon,\lambda,Q}\left(\lambda \notin \wh{\mathrm{CI}} \right)  \leq \sup_QP_{\epsilon,\lambda,Q}\left(\lambda \notin \wh{\mathrm{CI}}, \lceil\lambda \rceil \vee 1  \leq \wh{\lambda}_{\max}\right) + \sup_QP_{\epsilon,\lambda,Q}\left(\lceil\lambda \rceil \vee 1  > \wh{\lambda}_{\max}\right) \\
        & \overset{\textnormal{Lemma }\ref{Lem: poisson median}}\leq \sup_QP_{\epsilon,\lambda,Q}\left(\lambda \notin \wh{\mathrm{CI}}, \lceil\lambda \rceil \vee 1  \leq \wh{\lambda}_{\max}\right) + \alpha/6 \\
        & \overset{(a)}\leq \sup_QP_{\epsilon,\lambda,Q}\left(\lambda \notin \wh{\mathrm{CI}}_{\lceil\lambda \rceil \vee 1}, \lceil\lambda \rceil \vee 1  \leq \wh{\lambda}_{\max}\right) + \alpha/6  \leq \sup_QP_{\epsilon,\lambda,Q}\left(\lambda \notin \wh{\mathrm{CI}}_{\lceil\lambda \rceil \vee 1} \right) + \alpha/6  \overset{\eqref{Ineq: coverage guarantee of CI lambda ceiling}} \leq \alpha / 3,
    \end{split}
\end{equation*}
where in (a) we use the fact that $\wh{\mathrm{CI}}_{\lceil\lambda \rceil \vee 1} \subseteq \wh{\mathrm{CI}}$ as long as $\lceil\lambda \rceil \vee 1 \leq \wh{\lambda}_{\max}$. This finishes the proof for the coverage guarantee.

\vskip.2cm
{\noindent \bf (Length Guarantee of $\wh{\CI}$)} Before establishing the coverage guarantee of $\wh{\CI}$, we define a new confidence interval. Let
\begin{equation*}
\wh{\psi}_{\lambda,\epsilon,\infty}^- = 
\begin{cases}
\phi_{\lambda,\epsilon}^- \wedge \min_{\mu \in \mathbb{N}} \phi_{\mu,\epsilon}^- & \lambda \in [0,1) \\
\min_{\mu \in [\lfloor \lambda \rfloor,\infty)\cap\mathbb{N}} \phi_{\mu,\epsilon}^- &  \lambda \in [1, \infty), \\
\end{cases}
\end{equation*}
and
\begin{equation*}\label{eq:CI-lambda-infty}
\widehat{\CI}_{\infty}
= \left\{\lambda \in [0,\infty): 
\wh{\psi}_{\lambda,\epsilon}^+ = \wh{\psi}_{\lambda,\epsilon,\infty}^- = 0 
\text{ for all } \epsilon \in \mathcal{E} \right\}.
\end{equation*} Then, by the definitions of $\wh{\psi}_{\lambda,\epsilon}^-$ and $\wh{\CI}$, it is easy to check that $\wh{\CI} \subseteq \wh{\CI}_{\infty}$. Therefore, if we can show that $\wh{\CI}_{\infty}$ achieves the length guarantee, then $\wh{\CI}$ also achieves the length guarantee. 

We define 
\begin{equation*} 
		\begin{split}
		&\wh{\lambda}_{\infty,\rm{left}} = \inf\left\{\lambda \in [0,\infty): 
\wh{\psi}_{\lambda,\epsilon}^+ = 0 
\text{ for all } \epsilon \in \mathcal{E} \right\}, \\
& \wh{\lambda}_{\infty, \rm{right}}= \sup \left\{\lambda \in [0,\infty): 
\wh{\psi}_{\lambda,\epsilon,\infty}^- = 0 
\text{ for all } \epsilon \in \mathcal{E} \right\},
		\end{split}
	\end{equation*}
	\begin{equation*}
		\begin{split}
			\widebar{\lambda}_{\rm{left}} = \inf\{\lambda \in [0,\infty):  {\psi}_{\lambda,\epsilon}^+ = 0 \text{ for all }\epsilon\in\mathcal{E} \} \quad \text{and} \quad \widebar{\lambda}_{\rm{right}} = \sup\{\lambda \in [0,\infty):  {\psi}_{\lambda,\epsilon}^- = 0 \text{ for all }\epsilon\in\mathcal{E} \}.
		\end{split}
	\end{equation*}
	Then the closure of $\wh{\CI}_{\infty}$ and $\widebar{\CI}$ can be concisely written as $[\wh{\lambda}_{\infty, \rm{left}},\wh{\lambda}_{\infty, \rm{right}}]$ and $[\widebar{\lambda}_{\rm{left}}, \widebar{\lambda}_{\rm{right}}]$.
	The following lemma shows a few connections of $\wh{\lambda}_{\infty, \rm{left}}$/$\wh{\lambda}_{\infty,\rm{right}}$ with $\widebar{\lambda}_{\rm{left}}$/$\widebar{\lambda}_{\rm{right}}$ and its proof is similar to that of Lemma \ref{lm:connection-hatp-barp} and we omit the proof here. 

\begin{Lemma}\label{lm:connection-hatp-barp-pois}
\begin{itemize}
	\item (i) If $\widebar{\lambda}_{\rm{right}} < 1$, then $\wh{\lambda}_{\infty,\rm{right}} = \widebar{\lambda}_{\rm{right}}$.
	\item (ii) We always have $[\wh{\lambda}_{\infty,\rm{left}}, \wh{\lambda}_{\infty,\rm{right}}] \subseteq [\widebar{\lambda}_{\rm{left}}- 1, \widebar{\lambda}_{\rm{right}} + 1]$. 
\end{itemize}
\end{Lemma}
Following the same analysis as in Part III of Theorem~\ref{thm:upper} for the length guarantee of $\wh{\CI}$, we can get
\begin{equation*} \label{ineq:hatCI-length-pois}
	\sup_Q P_{\epsilon,\lambda,Q}\left(|\widehat{\CI}_{\infty}|\geq C^* \ell(n, \epsilon, \lambda) \right) \leq \alpha
\end{equation*} for some $C^* >0$, we omit the details for simplicity.

\subsubsection{Proof of Theorem \ref{thm:test-up-Poisson}}
The proof of Theorem~\ref{thm:test-up-Poisson} is similar to that of Theorem~\ref{thm:test-up}. Therefore, we omit some of the details here. A key lemma we will leverage is the following and its proof is provided in the subsequent subsections.

\begin{Lemma} \label{lm:r-t-property-Poisson}
	Suppose $ \frac{\log(2/\alpha)}{n} + \epsilon_{\max}$ is less than a sufficiently small universal constant. Then for the choices of $\overline{t}(\lambda, \epsilon)$ and $\overline{r}(\lambda, \epsilon)$ defined in (\ref{def:t-Poisson-upper}) and (\ref{def:r-Poisson-upper}), and $\underline{t}(\lambda, \epsilon)$ and $\underline{r}(\lambda, \epsilon)$ defined in (\ref{def:t-Poisson-under}) and (\ref{def:r-Poisson-under}), we have
	\begin{itemize}
		\item (i) $\lambda + \overline{r}(\lambda,\epsilon) \geq 0$ and $\lambda - \underline{r}(\lambda,\epsilon) \geq 0$ for all $\lambda \geq 0 $ and all $\epsilon \in [0, \epsilon_{\max}]$, moreover, both functions $\lambda \mapsto \lambda + \overline{r}(\lambda,\epsilon)$ and $\lambda \mapsto \lambda - \underline{r}(\lambda,\epsilon)$ are strictly increasing in $\lambda$ for $\lambda \in [0, \infty)$ given any fixed $\epsilon \in [0, \epsilon_{\max}]$;
		\item (ii) For all $\epsilon \in [0, \epsilon_{\max}]$ and all $\lambda \in [0, \infty)$, \begin{equation*} \label{ineq:r-t-property-1-Poisson}
		P_{\lambda + \overline{r}(\lambda, \epsilon)}\left( X \leq \overline{t}(\lambda, \epsilon) \right) > 10 \left( \epsilon + \sqrt{\frac{\log(24/\alpha)}{2n}} \right);
	\end{equation*}
        \item (iii) For all $\epsilon \in [0, \epsilon_{\max}]$ and all $\lambda \in [1, \infty)$, \begin{equation*} \label{ineq:r-t-property-2-Poisson}
		P_{\lambda - \underline{r}(\lambda, \epsilon)}\left( X \geq   \underline{t}(\lambda, \epsilon) \right) > 10 \left( \epsilon + \sqrt{\frac{\log(24/\alpha)}{2n}} \right).
	\end{equation*} 
	In addition, for all $\epsilon \in [0, \epsilon_{\max}]$ and all $\lambda \in \left[4 \left( \frac{10 \log(24/\alpha) }{n} + 3 \epsilon  \right), 1 \right)$, 
	\begin{equation*} \label{ineq:r-t-property-3-Poisson}
		P_{\lambda}\left( X \geq   \underline{t}(\lambda, \epsilon) \right) \geq 6 \epsilon + \frac{20\log(24/\alpha)}{n}.
	\end{equation*}
	\end{itemize}
\end{Lemma}

Next, we show the simultaneous Type-1 error control of $\phi_{\lambda, \epsilon}^+$ for $\lambda \in [0, \infty)$. We use a similar argument to that in the proof of Theorem~\ref{thm:test-up}, adapted to the Poisson setting via Lemma~\ref{lm:r-t-property-Poisson} (ii). Specifically, using a similar analysis to that in the argument of \eqref{ineq:type-1-regime1}, it suffices to show that
    \begin{equation} \label{ineq:type-1-Poisson-upper}
   P_{\lambda}(X \leq \overline{t}(\lambda, \epsilon )) > \frac{12}{10(1 - \epsilon_{\max} )} P_{\lambda + \overline{r}(\lambda, \epsilon)} \left( X \leq \overline{t}(\lambda,\epsilon) \right), \forall \epsilon\in[0,\epsilon_{\max}].
    \end{equation}
It is easy to check that \eqref{ineq:type-1-Poisson-upper} is satisfied when $\lambda = 0$. When $\lambda > 0$, \eqref{ineq:type-1-Poisson-upper} is implied by $\overline{r} (\lambda, \epsilon) \geq \frac{ \lambda}{4(\lambda - \overline{t} (\lambda, \epsilon))}$ following a similar analysis as in \eqref{ineq:type-1-regime2}. Notice that this condition is satisfied by the choice of $\overline{r}(\lambda,\epsilon)$ for all $\epsilon\in[0,\epsilon_{\max}]$ and all $\lambda > 0$.

The simultaneous Type-1 error control of $\phi_{\lambda, \epsilon}^-$ when $\lambda \in [1, \infty)$ can be derived using an argument entirely analogous to that for $\phi_{\lambda,\epsilon}^+$, now applying Lemma~\ref{lm:r-t-property-Poisson} (iii) instead of Lemma~\ref{lm:r-t-property-Poisson} (ii). When $\lambda \in [0,1)$, the analysis for the simultaneous Type-1 error control of $\phi_{\lambda, \epsilon}^-$ is similar to the simultaneous Type-1 error control of $\phi_{p, \epsilon}^+$ when $p \in (1-1/m,1]$ in Theorem ~\ref{thm:test-up}, we omit the details for simplicity. 

Next, we move to the Type-2 error control. Using a similar analysis to that in the proof of Theorem~\ref{thm:test-up}, specifically the argument of \eqref{ineq:type2 with rbar}, it suffices to show that $P_{\epsilon, \lambda + \overline{r}(\lambda, \epsilon), Q}(\phi_{\lambda, \epsilon}^{+} = 0) \leq \alpha / 12$ for all $\epsilon \in [0, \epsilon_{\max}]$ and all $\lambda \in [0, \infty)$, and $P_{\epsilon,\lambda - \underline{r}(\lambda, \epsilon),Q}(\phi_{\lambda, \epsilon}^-=0)\leq \alpha / 12$ for all $\epsilon \in [0, \epsilon_{\max}]$ and all $\lambda \in \left[ 4\left(\frac{10 \log(24 / \alpha)}{n} + 3\epsilon\right), \infty \right)$. The Type-2 error control of $\phi_{\lambda, \epsilon}^+$ for $\lambda \in [0, \infty)$ and of $\phi_{\lambda, \epsilon}^-$ for $\lambda \in [1,\infty)$ follows an argument similar to the Type-2 error control proof of $\phi^+_{p, \epsilon}$ for $p \in [0,1-1/m]$ in Theorem~\ref{thm:test-up}. Similarly, following a similar argument for the Type-2 error control proof of $\phi^+_{p, \epsilon}$ for $p \in (1-1/m,1]$ in Theorem~\ref{thm:test-up}, when $\lambda \in \left[4\left(\frac{10\log(24 / \alpha)}{n} + 3\epsilon\right),1\right)$, $ P_{\epsilon, \lambda - \underline{r}(\lambda, \epsilon), Q}(\phi_{\lambda, \epsilon}^{-} = 0)$ is bounded by $\alpha/ 12$ if the following conditions hold:
\begin{equation}\label{ineq:type-2-suff-cond-Poisson}
    \begin{split}
        (i): P_\lambda(X \geq \underline{t}(\lambda, \epsilon) ) \geq 6 \epsilon + \frac{20\log(24/\alpha)}{n}\quad \textnormal{ and }\quad (ii): 1 - e^{-\lambda} \geq 6 \left(1 - e^{-\lambda + \underline{r}(\lambda, \epsilon)}\right).
    \end{split}
\end{equation}
Notice that the first condition in \eqref{ineq:type-2-suff-cond-Poisson} is satisfied for all $\lambda \in \left[4\left(\frac{10\log(24 / \alpha)}{n} + 3\epsilon\right),1\right)$ as we have shown in Lemma \ref{lm:r-t-property-Poisson} (iii). Next, we are going to show that given any $\epsilon \in [0, \epsilon_{\max}]$,
$$ 1 - e^{-\lambda} \geq 6 \left(1 - e^{-\lambda + \underline{r}(\lambda, \epsilon)}\right) \,\text{ holds for all }\,  \lambda \in [0,1) \text{  with  } \underline{r}(\lambda, \epsilon) = (1 - 1/(6e) )\lambda.$$
Let $f(\lambda) = 1 - e^{-\lambda} - 6 \left(1 - e^{-\lambda/(6e)}\right)$. Then $f(0) = 0$. If we can show $f'(\lambda) \geq 0$ for all $\lambda \in [0, 1)$, then it implies that $f(\lambda) \geq 0$ for all $\lambda \in [0, 1)$.
\begin{equation*}
    \begin{split}
        f'(\lambda) = e^{-\lambda}\left(1 - \frac{1}{e} e^{(1 - 1/(6e))\lambda}\right) \geq e^{-\lambda}\left(1 - \frac{1}{e} e^{1 - 1/(6e)}\right) = e^{-\lambda} \left(1 - e^{-1/(6e)}\right) > 0.
\end{split}
\end{equation*}
This shows $f(\lambda) \geq 0$ for all $
\lambda \in [0, 1)$ and finishes the proof for the Type-2 error control. This also finishes the proof of this theorem.

\subsubsection{Proof of Lemma \ref{Lem: poisson median}} 

Fix $\lambda \in [0, \infty)$ and $Q$. For simplicity, let $F_{\epsilon, \lambda, Q}(\cdot)$, $F_{\lambda}(\cdot)$, and $F_{Q}(\cdot)$ denote the CDFs of $P_{\epsilon, \lambda, Q}$, $\mathrm{Poisson}(\lambda)$, and $Q$, respectively. Also, for any CDF $F(\cdot)$, let $F^{-}(\cdot)$ denote its generalized inverse, defined by
\begin{equation*}
F^{-}(u) = \inf\{x \in \mathbb{R} : F(x) \geq u\},    
\end{equation*}
for all $u \in (0,1)$. 
Similarly, for any empirical CDF $F_n(\cdot)$, we define $F_n^{-}(\cdot)$ in the same way. By the DKW inequality (see Lemma \ref{lm:DKW}), the following event occurs with probability at least $1 - \alpha/6$:
\begin{equation*}
    (A) = \left\{\sup_{x \in \bbR} |F_n(x) - F_{\epsilon, \lambda, Q}(x)| \leq \sqrt{\frac{\log(12 / \alpha)}{2n}} \right\}.
\end{equation*}
Given $(A)$ happens, for any $x < F_{\epsilon, \lambda, Q}^{-}\Big(\frac{1+\epsilon}{2}\Big)$, we have
\begin{equation*}
    F_n(x) \overset{(A)}\leq F_{\epsilon, \lambda, Q}(x) + \sqrt{\frac{\log(12 / \alpha)}{2n}} < \frac{1 + \epsilon}{2} + \sqrt{\frac{\log(12 / \alpha)}{2n}},
\end{equation*}
which directly implies that
\begin{equation*}
    x < F_{n}^{-}\left(\frac{1 + \epsilon}{2} + \sqrt{\frac{\log(12 / \alpha)}{2n}}\right).
\end{equation*}
Since the above inequality holds for any $x < F_{\epsilon, \lambda, Q}^{-}\Big(\frac{1+\epsilon}{2}\Big)$, it follows that
\begin{equation}\label{Ineq: F- and F_n- right end point}
    F_{\epsilon, \lambda, Q}^{-}\left(\frac{1+\epsilon}{2}\right) \leq F_{n}^{-}\left(\frac{1 + \epsilon}{2} + \sqrt{\frac{\log(12 / \alpha)}{2n}}\right) \leq F_n^{-}(3/4) \leq X_{(\lceil 3n/4 \rceil)},
\end{equation}
where the second inequality follows from the assumption $\frac{\epsilon}{2} + \sqrt{\frac{\log(12 / \alpha)}{2n}} \leq \frac{1}{4}$. Also, by the definition of $P_{\epsilon, \lambda, Q}$, we have
\begin{equation}\label{Ineq: upper bound for F}
    F_{\epsilon, \lambda, Q}(x) = ( 1 - \epsilon)F_{\lambda}(x) + \epsilon F_{Q}(x) \leq (1 - \epsilon)F_{\lambda}(x) + \epsilon,
\end{equation}
for all $x \in \bbR$. Thus, when $(A)$ happens, we have
\begin{equation}\label{Ineq: upper bound for median with empirical CDF}
    \begin{split}
       X_{(\lceil 3n/4 \rceil)} & \overset{\eqref{Ineq: F- and F_n- right end point}}\geq  F_{\epsilon, \lambda, Q}^-\left(\frac{1 + \epsilon}{2}\right) = \inf \left\{x \in \bbR : F_{\epsilon, \lambda, Q}(x) \geq \frac{1+\epsilon}{2}\right\} \\
        & \overset{\eqref{Ineq: upper bound for F}}\geq \inf \left\{x \in \bbR : (1 - \epsilon)F_{\lambda}(x) + \epsilon \geq \frac{1+\epsilon}{2}\right\} =F_{\lambda}^-\left(1/2\right).
    \end{split}
\end{equation}
Also, by Theorem 2 in~\cite{choi1994medians}, the median of the $\mathrm{Poisson}(\lambda)$ distribution satisfies
\begin{equation}\label{Ineq: median and mean for poisson}
    \lambda  \leq F_\lambda^{-}\left(1/2\right)  + 1.
\end{equation}
By combining \eqref{Ineq: upper bound for median with empirical CDF} and \eqref{Ineq: median and mean for poisson}, we have \begin{equation*} 
\begin{split}
    \lambda \leq X_{(\lceil 3n/4 \rceil)} + 1 = \wh{\lambda}_{\max},
\end{split}
\end{equation*}
with probability at least $1 - \alpha/6$. Since $\wh{\lambda}_{\max} \in \mathbb{N}$, $\lambda \leq \wh{\lambda}_{\max}$ implies $\lceil \lambda \rceil \vee 1 \leq \wh{\lambda}_{\max}$. This finishes the proof of this lemma.

\subsubsection{Proof of Lemma \ref{lm:r-t-property-Poisson}}
 
The proof of Lemma~\ref{lm:r-t-property-Poisson} is similar to that of Lemma~\ref{lm:r-t-property}. For convenience, let us denote $A = \epsilon + \sqrt{\frac{\log(24/\alpha)}{2n}}$. By assumption, $A$ is less than a sufficiently small constant.  

\vskip.2cm
{\noindent \bf (Part I: Proof of Claim (i))} We first show that the function $\lambda \mapsto \lambda + \overline{r}(\lambda, \epsilon)$ is strictly increasing in $\lambda$ for $\lambda \in [0, \infty)$ given any fixed $\epsilon \in [0, \epsilon_{\max}]$. Since $\lambda + \overline{r}(\lambda, \epsilon) = \lambda + \max\left\{\frac{1}{2}, 2\sqrt{\frac{\lambda}{\log(1 / A)}}\right\}$, the desired monotonicity follows trivially.

Next, we show that the function $\lambda \mapsto \lambda - \underline{r}(\lambda, \epsilon)$ is strictly increasing in $\lambda$ for $\lambda \in [0, \infty)$ given any fixed $\epsilon \in [0, \epsilon_{\max}]$. Notice that $1 \leq \log(1/ A)/16$ when $A$ is sufficiently small. When $\lambda \in [0, 1)$, we have $\lambda - \underline{r}(\lambda, \epsilon) = \lambda / (6e)$, which is clearly increasing with respect to $\lambda$. When $\lambda \in [1, \log(1 / A)/16]$, we have $\lambda - \underline{r}(\lambda , \epsilon) = \lambda - 1/2$, which is increasing in $\lambda$. When $\lambda \in (\log(1 / A)/16, \infty)$, we have $\underline{r}(\lambda , \epsilon) = 2\sqrt{\frac{\lambda}{\log(1 / A)}}$, and it follows that
\begin{equation*}
    \begin{split}
        \frac{\partial (\lambda - \underline{r}(\lambda, \epsilon))}{\partial \lambda} & = 1 - \frac{1}{\sqrt{\lambda \log(1 / A)}} \geq 1 - \frac{4}{ \log(1 / A)} > 0,
    \end{split}
\end{equation*}
where the last inequality is satisfied since $A$ is less than a sufficiently small constant. Hence, the function $\lambda \mapsto \lambda - \underline{r}(\lambda, \epsilon)$ is strictly increasing on $[1, \infty)$ as well. At $\lambda = 1$, the function $\lambda \mapsto \lambda - \underline{r}(\lambda, \epsilon)$ may not be continuous, but the right limit at $\lambda = 1$ is strictly larger than the left limit when $A$ is less than a sufficiently small constant. Therefore, the function $\lambda \mapsto \lambda - \underline{r}(\lambda, \epsilon)$ is strictly increasing in $\lambda$ for $\lambda \in [0, \infty)$.

Moreover, as both functions are increasing in $\lambda$ for $\lambda \in [0, \infty)$, it follows that for all $\lambda \in [0, \infty)$,
$$
\lambda + \overline{r}(\lambda, \epsilon) \geq 0 + \overline{r}(0, \epsilon) = \frac{1}{2}, 
\quad
\lambda - \underline{r}(\lambda, \epsilon) \geq 0-\underline{r}(0, \epsilon) = 0,
$$
which shows that both functions are non-negative.

\vskip.2cm
{\noindent \bf (Part II: Proof of Claim (ii))}
We will simply write $\overline{t}(\lambda, \epsilon)$ and $\overline{r}(\lambda, \epsilon)$ as $t$ and $r$ in this part. We divide the proof into two cases. 

\vskip.2cm
\begin{itemize}[leftmargin=*]
	\item (Case 1: $\lambda \in [0,\log(1/ A)/16]$) In this case, $t = \lambda / 2$ and $r = 1/2$. Then
\begin{equation*}
	\begin{split}
		P_{\lambda + r}(X \leq t) = P_{\lambda + \frac{1}{2} } \left( X \leq t \right) \geq P_{\lambda + \frac{1}{2} }\left( X = 0 \right) = \exp\left( - \lambda - \frac{1}{2} \right) 
         \geq \exp\left( - \frac{\log(1 / A)}{16} - \frac{1}{2} \right)
	\end{split}
\end{equation*} A sufficient condition for $\exp\left( - \frac{\log(1 / A)}{16} - \frac{1}{2} \right) > 10 A$ is derived as follows:
\begin{equation*}
	\begin{split}
		 \exp\left( - \frac{\log(1 / A)}{16} - \frac{1}{2} \right) > 10 A \Longleftrightarrow & - \frac{\log(1 / A)}{16} - \frac{1}{2} > \log(A) + \log(10) \\
		 \Longleftrightarrow & \log(1/A) > \frac{8}{15} + \frac{16\log(10)}{15}.
	\end{split}
\end{equation*} Notice that the last condition above is satisfied since $A$ is less than a sufficiently small constant. So we have shown $P_{\lambda + r}(X \leq t) > 10A$ when $\lambda \in [0, \log(1/ A)/16]$. 
\item (Case 2: $\lambda \in (\log(1/ A)/16, \infty)$) In this case, we have $t = \lambda - \frac{1}{8} \sqrt{\lambda \log\left(1/A\right)}$ and $ r = 2 \sqrt{ \frac{\lambda}{ \log(1/A)} }$. 
Moreover, since $A$ is less than a sufficiently small constant, we have 
\begin{equation}\label{ineq:t-r-Poisson}
	\begin{split}
		\lambda + r -t < 2 (\lambda - t) \quad \text{and} \quad \sqrt{t} < 2(\lambda - t).
	\end{split}
\end{equation}
When $\lambda \geq \log(1 / A)/16$, we have
$\frac{\partial \overline{t}(\lambda, \epsilon)}{\partial \lambda}  > 0$. Thus, for fixed $\epsilon \in [0, \epsilon_{\max}]$, $\overline{t}(\lambda, \epsilon)$ is increasing in $\lambda$ when $\lambda \geq \log(1 / A)/16$, and therefore
\begin{equation}\label{Ineq: Lowerbound for t for Poisson}
    \overline{t} (\lambda, \epsilon) \geq \overline{t} \left(\frac{\log(1 / A)}{16}, \epsilon \right) = \frac{\log(1/A)}{32} \geq 9, 
\end{equation}
where the last inequality holds since $A$ is less than a sufficiently small constant.

Next, we aim to provide a lower bound for $P_{\lambda+r}(X\leq t)$.
\begin{equation}\label{Lower bound for CDF of Poisson distribution}
    \begin{split}
        P_{\lambda + r} (X \leq t) & = \sum_{k \leq t} \frac{\exp(- \lambda - r) (\lambda + r )^k}{k!} \overset{(a)} \geq \sum_{t - \sqrt{t} < k \leq t} \frac{\exp(- t) t^k}{k!} \exp(- \lambda - r + t) \left(\frac{\lambda + r}{t}\right)^k \\
        & \geq \sum_{t - \sqrt{t} < k \leq t} \frac{\exp(- t) t^k}{k!} \min _{t - \sqrt{t} < k' \leq t}\exp(- \lambda - r + t) \left(\frac{\lambda + r}{t}\right)^{k'}  \\
        & \geq \sum_{t - \sqrt{t} < k \leq t} \frac{\exp(- t) t^k}{k!} \exp(- \lambda - r + t) \left(\frac{\lambda + r}{t} \right)^{t - \sqrt{t}}\\
        & = P_{t} \left(t - \sqrt{t} < X \leq t \right) \exp\left(- \lambda -r + t + (t - \sqrt{t})\log\left(\frac{\lambda + r}{t}\right)\right) \\
        & \overset{(b)}\geq P_{t} \left(t - \sqrt{t} < X \leq t \right) \exp\left(- \lambda -r + t + \frac{(t - \sqrt{t})(\lambda + r - t)}{\lambda + r}\right),
    \end{split}
\end{equation}
where in (a), we use the fact $t \geq 9$ by \eqref{Ineq: Lowerbound for t for Poisson}, which ensures $t - \sqrt{t} > 0$; in (b), we use the inequality $\log(1 + x) \geq x/(1 + x)$ for all $x > -1$. Next, we bound the two terms $ P_{t} \left(t - \sqrt{t} < X \leq t \right)$ and $$\exp\left(- \lambda -r + t + \frac{(t - \sqrt{t})(\lambda + r - t)}{\lambda + r}\right)$$ at the end of \eqref{Lower bound for CDF of Poisson distribution} separately. Let $\Phi(\cdot)$ denote the CDF of standard Gaussian. Then, by a Poisson-specific Berry–Esseen bound (see Lemma \ref{Lem: Berry esseen for Poisson}), we have
\begin{equation}\label{Application of Berry Esseen Theorem Poisson}
    P_{t} \left(t - \sqrt{t} < X \leq t \right) \geq (\Phi(0) - \Phi(- 1)) - \frac{7}{10\sqrt{t}} \overset{\eqref{Ineq: Lowerbound for t for Poisson}}\geq (\Phi(0) - \Phi(- 1)) - \frac{7}{30} > 0.1.
\end{equation}
At the same time,
\begin{equation}\label{ineq: exponent-bound Poisson}
        - \lambda -r + t + \frac{(t - \sqrt{t})(\lambda + r - t)}{\lambda + r} = -(\lambda + r - t)\left(\frac{\lambda + r - t + \sqrt{t}}{\lambda + r}\right) 
        \overset{\eqref{ineq:t-r-Poisson}}{>} - \frac{8(\lambda - t)^2}{\lambda + r}\geq -\frac{8(\lambda - t)^2}{\lambda }.
\end{equation}
By plugging \eqref{Application of Berry Esseen Theorem Poisson} and \eqref{ineq: exponent-bound Poisson} into \eqref{Lower bound for CDF of Poisson distribution}, we have
\begin{equation*}
     P_{\lambda + r} (X \leq t) > \frac{1}{10}\exp\left(-\frac{8(\lambda - t)^2}{\lambda }\right),
\end{equation*}
and a sufficient condition to guarantee $P_{\lambda + r}(X \leq t) > 10A$ is given as follows,
\begin{equation*}
    \begin{split}
        P_{\lambda + r}(X \leq t) > 10A & \Longleftarrow \exp\left(-\frac{8(\lambda - t)^2}{\lambda }\right) \geq 100A  \Longleftrightarrow \frac{8(\lambda - t)^2}{\lambda } \leq \log(1 / A) - \log(100) \\
        & \Longleftarrow \frac{8(\lambda - t)^2}{\lambda } \leq \frac{1}8\log(1 / A)  \Longleftrightarrow |\lambda - t| \leq \frac{1}{8}\sqrt{\lambda \log(1 / A).}
    \end{split}
\end{equation*}
Notice that the last condition is satisfied by the choice of $t$. Thus, we have shown $P_{\lambda+r}(X\leq t) > 10A$ when $\lambda \in (\log(1 / A)/16 , \infty)$.

\end{itemize}

\vskip.2cm
    {\noindent \bf (Part III: Proof of Claim (iii))}
We will simply write $\underline{t}(\lambda, \epsilon)$ and $\underline{r}(\lambda, \epsilon)$ as $t$ and $r$ in this part. We divide the proof into three cases. Notice that $1 \leq \log(1/ A)/16$ when $A$ is sufficiently small.

\vskip.2cm
\begin{itemize}[leftmargin=*]
	\item (Case 1: $\lambda \in  \left[4 \left( \frac{10 \log(24/\alpha) }{n} + 3 \epsilon  \right), 1 \right)$) Notice that in this regime, $t = 1$. Thus, $P_\lambda(X \geq t) = P_\lambda(X \neq 0) = 1 - e^{-\lambda}$. Then,
\begin{equation}\label{Cond: Poisson boundary}
    \begin{split}
      &  P_{\lambda}(X \geq t) \geq 6 \epsilon + \frac{20\log(24 / \alpha)}{n}  \Longleftrightarrow \lambda \geq - \log \left(1 - \left(6 \epsilon + \frac{20\log(24 / \alpha)}{n}\right)\right) \\
        & \Longleftarrow \lambda \geq \frac{6 \epsilon + \frac{20\log(24 / \alpha)}{n}}{1 - \left(6 \epsilon + \frac{20\log(24 / \alpha)}{n}\right)} \quad (\text{as } \log(1 + x) \geq x / (1 + x), \forall x > -1 ) \\
        & \overset{(a)}{\Longleftarrow} \lambda \geq 4\left(3 \epsilon + \frac{10\log(24 / \alpha)}{n}\right),
    \end{split}
\end{equation}
where (a) holds as long as $6 \epsilon + \frac{20\log(24 / \alpha)}{n} \leq 1/2$. The last condition in \eqref{Cond: Poisson boundary} is satisfied for $\lambda$ in this regime. 
\item (Case 2: $\lambda \in [1, \log(1/ A)/16]$) In this case, we have $t = \frac{3}{2}\lambda $ and $ r = 1/2$. We aim to provide a lower bound for $P_{\lambda-r}(X \geq t)$.
\begin{equation*}
    \begin{split}
       P_{\lambda-r}(X \geq t) & = P_{\lambda-\frac{1}{2}}\left(X \geq \frac{3}{2}\lambda\right) \overset{(a)}\geq P_{\frac{1}{2}\lambda}\left(X \geq \frac{3}{2}\lambda\right)  \geq P_{\frac{1}{2}\lambda}\left(X  = \left\lceil\frac{3}{2}\lambda \right \rceil\right) \\
       & = \frac{e^{-\lambda/2}\left(\frac{\lambda}{2}\right)^{\left\lceil\frac{3}{2}\lambda \right \rceil}}{\left(\left\lceil\frac{3}{2}\lambda \right \rceil\right)!} \geq \frac{e^{-\lambda/2}\left(\frac{\lambda}{2}\right)^{\left\lceil\frac{3}{2}\lambda \right \rceil}}{\left(\left\lceil\frac{3}{2}\lambda \right \rceil\right)^{\left\lceil\frac{3}{2}\lambda \right \rceil}} \overset{(b)}\geq \frac{e^{-\lambda/2}\left(\frac{\lambda}{2}\right)^{\left\lceil\frac{3}{2}\lambda \right \rceil}}{\left(\frac{5}{2}\lambda \right)^{\left\lceil\frac{3}{2}\lambda \right \rceil}} \\
       & = \frac{e^{-\lambda/2}}{5^{\left\lceil\frac{3}{2}\lambda \right \rceil}} \overset{(c)} \geq \frac{e^{-\lambda/2}}{5^{\frac{5}{2}\lambda}} \overset{(d)} > \exp\left(-\frac{11}{2}\lambda\right) \geq \exp\left(-\frac{11}{32}\log(1/A)\right),
    \end{split}
\end{equation*}
where (a) by Lemma \ref{lm:Poisson-prop}, together with the fact that $\lambda - 1/2 \geq \lambda/2$ in this regime; (b) and (c) hold since $\left \lceil \frac{3}{2} \lambda\right \rceil \leq \frac{3}{2}\lambda + 1 \leq \frac{5}{2}\lambda$ for all $\lambda \geq 1$; (d) is because $5^{\frac{5}{2}\lambda} < e^{5\lambda}$, which follows from $ 5 < e^2$. 

A sufficient condition for $\exp\left(-\frac{11}{32}\log(1/A)\right) > 10A$ is derived as follows:
\begin{equation*}
    \begin{split}
        \exp\left(-\frac{11}{32}\log(1/A)\right) > 10A & \Longleftrightarrow -\frac{11}{32}\log(1/A) > \log(A) + \log(10)  \\
        & \Longleftarrow \frac{21}{32}\log(1 / A) > \log(10).
    \end{split}
\end{equation*}
Notice that the last condition above is satisfied since $A$ is less than a sufficiently small constant. So we have shown $P_{\lambda - r}(X \geq t) > 10A$ when $\lambda \in [1, \log(1/ A)/16]$.
\item (Case 3: $\lambda \in  \left((\log(1/ A)/16, \infty \right)$) In this case, we have $t = \lambda + \frac{1}{8} \sqrt{\lambda \log\left(1/A\right)}$ and $r = 2 \sqrt{ \frac{\lambda}{ \log(1/A)} }$. Since $\lambda \geq 1$ and $A$ is less than a sufficiently small constant, it is easy to check that
\begin{equation}\label{Ineq: t > 9}
    t \geq 9 \quad \text{and} \quad
    \lambda \geq 2r.
\end{equation}
Also, using a similar argument to that in Part III, we have
\begin{equation}\label{ineq: t-lambda geq sqrt t-lower}
t - \lambda + r < 2(t - \lambda) \quad \textnormal{ and } \quad \sqrt{t} < 2( t - \lambda).
\end{equation}

Next, we aim to provide a lower bound for $P_{\lambda - r}(X \geq t)$.
\begin{equation}\label{Lower bound for CDF of Poisson distribution-lower}
    \begin{split}
        P_{\lambda - r} (X \geq t) & = \sum_{k \geq t} \frac{\exp(- \lambda + r) (\lambda - r )^k}{k!}  \geq \sum_{ t\leq k < t + \sqrt{t}} \frac{\exp(- t) t^k}{k!} \exp(- \lambda + r + t) \left(\frac{\lambda - r}{t}\right)^k \\
        & \geq \sum_{t\leq k < t + \sqrt{t}} \frac{\exp(- t) t^k}{k!} \min _{t\leq k < t + \sqrt{t}}\exp(- \lambda + r +  t) \left(\frac{\lambda - r}{t}\right)^{k'}  \\
        & \geq \sum_{t\leq k < t + \sqrt{t}} \frac{\exp(- t) t^k}{k!} \exp(- \lambda + r + t) \left(\frac{\lambda - r}{t} \right)^{t + \sqrt{t}}\\
        & = P_{t} \left(t \leq X < t + \sqrt{t} \right) \exp\left(- \lambda  + r + t + (t + \sqrt{t})\log\left(\frac{\lambda - r}{t}\right)\right) \\
        & \geq P_{t} \left(t \leq X < t + \sqrt{t} \right) \exp\left(- \lambda + r + t + \frac{(t + \sqrt{t})(\lambda - r - t)}{\lambda - r}\right),
    \end{split}
\end{equation}
where in the last inequality is because $\log(1 + x) \geq x/(1 + x)$ for all $x > -1$. Next, we bound the two terms $ P_{t} \left(t \leq X < t  + \sqrt{t} \right)$ and $$\exp\left(- \lambda + r + t + \frac{(t + \sqrt{t})(\lambda - r - t)}{\lambda - r}\right)$$ at the end of \eqref{Lower bound for CDF of Poisson distribution-lower} separately. First, by a Poisson-specific Berry–Esseen bound (see Lemma \ref{Lem: Berry esseen for Poisson}), we have
\begin{equation}\label{Application of Berry Esseen Theorem Poisson-lower}
    P_{t} \left(t - \sqrt{t} < X \leq t \right) \geq (\Phi(0) - \Phi(- 1)) - \frac{7}{10\sqrt{t}} \overset{\eqref{Ineq: t > 9}}\geq (\Phi(0) - \Phi(- 1)) - \frac{7}{30} > 0.1.
\end{equation}
At the same time,
\begin{equation}\label{ineq: exponent-bound Poisson-lower}
        - \lambda +r + t + \frac{(t + \sqrt{t})(\lambda - r - t)}{\lambda - r} = -(t  - \lambda + r )\left(\frac{t - \lambda + r + \sqrt{t}}{\lambda - r}\right) 
        \overset{\eqref{Ineq: t > 9}, \eqref{ineq: t-lambda geq sqrt t-lower}}{>} -\frac{16(t - \lambda)^2}{\lambda}.
\end{equation}
By plugging \eqref{Application of Berry Esseen Theorem Poisson-lower} and \eqref{ineq: exponent-bound Poisson-lower} into \eqref{Lower bound for CDF of Poisson distribution-lower}, we have
\begin{equation*}
     P_{\lambda - r} (X \geq t) > \frac{1}{10}\exp\left(-\frac{16(t - \lambda)^2}{\lambda }\right),
\end{equation*}
and a sufficient condition to guarantee $P_{\lambda - r}(X \geq t) > 10A$ is given as follows,
\begin{equation*}
    \begin{split}
        P_{\lambda - r}(X \geq t) > 10A & \Longleftarrow \exp\left(-\frac{16(t - \lambda )^2}{\lambda }\right) \geq 100A  \Longleftrightarrow \frac{16(t - \lambda )^2}{\lambda } \leq \log(1 / A) - \log(100) \\
        & \Longleftarrow \frac{16(t - \lambda )^2}{\lambda } \leq \frac{1}{4}\log(1 / A) \Longleftrightarrow |t - \lambda | \leq \frac{1}{8}\sqrt{\lambda \log(1 / A).}
    \end{split}
\end{equation*}
Notice that the last condition is satisfied by the choice of $t$. Thus, we have shown $P_{\lambda - r}(X\geq t) > 10A$ when $\lambda \in (\log(1 / A)/16, \infty)$.

\end{itemize}
This finishes the proof of this lemma.

\subsection{Proofs of Claim \eqref{eq:CI-tilde-poi-guarantee} and \eqref{eq:CI-bar-poi-guarantee}} \label{app:poi-additional-proof}
The coverage property of $\wt{\CI}$ in \eqref{eq:CI-tilde-poi-guarantee} follows the same analysis as the one in  Theorem \ref{thm:bino-upper-no-dis}, we omit it here for simplicity. Next, we consider the length guarantee of $\wt{\CI}$ in \eqref{eq:CI-tilde-poi-guarantee}. The following theorem and lemma will be useful for establishing the length guarantee and their proofs are given in the subsections.

\begin{Theorem}\label{thm:type2-Poisson}
 Suppose $ \frac{\log(2/\alpha)}{n} + \epsilon_{\max}$ is less than a sufficiently small universal constant. The testing function $\psi_{\lambda,\epsilon}^+$ defined by (\ref{def: psi+ pois real}) satisfies the simultaneous Type-1 error bound in the same sense as in (\ref{ineq:type-1-phi+ Poisson}) with $\phi_{\lambda,\epsilon}^+$ being replaced by $\psi_{\lambda,\epsilon}^+$. In addition, it also satisfies the Type-2 error bound,
$$
\underset{Q}{\sup}P_{\epsilon,\lambda + r,Q}\left(\psi_{\lambda,\epsilon}^+ = 0\right)\leq \alpha/6,
$$
for all $\epsilon \in [0, \epsilon_{\max}]$, all $\lambda \in [0, \infty)$ and all
$r \geq \overline{r}(\lambda, \epsilon)$, where $\overline{r}(\lambda, \epsilon)$ is given in (\ref{def:r-Poisson-upper}).
Similarly, the testing function $\psi_{\lambda,\epsilon}^-$ defined by (\ref{def: psi- pois real}) satisfies the simultaneous Type-1 error bound in the same sense as in (\ref{ineq:type-1-phi- Poisson}) with $\phi_{\lambda,\epsilon}^-$ being replaced by $\psi_{\lambda,\epsilon}^-$. In addition, it also satisfies the Type-2 error bound,
$$
\underset{Q}{\sup}P_{\epsilon,\lambda-r,Q}\left(\psi_{\lambda,\epsilon}^- = 0\right)\leq \alpha/6,
$$
for all $\epsilon \in [0, \epsilon_{\max}]$, all $\lambda \in \left[ 4\left( \frac{10 \log(24/\alpha)}{n} + 3 \epsilon \right), \infty \right)$, and all
$\lambda \in[\underline{r}(\lambda, \epsilon),\lambda]$, where $\underline{r}(\lambda, \epsilon)$ is given in (\ref{def:r-Poisson-under}).
\end{Theorem}

\begin{Lemma}\label{lm:r-property-length-pois}
	Suppose $ \frac{\log(2/\alpha)}{n} + \epsilon_{\max}$ is less than a sufficiently small universal constant. Given any $c \in (0,1)$, there exists a large constant $C_0 > 0$ only depending on $c$ and $\alpha$ such that for any $C \geq C_0$ and $\epsilon \in [0, \epsilon_{\max}]$,
    	\begin{equation*}\label{ineq:r-length-prop2-pois}
		\begin{split}
			&\bar{r}( \lambda - C\bar{r}(\lambda, \epsilon), \epsilon ) \leq C\bar{r}(\lambda, \epsilon), \quad \forall \lambda \in [0, \infty) \bigcap \{\lambda: \lambda - C\bar{r}(\lambda, \epsilon) \geq 0 \}, \\
			&\underline{r}( \lambda + C\underline{r}(\lambda, \epsilon), \epsilon ) \leq C\underline{r}(\lambda, \epsilon), \quad \forall \lambda \in [c, \infty).
		\end{split}
	\end{equation*}
\end{Lemma}
Also, we will use the following lemma, which establishes the asymptotic order of $\overline{r}(\lambda,\epsilon)$, $\underline{r}(\lambda,\epsilon)$, and $\ell(n,\epsilon,\lambda)$ in the Poisson setting. Its proof is similar to that of Lemma \ref{lm:test-rate-CI-length-connection}, so we omit the proof here.
\begin{Lemma} \label{lm:test-rate-CI-length-connection-pois}
	    Suppose $\epsilon \in [0, 1/2]$ and $n \geq 2$. For $\overline{r}(\lambda,\epsilon)$ defined by (\ref{def:r-Poisson-upper}), $\underline{r}(\lambda, \epsilon)$ defined by (\ref{def:r-Poisson-under}), and $\ell(n,\epsilon,\lambda)$ defined by \eqref{eq:l-poi}, we have
\begin{eqnarray*}
\overline{r}(\lambda, \epsilon) &\asymp& \sqrt{\lambda}\left(\frac{1}{\sqrt{\log n}}+\frac{1}{\sqrt{\log(1/\epsilon)}}\right) + 1, \\
\underline{r}(\lambda, \epsilon) &\asymp& \left(\sqrt{\lambda}\left(\frac{1}{\sqrt{\log n}}+\frac{1}{\sqrt{\log(1/\epsilon)}}\right) + 1 \right) \wedge \lambda, \\
\ell(n,\epsilon,\lambda) &\asymp&  \overline{r}(\lambda, \epsilon)\wedge \underline{r}(\lambda, \epsilon) + \frac{1}{n}+\epsilon, 
\end{eqnarray*}
     where $\asymp$ suppresses dependence on $\epsilon_{\max}$ and $\alpha$. 
\end{Lemma}
With these results, we are ready to show the length guarantee of $\wt{\mathrm{CI}}$. By the same analysis as that used for the length guarantee in the proof of Theorem \ref{thm:bino-upper-no-dis}, i.e., the analysis of \eqref{Ineq: Length guarantee}, we have, for any $\lambda \in [0,\infty)$ and $\epsilon \in [0, \epsilon_{\max}]$, 
\begin{equation}\label{Ineq: Length guarantee-pois}
    \begin{split}
        \sup_{ Q } P_{\epsilon, \lambda, Q}\left( |\widetilde{\CI}| \geq C' \ell(n,\epsilon,\lambda) \right) & \leq         \sup_{ Q } P_{\epsilon, \lambda, Q}\left( |\widetilde{\CI}| \geq C' \ell(n,\epsilon,\lambda), \lambda \in \widetilde{\CI}\right) + \alpha / 6.
    \end{split}
\end{equation}
Next, we show \begin{equation} \label{ineq:Poisson-length-upp}
	\sup_{ Q } P_{\epsilon, \lambda, Q}\left( |\widetilde{\CI}| \geq C' \ell(n,\epsilon,\lambda), \lambda \in \widetilde{\CI}\right) \leq \alpha/3.
\end{equation} Let $C_0' = 40 \log(24 / \alpha) \vee 12 + 1$ and $c = \frac{1}{2C'_0}$. We divide the proof into two cases: $\lambda \in [0, c)$ and $\lambda \in [c, \infty)$.

\begin{itemize}[leftmargin=*]
	\item (Case 1: $\lambda \in [0, c)$) In this case, it suffices to show that $\sup_{ Q } P_{\epsilon, \lambda, Q}\left( |\widetilde{\CI}| \geq C'_0 \ell(n,\epsilon,\lambda), \lambda \in \widetilde{\CI}\right) \leq \alpha/3
   $ where $C_0' = 40 \log(24 / \alpha) \vee 12 + 1$. Then, for any $C' \geq C_0'$, \eqref{ineq:Poisson-length-upp} holds. It is easy to check that $C_0'(\lambda + \frac{1}{n}+\epsilon) < 1$ as long as $\frac{1}{n}+\epsilon$ is sufficiently small, and that $\lambda + (C_0'-1)(\lambda + \frac{1}{n} + \epsilon) \geq 4(\frac{10\log(24 / \alpha)}{n} + 3\epsilon)$. By following the same approach as in Case 1 of the length guarantee proof of Theorem~\ref{thm:bino-upper-no-dis}, now using Theorem~\ref{thm:type2-Poisson} as well, it is easy to check that \eqref{ineq:Poisson-length-upp} holds
for some sufficiently large constant $C'$ depending on $\alpha$ only.

\item (Case 2: $\lambda \in [c, \infty)$)  By Lemma \ref{lm:test-rate-CI-length-connection-pois}, it is easy to check $\ell(n, \epsilon, \lambda)  \asymp \sqrt{\lambda}\left(\frac{1}{\sqrt{\log n}}+\frac{1}{\sqrt{\log(1/\epsilon)}}\right) + 1 \asymp \bar{r}(\lambda, \epsilon) \wedge \underline{r}(\lambda, \epsilon) \asymp \bar{r}(\lambda, \epsilon) \vee \underline{r}(\lambda, \epsilon)$. By the same analysis as in Case 2 in the length guarantee proof of Theorem~\ref{thm:bino-upper-no-dis}, we have
	\begin{equation*}
		\begin{split}
			&\sup_Q P_{\epsilon,\lambda,Q}\left(|\wt{\mathrm{CI}}|\geq C' \ell(n, \epsilon, \lambda)  ,\lambda\in\wt{\mathrm{CI}}\right) \\
&  \leq \sup_Q P_{\epsilon,\lambda,Q}\Bigg(\lambda - \frac{C''}{2} \bar{r}(\lambda, \epsilon) \in \wt{\CI} \Bigg)  + \sup_Q P_{\epsilon,\lambda,Q}\Bigg(\lambda + \frac{C''}{2}  \underline{r}(\lambda, \epsilon)   \in \wt{\CI}\Bigg)\\
& \leq \sup_Q P_{\epsilon,\lambda,Q}\Bigg(\lambda - \frac{C''}{2} \bar{r}(\lambda, \epsilon) \in \wt{\CI} \Bigg)  +  \sup_Q P_{\epsilon,\lambda,Q}\Bigg(\psi_{\lambda + \frac{C''}{2}  \underline{r}(\lambda, \epsilon), \epsilon}^-   = 0\Bigg).
		\end{split}
	\end{equation*}	
Notice that $
\sup_Q P_{\epsilon,\lambda,Q}\big(\lambda - \tfrac{C''}{2}\bar{r}(\lambda,\epsilon) \in \wt{\CI}\big)$
equals $0$ if $\lambda - \tfrac{C''}{2}\bar{r}(\lambda,\epsilon) < 0$, and is bounded above by $\sup_Q P_{\epsilon,\lambda,Q}\left(\psi_{\lambda - \tfrac{C''}{2}\bar{r}(\lambda,\epsilon),\epsilon}^+ = 0\right)$ if $\lambda - \tfrac{C''}{2}\bar{r}(\lambda,\epsilon) \geq 0$. When $\lambda - \tfrac{C''}{2}\bar{r}(\lambda,\epsilon) \geq 0$, Theorem \ref{thm:type2-Poisson} implies that we need the following condition to control the Type-2 error of $\psi^+_{\lambda - \frac{C''}{2} \bar{r}(\lambda, \epsilon), \epsilon }$:
	\begin{equation*}
		\begin{split}
				\frac{C''}{2} \bar{r}(\lambda, \epsilon) \geq \bar{r}\left(\lambda - \frac{C''}{2} \bar{r}(\lambda, \epsilon), \epsilon\right).
		\end{split}
	\end{equation*} The condition above holds as long as $C''$ is large by Lemma \ref{lm:r-property-length-pois}. Similarly, by Theorem \ref{thm:type2-Poisson}, to control the Type-2 error of $\psi_{\lambda + \frac{C''}{2} \underline{r}(\lambda, \epsilon), \epsilon}^-$ above, we need 
	\begin{equation*}
		\begin{split}
				(i): \lambda + \frac{C''}{2} \underline{r}(\lambda, \epsilon) \geq 4\left( \frac{10 \log(24/\alpha)}{n} + 3 \epsilon \right) \quad \textnormal{ and }
				\quad (ii):
				\frac{C''}{2} \underline{r}(\lambda, \epsilon) \geq \underline{r}\left(\lambda + \frac{C''}{2} \underline{r}(\lambda, \epsilon), \epsilon\right).
		\end{split}
	\end{equation*} The first condition above holds because $\lambda \geq c$ and $\frac{\log(2 / \alpha)}{n} + \epsilon_{\max}$ is sufficiently small; the second condition above holds as long as $C''$ is large by Lemma \ref{lm:r-property-length-pois}. In summary, as long as $C'$ is large enough to allow $C''$ to be taken sufficiently large, by Theorem \ref{thm:type2-Poisson}, \eqref{ineq:Poisson-length-upp} holds. 
\end{itemize} By plugging \eqref{ineq:Poisson-length-upp} into \eqref{Ineq: Length guarantee-pois}, we have
	\begin{equation*}
		 \sup_QP_{\epsilon,\lambda,Q}\left(|\widetilde{\mathrm{CI}}|\geq C' \ell(n, \epsilon, \lambda)  \right)\leq  \sup_QP_{\epsilon,\lambda,Q}\left(|\wt{\mathrm{CI}}|\geq C' \ell(n, \epsilon, \lambda)  ,\lambda\in\wt{\mathrm{CI}}\right) + \alpha/6 \leq \alpha/2.
	\end{equation*}
This finishes the length guarantee of $\widetilde{\CI}$.

The proof of \eqref{eq:CI-bar-poi-guarantee} is the same as in Part II of the proof of Theorem \ref{thm:upper}. We omit the proof here for simplicity.

\subsubsection{Proof of Theorem \ref{thm:type2-Poisson}}

The proof of Theorem~\ref{thm:type2-Poisson} is similar to that of Theorem~\ref{thm:type2}, we omit most of the details. The simultaneous Type-1 error control of $\psi_{\lambda,\epsilon}^+$ and $\psi_{\lambda,\epsilon}^-$ are directly implied by the simultaneous Type-1 error control of $\phi_{\lambda,\epsilon}^+$ and $\phi_{\lambda,\epsilon}^-$ as $\psi_{\lambda,\epsilon}^+\leq \phi_{\lambda,\epsilon}^+$ and $\psi_{\lambda,\epsilon}^-\leq \phi_{\lambda,\epsilon}^-$. The proof of Type-2 error control of $\psi_{\lambda,\epsilon}^+$ when $\lambda \in [0, \infty)$ and $\psi_{\lambda,\epsilon}^-$ when $\lambda \in [1, \infty)$ is similar to the Type-2 error control proof of $\psi_{p,\epsilon}^+$ when $p \in [0,1-1/m]$ in Theorem~\ref{thm:type2}. 

The proof of Type-2 error control of $\psi_{\lambda,\epsilon}^-$ when $\lambda \in [0, 1)$ is similar to the Type-2 error control proof of $\psi_{p,\epsilon}^+$ when $p \in (1-1/m,1]$ in Theorem~\ref{thm:type2} with the observation that when $\lambda \in [0,1)$, $\phi_{\lambda, \epsilon}^{-}$ is non-decreasing as $\lambda$ increases.

\subsubsection{Proof of Lemma \ref{lm:r-property-length-pois}}
For any $\lambda \in [0, \infty)$ and $C' > 0$ such that $\lambda - C'\overline{r}(\lambda,\epsilon) \geq 0$, it is easy to check that $\overline{r}(\lambda - C' \overline{r}(\lambda ,\epsilon), \epsilon) \leq \overline{r}(\lambda,\epsilon)$ since $\overline{r}(\lambda ,\epsilon)$ is increasing in $\lambda$ for any fixed $\epsilon \in [0, \epsilon_{\max}]$. Thus, setting $C_0 = 1$ yields $\overline{r}(\lambda - C_0 \overline{r}(\lambda ,\epsilon), \epsilon) \leq C_0\overline{r}(\lambda,\epsilon)$ provided that $\lambda - C_0\overline{r}(\lambda,\epsilon) \geq 0$. Note that the same conclusion holds if we replace $C_0$ by any $ C \geq C_0$.

Next, we show that there exists some constant $C_0 > 0 $ such that for any $ C \geq C_0$ and $\epsilon \in [0, \epsilon_{\max}]$,
$$\underline{r}( \lambda + C\underline{r}(\lambda, \epsilon), \epsilon ) \leq C\underline{r}(\lambda, \epsilon), \quad \forall \lambda \in [c, \infty).$$ 

By Lemma \ref{lm:test-rate-CI-length-connection-pois}, it is easy to check that for any $\lambda \in [c, \infty)$,
\begin{equation*}
    \begin{split}
        \underline{r}(\lambda, \epsilon) & \asymp \left(\sqrt{\lambda }\left(\frac{1}{\sqrt{\log n}}+\frac{1}{\sqrt{\log(1/\epsilon)}}\right) + 1\right) \wedge \lambda  \asymp \sqrt{\lambda }\left(\frac{1}{\sqrt{\log n}}+\frac{1}{\sqrt{\log(1/\epsilon)}}\right) + 1.
    \end{split}
\end{equation*}
Thus, for any $C_0 > 1$ and $\lambda \in [c, \infty)$,
\begin{equation*}
\begin{split}
      \underline{r}(\lambda + C_0 \underline{r}(\lambda ,\epsilon)) &\lesssim \sqrt{\lambda + C_0 \underline{r}(\lambda ,\epsilon)}\left(\frac{1}{\sqrt{\log n}}+\frac{1}{\sqrt{\log(1/\epsilon)}}\right) + 1 \\
      & \overset{(a)}\lesssim \sqrt{(C_0 + 1)\lambda }\left(\frac{1}{\sqrt{\log n}}+\frac{1}{\sqrt{\log(1/\epsilon)}}\right) + 1 \lesssim \sqrt{C_0 }\underline{r}(\lambda ,\epsilon),
\end{split}
\end{equation*} where (a) is because $\underline{r}(\lambda ,\epsilon) \lesssim \lambda$. Therefore, as long as $C_0$ is large enough, we have $\underline{r}(\lambda + C_0 \underline{r}(\lambda, \epsilon), \epsilon) \leq C_0  \underline{r}(\lambda, \epsilon)$ and the same conclusion holds if we replace $C_0$ by any $C \geq C_0$. This finishes the proof of this lemma.

\section{Proofs for the Erd\Horig{o}s--R\'{e}nyi Model with Node Contamination} 

This section collects the proofs of Theorem \ref{thm:er-up}, Proposition \ref{prop:cd-lower} and Theorem \ref{thm:er-low}.

\subsection{Proof of Theorem \ref{thm:er-up}} \label{thm:er-up proof}
We begin by stating the following lemma, which describes key properties of $\|\cdot\|_{\mathcal{U}}$ defined by \eqref{def: norm}. Its proof is provided in the subsections. 
\begin{Lemma} \label{lm:norm property}
   For any $n \geq 1$, $\|\cdot \|_{\mathcal{U}}$ defined by (\ref{def: norm}) satisfies
   \begin{itemize}
		\item (i) For any $B, B' \in \bbR^{n \times n}$, $\|B + B'\|_{\mathcal{U}} \leq \|B\|_{\mathcal{U}}  + \|B'\|_{\mathcal{U}} $;
            \item (ii) For any $B \in \bbR^{n \times n}$ and $c \in \bbR$, $\|cB \|_{\mathcal{U}} = |c|\|B\|_{\mathcal{U}}$;
            \item (iii) For any $B \in \bbR^{n \times n}$, $\|B\|_{\mathcal{U}} \geq 0$;
		\item (iv) For any $B \in \bbR^{n \times n}$ and $S' \subseteq S \subseteq [n]$,  $\|B_{S'\times S'}\|_{\mathcal{U}} \leq \|B_{S\times S}\|_{\mathcal{U}}$;
        \item (v) For any $c \in \bbR$ and $S \subseteq[n]$, $\|(cJ)_{S \times S}\|_{\mathcal{U}} = |c||S|(|S|-1)$.
    \end{itemize}
    By (i), (ii), and (iii), $\|\cdot \|_{\mathcal{U}}$ defines a seminorm.
\end{Lemma}

As \eqref{Ineq: er-est} holds trivially for $n$ of constant order, we only need to consider the case of sufficiently large $n$. Fix $p \in [0,1]$, and $P\in \mathcal{G}(n, p, \epsilon_{\max})$. Suppose $A \sim P$ and let $G \subseteq [n]$ be the set of uncontaminated nodes under $A$. By Bernstein's inequality (see Lemma \ref{lm:Binomial-prop} (ii)), we have with probability at least $1 - \alpha/3$, the following event holds:
\begin{equation*}
    (A) = \left\{ |G| \geq \frac{3n}{4} \right\}.
\end{equation*}
For simplicity, define $\wh{p}_S$ for any $S \subseteq [n]$ as 
\begin{equation*}
    \wh{p}_{S} = \frac{1}{\#S(\#S - 1)}\sum_{i\in S}\sum_{j\in S}A_{ij}.
\end{equation*}
Then, given $(A)$ happens, we have
\begin{equation}\label{ineq:SBM upper bound basic ineq}
    \begin{split}
   &\frac{n}{2}\Big(\frac{n}2{} - 1\Big)  |\wh{p} - p|    \overset{(a)}\leq    |\wh{S} \cap G|(|\wh{S} \cap G| - 1)  |\wh{p} - p|\overset{\textnormal{Lemma }\ref{lm:norm property}\,(v)} = \|((\wh{p} - p)J)_{\wh{S} \cap G \times \wh{S} \cap G} \|_{\mathcal{U}} \\
        & \overset{\textnormal{Lemma }\ref{lm:norm property}\,(i)}\leq \|(A - \wh{p}J)_{\wh{S} \cap G \times \wh{S} \cap G} \|_{\mathcal{U}} + \|(A - pJ)_{\wh{S} \cap G \times \wh{S} \cap G} \|_{\mathcal{U}} \\
        & \overset{\textnormal{Lemma }\ref{lm:norm property}\,(iv)}\leq \|(A - \wh{p}J)_{\wh{S} \times \wh{S} } \|_{\mathcal{U}} + \|(A - pJ)_{G \times  G} \|_{\mathcal{U}} \\
        & \overset{(b)} =\|(A - \wh{p}_{\wh{S}} J)_{\wh{S} \times \wh{S} } \|_{\mathcal{U}} + \|(A - pJ)_{G \times G} \|_{\mathcal{U}}  \overset{(c)} \leq \|(A - \wh{p}_{G} J)_{G \times G} \|_{\mathcal{U}} + \|(A - pJ)_{G \times  G} \|_{\mathcal{U}} \\
        & \overset{\textnormal{Lemma }\ref{lm:norm property}\,(i)} \leq 2\|(A - p J)_{G \times G} \|_{\mathcal{U}} + \|((p - \wh{p}_{G})J)_{G \times  G} \|_{\mathcal{U}} \\
        & \overset{\textnormal{Lemma }\ref{lm:norm property}\,(v)} = 2\|(A - p J)_{G \times G} \|_{\mathcal{U}} + |G|(|G|-1)|p - \wh{p}_{G}|
    \end{split}
\end{equation}
where in (a) we use the fact that $|\wh{S} \cap G| \geq \frac{n}{2}$, since $|\wh{S}|, |G| \geq \frac{3n}{4}$ under event $(A)$; (b) is by the definition of $\wh{p}$; (c) is by the definition of $\wh{S}$ in \eqref{eq:find-whS} together with the fact that $|G| \geq \frac{3n}{4}$ when $(A)$ occurs. 

We next derive upper bounds on the two terms on the right-hand side of \eqref{ineq:SBM upper bound basic ineq} separately, each of which holds with probability at least $1 - \alpha/3$.

First, let $t = \frac{16\log(6\cdot2^n/\alpha)}{3} + 8\sqrt{p(1-p) \left(\frac{n(n-1)}{2}\right)\log(6\cdot2^n/\alpha)}$. Note that for all $(i, j) \in G \times G$ with $i < j$, the entries $A_{ij}$ are independent and identically distributed according to $\mathrm{Bernoulli}(p)$. Then 
\begin{equation*}
    \begin{split}
        &\bbP(2\|(A - p J)_{G \times G} \|_{\mathcal{U}} \geq t)  = \bbP\left(4\sup_{S \subseteq [n]}\left|\sum_{(i, j) \in S\cap G \times S\cap G:i<j }(A_{ij} - p) \right|  \geq t\right) \\
        & \leq \sum_{S \subseteq [n]} \bbP\left(\left|\sum_{(i, j) \in S\cap G \times S\cap G:i<j }(A_{ij} - p) \right| \geq \frac{t}{4}\right) \leq \sum_{S \subseteq [n]} \frac{\alpha}{3\cdot 2^{n}} =\alpha /3,
    \end{split}
\end{equation*}
where in the first inequality we use the union bound and in the last inequality we use Bernstein's inequality (see Lemma \ref{lm:Binomial-prop} (ii)). Hence, for some large constant $C_1>0$ only depending on $\alpha$, we have with probability at least $1 - \alpha/3$, the following event holds:
$$
(B_1) = \left\{2\|(A - p J)_{G \times G} \|_{\mathcal{U}} \leq C_1\left(n + \sqrt{p(1-p)n^3}\right)\right\}.
$$ In addition, by Bernstein's inequality again, we obtain that for some large constant $C_2>0$ only depending on $\alpha$, the following event holds with probability at least $1 - \alpha/3$:
$$
(B_2) = \left\{|G|(|G| - 1)|p - \wh{p}_{G}| \leq C_2 \left(1 + \sqrt{p(1-p)n^2}\right) \right\}.
$$
Thus, by \eqref{ineq:SBM upper bound basic ineq}, when the events $(A)$, $(B_1)$, and $(B_2)$ occur simultaneously, we have
\begin{equation*}
 |\wh{p} - p| \lesssim \sqrt{\frac{p(1-p)}{n}} + \frac{1}{n}.
\end{equation*}
 Notice that $(A)$, $(B_1)$, and $(B_2)$ happen simultaneously with probability at least $1 - \alpha$ by the union bound. Also, the above arguments hold for any $p \in [0,1]$ and any $P \in \mathcal{G}(n, p, \epsilon_{\max})$. Therefore, we have shown that \eqref{Ineq: er-est} holds.

Next, by a similar analysis to that in the proof of Proposition~\ref{prop:bench}, where we defined a confidence interval using the estimator as in \eqref{Def: CI using estimator} and proved its coverage and length guarantees, it is easy to check that the interval~\eqref{eq:er-conser} also satisfies the coverage and length guarantees. This finishes the proof of this theorem.

\subsubsection{Proof of Lemma \ref{lm:norm property}}
The proofs of (i), (ii), and (iii) are straightforward so we omit them. Now, fix $B \in \mathbb{R}^{n \times n}$ and $S' \subseteq S \subseteq [n]$. Then, there exists $U' \in \mathcal{U}$ such that $\|B_{S' \times S'}\|_{\mathcal{U}} = 
|\iprod{B_{S' \times S'}}{U'}|$. It is easy to check that $$\iprod{B_{S' \times S'}}{U'} = \iprod{B_{S' \times S'}}{U'_{S' \times S'}} = \iprod{B_{S \times S}}{U'_{S' \times S'}}.$$
Thus, we have
$$
\|B_{S' \times S'}\|_{\mathcal{U}}= |\iprod{B_{S \times S}}{U'_{S' \times S'}}|\leq \sup_{U \in \mathcal{U}} |\iprod{B_{S \times S}}{U}| = \|B_{S \times S}\|_{\mathcal{U}}.
$$ 
Finally, for any $c \in \bbR$ and $S \subseteq [n]$, we have
\begin{equation*}
    \|(cJ)_{S \times S}\|_{\mathcal{U}} \overset{\text{Claim }(ii)}= |c|\|J_{S \times S}\|_{\mathcal{U}} = |c||S|(|S|-1),
\end{equation*}
where the second equality is trivial by the definition of $\|\cdot\|_{\mathcal{U}}$. This finishes the proof of this lemma.

\subsection{Proof of Proposition \ref{prop:cd-lower}}
 The proof of Proposition \ref{prop:cd-lower} is a simplification of the proof of Theorem 3.3 in \cite{jin2021optimal}, adapted to our setting. It suffices to consider the case $p \in [1/n, 1/2]$. Indeed, once the claim is established for $p \in [1/n, 1/2]$, the case $p \in [1/2, 1 - 1/n]$ follows by symmetry. 
Now, let
$P = \mathcal{G}(n, p, 0)$ and $Q = \mathrm{SBM}\left(n, p + r, p + \left( \frac{1 - \epsilon_{\max}}{\epsilon_{\max}} \right)^2 r, p - \frac{1 - \epsilon_{\max}}{\epsilon_{\max}} r,  \epsilon_{\max} \right)$ for simplicity. For any $p \in [1/n, 1/2]$ and $ r \leq c \left(\sqrt{\frac{p(1-p)}{n}} + \frac{1}{n}\right)$, it is easy to check that $Q$ is a valid distribution as long as $c$ is sufficiently small. Now, choose $c > 0$ sufficiently small so that $c \leq \frac{\epsilon_{\max}\alpha}{2 + \sqrt{2}}$. Then, 
\begin{equation} \label{Ineq: SBM r upper bound}
    \begin{split}
        r & \leq c \left(\sqrt{\frac{p(1-p)}{n}} + \frac{1}{n}\right) \overset{(a)}\leq  c\left(\sqrt{\frac{p(1-p)}{n}} + \sqrt{\frac{p}{n}}\right) \\
    & \overset{(b)}\leq c\left(\sqrt{\frac{p(1-p)}{n}} + \sqrt{\frac{2p(1 - p)}{n}}\right) =  (1 +\sqrt{2})c\sqrt{\frac{p(1-p)}{n}}  \overset{(c)} \leq \frac{ \epsilon_{\max}\alpha}{\sqrt{2}}\sqrt{\frac{p (1 - p)}{n}},
    \end{split}
\end{equation}
where (a) holds for all $p \geq 1/n$; (b) holds for all $p \leq 1/2$; (c) holds since $c \leq \frac{\epsilon_{\max}\alpha}{2 + \sqrt{2}}$.

Next, we will show that when $p \in [1/n, 1/2]$ and $r \leq c\left(\sqrt{\frac{p(1-p)}{n}} + \frac{1}{n}\right)$, then $\TV(P, Q) \leq \alpha$. We just need to show
    $$\chi^2 \left(Q \middle\|P  \right) = \int \left[\frac{dQ}{dP}\right]^2dP - 1\leq 2\alpha^2,$$
since $\TV(P, Q) \leq \sqrt{\frac{1}{2}\chi^2(Q\| P)}$ for any distributions $P$ and $Q$. Let $Z = (Z_1, \ldots, Z_n)$ be a random vector whose components $Z_i$ are i.i.d. random variables satisfying $P(Z_i=\epsilon_{\max})=1-\epsilon_{\max}$ and $P(Z_i=-(1-\epsilon_{\max}))=\epsilon_{\max}$ for all $i \in [n]$. Let 
$$q_{ij}(Z) = p + \frac{ Z_i Z_j}{ (\epsilon_{\max})^2}r,$$
for all $1 \leq i < j \leq n$. Conditioned on $Z$, consider the random adjacency matrix $A \in \{0,1\}^{n \times n}$ whose upper-triangular entries $\{A_{ij} :1 \leq  i < j \leq n \}$ are conditionally independent and distributed as $A_{ij} \mid Z \sim \mathrm{Bernoulli}(q_{ij}(Z))$. Then, it is easy to check that the marginal distribution of $A$ is $Q$. Let $\tilde{Z}=(\tilde{Z_1},\dots, \tilde{Z_n})$ be the independent copy of $Z$. Then, 
  \begin{equation*}
        \begin{split}
            &\int \left[\frac{dQ}{dP}\right]^2dP  = \mathbb{E}\left[\prod_{ 1 \leq i < j \leq n} 
            \sum_{a_{ij} \in \{0, 1\}} \frac{\left(q_{ij}(Z)^{a_{ij}}(1-q_{ij}(Z))^{1-a_{ij}}\right) \left( q_{ij}(\tilde{Z})^{a_{ij}}(1-q_{ij}(\tilde{Z}))^{1-a_{ij}} \right)}{p^{a_{ij}}(1-p)^{1-a_{ij}}} \right]\\
            & = \mathbb{E}\left[\prod_{ 1 \leq i < j \leq n} \left\{\frac{q_{ij}(Z)q_{ij}(\tilde{Z})}{p}+\frac{(1-q_{ij}(Z))(1-q_{ij}(\tilde{Z}))}{1-p} \right\}  \right] \\
            & =  \mathbb{E}\left[\prod_{ 1 \leq i < j \leq n} \left\{1+\frac{(q_{ij}(Z)-p)(q_{ij}(\tilde{Z})-p)}{p(1-p)}\right\}  \right] = \mathbb{E}\left[\prod_{ 1 \leq i < j \leq n} \left\{1+\frac{Z_iZ_j\tilde{Z_i}\tilde{Z_j}}{p(1-p)(\epsilon_{\max})^2}r^2\right\}  \right].
        \end{split}
    \end{equation*}
Therefore, if we let $S=\frac{1}{p(1-p)(\epsilon_{\max})^2}r^2$, we have
        \begin{equation}\label{chi-squared divergence}
        \begin{split}
            &\int \left[\frac{dQ}{dP}\right]^2dP = 
             \mathbb{E}\left[\prod_{ 1 \leq i < j \leq n} \left\{1+S Z_iZ_j\tilde{Z_i}\tilde{Z_j}\right\}  \right] \\
             & \leq \mathbb{E}\left[\exp\left(S\sum_{ 1 \leq i < j \leq n}Z_iZ_j\tilde{Z_i}\tilde{Z_j} \right) \right] = \mathbb{E}\left[\exp\left(\frac{S}{2}\left(\Big[\sum_{i =1 }^n Z_i\tilde{Z_i}\Big]^2-\sum_{i = 1}^n Z_i^2\tilde{Z_i}^2\right) \right) \right]\\
             &\leq \mathbb{E}\left[\exp\left(\frac{S}{2}\Big[\sum_{i=1}^n Z_i\tilde{Z_i}\Big]^2 \right) \right]= 1+\int_{0}^{\infty}e^t \mathbb{P}\bigg(S\Big[\sum_{i = 1}^n Z_i\tilde{Z_i}\Big]^2>2t\bigg)dt,
        \end{split}
    \end{equation}
    where in the first inequality we use the fact that $1+x \leq e^x$ for all $x \in \mathbb{R}$. Notice that $\sum_{i=1}^n Z_i\tilde{Z_i}$ is sum of independent, mean-zero random variables with $|Z_i\tilde{Z_i}|\leq1$ for all $i \in [n]$. Therefore, by Hoeffding's inequality, we have 
    \begin{equation}\label{Hoeffding's inequality}
        \mathbb{P}\left(\Big|\sum_{i=1}^n Z_i\tilde{Z_i}\Big|>\sqrt{\frac{2t}{S}}\right) \leq 2\exp\left(-\frac{t}{nS}\right),
    \end{equation}
for any $t > 0$. By \eqref{Ineq: SBM r upper bound}, $nS =\frac{nr^2}{p(1-p)(\epsilon_{\max})^2} \leq \alpha^2 /2 \leq 1/2$. Combining \eqref{chi-squared divergence} and \eqref{Hoeffding's inequality}, it follows that
    \begin{equation*}
        \begin{split}
            \int \left[\frac{dQ}{dP}\right]^2dP &\leq 1+2\int_{0}^{\infty}e^{\big(1-\frac{1}{nS}\big)t} dt  = 1+\frac{2nS}{1-nS} \quad (\text{as } nS < 1)\\
            & \leq 1 + 4nS \quad (\text{as } nS \leq 1/2)\\
            & \leq 1 + 2\alpha^2 \quad (\text{as } nS \leq \alpha^2 /2).
        \end{split}
    \end{equation*}
Therefore, we have
    $
            \chi^2(Q \| P) = \int \left[\frac{dQ}{dP}\right]^2dP - 1 \leq 2\alpha^2.
$
This finishes the proof of this proposition.

\subsection{Proof of Theorem \ref{thm:er-low}}

The proof of Theorem~\ref{thm:er-low} is similar to that of Theorem~\ref{thm:lower}, which follows directly from Lemma~\ref{Lem: TV to lower bound} and Theorem~\ref{thm:test-low}. In our setting, Theorem~\ref{thm:test-low} is replaced by Proposition~\ref{prop:cd-lower}. Also the following lemma serves as the analogue of Lemma~\ref{Lem: TV to lower bound}, and its proof is similar and thus we omit the proof here. 

\begin{Lemma}\label{lm: TV to lower bound SBM}
    For any $\alpha \in (0, 1/4)$, $\epsilon, \epsilon_{\max} \in [0, 1]$ with $\epsilon \leq \epsilon_{\max}$, and $p,q \in [0,1]$ satisfying $
        \inf_{P \in \mathcal{G}(n, p, \epsilon),Q \in \mathcal{G}(n, q, \epsilon_{\max})} \TV\left( P, Q\right)   \leq \alpha$,
    we have $r_{\alpha}^{\rm ER}(\epsilon,p,\epsilon_{\max}) > |p - q|$.
\end{Lemma}
This finishes the proof of this theorem.

\section{Additional Technical Lemmas}

The next lemma collects a few properties regarding $\bar{r}(p, \epsilon)$ defined in \eqref{def:r}.
\begin{Lemma} \label{lm:additional-r-property}
    Denote $A := \epsilon + \sqrt{\frac{\log(24/\alpha)}{2n}}$. Suppose $\epsilon_{\max} +  \sqrt{\frac{\log(24/\alpha)}{2n}}$ is less than a sufficiently small constant. Then for any $\epsilon \in [0, \epsilon_{\max}]$,
        \begin{itemize} 
        \item (i) \begin{equation*}
		\begin{split}
			\bar{r}(p, \epsilon) = \left\{ \begin{array}{ll}
				\frac{1}{2m} \vee 2 \sqrt{ \frac{p(1-p)}{m \log(1/A)} } &  p \in [0,1-1/m]\\
				(1- 1/(6e) )(1-p) &  p \in (1-1/m, 1].
			\end{array}  \right.
		\end{split}
	\end{equation*}
        \item (ii) When $\log(1/A) \geq 4m$, 
        	\begin{equation*}
		\begin{split}
			\bar{r}(p, \epsilon) = \left\{ \begin{array}{ll}
				\frac{1}{2m} &  p \in [0,1-1/m]\\
				(1- 1/(6e) )(1-p) &  p \in (1-1/m, 1].
			\end{array}  \right.
		\end{split}
	\end{equation*}
    \item (iii) When $\log(1/A) < 4m$, let $0\leq p_1 < 1/2 < p_2 \leq 1$ be the solution of the equation $\log(1/A) = 16 m p(1-p)$, then
    	\begin{equation*}
		\begin{split}
			\bar{r}(p, \epsilon) = \left\{ \begin{array}{ll}
				 \frac{1}{2m} &  p \in [0,p_1]\\
				2 \sqrt{ \frac{p(1-p)}{m \log(1/A)} } &  p \in (p_1, p_2) \\
				\frac{1}{2m} &  p \in [p_2,1-1/m] \\
				(1- 1/(6e) )(1-p) &  p \in (1-1/m, 1].
			\end{array}  \right.
		\end{split}
	\end{equation*}
    \end{itemize}
\end{Lemma}
\begin{proof} (i) is a direct consequence of plugging the formula of $\bar{t}(p, \epsilon)$ into $\bar{r}(p, \epsilon)$.
    When $\log(1/A) \geq 4m$, then $\frac{1}{2m} \geq  2 \sqrt{ \frac{p(1-p)}{m \log(1/A)} }$ for any $p \in [0,1]$, it is easy to verify the formula for $\bar{r}(p, \epsilon)$ from (i). When $\log(1/A) < 4m$, notice that when $p = 1/m$ or $1-1/m$, $16 m p(1-p) \leq 16 \leq \log(1/A)$ as $A$ is less than a sufficiently small constant. As a result, we have $1/m \leq p_1 \leq 1/2 \leq p_2 \leq 1-1/m$. It is also easy to verify the formula for $\bar{r}(p, \epsilon)$ from (i).
\end{proof}

The following is the Dvoretzky-Kiefer-Wolfowitz-Massart inequality (DKW inequality) from \cite{massart1990tight}.
\begin{Lemma}\label{lm:DKW}
	Suppose $n$ is a positive integer. Let $X_1, \ldots, X_n$ be i.i.d. real-valued random variables drawn from a distribution with CDF $F(\cdot)$. Then for any $\alpha \in (0,1)$, we have
	\begin{equation*}
		\bbP\left( \sup_{t \in \bbR} |F_n(t) - F(t)| > \sqrt{ \frac{\log(2/\alpha)}{2n} } \right) \leq \alpha.
	\end{equation*}
\end{Lemma}

The following two lemmas provide Chernoff bounds for the concentration of the binomial and Poisson distributions. They can be derived by a standard Chernoff bound argument, e.g., see \cite{hoeffding1963probability} and \cite{wainwright2019high}[Chapter 2].
\begin{Lemma} \label{Lem: Chernoff bound}
Suppose $p \in (0,1)$, $m \in \mathbb{N}$, and $k \in [0, mp]$. Then, we have
\begin{equation*}
    \bbP(\mathrm{Binomial}(m,p) \leq k)\leq \exp\left(-mD\left(\mathrm{Bernoulli}\left(\frac{k}{m}\right) \parallel \mathrm{Bernoulli}(p) \right)\right),
\end{equation*}
for all $t \in (0,1)$ and $p \in (0,1)$.
\end{Lemma}

\begin{Lemma} \label{Lem: Chernoff bound for Poisson}
Suppose $\lambda >0$. Then, we have
\begin{align*}
   & \bbP(\mathrm{Poisson}(\lambda) \leq x) \leq \frac{\exp(x - \lambda) \lambda^x}{x^x}, \quad \forall x <\lambda, \\
   & \bbP(\mathrm{Poisson}(\lambda) \geq x) \leq \frac{\exp(x - \lambda) \lambda^x}{x^x}, \quad \forall x > \lambda. 
\end{align*}
\end{Lemma}

The following lemma provides a standard Berry-Esseen bound.
\begin{Lemma}[\cite{shevtsova2011absolute}]\label{Lem: Berry esseen}
Let $X,\cdots, X_n$ be i.i.d. real-valued random variables with $\mathbb{E}(X)=0$, $\mathbb{E}(X^2)=\sigma^2>0$, and $\mathbb{E}(|X|^3)=\rho <\infty$. Let $F^{(n)}(\cdot)$ be the CDF of \begin{equation*}
    Y_n=\frac{\sum\limits_{i=1}^n X_i}{\sqrt{n}\sigma},
\end{equation*}
    and $\Phi(\cdot)$ denote the CDF of standard Gaussian. Then, for all $n\in \mathbb{N}$, we have
    \begin{equation*}
        \underset{x\in \mathbb{R}}{\sup}|F^{(n)}(x)-\Phi(x)|\leq \frac{7\rho}{20\sigma^3\sqrt{n}} +\frac{1}{6\sqrt{n}}.
    \end{equation*}
\end{Lemma}

The following is an application of the Berry–Esseen bound to the Poisson distribution.

\begin{Lemma}\label{Lem: Berry esseen for Poisson}
 Suppose $\lambda  > 0$. Let $\Phi(\cdot)$ denote the CDF of the standard Gaussian. Then, we have
    \begin{equation}\label{Ineq: Berry esseen for Poisson 1}
        \underset{x\in \mathbb{R}}{\sup} | \bbP (\mathrm{Poisson}(\lambda) \leq \lambda + \sqrt{\lambda}x) - \Phi(x) | \leq \frac{7}{20\sqrt{\lambda}}
    \end{equation}
    and 
    \begin{equation}\label{Ineq: Berry esseen for Poisson 2}
        \underset{x\in \mathbb{R}}{\sup} | \bbP (\mathrm{Poisson}(\lambda) \geq \lambda + \sqrt{\lambda}x) - \Phi(-x) | \leq \frac{7}{20\sqrt{\lambda}}.
    \end{equation}
\end{Lemma}

\begin{proof}
   We will only prove inequality \eqref{Ineq: Berry esseen for Poisson 1}. The proof for inequality \eqref{Ineq: Berry esseen for Poisson 2} is analogous, as it essentially involves applying the Berry-Esseen theorem with the direction of the inequality reversed. Consider an arbitrary $n \in \bbN$ and suppose $Z_1, \cdots, Z_n \overset{i.i.d.}\sim \mathrm{Poisson}(\lambda / n)$. Notice that $\bbE(Z_1 - \lambda / n) = 0$ and $\bbE((Z_1 - \lambda / n)^2) = \lambda / n$. Also, we have
    \begin{equation*}
        \begin{split}
            \bbE (|Z_1 - \lambda / n|^3) & \leq \bbE ((Z_1 + \lambda / n)^3) = \frac{\lambda}{n} + \frac{6 \lambda^2}{n^2} + \frac{8 \lambda^3}{n^3}.
        \end{split}
    \end{equation*}
    By the property of the Poisson distribution, we have $\sum_{i=1}^n Z_i \sim \mathrm{Poisson}(\lambda)$. Therefore, we have
    \begin{equation}\label{Eq: Berry esseen in Poisson}
        \begin{split}
            \underset{x\in \mathbb{R}}{\sup} | \bbP (\mathrm{Poisson}(\lambda) \leq \lambda + \sqrt{\lambda}x) - \Phi(x) | & =  \underset{x\in \mathbb{R}}{\sup} \left| \bbP \left(\sum_{i=1}^n (Z_i - \lambda/n)  \leq \sqrt{\lambda} x \right) - \Phi(x) \right| \\
             & \overset{\textnormal{Lemma } \ref{Lem: Berry esseen} } {\leq}  \frac{7}{20\sqrt{\lambda}} \left(1 + \frac{6 \lambda}{n} + \frac{8 \lambda^2}{n^2} \right) + \frac{1}{6\sqrt{n}}.
        \end{split}
    \end{equation}
    Notice that \eqref{Eq: Berry esseen in Poisson} holds for all $n \in \bbN$. Therefore, taking the limit as $n \to \infty$ completes the proof.
    
\end{proof}

The following lemma states a few standard properties of the binomial distribution, and we omit the proof for simplicity. 
\begin{Lemma} \label{lm:Binomial-prop} Given any positive integer $n$, $\alpha \in (0,1)$ and $0\leq p \leq 1$,
\begin{itemize}
	\item (i) (Hoeffding's inequality) 
	\begin{equation*}
		\bbP\left(|\textnormal{Binomial}(n,p) - np| \geq \sqrt{ \frac{n\log(2/\alpha)}{2} } \right) \leq \alpha.
	\end{equation*} 
	\item (ii) (Bernstein's inequality)
	\begin{equation*}
		\bbP\left(|\textnormal{Binomial}(n,p) - np| \geq \frac{4\log(2/\alpha) }{3}  \vee 2\sqrt{ p(1-p)n\log(2/\alpha) } \right) \leq \alpha.
	\end{equation*} 
    Since $ \frac{np(1-p)}{C} + C \log(2/\alpha) \geq \frac{4\log(2/\alpha) }{3}  \vee 2\sqrt{ p(1-p)n\log(2/\alpha) }  $ holds for any $C \geq 4/3$, the inequality above holds if we replace $\frac{4\log(2/\alpha) }{3}  \vee 2\sqrt{ p(1-p)n\log(2/\alpha) } $ by $ \frac{np(1-p)}{C} + C \log(2/\alpha) $ for any $C \geq 4/3$.

	\item (iii) (see \cite{roch2024modern} Example 4.2.4) Given any other $1\geq p' \geq p$, we have $P( \textnormal{Binomial}(n,p) \leq t ) \geq P( \textnormal{Binomial}(n,p') \leq t )$ for any $t \in [0,n]$.

\end{itemize}
\end{Lemma}

The following lemma states a standard monotonicity property of the Poisson distribution (see \cite{roch2024modern} Example 4.2.5), and we omit the proof for simplicity.
\begin{Lemma} \label{lm:Poisson-prop}
Given any $\lambda' \geq \lambda \geq 0$ and $t \geq 0$, we have
\[
P( \textnormal{Poisson}(\lambda) \leq t ) \geq P( \textnormal{Poisson}(\lambda') \leq t )  \quad \textnormal{and} \quad
P( \textnormal{Poisson}(\lambda) \geq t ) \leq P( \textnormal{Poisson}(\lambda') \geq t ).
\]
\end{Lemma}

The following lemma captures a standard monotonicity of total variation distance with respect to the binomial or Poisson parameter.

\begin{Lemma} \label{lm:TV monotonicity}
For any $0\leq p' \leq p \leq 1$, we have
\begin{equation*}
    \TV(\mathrm{Binomial}(m, p), \mathrm{Binomial}(m, p')) \leq     \TV(\mathrm{Binomial}(m, p), \mathrm{Binomial}(m, 0)).
\end{equation*}
Similarly, for any $0 \leq \lambda' \leq \lambda$, we have
\begin{equation*}
    \TV(\mathrm{Poisson}(\lambda), \mathrm{Poisson}(\lambda')) \leq     \TV(\mathrm{Poisson}(\lambda), \mathrm{Poisson}(0)).
\end{equation*}
\end{Lemma}

\begin{proof}
    We begin with the binomial case. Let $\mathcal{I} = \{i \in [m] \cup \{0 \}:  (p^i (1-p)^{m-i} \leq (p')^i (1-p')^{m-i} \}$, then it is easy to check that $\mathcal{I} = \{0, 1, \dots,k\}$ for some $k \in [m] \cup \{0\}$. Then 
	\begin{equation*}
		\begin{split}
			\TV(P_{p},P_{p'}) & = \frac{1}{2} \sum_{i=0}^m {m \choose i} \left| p^i (1-p)^{m-i} - (p')^i (1-p')^{m-i} \right| \\ 
            & = P_{p'}(X \leq k) -P_{p}(X \leq k)  \overset{\text{Lemma } \ref{lm:Binomial-prop} \,(iii)}\leq P_{0}(X \leq k) - P_{p}(X \leq k)  \leq \TV(P_{p},P_{0}).
		\end{split}
	\end{equation*} 
    
We now consider the Poisson case. Let $\mathcal{I} = \{i \in \bbN_0:  \exp(-\lambda )\lambda^{i} \leq \exp(-\lambda') (\lambda')^{i} \}$, then it is easy to check that $\mathcal{I} = \{0, 1, \dots,k\}$ for some $k \in \bbN_0$. Then 
	\begin{equation*} 
		\begin{split}
		 \TV(P_{\lambda},P_{\lambda'}) & = \frac{1}{2}\sum_{i=0}^{\infty}\frac{1}{i!}\left|\exp(-\lambda )\lambda ^{i} - \exp(-\lambda') (\lambda')^{i}\right| \\
& = P_{\lambda'}(X \leq k) - P_{\lambda}(X \leq k)  \overset{\textnormal{Lemma }\ref{lm:Poisson-prop}} \leq P_{0}(X \leq k) - P_{\lambda }(X \leq k)  \leq 		 \TV(P_{\lambda},P_{0})	.	\end{split}
	\end{equation*} 

    This finishes the proof of this lemma.
\end{proof}

\begin{Lemma}\label{lem: TV}
 For any distributions $P_1$ and $P_2$, and any $\epsilon \in [0,1)$ satisfying $
\TV(P_1, P_2) \leq \frac{\epsilon}{1 - \epsilon}$, there exist distributions $Q_1$ and $Q_2$ such that $(1 - \epsilon) P_1 + \epsilon Q_1 = (1 - \epsilon) P_2 + \epsilon Q_2$.
\end{Lemma}
\begin{proof}
See Theorem 5.1 of \cite{chen2018robust}.
\end{proof}

\end{document}